\documentclass[reqno,11pt]{preprint}
\usepackage[marginparwidth=1in]{geometry}
\usepackage[full]{textcomp}
\usepackage[osf]{newtxtext}
\usepackage{cabin}
\usepackage{enumerate}
\usepackage{enumitem}
\usepackage{bbm}
\usepackage{csquotes}
\usepackage{longtable}

\usepackage{circuitikz}

\usepackage[hypcap=true]{caption}
\usepackage[hypcap=true,list=true]{subcaption}
\usepackage{hyperref}

\usepackage[varqu,varl]{inconsolata}
\usepackage[scr=boondoxupr]{mathalfa}

\usepackage{breakurl}

\usepackage{booktabs}

\usepackage{nicefrac}

\usepackage{tikz}
\usepackage{tikz-cd}
\tikzcdset{arrow style=tikz,
	diagrams={->}
} 

\usetikzlibrary{patterns}

\usepackage{amsmath}

\usepackage{cleveref}
\usepackage{mhequ} 
\usepackage{mhsymb} 
\usepackage{mhenvs} 
\usepackage{microtype}
\usepackage{amssymb} 
\usepackage{amsfonts}

\usepackage{algorithm}
\usepackage{algpseudocode}

\usepackage{wasysym}
\usepackage{centernot}

\usepackage{soul}

\usepackage{float}
\usepackage{subfiles}
\usetikzlibrary{plotmarks}
\usetikzlibrary{decorations.pathreplacing}
\usetikzlibrary{decorations.pathmorphing}
\usetikzlibrary{shapes.misc}
\usetikzlibrary{shapes.symbols}
\usetikzlibrary{shapes.geometric}
\usetikzlibrary{decorations}
\usetikzlibrary{decorations.markings}
\usetikzlibrary{arrows.meta}
\usetikzlibrary{calc}
\usetikzlibrary{snakes} 
\usetikzlibrary{intersections} 

\makeatletter
\newcommand{\globalcolor}[1]{%
	\color{#1}\global\let\default@color\current@color
}
\makeatother

\usetikzlibrary{calc}
\usetikzlibrary{decorations}
\usetikzlibrary{positioning}
\usetikzlibrary{shapes}
\usetikzlibrary{external}
\usetikzlibrary{decorations.markings}

\usepackage{commath}

\newtheorem{question}{Question}

\newcommand{\Cdot}{\boldsymbol{\cdot}}

\newcommand{\T}{\mathbb{T}}

\def\minus{\scalebox{0.7}[0.8]{\raisebox{-0.7pt}{-}}}
\newcommand{\simhor}{\sim_{\mathrm{hor}}}

\newcommand{\hor}{\text{-}\mathrm{hor}}
\newcommand{\ax}{\minus\mathrm{ax}}
\newcommand{\gr}{\minus\mathrm{gr}}
\newcommand{\tri}{\minus\mathrm{tr}}

\newcommand{\ver}{\mathrm{ver}}

\newcommand{\bsg}{{\boldsymbol{g}}}
\newcommand{\Sym}{\operatorname{Sym}}
\newcommand{\Axial}{\fA^{\mathrm{ax}}}
\newcommand{\Caxial}{\fA^{\mathrm{ax},0}}
\newcommand{\Span}{\operatorname{span}}
\newcommand{\lc}{{\scalebox{1.2}{$c$}\hspace{0.1ex}}}
\newcommand{\coul}{\mathrm{coul}}

\newcommand{\hov}{\mathrm{hov}}
\newcommand{\out}{\mathrm{out}}

\newcommand{\ot}{\leftarrow}
\newcommand{\Cas}{\mathrm{Cas}}
\newcommand{\dist}{\operatorname{dist}}
\newcommand{\floor}[1]{\left\lfloor #1 \right\rfloor}

\newcommand{\Dir}{{\scalebox{0.6}{$\sf{D}$}}}
\newcommand{\Free}{%
  {\raisebox{-0.3pt}{\scalebox{0.6}{$\bR^2$}}}%
}

\newcommand{\ext}{\textnormal{ext}}

\newcommand{\all}{\textnormal{all}}
\newcommand{\diag}{\operatorname{diag}}

\renewcommand{\dim}{\mathrm{dim}\,}

\newcommand{\patch}{{\mathrm{patch}}}
\newcommand{\plush}{\scalebox{0.65}{+}\scalebox{0.5}{\ensuremath{(1,0)}}}

\newcommand{\minush}{\minus\scalebox{0.5}{\ensuremath{(1,0)}}}
\newcommand{\plusv}{\scalebox{0.65}{+}\scalebox{0.5}{\ensuremath{(0,1)}}}
\newcommand{\plushv}{\scalebox{0.65}{+}\scalebox{0.5}{\ensuremath{(1,1)}}}

\newcommand{\Lambdaex}{\Lambda^{\!+1}}

\newcommand{\rlog}{\rmr_{\mathrm{log}}}

\newcommand{\1}{\mathbf{1}}
\newcommand{\bone}{\blue{\1_G}}

\newcommand{\diff}{\mathrm{d}}

\newcommand{\sol}{\mathrm{sol}}
\newcommand{\Int}{\mathrm{int}}
\newcommand{\alg}{\mathrm{alg}}
\newcommand{\lin}{\mathrm{lin}}

\newcommand{\mcF}{\mathcal{F}}

\newcommand{\mcQ}{\mathcal{Q}}

\newcommand{\mfg}{\mathfrak{g}}
\newcommand{\mfK}{\mathfrak{K}}

\newcommand{\mfG}{\mathfrak{G}}

\newcommand{\fg}{\mathfrak{g}}

\newcommand{\s}{\mathfrak{s}}

\newcommand{\restr}{\mathord{|}}

\renewcommand{\R}{\mathbb{R}}
\renewcommand{\C}{\mathbb{C}}
\renewcommand{\N}{\mathbb{N}}

\renewcommand{\T}{\mathbb{T}}
\renewcommand{\Z}{\mathbb{Z}}
\renewcommand{\E}{\mathbb{E}}
\renewcommand{\P}{\mathbb{P}}

\newcommand{\da}{\downarrow}
\newcommand{\ua}{\uparrow}

\newcommand{\h}{\mathbf{h}}

\newcommand{\YM}{\mathrm{YM}}

\newcommand{\hol}{\textnormal{hol}}
\newcommand{\Area}{\mathrm{Area}}

\def\bfOmega{{\boldsymbol{\Omega}}}

\newcommand{\Hol}[1]{#1\textnormal{-H{\"o}l}}

\DeclareFontFamily{U}{matha}{\hyphenchar\font45}
\DeclareFontShape{U}{matha}{m}{n}{
	<5> <6> <7> <8> <9> <10> gen * matha
	<10.95> matha10 <12> <14.4> <17.28> <20.74> <24.88> matha12
}{}
\DeclareSymbolFont{matha}{U}{matha}{m}{n}
\DeclareFontSubstitution{U}{matha}{m}{n}
\DeclareFontFamily{U}{mathx}{\hyphenchar\font45}
\DeclareFontShape{U}{mathx}{m}{n}{
	<5> <6> <7> <8> <9> <10>
	<10.95> <12> <14.4> <17.28> <20.74> <24.88>
	mathx10
}{}
\DeclareSymbolFont{mathx}{U}{mathx}{m}{n}
\DeclareFontSubstitution{U}{mathx}{m}{n}
\DeclareMathDelimiter{\vvvert}{0}{matha}{"7E}{mathx}{"17}

\def\triple#1{{\vvvert #1 \vvvert}}

\newcommand{\id}{\operatorname{id}}

\def\nfrac#1#2{\nicefrac{#1}{#2}}
\newcommand{\Ad}{\mathrm{Ad}}

\definecolor{pagebackground}{rgb}{1,1,1}
\colorlet{symbols}{blue!70!black!70}
\colorlet{connection}{red!30!black}

\pagecolor{pagebackground}

\colorlet{darkblue}{blue!90!black}
\colorlet{darkred}{red!90!black}
\colorlet{darkgreen}{green!50!black}
\colorlet{darkbrown}{brown!80!black}

\newcommand{\blue}[1]{\textcolor{darkblue}{#1}}

\renewcommand{\emptyset}{\varnothing}

\makeatletter
\pgfdeclareshape{crosscircle}
{
	\inheritsavedanchors[from=circle] 
	\inheritanchorborder[from=circle]
	\inheritanchor[from=circle]{north}
	\inheritanchor[from=circle]{north west}
	\inheritanchor[from=circle]{north east}
	\inheritanchor[from=circle]{center}
	\inheritanchor[from=circle]{west}
	\inheritanchor[from=circle]{east}
	\inheritanchor[from=circle]{mid}
	\inheritanchor[from=circle]{mid west}
	\inheritanchor[from=circle]{mid east}
	\inheritanchor[from=circle]{base}
	\inheritanchor[from=circle]{base west}
	\inheritanchor[from=circle]{base east}
	\inheritanchor[from=circle]{south}
	\inheritanchor[from=circle]{south west}
	\inheritanchor[from=circle]{south east}
	\inheritbackgroundpath[from=circle]
	\foregroundpath{
		\centerpoint%
		\pgf@xc=\pgf@x%
		\pgf@yc=\pgf@y%
		\pgfutil@tempdima=\radius%
		\pgfmathsetlength{\pgf@xb}{\pgfkeysvalueof{/pgf/outer xsep}}%
		\pgfmathsetlength{\pgf@yb}{\pgfkeysvalueof{/pgf/outer ysep}}%
		\ifdim\pgf@xb<\pgf@yb%
		\advance\pgfutil@tempdima by-\pgf@yb%
		\else%
		\advance\pgfutil@tempdima by-\pgf@xb%
		\fi%
		\pgfpathmoveto{\pgfpointadd{\pgfqpoint{\pgf@xc}{\pgf@yc}}{\pgfqpoint{-0.707107\pgfutil@tempdima}{-0.707107\pgfutil@tempdima}}}
		\pgfpathlineto{\pgfpointadd{\pgfqpoint{\pgf@xc}{\pgf@yc}}{\pgfqpoint{0.707107\pgfutil@tempdima}{0.707107\pgfutil@tempdima}}}
		\pgfpathmoveto{\pgfpointadd{\pgfqpoint{\pgf@xc}{\pgf@yc}}{\pgfqpoint{-0.707107\pgfutil@tempdima}{0.707107\pgfutil@tempdima}}}
		\pgfpathlineto{\pgfpointadd{\pgfqpoint{\pgf@xc}{\pgf@yc}}{\pgfqpoint{0.707107\pgfutil@tempdima}{-0.707107\pgfutil@tempdima}}}
	}
}
\makeatother

\def\decorate#1#2{
	\ifnum#2>0
	\foreach \count in {1,...,#2}{
		let
		\p1 = (sourcenode.center),
		\p2 = (sourcenode.east),
		\n1 = {\x2-\x1},
		\n2 = {1mm},
		\n3 = {(1.3+0.6*(\count-1))*\n1},
		\n4 = {0.7*\n1}
		in 
		node[rectangle,fill=symbols,rotate=30,inner sep=0pt,minimum width=0.2*\n2,minimum height=\n2] at ($(sourcenode.center) + (\n3,\n4)$) {}
	}
	\fi
	\ifnum#1>0
	\foreach \count in {1,...,#1}{
		let
		\p1 = (sourcenode.center),
		\p2 = (sourcenode.east),
		\n1 = {\x2-\x1},
		\n2 = {1mm},
		\n3 = {(1.3+0.6*(\count-1))*\n1},
		\n4 = {0.7*\n1}
		in 
		node[rectangle,fill=symbols,rotate=-30,inner sep=0pt,minimum width=0.2*\n2,minimum height=\n2] at ($(sourcenode.center) + (-\n3,\n4)$) {}
	}
	\fi
}

\tikzset{
	dectriangle/.style 2 args={
		triangle,
		alias=sourcenode,
		append after command={\decorate{#1}{#2}}
	},
	dectriangle/.default={0}{0},
}

\tikzset{
	cross/.style={path picture={ 
			\draw[black]
			(path picture bounding box.south east) -- (path picture bounding box.north west) (path picture bounding box.south west) -- (path picture bounding box.north east);
	}},
	tinydot/.style={circle,fill=black,inner sep=0pt, minimum size=1mm, scale=0.6},
	dot/.style={circle,fill=symbols,inner sep=0pt, minimum size=1mm},
	tinydiamond/.style={draw=red, diamond, fill=symbols!0, inner sep=0pt, minimum size=1mm,scale=0.6},
	diamonds/.style={diamond, fill=symbols, inner sep=0pt, minimum size=1mm},
	diamondfill/.style={diamond, draw=red, fill=symbols!00, inner sep=0pt, minimum size=1mm,scale=1},
	stars/.style={star, line width=0.3pt, star point height=0.1cm, minimum size=1mm, fill=lbluenode, draw=symbols, scale=0.15},
	stars2/.style={star, line width=0.3pt, star point height=0.5cm, minimum size=1mm, fill=lbluenode, draw=symbols, scale=0.08},
	root/.style={circle,fill=green!50!black,inner sep=0pt, minimum size=1.2mm},
	dot/.style={circle,fill=black,inner sep=0pt, minimum size=0.8mm},
	dotred/.style={circle,fill=pageforeground!50!pagebackground,inner sep=0pt, minimum size=2mm},
	var/.style={circle,fill=pageforeground!10!pagebackground,draw=pageforeground,inner sep=0pt, minimum size=3mm},
	kernel/.style={semithick,shorten >=2pt,shorten <=2pt},
	kernels/.style={snake=zigzag,shorten >=2pt,shorten <=2pt,segment amplitude=1pt,segment length=4pt,line before snake=2pt,line after snake=5pt,},
	rho/.style={densely dashed,semithick,shorten >=2pt,shorten <=2pt},
	testfcn/.style={dotted,semithick,shorten >=2pt,shorten <=2pt},
	renorm/.style={shape=circle,fill=pagebackground,inner sep=1pt},
	labl/.style={shape=rectangle,fill=pagebackground,inner sep=1pt},
	xic/.style={very thin,circle,draw=symbols,fill=symbols,inner sep=0pt,minimum size=1.2mm},
	g/.style={very thin,rectangle,draw=symbols,fill=symbols!10!pagebackground,inner sep=0pt,minimum width=2.5mm,minimum height=1.2mm},
	bxi/.style={very thin,circle,draw=black,fill=symbols!10!pagebackground,inner sep=0pt,minimum size=1.2mm},
	bxicross/.style={very thin,circle,draw=black,fill=black!10!pagebackground,cross,inner sep=0pt,minimum size=1.2mm},
	xi/.style={very thin,circle,draw=symbols,fill=symbols!10!pagebackground,inner sep=0pt,minimum size=1.2mm},
	xi1/.style={very thin,rectangle,fill=symbols!10!pagebackground,draw=symbols,inner sep=0pt,minimum size=1.2mm},
	bxi1/.style={very thin,rectangle,fill=black!10!pagebackground,draw=black,inner sep=0pt,minimum size=1.2mm},
	bxi1cross/.style={very thin,rectangle,fill=black!10!pagebackground,draw=black,cross,inner sep=0pt,minimum size=1.2mm},
	taunode/.style={very thin,star,fill=symbols!10!pagebackground,draw=symbols,inner sep=0pt,minimum size=1.3mm},
	taunodefill/.style={very thin,star,fill=symbols!75!pagebackground,draw=symbols,inner sep=0pt,minimum size=1.3mm},
	xi1cross/.style={very thin,rectangle, cross, fill=symbols!10!pagebackground,draw=symbols,inner sep=0pt,minimum size=1.2mm},
	gxi/.style={very thin,circle,draw=lightgray,fill=lightgray!10!pagebackground,inner sep=0pt,minimum size=1.2mm},
	xies/.style={very thin,rectangle,fill=green!50!black!25,draw=symbols,inner sep=0pt,minimum size=1.1mm},
	xiesf/.style={very thin,rectangle,fill=blue!70!black,draw=symbols,inner sep=0pt,minimum size=1.1mm},
	xix/.style={very thin,crosscircle,fill=symbols!10!pagebackground,draw=symbols,inner sep=0pt,minimum size=1.2mm},
	X/.style={very thin,cross,rectangle,fill=pagebackground,draw=symbols,inner sep=0pt,minimum size=1.2mm},
	xib/.style={thin,circle,fill=symbols!10!pagebackground,draw=symbols,inner sep=0pt,minimum size=1.6mm},
	xie/.style={thin,circle,fill=green!50!black,draw=symbols,inner sep=0pt,minimum size=1.6mm},
	xid/.style={thin,circle,fill=symbols,draw=symbols,inner sep=0pt,minimum size=1.6mm},
	xibx/.style={thin,crosscircle,fill=symbols!10!pagebackground,draw=symbols,inner sep=0pt,minimum size=1.6mm},
	kernels2/.style={very thick,draw=connection,segment length=12pt},
	gkernels2/.style={very thick,draw=lightgray,segment length=12pt},
	keps/.style={thin,draw=symbols,->},
	kepspr/.style={thick,draw=connection,->},
	krho/.style={thin,draw=symbols,superdense,->},
	krhopr/.style={thick,draw=connection,superdense,->},
	triangle/.style = { regular polygon, regular polygon sides=3},
	point/.style={thin,circle,draw=black,fill=black,inner sep=0pt,minimum size=0.4mm},
	not/.style={thin,circle,draw=connection,fill=connection,inner sep=0pt,minimum size=0.5mm},
	gnot/.style={thin,circle,draw=lightgray,fill=lightgray,inner sep=0pt,minimum size=0.5mm},
	diff/.style = {very thin,draw=symbols,triangle,fill=red!50!black,inner sep=0pt,minimum size=1.6mm},
	diff1/.style = {very thin,dectriangle={1}{0},fill=red!50!black,draw=symbols,inner sep=0pt,minimum size=1.6mm},
	diff2/.style = {very thin,dectriangle={1}{1},fill=red!50!black,draw=symbols,inner sep=0pt,minimum size=1.6mm},
	diffmini/.style = {very thin,rectangle,fill=black,draw=black,inner sep=0pt,minimum size=0.75mm},
	kernelsmod/.style={very thick,draw=connection,segment length=12pt},
	rec/.style = {very thin,rectangle,fill=black,draw=black,inner sep=0pt,minimum size=2mm},
	cerc/.style={very thin,circle,draw=black,fill=symbols,inner sep=0pt,minimum size=2mm},
	trup/.style={very thin,regular polygon,regular polygon sides=3,shape border rotate=180,draw=black,fill=red,inner sep=0pt,minimum size=3mm},
	stars/.style={very thin,star,star points=6,star point ratio=0.5, draw=black,fill=red,inner sep=0pt,minimum size=0.7mm},
	>=stealth,
	kernelsgen/.style={draw=symbols,thick,densely dashed, segment length=4pt}, 
	kernelsgen2/.style={draw=symbols, snake=zigzag, segment length=4pt, segment amplitude=0.4mm}, 
	kernelsgen3/.style={draw=symbols, thick, snake=zigzag, segment length=4pt, segment amplitude=0.4mm}, 
	kernels1/.style={draw=symbols, segment length=12pt},
	kernels11/.style={very thick, draw=symbols,segment length=12pt},
	kernels2/.style={draw=symbols,snake=zigzag, segment aspect=0, segment length=1pt, segment amplitude=0.2mm},
	kernels22a/.style={draw=symbols, snake=zigzag, semithick, raise snake=0.3pt, segment aspect=0pt, segment length=1pt, segment amplitude=0.2mm},
	kernels22b/.style={draw=symbols, snake=zigzag, semithick, raise snake=-0.3pt, segment aspect=0pt, segment length=1pt, segment amplitude=0.2mm},
	kernels12r/.style={draw=symbols, snake=snake, very thick, segment aspect=0pt, segment length=3pt, segment amplitude=0.2mm},
	kernels12l/.style={draw=symbols, snake=snake, very thick, segment aspect=0pt, segment length=3pt, segment amplitude=-0.2mm},
	kernels21/.style={snake=snake, very thick, segment aspect=0pt, segment length=3pt, segment amplitude=0.2mm},
	bkernel/.style={draw=black,thick,densely dotted, segment length=4pt}, 
	bkernels1/.style={draw=black, segment length=12pt},
	bkernels11/.style={very thick, draw=black,segment length=12pt},
	bkernels2/.style={draw=black,snake=zigzag, segment aspect=0, segment length=1pt, segment amplitude=0.2mm},
	bkernels22a/.style={draw=black, snake=zigzag, semithick, raise snake=0.3pt, segment aspect=0pt, segment length=1pt, segment amplitude=0.2mm},
	bkernels22b/.style={draw=black, snake=zigzag, semithick, raise snake=-0.3pt, segment aspect=0pt, segment length=1pt, segment amplitude=0.2mm},
	bkernels12r/.style={draw=black, snake=snake, very thick, segment aspect=0pt, segment length=3pt, segment amplitude=0.2mm},
	bkernels12l/.style={draw=black, snake=snake, very thick, segment aspect=0pt, segment length=3pt, segment amplitude=-0.2mm},
	bkernels21/.style={draw=black, snake=snake, very thick, segment aspect=0pt, segment length=3pt, segment amplitude=0.2mm},
}

\makeatletter
\def\DeclareSymbol#1#2#3{%
	\expandafter\gdef\csname MH@symb@#1\endcsname{\tikzsetnextfilename{symbol#1}%
		\tikz[baseline=#2,scale=0.15,draw=symbols,line join=round]{#3}}%
	\expandafter\gdef\csname MH@symb@#1s\endcsname{\scalebox{0.75}{\tikzsetnextfilename{symbol#1}%
			\tikz[baseline=#2,scale=0.15,draw=symbols,line join=round]{#3}}}%
	\expandafter\gdef\csname MH@symb@#1ss\endcsname{\scalebox{0.65}{\tikzsetnextfilename{symbol#1}%
			\tikz[baseline=#2,scale=0.15,draw=symbols,line join=round]{#3}}}%
}
\def\<#1>{\ifthenelse{\boolean{mmode}}{\mathchoice{\csname MH@symb@#1\endcsname}{\csname MH@symb@#1\endcsname}{\csname MH@symb@#1s\endcsname}{\csname MH@symb@#1ss\endcsname}}{\csname MH@symb@#1\endcsname}}
\makeatother

\DeclareSymbol{Qn}{1.4}{\draw[color=black, semithick] (0,0) node[] {} -- (0.5,2)  node[] {} -- (5,2) node[] {} -- (4.5,0) node[label={[xshift=0.1cm,yshift=-0.2cm]above left:{\tiny $n$}}] {} -- (0,0);}

\DeclareSymbol{Qnp}{1.4}{
	\draw[color=black, semithick] (0,0) node[] {} -- (0.5,2)  node[] {} -- (5,2) node[] {} -- (4.5,0) node[label={[xshift=0.3cm,yshift=-0.25cm]above left:{\tiny $n+1$}}] {} -- (0,0);
}

\makeatletter 
\newcommand*{\bigcdot}{}
\DeclareRobustCommand*{\bigcdot}{%
	\mathbin{\mathpalette\bigcdot@{}}%
}
\newcommand*{\bigcdot@scalefactor}{.5}
\newcommand*{\bigcdot@widthfactor}{1.15}
\newcommand*{\bigcdot@}[2]{%
	\sbox0{$#1\vcenter{}$}
	\sbox2{$#1\cdot\m@th$}%
	\hbox to \bigcdot@widthfactor\wd2{%
		\hfil
		\raise\ht0\hbox{%
			\scalebox{\bigcdot@scalefactor}{%
				\lower\ht0\hbox{$#1\bullet\m@th$}%
			}%
		}%
		\hfil
	}%
}
\makeatother

\def\dash{\leavevmode\unskip\kern0.18em--\penalty\exhyphenpenalty\kern0.18em}
\def\slash{\leavevmode\unskip\kern0.15em/\penalty\exhyphenpenalty\kern0.15em}

\begin{document}
	
	\title{A PDE approach to the 2D Yang--Mills measure}
	\author{Ilya Chevyrev$^1$,  Tom Klose$^2$, Abdulwahab Mohamed$^3$}
 
	\institute{SISSA, Trieste, Italy \and
		University of Oxford, Oxford, United Kingdom \and Max Planck Institute for Mathematics in Sciences, Leipzig, Germany\\
		\email{ichevyrev@gmail.com, tom.klose@outlook.com, \\abdulwahab.mohamed@mis.mpg.de}}
	
	\maketitle
	
	\begin{abstract}
	We introduce the concept of rough additive functions, which extends rough paths theory to line integrals of distributional 1-forms. In the context of gauge theory, we use this notion to define controlled gauge transformations and holonomies via RDEs. One of our main results is a rough version of Uhlenbeck compactness for distributional connections based on rough additive functions. The main ingredient is a singular elliptic PDE to obtain a Coulomb gauge. We solve this PDE using regularity structures and a method inspired by the implicit function theorem. Surprisingly, the model needed to solve the equation is determined entirely from rough additive functions, which is a simpler and geometrically more natural object. We apply our results to the Yang--Mills measure on the unit square, showing that it has a gauge-fixed representation with optimal regularity. 
	Although we focus on the unit square in this article, we expect our results to apply to more general surfaces.
	\\[.4em]
	\end{abstract}

	\medskip 
	
	\setcounter{tocdepth}{2}       
	\tableofcontents

     \section{Introduction}
    Let~$G$ be a compact, connected Lie group,~$\mfg$ its associated Lie algebra, and~$\Lambda$ a $d$-manifold (possibly with boundary; we take later $d=2$ and $\Lambda = [0,1]^2$).
	We assume that~$G$ is a subgroup of the unitary group~$\rmU(N)$ for some~$N \geq 1$.

	Formally, the \emph{Yang--Mills measure} on~$\Lambda$ is the probability measure given by
	\begin{equation}
		\mu_{\YM}(\dif A) \propto \exp\del[1]{- \norm[0]{F^A}_{L^2(\Lambda;\mfg)}^2} \dif A
		\label{intro:ym_measure}
	\end{equation}
	where~$A$ is a~$\mfg$-valued $1$-form on~$\Lambda$, also called a \emph{connection form}, 
		\begin{equation}
			F^A = \diff A + [A \wedge A]
			\label{intro:ym_curvature}
		\end{equation}
	the associated~$\mfg$-valued $2$-form called the \emph{curvature} of~$A$, and~$\dif A$ is the \enquote{Lebesgue measure} on~$\CA := \Omega^1(\Lambda;\mfg)$.
	The $L^2$ norm in \eqref{intro:ym_measure} is defined with respect to an $\Ad$-invariant inner product~$\scal{\Cdot,\Cdot}_{\mfg}$ on~$\mfg$, which we fix henceforth.

	In coordinates, we may write $A = \sum_{i=1}^d A_i\dif x_i$ and
	\begin{equation*}
			F^A_{ij}(x) = \partial_i A_j(x) - \partial_j A_i(x) + [A_i(x),A_j(x)], \quad i,j \in \{1,\ldots, d\}.
		\end{equation*}
	Since~$\CA$ is infinite-dimensional, it is not clear how to give precise mathematical sense to~$\mu_\YM$, and in dimensions $d=3,4$, making sense of $\mu_\YM$ is a long-standing open problem, even for compact $\Lambda$ such as the torus $\T^d=\R^d/\Z^d$.
	
	Ignoring this issue for the moment, Yang-Mills theory exhibits a symmetry encoded by~$G$.
	The \emph{gauge group} $\mfG := C^\infty(\Lambda;G)$
	carries an action on $\CA$, given for $g\in \mfG$ and $A\in\CA$ by
		\begin{equation}
			A \mapsto A^g \equiv (g \Cdot A) := gA g^{-1} - (\diff g) g^{-1}
			\label{eq:trafo_conn}
		\end{equation} 
	which also induces an action on the curvature, namely~$F^{A^g}(x) = g(x) F^A(x) g(x)^{-1}$. Since the inner product on $\mfg$ is $\Ad$-invariant, it follows that
		\begin{equation}
			\norm[0]{F^{A^g}}_{L^2(\Lambda;\mfg)} = \norm[0]{F^{A}}_{L^2(\Lambda;\mfg)}\;, \quad g \in \mfG.\; 
			\label{eq:inv_curv_L2}
		\end{equation}
	In particular, this shows that the action of~$\mfG$ leaves the exponent in~\eqref{intro:ym_measure} invariant while simultaneously \enquote{translating} the \enquote{Lebesgue measure}~$\dif A$. Consequently, the measure~$\mu_\YM$ needs to be interpreted on the \emph{gauge orbit space} $\CO:= \CA \slash \mfG$ for it to be invariant under \emph{gauge transformations}~$g \in \mfG$.

	\paragraph{The two-dimensional case.} \label{sec:two-dim}
	In the case of $2$-dimensional space-time~$\Lambda$, one can indeed give meaning to $\mu_{\YM}$ by leveraging gauge symmetry.
	To see why two dimensions is special,
	consider the simple case $\Lambda = [0,1]^2$.
	Then, for any $B\in\CA$, there exists a
	gauge transformation~$g \in \mfG$ (unique up to the choice of $g(0)$)
	such that $A:=B^g$ is in the \emph{complete axial gauge}, i.e.~$A_2 \equiv 0$ and~$A_1(\Cdot,0) \equiv 0$.
	We can therefore identify the gauge orbit space~$\CO$ with the subspace of $1$-forms $A\in \CA$ in the complete axial gauge.
	For such $A$, one has~$[A \wedge A] \equiv 0$ and then
		\begin{equation*}
			F^A = \diff A = -\partial_2 A_1\,\diff x^1\wedge \diff x^2 =: F^A_{12}\, \diff x^1 \wedge \diff x^2\;.
		\end{equation*}
	In the complete axial gauge, the measure~$\mu_\YM$ becomes 
		\begin{equation}
			\mu_\YM(\diff A) \propto \exp\del[2]{- \int_\Lambda \abs[0]{\partial_2 A_1(x)}_{\mfg}^2 \dif x} \dif A
			\label{eq:ym_axial}
		\end{equation}
	with no Jacobian factor for the transformation of~$\diff A$.
	Upon identifying~$F \equiv F^A$ with~$F_{12}^A = -\partial_2 A_1$, a linear change of variables in~\eqref{eq:ym_axial} leads to
		\begin{equation*}
			\mu_\YM(\diff F) \propto \exp\del[2]{- \int_\Lambda \abs[0]{F(x)}_{\mfg}^2 \dif x} \dif F
			\;,
		\end{equation*}	  
	which is the expression for the measure of a~$\mfg$-valued white noise~$\xi$.
	So, modulo gauge equivalence, one can view $\mu_\YM$ on $[0,1]^2$ as the law of the random distributional $1$-form $A$ determined by
	\begin{equation}\label{eq:A_def}
	A_1 := \partial_2^{-1}\xi\;,\quad A_1(\Cdot,0)=0\;,\quad A_2=0\;.
	\end{equation}

While such a simple construction is not possible for other $2$-manifolds, a significant trace of this exact solvability remains and manifests as `invariance under subdivision' for discretisations of the manifold, first observed in \cite{Migdal75}.
This has enabled precise constructions of $\mu_{\YM}$ in a gauge invariant manner, see \cite{GKS89, Driver89,Sengupta92,Sengupta97,Levy03, Levy06, Levy20} and \cite{Witten91} for related results.
These works primarily study $\mu_{\YM}$ through gauge-invariant observables, such as Wilson loops.

The question we are concerned with in this article is the \emph{small scale regularity} of $\mu_{\YM}$:

\begin{question}\label{ques:regularity}
Suppose $A$ is a random $1$-form such that its gauge orbit $[A] = \{A^g\,:\,g\in\mfG\}$ is distributed according to $\mu_{\YM}$.
What is the possible regularity of $A$?
\end{question}

This question has received attention recently, especially due to recent progress in stochastic quantisation, where the regularity of quantum fields is a central question.

\paragraph{Prior work.}
Of course, one answer is to take $A$ as in \eqref{eq:A_def}.
For this choice, measuring regularity in H\"older--Besov spaces $C^\beta$ (see \ref{sec:notation} for notations),
one can show that~$A_1 = -\partial_2^{-1} \xi \in C^{-\nicefrac{1}{2}-\kappa}$ almost surely, for any $\kappa>0$.
However, from perturbation theory, one expects that the best regularity of $A$ is the same as that of the Gaussian free field, which is $C^{-\kappa}$ for any $\kappa>0$.
The issue with the axial gauge is that we take one component $A_2=0$ to be smooth, and this forces $A_1$ to behave in a rough way,
and it is expected that a better choice of gauge should spread regularity more evenly across $A_1$ and $A_2$.

Indeed, it was first shown in \cite{Chevyrev19} that the YM measure $\mu_{\YM}$ on $\T^2$ for simply connected $G$ admits a representative $A$ such that $A\in C^{-\kappa}$ almost surely, and the same proof applies to $[0,1]^2$ with no restrictions on $G$.
More strongly, \cite{Chevyrev19} introduced spaces of distributional $1$-forms, strictly smaller than $C^{-\kappa}$, with growth and continuity of line integrals along axis-parallel lines, and showed that these spaces support representatives of $\mu_{\YM}$.
For Abelian $G$ and with coupling to a Higgs field, a sharpened version of this result was obtained in \cite{CC24}.

Refinements of these spaces (in which one drops the axis-parallel line constraint) were studied in \cite{CCHS2d}, see also the survey \cite{Chevyrev22YM}.
We denote these spaces by $\Omega^1_\alpha$ and review their properties in Section~\ref{sec:AF} below.
Some of their key features are that Wilson loops are continuous functions on $\Omega^1_\alpha$, they carry a natural action of the gauge group $C^{1-\kappa}(\T^2,G)$ for which the quotient orbit space is Polish,
and the YM Langevin dynamic induced a Markov process on this orbit space.

Further progress was made in \cite{CS26}, in which it is shown that, for the 2D torus $\T^2$, the invariant measure of the aforementioned Markov process is precisely $\mu_{\YM}$.
A consequence of this result is that $\mu_{\YM}$ admits a representative of the form $\Psi+B$ where $\Psi=(\Psi_1,\Psi_2)$ is a couple of $\mfg$-valued independent Gaussian free fields, and $B$ is a random $1$-form such that $\E\|B\|_{C^{1-\kappa}}^p<\infty$ for every $p$.
That is, $\mu_{\YM}$ can be seen as an almost Lipschitz perturbation of the GFF, which has a clear bearing on Question \ref{ques:regularity}.

Note that the above works apply only to the torus, and it is far from trivial to extend them to other surfaces.
More recently, \cite{Dang_Nohra_26} initiated a different approach to Question \ref{ques:regularity} by introducing a new `Morse gauge', allowing them to construct a representative $A$ of $\mu_{\YM}$ on general surfaces.
See also \cite{Dang_Nohra_26_semiclassical} for applications of this gauge to semiclassical limits and \cite{Sengupta92, Sengupta97} for related results on surfaces via conditioned white noise measures.
However, the regularity of the representative $A$ is essentially the same as that of the axial gauge representative on $[0,1]^2$
so is expected to be suboptimal.

\paragraph{Summary of contributions.}
In this article, we provide a new approach to Question~\ref{ques:regularity} to obtain essentially optimal regularity.
We henceforth work over the unit square $\Lambda=[0,1]^2$.
A natural approach is to impose the Coulomb gauge
\[
    \diff^* A^g := \sum_{i=1}^2 \partial_i A^g_i =0\;.
\]
Indeed, by the identity
$
    F^{A^g}=\Ad_g F^A
$,
one would have
\begin{equation}\label{eq:YM_Coulomb}
    \diff A^g = \Ad_g F^A - [A^g\wedge A^g],
    \qquad
    \diff^*A^g=0.
\end{equation}
One then starts with $A$ as the representative of the YM measure in the axial gauge, for which $F^A=\xi$ is white noise, and tries to find $g$ that solves \eqref{eq:YM_Coulomb}. Power counting suggests that the first term on the right-hand side has regularity $C^{-1-\kappa}$, so, at a formal level, elliptic regularity for $\diff\oplus \diff^*$ suggests that $A^g$ should have regularity $C^{-\kappa}$.
Of course, this argument is only heuristic: neither the product $\Ad_gF^A$ nor the quadratic term $[A^g\wedge A^g]$ is classically meaningful at the expected regularities.

It is a priori clear (at least for smooth $A$) that Coulomb gauge representatives exist since every minimiser of the functional $\|A^g\|_{L^2}$ over $g\in\mfG$ must satisfy $\diff^*A^g=0$.
What is far from clear is whether one of these representatives would have good regularity.
Our approach to solve for $A^g$ in the Coulomb gauge and determine its regularity is based on rough path theory and regularity structures.
Our main contributions can be summarised as follows (see Section \ref{sec:main_results} for details):

\begin{itemize}
	\item \textbf{Rough additive functions.} 
	We introduce a (non-linear) space of rough additive functions $\bfOmega_{\alpha\ax}^1$ for $\alpha\in (\frac 13,\frac 12)$, which is a generalisation of the space $(\Omega^1_\beta,|\cdot|_{\beta})$ from Section~\ref{sec:AF}. The space $\bfOmega_{\alpha\ax}^1$ is designed to encode rough path lifts of line integrals of $1$-forms.

	\item \textbf{Rough Uhlenbeck compactness.}
	For every $\bfA\in \bfOmega_{\alpha\ax}^1$, we define a space of controlled gauge transformations $\fG^{2\alpha}_A$ and show that there exists $g\in \fG^{2\alpha}_A$ such that $A^g$ has good regularity and continuity properties,
	see Theorem \ref{thm:Rough_Uhlenbeck}.
	This result is entirely deterministic and we achieve it by analysing a Coulomb gauge on sufficiently small squares and `patching' the result together.
	This can be seen as a rough version of Uhlenbeck compactness \cite{Uhlenbeck82,Wehrheim04}, a fundamental result in gauge theory,
	which we prove with the help of regularity structures 
\cite{Hairer14}.

	\item \textbf{Gauge-fixed YM measure.}
	By building a canonical rough additive function $\bfA$ from the YM measure, we apply rough Uhlenbeck compactness to obtain a representative $A^g$ of $\mu_{\YM}$ with good regularity,
	see Theorem \ref{thm:RUC_YM}.
	We moreover show a stability result: smooth approximations of the YM curvature in the axial gauge yield smooth approximations of $A^g$.
\end{itemize}

Along the way in our proof of rough Uhlenbeck compactness, we develop a novel way to solve geometric (group-valued) singular PDEs in the context of regularity structures by analysing the tangent equation (\eqref{eq:intro_PDE_g} and Section \ref{sec:solution_theory}).
Our method is inspired by the implicit function theorem and bears similarity to (but is different from) the recent work \cite{BOS_25_Phi4}.

The only place where we use that we are working on $[0,1]^2$ is in the construction of the rough additive function $\bfA$ from the YM measure.
In analogy with how the main analytic step in classical Uhlenbeck compactness is to find a Coulomb gauge on the unit ball \cite{Uhlenbeck82},
we anticipate that our rough Uhlenbeck compactness for $[0,1]^2$ could be applied to general surfaces without change.
Therefore the main missing ingredient to extend our probabilistic results to general surfaces is to construct (locally) rough additive functions from the YM measure on these surfaces.
It would be interesting if ideas from \cite{Dang_Nohra_26} could be applied to this problem.

\paragraph{Comparison with prior works.}
Our work should be compared to \cite{Chevyrev19}.
Both works employ the Coulomb gauge condition $\diff^* A=0$ for small scale regularity and leverage solvability of the YM measure in axial-type gauges to get probabilistic bounds.
In terms of answers to Question \ref{ques:regularity}, both establish the existence of a representative $A$ of $\mu_{\YM}$ with $A\in \Omega^1_\beta$ for some $\beta<1$,
but the space $\Omega^1_\beta$ in \cite{Chevyrev19} is restricted to axis-parallel lines, while for us this space is the `isotropic' version from \cite{CCHS2d} that allows integration against any smooth curves (see also \cite{ChandraSingh25_rough} for a related notion of distributional $k$-forms in arbitrary dimensions).

The main advantage of our work is that it is much more transparent and generalisable.
It works directly in the continuum (but could be implemented on the lattice),
provides explicit continuity statements for $\bfA\mapsto A^g$ (see Theorem \ref{thm:Rough_Uhlenbeck})
which are not available in \cite{Chevyrev19},
and is closer in spirit to the work of Uhlenbeck \cite{Uhlenbeck82}.
The approach in \cite{Chevyrev19} strongly relies on (square) lattice approximations and is, in many ways, more technical. The basic idea in \cite{Chevyrev19}  is to proceed by induction on scales, zooming in and applying an approximate Coulomb gauge iteratively.
In contrast, we zoom in only once to get smallesness of a scale parameter, and then solve for $A^g$ directly via regularity structures.
As described above, we believe this makes our method easier to generalise to other manifolds.
Moreover, we believe our approach is more suitable for the study of large-$N$ limits (see below), especially in light of recent progress in noncommutative regularity structures \cite{CHP25}.

Comparing to \cite{CS26}, our approach is again more direct and does not rely on the Langevin dynamics, which is a more involved construction.
Moreover, the results of \cite{CS26} still use the lattice, in particular a refinement of \cite{Chevyrev19}, and would be challenging to extend to other domains/surfaces.
To illustrate the difference, we obtain much sharper bounds on the growth of moments of $A^g$ compared to \cite{CS26}, see Theorem \ref{thm:RUC_YM} and Remark \ref{rem:compare_CS}.
In contrast to \cite{CS26}, we do not prove that $\mu_{\YM}$ is a perturbation of a Gaussian free field, however we do still obtain a decomposition into objects of different regularities.

\paragraph{Further related works.}
The question of continuum (ultraviolet) limits and regularity of the YM measure has a long history.
Classical works include \cite{Gross83, Driver87}, where a construction of the continuum limit is given for Abelian gauge groups on $\R^3$ and $\R^4$,
and~\cite{Balaban85IV,Balaban89,MRS93} for progress towards a continuum construction on $\T^3$ and $\T^4$.

There has been a lot of recent interest in probabilistic aspects of quantum YM theory.
In addition to the works already mentioned, we highlight the derivation of Makeenko--Migdal (or Dyson--Schwinger) equations in the continuum \cite{Levy17_master, Dahlqvist16_free_energeis, DriverHallTodd_17_MM_eq, DriverGabrielHallKemp_17_MM_eq,Driver19_MM_eq}
and on the lattice \cite{Chatterjee19,Jafarov16_SUN,SSZ_24_master_loop},
large-$N$ limits and connections to free probability \cite{Hall18_LargeN, BasuGanguly18_planar, DahlqvistNorris20_master_sphere,
DahlqvistLemoine23_largeN_I,
DahlqvistLemoine25_largeN_II,
LemoineMaida25_dual,
Lemoine25_almost_flat},
random surface representations \cite{Cao_Park_Sheffield_25,PPSU_26,Lemoine26_dual},
and infrared behaviour of lattice gauge theory for continuous groups
\cite{Chatterjee21_confinement,Garban_Sepulveda_23,SZZ23,SZZ26},
and discrete groups \cite{Chatterjee_2020_Ising_LGT,Cao20,FLV20,FLV21,Forsstrom24}.
This list is far from complete and we refer to references therein for further works.

There has also been much progress on the  stochastic quantisation approach to Yang--Mills \cite{CCHS2d, CCHS3d, BringmannCao24_Higgs, BringmannCao26, CS26, ChevyrevShen26_Renorm}.
Although the problem we address is distinct, our work is related to this programme because, like these works, we develop and employ the theory singular SPDE.

\paragraph{Open problems.}
The next natural step following this work is to obtain a complete regularity theory for the YM measure on any compact surface.
As we mentioned, we expect Theorem \ref{thm:Rough_Uhlenbeck} to provide the core analytic ingredient, and the main remaining difficulty lies in the construction of rough additive functions defined on local charts that satisfy suitable compatibility conditions on overlaps.
We anticipate that ideas from \cite{Sengupta92,Dang_Nohra_26} could be useful for this problem.

Another open problem which requires substantially new ideas is when the base manifold is $3$-dimensional.
First, little is known about the 3D YM measure in the continuum. It is far from clear if the curvature in the axial gauge is well-behaved (we expect it is not)
and no gauge-invariant quantities are known to be controlled in the continuum limit.
It is therefore an open question how to formulate an Uhlenbeck compactness theorem applicable to the 3D YM measure.

A closely related problem is that of a state space for the 3D YM measure. Candidate state spaces were constructed in~\cite{CCHS3d, Sourav_state} (see also~\cite{Sourav_flow}), but there is a lack of canonical gauge group associated with these spaces.
We expect our notion of rough additive functions and `controlled gauge transformations'
to have a bearing on this problem because the 3D GFF is in $C^{-\frac 1 2-\kappa}$, $\kappa>0$, which is the same regularity as the axial gauge representative of the 2D Yang-Mills measure.
The key difficult in 3D is that, unlike in 2D, line integrals are not well-defined,
and a possible approach to the problem is to combine rough path theory with gauge-covariant regularisation, e.g. the YM heat flow used in 
\cite{CG13, CCHS3d, Sourav_state}.

Finally, there is considerable interest in understanding the behaviour of $\mu_{\YM}$ with structure group $G=U(N)$, and other classical groups, as $N\to\infty$.
In particular, there is recent progress in constructing the limiting `master field' on arbitrary compact surfaces \cite{DahlqvistLemoine23_largeN_I, DahlqvistLemoine25_largeN_II}.
A natural question that emerges is whether one can construct a gauge-fixed representative of the master field with natural regularity.
It would be interesting if the methods of this article, combined with \cite{CHP25}, could be applied to this problem.


%
\subsection{Main results, strategy, and outline}
\label{sec:main_results}

We now describe our main results and strategy in detail.
Rather than solving for $A^g$ in \eqref{eq:YM_Coulomb} directly, we solve for the gauge transformation $g$ via $\diff^*A^g=0$. 
Since $A$ is in the complete axial gauge, $A=A_1\,\diff x^1$, the condition $\diff^*A^g=0$ becomes equivalent to the Lie-group valued equation
\begin{align}\label{eq:intro_PDE_g}
\sum_{i=1}^2\partial_i ((\partial_i g) g^{-1})-\partial_1(gA_1g^{-1})=0,
\end{align}
with suitable boundary conditions.
Below, we write $\alpha^-$ to denote a number $\alpha-\kappa$ for $\kappa>0$ arbitrarily small.

We treat~\eqref{eq:intro_PDE_g} as an elliptic SPDE for $g$ with a random $A_1$ as an `input noise'.
For $A_1$ sampled from the YM measure in the axial gauge, this equation is singular:
since $A_1\in C^{-\nfrac{1}{2}^-}$, one expects $g\in C^{\nfrac{1}{2}^-}$ by Schauder estimates, and consequently the products $gA_1g^{-1}$ and $(\partial_i g) g^{-1}$ are singular.
However, one can check that \eqref{eq:intro_PDE_g} is subcritical for $A_1\in C^{-\nfrac{1}{2}^-}$ (it becomes critical for $A_1\in C^{-1}$),
which places the equation in the class of singular elliptic equations for which one expects regularity structures.
Part of this solution theory is to enhance $A_1$ with further elements, e.g.~products of analytically ill-defined distributions.
We next describe the steps we take to solve \eqref{eq:intro_PDE_g} and how we obtain the desired regularity of $A^g$.

\paragraph{Rough additive functions.}
One of the insights of this paper is that the enhancement needed for the singular SPDE~\eqref{eq:intro_PDE_g} is encoded in a geometrically natural object that we refer to as a \emph{rough additive function (RAF)},
introduced and studied in Section \ref{sec:RAF}.
On first sight, this object is simpler and contains less data than what one expects from a naive solution theory of \eqref{eq:intro_PDE_g}.

A RAF is a pair of functions $\bfA=(A,\bA)$, each mapping lines to the Lie algebra $\mfg$, in a way that respects Chen identity and satisfies the analytic bounds
\[
|A(\ell)|\lesssim |\ell|^\alpha\triple{\bfA}_{\alpha\ax}, \qquad |\bA(\ell)|\lesssim |\ell|^{2\alpha}\triple{\bfA}_{\alpha\ax}
\]
for some $\alpha\in (\frac 13,\frac 12)$, together with suitable continuity in $\ell$, where $\triple{\bfA}_{\alpha\ax}$ is a \enquote{norm}\footnote{The space of RAFs is not a vector space, but can be seen to be a subset of a Banach space.} for the space of RAFs.
One should think of $A$ as a line integral and $\bA$ as the iterated line integral.
Note $2\alpha$ appearing for the bound of $\bA$, reminiscent of the second level lift in rough path theory.

We denote the space of RAFs by $\bfOmega_{\alpha\ax}^1$ with $\alpha\in (\frac 13,\frac 12)$, where \enquote{$\mathrm{ax}$} is emphasises the axial gauge. (Though the spaces make sense in more generality than just the axial gauge.)
The space is similar in spirit to $\Omega^1_\beta$ from Section~\ref{sec:AF},
but with the key difference that we encode rough path enhancements.

An important fact is that the first component projection $(A,\bA)\mapsto A$ is a continuous map $\bfOmega_{\alpha\ax}^1 \to \Omega^1_{\alpha} \hookrightarrow \Omega^1 C^{\alpha-1}$ (see~\eqref{eq:omega_beta_embedding}),
so $\bfOmega_{\alpha\ax}^1$ is indeed a strengthening of both the space $\Omega^1_{\alpha}$ from \cite{Chevyrev19,CCHS2d} and of classical H\"older--Besov spaces.

The space $\bfOmega_{\alpha\ax}^1$ is also natural from a stochastic point of view: if $A$ is the YM measure representative in the axial gauge, then $A(\ell)$ are essentially $\mfg$-valued Brownian motions, of which the Stratonovich lift that we consider here is a central object in rough path theory \cite{FH20}.
One can therefore see the role of $\bfOmega_{\alpha\ax}^1$ as replacing the stochastic analytic arguments in \cite{GKS89, Sengupta92} by rough (pathwise) arguments with the usual advantages of rough path theory, such as continuity and stability of the solution map.

One of the main features of the space $\bfOmega_{\alpha\ax}^1$ is that it comes with a natural generalisation of gauge transformations. We introduce \emph{controlled gauge transformations}, denoted by $\fG^{\beta}_A$, which depend on the underlying RAF $\bfA=(A,\bA)\in\bfOmega_{\alpha\ax}^1$, reminiscent of Gubinelli's controlled rough paths \cite{MaxControl} (see Section~\ref{sec:gauge_transformation} for further details).

\paragraph{Rough Uhlenbeck compactness.}
The above ingredients allow us to formulate our main result:

\begin{theorem}[Rough Uhlenbeck compactness]\label{thm:Rough_Uhlenbeck}
Let $\alpha\in (\frac {4}{9},\frac 12)$ and $\beta\in (0,6\alpha-2)$.\footnote{These parameters seem arbitrary, but they arise from several places in the proof: the requirements $\beta<6\alpha-2$ and $\beta>\frac 23$ arise from  Lemma~\ref{lem:integration_af_Young} and  Theorem~\ref{thm:patching} respectively, which in turn impose $\alpha>\frac{4}{9}$ for the sets of $\beta$ to be non-empty.} 
There exist $C>0$, depending only on $\alpha,\beta$ and $G$,
with the following property: for any two rough additive functions $\bfA=(A,\bA),\bar\bfA=(\bar A,\bar \bA)\in \bfOmega^1_{\alpha\ax}$
in the axial gauge, there exist controlled gauge transformation $g\in  \fG^{3\alpha/2 }_A$,
$\bar g\in  \fG^{3\alpha/2 }_{\bar A}$  such that $A^g,\bar A^{\bar g}\in\Omega^1_{\beta}$, and 
\begin{align}\label{eq:RUC_bound}
|A^g|_{\beta}
&\leq C(1+\triple{\bfA}_{\alpha\ax})^{{\frac {2+\beta-\alpha}{\alpha}}}\triple{\bfA}_{\alpha\ax}\;,\\
|A^g-\bar A^{\bar g}|_{\beta}
&\leq C(1+\triple{\bfA}_{\alpha\ax}+\triple{\bar\bfA}_{\alpha\ax})^{{\frac {2+\beta-\alpha}{\alpha}}}\triple{\bfA;\bar\bfA}_{\alpha\ax}\;.
\end{align}
%
\end{theorem}

A more precise statement is given in Theorem \ref{thm:RUC_bounded_sets}, which shows that in fact one can choose the map $\bfA\mapsto (g,A^g)$ to be Lipschitz continuous on bounded sets.
Although we do not do it here, one can readily show from Theorem \ref{thm:RUC_bounded_sets} that the map $\bfA\mapsto A^g$ can be chosen measurable.

Before describing the proof and applications of Theorem~\ref{thm:Rough_Uhlenbeck}, we remark on its connection to classical Uhlenbeck compactness \cite{Uhlenbeck82,Wehrheim04}.
Suppose the connection~$B$ on~$[0,1]^d \subset \R^d$ has bounded~$W^{1,p}$ norm for~$p > d/2$. It is then easy to see by Sobolev embedding that~$F^B = \diff B + [B\wedge B]$ has bounded~$L^p$ norm.

	Uhlenbeck's theorem concerns the converse direction: if~$\norm[0]{F^B}_{L^p} < \infty$, can one bound $\norm[0]{B}_{W^{1,p}}$ in terms of $\|F^B\|_{L^p}$?
	This is of course not true even for a pure gauge $B=0^g=-(\dif g)g^{-1}$, since then $F^B=0$ but~$\norm[0]{B}_{W^{1,p}}$ can be arbitrarily large if $g$ is rough.
	The correct question is thus whether there exists a gauge transformation~$g$ such that~$\norm[0]{B^g}_{W^{1,p}}$ is bounded in terms of $\norm[0]{F^B}_{L^p}$.
	Since
	\[
	\norm[0]{F^{B^g}}_{L^p} =\norm[0]{\Ad_g F^{B}}_{L^p} = \norm[0]{F^B}_{L^p}\;,
	\]
	the question in abstract terms is asking to bound the regularity of $B$, modulo gauge transformations, in terms of a gauge-invariant quantity.

As described in \eqref{eq:YM_Coulomb}, Uhlenbeck's approach is to fix the Coulomb gauge by $\diff^* B=0$.
Then, if $\|B\|_{W^{1,p}}$ is small, one indeed has $\|B\|_{W^{1,p}} \lesssim \|F^B\|_{L^p}$ by elliptic regularity.
Then with a continuity argument, one can show that, under only a smallness assumption on $\|F^B\|_{L^p}$ (and not $B$), one can find a gauge transformation $g$ such that $B^g$ is in the Coulomb gauge and $\|B^g\|_{W^{1,p}} \lesssim \|F^B\|_{L^p}$.
Patching together these local estimates leads to a global estimate on $\|B^g\|_{W^{1,p}}$ in terms of $\|F^B\|_{L^p}$, which is Uhlenbeck's result.

Theorem \ref{thm:Rough_Uhlenbeck} can be seen as a rough version of this result:
for any (say smooth) connection $B$,
let $A$ be its (essentially unique) complete axial gauge and $\bfA$ the canonical RAF.
Then Theorem \ref{thm:Rough_Uhlenbeck} implies the existence of $g\in\mfG$ such that $\|B^g\|_{C^{\beta-1}}\lesssim |B^g|_{\beta}$ is bounded in terms of the gauge-invariant quantity $\triple{\bfA}_{\alpha\ax}$.
That is, we replace the $L^p$ norm of $F^B$ by $\triple{\bfA}_{\alpha\ax}$, which requires much less regularity on the underlying fields.

We highlight that $\triple{\bfA}_{\alpha\ax}$ can be derived from a geometrically natural quantity, namely the lasso field of Gross \cite{Gross85}.
Indeed, $\triple{\bfA}_{\alpha\ax}$ is given by H\"older norms of (iterated) surface integrals of the curvature $F^A$.
For any connection $B$ with axial gauge representative $A$, the curvature of $F^A$ is a lasso field of $B$ in the sense of \cite{Gross85}.
The key feature of lassos is that, while smeared averages curvature like $\scal{F^B,\psi}$ for $\psi\in C^\infty_c$ are not gauge covariant (thus general Besov norms of $F^B$ are not gauge invariant), smeared averages of lasso fields \emph{are} gauge covariant.
This makes the lasso field (equivalently the curvature of the axial gauge representative) a natural object to work with in low regularity regimes where $L^p$ norms are not controlled.

\paragraph{Proof outline.}
Since Theorem~\ref{thm:Rough_Uhlenbeck} is our main result, we highlight some the key ideas in its proof.
There are three main steps: the solution theory for the singular PDE \eqref{eq:intro_PDE_g} in Section \ref{sec:solution_theory}, model bounds based on rough additive functions in Section \ref{s:model_bounds}, and patching in Section \ref{sec:patching}.

\paragraph{Solution theory.}
Let $A=(A_1,0)$ be in the complete axial gauge,
As explained after \eqref{eq:YM_Coulomb}, we wish to find $g$ such that $\diff^*A^g=0$, which is equivalent to \eqref{eq:intro_PDE_g}.
Just like in~\cite{Uhlenbeck82}, it is important for us to have smallness of $A$ (more specifically of $\triple{\bfA}_{\alpha\ax}$) to solve the elliptic problem.
We achieve this by zooming to a small square and rescaling $\bfA$ back to the unit square, which has the effect of shrinking the norm $\triple{\bfA}_{\alpha\ax}$.
It is technically convenient to then replace $A_1$ in 
\eqref{eq:intro_PDE_g} by $\sigma A_1$ for $\sigma>0$ small $A_1$ of order $1$.
We choose Dirichlet boundary condition $g|_{\partial\Lambda}=1_G$ (this is an arbitrary but convenient choice),
so we are after solutions to
\begin{equation}\label{eq:intro_Psi}
\Psi(\sigma,g):=\sum_{i=1}^2\partial_i ((\partial_i g)g^{-1})-\sigma \partial_1(gA_1g^{-1})=0\;, \qquad g|_{\partial\Lambda}=1_G
\;.
\end{equation}
Clearly~$\Psi(0,1_G)=0$ and, for smooth~$A_1$, the derivative~$D_2\Psi(0,1_G)(u)$ is equal to $\Delta u$ on $C^\infty(\Lambda,\mfg)$, so one could solve \eqref{eq:intro_Psi} for small~$\sigma$ by the implicit function theorem (IFT).

In Section \ref{sec:solution_theory},
we adapt this IFT argument to a rough setting, in which $A_1\in C^{-\nfrac{1}{2}^-}$ is distributional.
It turns out technically simpler to  introduce an integrated version of $\Psi$ by writing $\Phi = G^\Dir*\Psi$, where $G^\Dir\label{symb:Green_Dir}$ is the Dirichlet Green's function on $\Lambda$, and lift $\Phi$ to the level of modelled distributions over a suitable regularity structure, which we introduce in Sections \ref{sec:reg_struct}-\ref{sec:modelled_distributions}.

Our regularity structure is somewhat non-standard.
One of the reasons is that it is very useful to use that $\partial_2A_1\in C^{-1^-}$, and not just $C^{-\nfrac{3}{2}^-}$, for $A_1$ arising from a RAF (think of $\partial_2A_1=\xi$ as white noise).
Thus $\partial_2$ has a `better than expected' deregularising effect on some terms in the regularity structure, and we encode this with non-standard homogeneities and abstract derivation operators.
Our construction is completely explicit and does not rely on black-box algebraic results such as \cite{Hairer14, BHZ19, BCCH21, Chevyrev_22_Alg}.
A simplifying feature of our problem, compared to many works in singular SPDE, is that, as is often the case in rough path theory, we do not need to consider `negative renormalisation' and thus work primarily with canonical models.

In Section \ref{sec:existence_stability}, we lift the map $\Phi$ to modelled distributions and show that $D_2\Phi(\sigma,g)$ is invertible for small $\sigma$.
Our main result in Section \ref{sec:solution_theory} is Theorem \ref{thm:existence_solution}, which states that there exists a solution $g$ to $\Phi(\sigma,g)=0$ and that $g$ is locally Lipschitz continuous in the model.
We achieve this by studying the ODE
\begin{equation}\label{eq:intro_ODE}
	\partial_\sigma g_\sigma
	=
	-\bigl(\partial_2\Phi(\sigma,g_\sigma)\bigr)^{-1}
	\partial_1\Phi(\sigma,g_\sigma)\;,
	\qquad
	g_0=1_G\;,
\end{equation}
which is equivalent to $\partial_\sigma \Phi(\sigma,g_\sigma)=0$.
This is a more explicit version of the IFT argument sketched above,
with the main advantage that the ODE formulation makes it easy to show continuity in the model.
This approach appears different from standard fixed point arguments in singular SPDE, but is reminiscent of \cite{BOS_25_Phi4}; see also \cite{FK22_Laplace} for applications of the IFT in regularity structures.

In Section \ref{sec:properties_solution},
we derive the crucial properties of the solution $g$ needed in the rest of the paper.
In particular, we identify the coefficients of the solution $g$ in its modelled distribution expansion and show, with a careful cancellation, that $A^g$ is a modelled distribution that reconstructs to a function of regularity $0^-$ (vs. $-1/2^-$ for $A_1$).

\paragraph{Construction of singular objects.}\label{sec:proof_idea_construction_model}
A key insight in the proof of Theorem~\ref{thm:Rough_Uhlenbeck} is that any rough additive function $\bfA=(A,\bA)$ determines canonically a model $\sfZ=(\Pi,\Gamma)$ necessary for the preceeding `solution theory' step. This is shown in Section \ref{s:model_bounds}.
In a stochastic setting, this factorises the measurable map taking the `noise' to the model, $A_1\longrightarrow \sfZ$, as
\begin{equation}\label{eq:A_to_Z}
    A_1
    \longrightarrow
    \bfA=(A,\bA)
    \longrightarrow
    \mathsf Z\;,
\end{equation}
where the first map is measurable and the last map is locally Lipschitz continuous.
This factorisation simplifies the construction of models and shows they arise from a geometrically simple object (the RAF) built from iterated line/area integrals.

To explain the ideas behind the construction of the model $\sfZ$, consider two elements of the structure space, $\blue{\CI_1 A_1}$ and $\blue{A_1\CI_1A_1}$,
where, in these expressions, $\blue{A_1}$ is the abstract placeholder for the first component of $A\in\Omega^1(\Lambda,\mfg)$,
and $\CI_1$ represents integration against $K_1 := \partial_1 K$ for a truncation $K$ of the Green's function of the Laplacian $\Delta$ on $\R^2$.
These symbols have (standard) homogeneities $\nfrac12^-$ and $0^-$, respectively.

Now let $\ell^{x;y}\label{symb:ell_xy}$ be the line segment connecting $x$ to $y$.
In the usual notation for a model $\sfZ= (\Pi,\Gamma)$ in regularity structures,
one then has (see Proposition~\ref{prop:regularity_d*KFA})
\[
\Pi_x \blue{\cI_1A_1}(y)=K_1*A_1(y)-K_1*A_1(x)=A(\ell^{x;y})+Q^A(\ell^{x;y}),
\]
where $Q^A\in\Omega^1_{1^-}$ with $|Q^A|_{1^-}\lesssim |A|_{\nfrac12^-\ax}$.
The point is that any model arising from a RAF realises $\blue{\cI_1A_1}$ as the line integral of $A$ plus the smoother error $Q^A$.
It follows that
\[
\Pi_x \blue{A_1\cI_1A_1}(y)=\partial_{y_1}\bA(\ell^{(x_1,y_2);y})+\partial_{y_1}(\text{Young integrals})\;,
\]
which relates the definition of the model to derivatives of (iterated) line integrals. 

\paragraph{Patching and finalising the proof.}
The output of the previous two steps are $1$-forms $B_n = A^{g_n}$, each in the Coulomb gauge with good regularity, defined on sufficiently small squares $U_n$ which cover $\Lambda$.
The final step combines these local $1$-forms into a global $1$-form via patching. This follows a classical strategy of \cite{Uhlenbeck82,Wehrheim04}, but requires care as we work with distributional connection forms.
In \ref{sec:patching} we state the core patching theorem, which works with $1$-forms in the space $\Omega^1_\beta$.
This step can be read independently of the rest of the paper and we believe it is of independent interest.

In Section \ref{ch:RUC} we combine the (up to now independent) Sections \ref{sec:solution_theory}, \ref{s:model_bounds}, and \ref{sec:patching} to prove Theorem~\ref{thm:Rough_Uhlenbeck}.

\paragraph{Gauge-fixed YM measure.}
While Theorem \ref{thm:Rough_Uhlenbeck} is a deterministic statement, it is inspired by its probabilistic applications.
It turns out one can construct canonical random rough additive functions $\bfA$ from the axial gauge representative $A$ of the YM measure on $[0,1]^2$; indeed we a prove a version of Kolmogorov continuity theorem for RAFs (Theorem \ref{thm:Kolmogorov}), which allows us to
lift $A$ to a random element of $\bfOmega^1_{\alpha\ax}$ with good moment bounds.
(This can be seen as an extension of the rough path Kolmogorov continuity theorem \cite[Thm~3.1]{FH20} to a spatial setting.)
In combination with Theorem \ref{thm:Rough_Uhlenbeck} (or more precisely Theorem \ref{thm:RUC_bounded_sets}),
this allows us to prove our main result on the YM measure:

\begin{theorem}[Gauge-fixed YM measure]\label{thm:RUC_YM}
	Let $\beta \in (0,1)$, $\delta>0$.
	There exists a representative $B\in \Omega_\beta^1$ of the YM measure on $\Lambda=[0,1]^2$ (see Definition \ref{def:YM_measure_precise}) such that, for any $p\geq1$,
    \begin{equation}\label{eq:B_moments}
    \E[|B|^p_\beta]^{1/p}\lesssim_{\beta,\delta} p^{\frac{2+\beta}{2}+\delta}\;.
	\end{equation}
	Furthermore, fix a non-negative radial function $\rho \in C_c^\infty(\R^2)$ with support on $B(0,1)$ and $\int_{\R^2} \rho=1$. For $\eps \in (0,1]$, denote
$
\rho^\eps(x) = \eps^{-2}\rho(\eps^{-1}x) $.
Let $\xi$ be a $\mfg$-valued white noise on $\R^2$ and $\xi^\eps = \xi*\rho^\eps$.
Let $A^\eps_1 \in C^\infty(\R^2,\mfg)$ such that $A_1(\Cdot,0)=0$ and $\partial_2A^\eps_1=\xi^\eps$.
Denote $A^\eps = (A_1^\eps,0)\in \Omega^1 C^\infty$.
Then one can define $B$ above on the same probability space as $\xi$ and there exist random $g^\eps\in C^\infty(\Lambda;G)$ such that $B^\eps := (A^\varepsilon)^{g^\eps}$ converges to $B$ as $\eps\downarrow0$ in $\Omega_\beta^1(\Lambda)$ a.s. and in $L^p$ for any $p\geq 1$
	and there exists $\kappa>0$ such that
	\begin{equation}\label{eq:B_eps_moments}
	\E\Big[\Big|\sup_{0\leq \bar\eps<\eps\leq 1}\frac{ |B^\eps - B^{\bar\eps}|_{\beta}}{|\eps-
	\bar\eps|^\kappa}\Big|^p\Big]^{1/p} \lesssim_{\beta,\delta} p^{\frac{2+\beta}{2}+\delta}\;.
	\end{equation}
\end{theorem}

The proof of Theorem \ref{thm:RUC_YM} is given in Section \ref{sec:proof_RUC_YM}.
The final part of Theorem \ref{thm:RUC_YM} should be seen as a Wong--Zakai-type theorem for the YM measure in (local) Coulomb gauges and is a consequence of the continuity of the map $\bfA\mapsto A^g$ on bounded sets implied by Theorem \ref{thm:RUC_bounded_sets}.

\begin{remark}\label{rem:compare_CS}
A refinement of \cite{Chevyrev19} derived in \cite{CS26} (see the proof of Theorem 9.1 therein)
shows that there exists a gauge-fixed representative $B$ of the YM measure with $|B|_\beta$ stochastically bounded above by $e^{C|\log(2+|X|)|^2}$, where $X$ is a standard Gaussian random variable.
The bound \eqref{eq:B_moments} considerably sharpens this and shows that $|B|_\beta$ is bounded stochastically by $C(1+|X|^{3/2})$.
\end{remark}


\subsection{Notation and preliminaries}
\label{sec:notation}
We end the introduction with  notational conventions used throughout the paper.

We let $X\lesssim Y$ denote $X\leq CY$ for some constant $C>0$ not depending on $X,Y$.
When we wish to emphasise dependence of $C$ on parameters $\alpha,\beta,\ldots$, we write $\lesssim_{\alpha,\beta,\ldots}$.
We equip $\R^d$ with the Euclidean norm $|\Cdot|$.

\paragraph{Metric spaces.}
Let $(X,d_X)$ be a metric space. We denote by $B(x,r) = \{y\in X\colon d_X(x,y) < r\}$ the open ball of center $x$ and radius $r$.
For a set $U\subset X$ we denote by $U^\circ\label{symb:interior}$ the interior of $U$
and by $\overline U$ the closure of $U$.
For any $\varepsilon>0$, we let\label{symb:fattening}
\begin{equation}\label{eq:ex_def}
U^{+\varepsilon}=\bigcup_{y\in U}\{x\in X\colon d_X(x,y)\leq \varepsilon\}
\end{equation}
denote the $\varepsilon$-fattening of $U$.
For sets $U,V\subset \R^d$ we write $U\Subset V\label{symb:Subset}$ if $\overline{U}\subset V$.  

\paragraph{Matrices and Lie groups.}
Let $\rmM_\bC(N)$ denote the vector space of $N\times N$ complex matrices endowed $\rmM_\bC(N)$ with the Frobenius norm.
Let $\rmU(N)\subset \rmM_\bC(N)$ be the unitary group and $G\subset \rmU(N)\label{symb:G_Lie_group}$ a compact connected Lie group.
We denote by $\fg\subset \rmM_\bC(N)\label{symb:Lie_algebra}$ the Lie algebra of $G$.
The identity element of $G$ is denoted by $1_G\label{symb:identity_G}$
(this is just the identity matrix in $\rmM_\bC(N)$).
We use the notation $O_G\subset G$ (resp. \  $O_\fg\subset \fg$)
for the open sets around $1_G\in O_G$ (resp.\ $0\in O_\fg$) such that the Lie group exponential map $\exp:\fg\to G$ 
restricted on $O_G$ is a diffeomorphism between $O_G$ and $O_\fg$ with inverse $\log:O_G\to O_\fg$.

\paragraph{Function spaces.}
We denote henceforth $\Lambda\label{symb:Lambda} = [0,1]^2\subset \R^2$.

Let $V,W$ be vector spaces. We denote by $L(V;W)\label{symb:LVW}$ the set of linear operators $V\to W$. We set $L(V):=L(V;V)$.

For $\beta\in[0,1]$ and a function $f\colon E\to F$ between metric spaces $(E,d_E)$ and $(F,d_F)$, we denote by $|f|_{\Hol\beta}$ the $\beta$-H\"older seminorm of $f$:
\[
|f|_{\Hol\beta} = \sup_{x,y\in E}\frac{d_F(f(x),f(y))}{d_E(x,y)^\beta}\;.
\]

We denote by $\floor{\beta}$ the floor of $\beta\in\R$.
For $\beta > 0$, let $C^\beta(\R^d)\label{symb:Cbeta}$ be the Banach space of $\beta$-H\"older functions equipped with the norm
\[
|\phi|_{C^\beta} = \max_{|k|\leq\floor{\beta}}|\partial_k\phi|_{\infty} + \max_{|k| = \floor{\beta
}} |\partial_k \phi|_{\Hol{(\beta-\floor\beta)}}
\]
where $|\Cdot|_\infty$ is the uniform norm and the $\max$ runs over multi-indices $k\in\N^d$ and we denote $|k| = \sum_{i=1}^d|k_i|$.
For instance, $C^1$ means Lipschitz.
For a set $U\subset \R^d$, we treat by $C^\beta(U)$ by restriction:
\[
C^\beta(U) = \{\zeta|_U \,:\,\zeta \in C^\beta(\bR^d)\}\;,
\quad
|\xi|_{C^\beta(U)} = \inf\{|\zeta|_{C^\beta(\bR^d)} \,:\,\zeta|_U = \xi\}\;.
\]
We sometimes write $|\Cdot|_{C^\beta}$ for $|\Cdot|_{C^\beta(U)}$ when the domain $U$ is clear from the context.

For $\beta\leq 0$, we let $C^\beta(\R^d)$ be the Banach space of distributions $\xi\in \cD'(\R^d)$ such that
\[
|\xi|_{C^\beta} := \sup_{x\in\R^d}\sup_{\lambda\in(0,1]}\sup_{\phi\in \CB^r} \lambda^{-\beta}|\scal{\xi,\phi_x^\lambda}|<\infty\;,
\]
where $r=\floor{-\beta+1}$, $\CB^r = \{\phi\in C^r(\R^d)\colon |\phi|_{C^r}\leq 1\,,\, \supp(\phi) \subset B(0,1)\}$,
\[
\phi_x^\lambda(y) = \lambda^{-d}\phi(\lambda^{-1}(y-x))\;,
\]
and we denote the pairing of functions and distributions by $\scal{\xi,\phi} =\xi(\phi) = \int_{\R^d}\xi(y)\phi(y)\diff y$.

We denote by $\lc^\beta$ the closure of smooth functions in $C^\beta$, i.e.\ the little H\"older spaces.
For function space $X$, we denote by $\Omega^1 X$, the space of $1$-forms $A=A_i\diff x^i$ for which $A_i\in X$.
Recall the following Young product theorem (see, e.g.~\cite[Thm.~14.1]{CZ20_reconstruction}).

\begin{lemma}\label{lem:Young_product}
For $1\geq \alpha \geq \beta$ with $\alpha+\beta>0$, there exists a continuous bilinear map
\[
C^\alpha(\R^d)\times C^\beta(\R^d)\ni (f,g)\mapsto f g\in C^{\beta}(\R^d)\;,
\]
which extends classical multiplication $C^\infty\times C^\infty\to C^\infty$.
Moreover, for $r = \floor{-\beta+1}\vee 1$,
\begin{equation}\label{eq:Young_bound}
\sup_{x\in\R^d}\sup_{\lambda\in(0,1]}\sup_{\phi\in \CB^r} \frac{|\scal{fg - f(x)g,\phi_x^\lambda}|}{\lambda^{\alpha+\beta}} \lesssim_{\alpha,\beta} |f|_{C^\alpha} |g|_{C^\beta}\;.
\end{equation}
\end{lemma}

\section{Rough additive functions}\label{sec:RAF}
We want to define a state space for the axial gauge of the 2D YM measure on $\Lambda$. 
It should be strong enough to consist of connection $1$-forms for which a suitable notion of holonomy and gauge transformation exists. 
A natural approach that has been initiated in \cite{Chevyrev19} and extended in \cite{CCHS2d} is to enforce elements of the state space to have well-defined line integrals. 
These works, however, do not cover the \emph{axial gauge} representative of the YM measure as its line integrals are of worse regularity than covered in \cite{CCHS2d}. When it comes to line integration, the latter work corresponds to the Young regime while it is known since the works~\cite{GKS89} and~\cite{Driver89} that the integrals of the YM measure in the axial gauge behave like Brownian motions. 
In this section, we will therefore take inspiration from rough paths theory to generalise the state space defined in~\cite{CCHS2d}.

\subsection{Preliminaries on additive functions}\label{sec:AF}
We start by recalling some basic definitions from \cite{CCHS2d}. We fix a finite-dimensional Banach space $E$. We denote by $\cX\label{symb:cX}$ the set of line segments $\ell=(x,v)$ where $x\in \Lambda$ is the starting point and $v\in\R^2$ is the direction vector satisfying $|v|\leq \frac 14$ and $x+v\in \Lambda$. The length of $\ell$ is denoted by $|\ell|=|v|$. 
We say two lines $\ell=(x,v),\bar\ell=(\bar x,\bar v)\in\cX$ are \emph{joinable} if $\bar x=x+v$, $v=cw$ and $\bar v=\bar cw$ for some $|w|\leq 1$ and $c,\bar c\in [-\frac 1 4,\frac 14 ]$ satisfying $|c+\bar c|\leq \frac 1 4$. In such case, we write $\ell\sqcup\bar\ell\label{symb:line_concatenation}=(x,v+\bar v)$.

\begin{definition}[Additive functions]\label{def:additive_function}
A function $A:\cX\to E$ is called \emph{additive} if for all joinable lines $\ell,\bar\ell\in\cX$, we have
\[
A(\ell\sqcup\bar\ell) = A(\ell) + A(\bar\ell).
\]
\end{definition}

Recall the space $\Omega^1 C^\infty(\Lambda;E)$  consisting of smooth $E$-valued $1$-forms which we subsequently simply write as $\Omega^1C^\infty$. Let $A\in\Omega^1 C^\infty$ and $\ell\in \cX$, and consider the line integral of $A$ along $\ell$ which we denote by 
\[
A(\ell):=\int_\ell A.
\]  
Note the obvious abuse of notation that $A$ is both the integrated $1$-form and the $1$-form itself. This abuse of notation is harmless since one may recover the original $1$-form from the integrated $1$-form via the fundamental theorem of calculus. Crucially, the line integral of a $1$-form defines an additive function. For any $\beta\in (0,1]$, we denote by $\Omega_{\beta\gr}^1\label{symb:Omega_gr}$ the completion of $\Omega^1 C^\infty$ under the norm 
\begin{align}\label{eq:gr_norm}
|A|_{\beta\gr}:=\sup_{\ell\in\cX,|\ell|\neq 0}\frac{|A(\ell)|}{|\ell|^\beta}.
\end{align}

The same definition applies to $1$-forms defined on a rectangle $U\subset \R^2$ with the norm $|\Cdot|_{\beta\gr;U}$ defined similarly with line segments in $U$ of length at most $\frac 14$. The space is then denoted by $\Omega^1_{\beta\gr}(U)$.\footnote{When the domain is unspecified, we mean $U=\Lambda$.}

It is useful to be able to control differences of line integrals as well, for example, $A(\ell)$ should somehow vary continuously with respect to a suitable topology on lines. To that end, given a smooth $1$-form $A\in\Omega^1 C^\infty$ and a triangle $P$, we denote by $A(\partial P)$ the line integral of $A$ along $\partial P$ and by $|P|$ the area enclosed by the triangle. We now recall the following seminorm from \cite{CCHS2d}, namely:  
\begin{align}\label{eq:def_beta_tri}
|A|_{\beta\tri}:=\sup_{|P|>0}\frac{|A(\partial P)|}{|P|^{\beta/2}},
\end{align}
where the supremum is taken over all triangles in $\Lambda$ with diameter at most $\frac 1 4$. 
Let us define\label{symb:additive_norms} \[
|A|_\beta:=|A|_{\beta\gr}+|A|_{\beta\tri},
\]
and denote by $\Omega^1_\beta\label{symb:Omega_beta}$ the closure of $\Omega^1 C^\infty$ with respect to the norm $|\Cdot|_\beta$. It turns out that any $A$ in $\Omega^1_\beta$ possesses continuity with respect to the lines \cite[Thm.~3.11]{CCHS2d}.
Similarly as before, we can use this definition for any rectangle $U\subset \R^2$ in which case we consider triangles of diameter at most $\frac 14$, and we write the (semi)norms as $|\Cdot|_{\beta\tri;U}$ and $|\Cdot|_{\beta;U}:=|\Cdot|_{\beta\gr;U}+|\Cdot|_{\beta\tri;U}$ and the corresponding space as $\Omega^1_{\beta}(U)$.

The following embedding from \cite[Cor.\ 3.23]{Chevyrev19} is important for us: for any $\beta\in (0,1)$,
\begin{align}\label{eq:omega_beta_embedding}
    \Omega^1\lc^{\beta/2}(U)\hookrightarrow \Omega^1_\beta(U) \hookrightarrow \Omega^1_{\beta\gr}(U)\hookrightarrow \Omega^1\lc^{\beta-1}(U). 
\end{align}

\subsection{Recap of rough paths theory}\label{sec:rough_paths}

We briefly recall the rough path facts used throughout the paper. We refer to
\cite{FH20} for a detailed treatment.
Let $\alpha\in(\frac13,\frac12)$ and let $E$ be a finite-dimensional normed
vector space. For a path $X:[0,1]\to E$ we define $X_{s,t}:=X_t-X_s$ for $s,t\in [0,1]$. We denote by $\scC^\alpha_{\rmg}\label{symb:scC}([0,1];E)$ the space of
weakly geometric $\alpha$-Hölder rough paths. More precisely, an rough path
$\bfX=(X,\bX)\in \scC^\alpha_{\rmg}([0,1];E)$ consists of maps
\[
    X:[0,1]\to E,
    \qquad
    \bX:[0,1]^2_{\leq}\to E\otimes E,
\]
such that
\begin{equation}\label{eq:Hol_X_def}
    |X|_{\Hol\alpha}=
    \sup_{s<t}\frac{|X_{s,t}|}{|t-s|^\alpha}
    <\infty,
    \qquad
	 |\bX|_{C^{2\alpha}_2} :=
    \sup_{s<t}\frac{|\bX_{s,t}|}{|t-s|^{2\alpha}}
    <\infty,
\end{equation}
and  for all $s\leq u\leq t$, the second level $\bX$ satisfies
\[
    \bX_{s,t}
    =
    \bX_{s,u}+\bX_{u,t}+X_{s,u}\otimes X_{u,t},
\]
and the geometricity condition
\[
    \Sym\bX_{s,t}
    =
    \frac12 X_{s,t}\otimes X_{s,t}.
\]

Let $\bfX=(X,\bX)\in\scC^\alpha_{\rmg}([0,1];E)$ and let $F$ be another
finite-dimensional normed vector space. A path $Y:[0,1]\to F$ is said to be
controlled by $X$ if there exists a path
$Y':[0,1]\to L(E;F)$ such that
\begin{align}\label{eq:controlled_RP}
    Y_{s,t}
    =
    Y'_sX_{s,t}
    +
    R^Y_{s,t},
\end{align}
where $Y'$ is $\alpha$-Hölder and
\begin{align}\label{eq:remainder_RP}
  |R^Y|_{C^{2\alpha}_2} :=  \sup_{s<t}\frac{|R^Y_{s,t}|}{|t-s|^{2\alpha}}
    <\infty.
\end{align}
We denote the corresponding space by $\cD^{2\alpha}_X(F)\label{symb:controlled_rough_paths}$.

For $(Y,Y')\in\cD^{2\alpha}_X(L(E;F))$, the rough integral
$\int Y\,\dif\bfX$ is defined as the limit of compensated Riemann sums
\[
    \sum_{[u,v]\in\cP}
    \Bigl(
        Y_u X_{u,v}
        +
        Y'_u \bX_{u,v}
    \Bigr),
\]
as the mesh of the partition $\cP$ tends to zero. We use without further
comment the standard continuity estimates for this integral.

We shall also use the following standard consequences of controlled rough
path calculus. First, if $(Y,Y')\in\cD^{2\alpha}_X(F)$ and
$\Phi:F\to H$ is smooth, then $\Phi(Y)$ is again controlled by $X$. Second,
if $(Z,Z')\in\cD^{2\alpha}_X(H)$, then one can define integrals of the form
\[
    \int Y\dif Z
\]
by viewing both $Y$ and $Z$ as controlled by the same rough path $\bfX$.
Equivalently, the integral is obtained as the limit of the compensated sums
\[
    \sum_{[u,v]\in\cP}
    \Bigl(
        Y_u Z_{u,v}
        +
        Y'_u Z'_u \bX_{u,v}
    \Bigr),
\]
with the natural interpretation of the products in the relevant finite
dimensional spaces. In particular, products, inverses on Lie groups, adjoint
actions, and rough product rules used below are all understood in this
standard controlled rough path sense.

\subsection{Definition of rough additive functions}
Inspired from rough paths theory, we enhance Definition~\ref{def:additive_function} as follows: 
\begin{definition}[Enhanced additive functions]\label{def:enhanced_AF}
We define an \emph{enhanced additive function} to be a pair
$\bfA=(A,\bA)$ consisting of an additive function $A:\cX\to E$ and a map
$\bA:\cX\to E\otimes E$ satisfying the following:
\begin{itemize}
\item Chen's identity: for any joinable lines $\ell,\bar\ell\in\cX$, we have
\[
\bA(\ell\sqcup\bar\ell)
=
\bA(\ell)+\bA(\bar\ell)+A(\ell)\otimes A(\bar\ell);
\]
\item Weak geometricity condition: for any $\ell\in\cX$,
\[
\Sym\bA(\ell)
=
\frac12 A(\ell)\otimes A(\ell).
\]
\end{itemize}
We denote by $\fA$ the collection of enhanced additive functions.
\end{definition}
Let $A\in\Omega^1 C^\infty(\Lambda;E)$, and define $\ell_A:[0,1]\to E$ to be the path
\begin{equation}\label{eq:ell_A}
  \ell_A(t):=A((x,tv)).  
\end{equation}
Note that $\ell_A\in C^1$ and as such we can define 
\begin{align}\label{eq:def_bA_one_arg}
\bA(\ell):=\int^1_0\int^r_0 \dif\ell_A(r') \otimes\dif\ell_A(r)=\int^1_0\ell_A(r)\otimes\dif\ell_A(r).
\end{align}
In particular, any $A\in \Omega^1 C^\infty(\Lambda;E)$ induces a canonical enhanced additive function. 

Similar as for additive function, we define a norm for enhanced additive functions. Let $\alpha \in (\frac 1 3 ,1]$ and define given an enhanced additive function $\bfA=(A,\bA)$,  the following quantity:
\begin{equation} \label{e:RAF_gr_norm_2}
    \|\bA\|_{2\alpha\gr}:=\sup_{\ell\in\cX,|\ell|\neq 0}\frac{|\bA(\ell)|}{|\ell|^{2\alpha}}. 
\end{equation}
This leads us to define a metric on $\fA$ by 
\begin{align}\label{eq:RAF_gr_metric}
    \triple{\bfA;\bar\bfA}_{\alpha\gr}:=|A-\bar A|_{\alpha\gr}+\|\bA-\bar\bA\|_{2\alpha\gr}^{1/2},\qquad \bfA,\bar\bfA\in\fA,
\end{align}
where we recall the norm $|\Cdot|_{\alpha\gr}$ from \eqref{eq:gr_norm}.
We also set $\triple{\bfA}_{\alpha\gr}:=\triple{\bfA;\mathbf{0}}_{\alpha\gr}$.

\begin{definition}[Rough additive functions]\label{d:RAF}
Let $\alpha\in(\frac13,1]$. We define
\[
    \fA_{\alpha\gr}
    :=
    \{\bfA\in\fA:\triple{\bfA}_{\alpha\gr}<\infty\},
\]
equipped with the metric $\triple{\Cdot;\Cdot}_{\alpha\gr}$. Elements of
$\fA_{\alpha\gr}$ are called \emph{rough additive functions}.
Moreover, we denote by $\bfOmega^1_{\alpha\gr}$ the closure in
$\fA_{\alpha\gr}$ of the canonically enhanced smooth $1$-forms.
\end{definition}

The name \enquote{rough additive functions} is inspired from rough path theory, as we have extended the notion of additive functions from \cite[Def.\ 3.1]{CCHS2d} to the rougher regime.
\begin{remark}\label{rem:inclusions_RAF}
	For $\alpha\in (\frac 12,1] $, we have the embedding $A\in\Omega^1_{\alpha\gr}\hookrightarrow \bfOmega_{\alpha\gr}^1$ by Young integration.
	Furthermore, for $\alpha\in (\frac 13,1]$, we have a canonical projection~$\bfOmega^1_{\alpha\gr}\to\Omega^1_{\alpha\gr}$,
	and recalling the embedding $\Omega^1_{\alpha\gr}\hookrightarrow \Omega^1 C^{\alpha-1}$ from \eqref{eq:omega_beta_embedding}, we obtain a canonical projection $\bfOmega^1_{\alpha\gr}\to \Omega^1 C^{\alpha-1}$. 
\end{remark} 
For a line segment $\ell=(x,v)\in\cX$ and an enhanced additive function
$\bfA=(A,\bA)\in\fA$, we can associate a first level path $\ell_A$ and  second level enhancement is defined by
\[
    \ell_\bA(s,t)
    :=
    \bA(x+sv,(t-s)v),
\]
for $0\leq s\leq t\leq 1$. It is not difficult to check that  $\ell_\bA$ satisfies the classical Chen's relation and as such the following lemma should not come as a surprise:
\begin{lemma} \label{l:raf_to_rp}
    Let $\bfA$~$\in\fA_{\alpha\gr}$, then $\ell_\bfA:=(\ell_A,\ell_\bA)\in\scC^\alpha_\rmg$ and 
    $$|\ell_A|_{C^\alpha}\leq |\ell|^\alpha |A|_{\alpha\gr}, \ \ \ |\ell_\bA|_{C^{2\alpha}_2}\leq |\ell|^{2\alpha}\|\bA\|_{2\alpha\gr}.$$
    Furthermore, the map $\bfA\mapsto\ell_\bfA$ is Lipschitz continuous. 
    \end{lemma}

\subsection{Controlled gauge transformations}\label{sec:gauge_transformation}
Recall  the definition of gauge transformation for a smooth connection form $A\in \Omega^1 C^\infty(\Lambda;\fg)$: for any $g\in C^\infty(\Lambda;G)$ we define
\[
A^g:=\Ad_g A-\diff gg^{-1},
\]
where $\Ad_g\in L(\fg)$ is the adjoint action $\Ad_g X=gXg^{-1}\label{symb:Adg}$.  Of course, we can integrate $A^g$, but more importantly, we can write the integral of $A^g$ along $\ell=(x,v)\in\cX$ as 
	\begin{align}\label{eq:line_int_gauge_transformation_smooth}
	\int_\ell A^g=\int_0^1 \Ad_{\ell_g(t)}\dif\ell_A(t)-\int_0^1 \dif \ell_g(t)\ell_{g^{-1}}(t),
	\end{align}
	where we have used the notation $\ell_f:[0,1]\to F$ for a given function $f:\Lambda\to F$ and a set $F$ to mean the path $\ell_f(t):=f(x+tv)$.
    
    For $\bfA\in\bfOmega_{\alpha\gr}^1$, we would like to now make sense of $\bfA^g$. Since $\ell_\bfA$ is a rough path, a suitable controlled rough path structure for $\ell_g$ (as in~\eqref{eq:controlled_RP}) is necessary to make sense of the first integral on the right-hand side of \eqref{eq:line_int_gauge_transformation_smooth}.

Inspired by those notions, in particular the norm for remainders in controlled rough paths~\eqref{eq:remainder_RP}, we define the space  $C^{2\alpha}_2\label{symb:C2alpha2}(\Lambda;E)$ to consist of all functions $R:\Lambda\times \Lambda\to E$ such that
\begin{align}\label{eq:def_C_2_spatial}
\|R\|_{C^{2\alpha}_2}:=\sup_{x\neq y\in\Lambda}\frac{|R(x,y)|}{|x-y|^{2\alpha}}<\infty.
\end{align}
With these notations, we are ready to introduce the notion of controlled gauge transformations: 
\begin{definition}[Controlled gauge transformations]\label{def:gauge_transformation}
Let $\bfA=(A,\bA)\in\fA_{\alpha\gr}$. A \emph{controlled gauge transformation}
over $\bfA$ is a pair $(g,g')$ with
\[
    g\in C^\alpha(\Lambda;G),
    \qquad
    g'\in C^\alpha(\Lambda;L(\rmM_\C(N))),
\]
such that there exists $R^g\in C^{2\alpha}_2(\Lambda;\rmM_
\C(N))$ satisfying
\[
    g(y)-g(x)
    =
    g'(x)A(\ell^{x;y})
    +
    R^g(x,y)
\]
whenever $\ell^{x;y}\in\cX$ denotes the line segment from $x$ to $y$.
We denote the collection of controlled gauge transformations over $\bfA$ by
$\fG_A^{2\alpha}$. Finally, we use both notations $g\in\fG^{2\alpha}_A$ and $(g,g')\in\fG^{2\alpha}_A$ interchangeably.
\end{definition}

\begin{remark}
The definition of $\fG^{2\alpha}_A$ depends only on the first level $A$ of
$\bfA=(A,\bA)$. Nevertheless, we state it for $\bfA$, since we only consider
controlled gauge transformations for rough additive functions.
\end{remark}

\begin{remark}
Let $(g,g')\in\fG_A^{2\alpha}$ and let $\ell=(x,v)\in\cX$. We recall the notations
$
    \ell_g(t):=g(x+tv),
$
and
$
    \ell_{g'}(t):=g'(x+tv).
$
Then $(\ell_g,\ell_{g'})$ is a controlled rough path over
$\ell_\bfA=(\ell_A,\ell_\bA)$. Indeed,
\[
    \ell_g(t)-\ell_g(s)
    =
    \ell_{g'}(s)\ell_A(s,t)
    +
    R^g(x+sv,x+tv).
\]
Thus all rough integrals involving $g$ below are ordinary one-dimensional
controlled rough path integrals along line segments.
\end{remark}
	We compare $(g,g')\in \fG_A^{2\alpha}$ with $(\bar g,\bar g')\in \fG_{\bar A}^{2\alpha}$, which may not live in the same space, via the quantities
	\[
		\| (g,g');(\bar g,\bar g')\|_{2\alpha}:=|g'-\bar g'|_{\Hol\alpha}+\|R^g-R^{\bar g}\|_{C^{2\alpha}_2},
	\]
	and 
    \label{symb:controlled_gauge_metric}\begin{align}\label{eq:def_metric_controlled_RF}
	\triple{(g,g');(\bar g,\bar g')}_{2\alpha}:=|g-\bar g|_\infty+|g'-\bar g'|_\infty+ \| (g,g');(\bar g,\bar g')\|_{2\alpha}.
	\end{align}
	Technically speaking, one would have to specify the underlying $A$ and $\bar A$ in these norms, but we avoid doing so, as the role of $A$ and $\bar A$ is always clear from the context.

\begin{proposition}[Gauge transformation of rough additive functions]
\label{prop:controlled_gauge_matrix}
Let $\bfA=(A,\bA)\in\fA_{\alpha\gr}$ and let
$(g,g')\in\fG^{2\alpha}_A$.  Then 
        \begin{align}\label{e:gauge_trafo_connection:2}
            A^g(\ell):=\int^1_0 \Ad_{\ell_g(t)}\dif\ell_{\bfA}(t)-\dif\ell_g(t)\ell_{g^{-1}}(t), \quad \ell \in \cX,
        \end{align}
        is well-defined, and can be enhanced with $\bA^g$ to gauge transformed rough additive function $\bfA^g:=(A^g,\bA^g)\label{symb:gauge_transformed_RAF}\in \fA_{\alpha\gr}$. Furthermore, the map $(\bfA,g)\mapsto \bfA^g$ is locally Lipschitz continuous, i.e. for any $L>0$ there exists $C_L>0$ such that
    $$\triple{\bfA^g;\bar \bfA^{\bar g}}_{\alpha\gr}\leq C_L (\triple{\bfA;\bar \bfA}_{\alpha\gr}+\triple{(g,g');(\bar g,\bar g')}_{2\alpha}),$$
	for all $\triple{\bfA}_{\alpha\gr},\triple{\bar\bfA}_{\alpha\gr}\leq L$, $g\in \fG^{\alpha\gr}_{A}$ and $\bar g\in\fG^{2\alpha}_{\bar A}$ satisfying $\triple{g,g'}_{2\alpha},\triple{\bar g,\bar g'}_{2\alpha}\leq L$.
\end{proposition}

\begin{proof}
For each $\ell\in\cX$, the pair $\ell_\bfA$ is a weakly geometric
$\fg$-valued rough path, and by the definition of controlled gauge
transformations, $(\ell_g,\ell_{g'})$ is controlled by $\ell_\bfA$.
Lemma~\ref{lem:rp_gauge_transform} below, therefore gives a weakly geometric
$\fg$-valued rough path $\ell_{\bfA^g}$. This defines $A^g(\ell)$ and
$\bA^g(\ell)$. From here, one can readily verify the additivity, Chen's identity and the bounds which we leave as an exercise to the reader. 
\end{proof}
    As we have just noted, the proof boils down to a clean statement in terms of rough paths theory: 
\begin{lemma}[Gauge transformation of rough paths]\label{lem:rp_gauge_transform}
Let $\alpha\in(\frac13,\frac12)$ and let
$\bfX=(X,\bX)$ be a weakly geometric $\fg$-valued $\alpha$-Hölder rough path.
Let $(g,g')$ be a controlled rough path over $\bfX$ with values in
$\rmM_\C(N)$, and assume that $g_t\in G$ for every $t$. Then the expression
\[
    X^g_{s,t}
    :=
    \int_s^t \Ad_{g_r}\dif\bfX_r
    -
    \int_s^t \dif g_r g_r^{-1}
\]
is well-defined as an $\fg$-valued increment. Moreover, there is a canonical
second level
\[
    \bX^g_{s,t}
    :=
    \int_s^t X^g_{s,r}\otimes \dif X^g_r
\]
such that $\bfX^g=(X^g,\bX^g)$ is a weakly geometric $\fg$-valued
$\alpha$-Hölder rough path. The construction is continuous in the standard
rough path and controlled rough path topologies.
\end{lemma}

\begin{proof}
The two integrals are well-defined by the usual calculus of controlled rough
paths, working in the ambient matrix algebra $\rmM_\C(N)$. Indeed,
$\Ad_g$ and $g^{-1}$ are controlled by $\bfX$ as soon as $g$ is, since they
are obtained from $g$ by smooth maps on $G\subset\rmM_\C(N)$. So both rough integrals make sense by integration of controlled rough path against a (controlled) rough path.   

It remains to see that the resulting rough path is $\fg$-valued and weakly
geometric. Since $g$ is continuous and $[0,1]$ is compact, we may subdivide
$[0,1]$ into finitely many intervals such that, on each subinterval,
$g$ takes values in a single logarithmic chart $g_0O_G$. Define on each such subinterval
\[
    h_t:=\log(g_0^{-1}g_t).
\]
Since $g$ is controlled by $\bfX$ in the ambient matrix space and $\log$ is
smooth on the chosen chart, the usual composition theorem for controlled
rough paths shows that $h$ is controlled by $\bfX$ as an
$\rmM_\C(N)$-valued path. Moreover, $h$ actually takes values in the linear
subspace $\fg\subset\rmM_\C(N)$ and by projection we can view it as a genuine $\fg$-valued controlled rough path.  By the approximation result for
controlled rough paths, see \cite[Ex.~4.8]{FH20}, we may approximate
$\bfX$ by smooth $\fg$-valued rough paths $\bfX^\varepsilon$ and $h$ by
smooth $\fg$-valued controlled paths $h^\varepsilon$ over
$\bfX^\varepsilon$, in their corresponding (rough path) topologies with any weaker Hölder exponent. Setting
\[
    g^\varepsilon_t:=g_0\exp(h^\varepsilon_t)
\]
gives smooth $G$-valued paths converging to $g$ in the corresponding
controlled rough path topology. 

For the smooth approximations, set
\[
    X^{\varepsilon,g^\varepsilon}_{s,t}
    :=
    \int_s^t \Ad_{g_r^\varepsilon}\dif X_r^\varepsilon
    -
    \int_s^t \dif g_r^\varepsilon (g_r^\varepsilon)^{-1},
\]
and let $\bX^{\varepsilon,g^\varepsilon}$ be its canonical iterated integral.
By continuity of rough integration,
$
    (X^{\varepsilon,g^\varepsilon},\bX^{\varepsilon,g^\varepsilon})
$
converges to
$
    (X^g,\bX^g)
$
in any rough path topology with weaker H\"older exponent.
But since $(X^{\eps,g^\eps},\bX^{\eps,g^\eps}) $ is $\fg$-valued and geometric, we see that $\bfX^g$ is $\fg$-valued and weakly geometric. For now at least on the chosen subinterval, but concatenating over the finitely many subintervals together with Chen's identity finishes the proof.

\end{proof}

\begin{remark}\label{rem:left_action} Let $\bfA\in\bfOmega_{\alpha\gr}^1$.
Given $(g,g')\in\fG^{2\alpha}_A$ and $h\in C^\beta(\Lambda;G)$ with $\frac 1 2<\beta\leq 2\alpha$ such that $\beta+\alpha >1$, then $(hg,L_h g')\in \fG^\beta_A$, where
\[
L_g \colon \rmM_\C(N) \to \rmM_\C(N)\;,
\qquad L_h(y)=hy
\label{symb:left_multiplication}
\]
denotes left multiplication by $h$. Moreover, it is not difficult to show that $(\bfA^g)^h=\bfA^{hg}$, where to make sense of the left hand side, we use that $(h,0)\in \fG^\beta_{\bar A}$ for any $\bar\bfA\in \bfOmega_{\alpha\gr}^1$ .   
\end{remark}

\subsection{Parallel transport}
Let $\bfA\in\bfOmega^1_{\alpha\gr}$ and~$\ell \in \cX$. We consider the \emph{linear} rough differential equation~(RDE)
\begin{equation}\label{eq:holonomy_RDE}
\dif y(t)=y(t)\dif\ell_\bfA(t)\;, \qquad y(0)=1_G\;,
\end{equation}
which has a unique solution in the space of controlled rough paths.
Furthermore, we have continuity of the solution with respect to $\bfA$. In particular, we can readily get the following result: 
\begin{proposition}\label{prop:holonomy_bounds}
	Let $\bfA\in\fA_{\alpha\ax}^1$ and $\ell\in\cX$. Then there exists a unique solution $(y,y')\in\cD^{2\alpha}_{\ell_A}$ to \eqref{eq:holonomy_RDE} such that $y(t)\in G$ for all $t\in [0,1]$. 
	\end{proposition}
\begin{proof}
The existence, uniqueness and bounds are well-known results in rough path theory \cite{FH20}. 
It remains only to record why the solution is $G$-valued. This is a
one-dimensional statement. Since $\ell_\bfA$ is weakly geometric, it can be
approximated in any weaker Hölder rough path topology by smooth
$\fg$-valued rough paths. The corresponding smooth ODE solutions remain in
$G$, and the conclusion follows by continuity of the Itô--Lyons map and
closedness of $G\subset\rmM_\C(N)$.
\end{proof}
Now we define the notion of holonomy for rough additive functions:
\begin{definition}[Holonomy]\label{def:holonomy}
	For~$\bfA\in\bfOmega^1_{\alpha\gr}$ and~$\ell \in \cX$, we define the \emph{holonomy} of~$\bfA$ around~$\ell$ by $\hol(\bfA,\ell) := y(1)$ where~$y$ denotes the unique solution to the RDE~\eqref{eq:holonomy_RDE}. 
\end{definition}
For joinable line segments $\ell,\bar\ell\in\cX$, the uniqueness of the RDE
solution and Chen's identity imply
\[
    \hol(\bfA,\ell\sqcup\bar\ell)
    =
    \hol(\bfA,\ell)\hol(\bfA,\bar\ell).
\]
In particular,
\[
    \hol(\bfA,\ell)^{-1}
    =
    \hol(\bfA,\ell^\leftarrow),
\]
where $\ell^\leftarrow$ denotes the reversed line segment. We also define
$\hol(\bfA,\gamma)$ for a piecewise affine curve $\gamma$ by taking the
ordered product of the holonomies along its line-segment pieces.
We can also recover the familiar holonomy transformation formula under gauge transformation:
\begin{lemma}\label{lem:holonomy_gauge_tr}
Let $\bfA\in\fA^1_{\alpha\gr}$ and $g\in\fG^{2\alpha}_A$. Then for any $\ell\in\cX$ one has
\begin{align}\label{eq:holonomy_gauge_transform}
	\hol(\bfA^g,\ell)=g(\ell_0)\hol(\bfA,\ell)g(\ell_1)^{-1}.
\end{align}
where~$\ell_0$ denotes the starting and~$\ell_1$ the end point of the line~$\ell$.
\end{lemma}
\begin{proof}
This is again a one-dimensional statement along the line segment $\ell$.
For smooth paths it is the classical gauge transformation formula for
parallel transport. For weakly geometric rough paths, approximate the
driving rough path and the controlled gauge transformation locally in
logarithmic coordinates as in the proof of
Lemma~\ref{lem:rp_gauge_transform}. The formula then follows by continuity
of the Itô--Lyons map and of rough integration.
\end{proof}
\subsection{State space of rough axial gauges}\label{sec:RAF_ax}
The space $\fA_{\alpha\gr}$ (as well as $\bfOmega_{\alpha\gr}^1$), which we have previously defined, does not implement continuity of $\ell\mapsto\bfA(\ell)$: It only measures its growth as the length~$\abs[0]{\ell}$ of~$\ell$ varies. 
We have previously seen that the norm $|\Cdot|_\beta$ defined in Section~\ref{sec:AF} is incorporating such continuity through $|\Cdot|_{\beta\tri}$ defined in \eqref{eq:def_beta_tri}. 
We want to define a similar notion of continuity for elements in $\fA_{\alpha\gr}$  that is adapted to connection $1$-forms in \emph{axial gauges} (cf. \cite{Driver89,GKS89}).
\begin{definition}[Axial gauge]\label{def:axial_gauge}
We say a $1$-form $A\in\Omega^1 C^\infty$ is in the \emph{axial gauge}, if $A_2\equiv 0$. The space of all such $1$-forms is denoted by $\Omega^1_{\infty\ax}$ . Furthermore, we say $A\in\Omega^1 C^\infty$ is in the \emph{complete} axial gauge if $A$ is in the axial gauge and moreover $A_1(\Cdot,0)\equiv 0$, i.e.~$A_1$ vanishes on the~$x_1$-axis.
\end{definition}
The line integral of $1$-forms in the axial gauge vanishes on vertical line segments, which motivates the following definition: 
\begin{definition}[Non-vertical line segments]\label{def:non_vert_lines}
For a line segment $\ell=(x,v)\in\cX$, we define its horizontal projection
$\rmP_{\cX_{\rmh\ax}}(\ell)\label{symb:horizontal_projection}$ to be the line segment~$((x^1,0),(v^1,0))$.
We write
\[
    \ell\simhor\bar\ell
    \quad\text{if}\quad
    \rmP_{\cX_{\rmh\ax}}(\ell)
    =
    \rmP_{\cX_{\rmh\ax}}(\bar\ell).
\]
Given $\ell\simhor\bar\ell$, we define $\Area(\ell,\bar\ell)\label{symb:Area}$ to be the
Euclidean area enclosed between the two line segments. Finally, we denote by $\cX_{\rmh}$ the set of non-vertical line segments, i.e. those
$\ell=(x,v)$ with $v^1\neq 0$, and by $\cX_{\rmh\ax}\subset \cX_{\rmh}$ the subset of line segments lying on the $x_1$-axis.
\end{definition}

The concepts introduced in the previous definition are illustrated in Figure~\ref{fig:non_vert_lines}.
\begin{figure}[H]
	\centering
	\begin{minipage}[t]{.45\textwidth}
		\centering
		\tikzset{>={Latex[width=3pt,length=3pt]}}
		\begin{tikzpicture}[scale=3] 
			%
			%
			\draw (0,0) rectangle (1,1);
			%
			\draw[-{Latex[length=1.5mm]}, thick] (0,0) -- (1.1,0); 
			\draw[-{Latex[length=1.5mm]}, thick] (0,0) -- (0,1.1); 
			%
			\coordinate (A) at (0.2,0);       
			\coordinate (B) at (0.8,0);       
			\coordinate (C) at (0.2,0.5);     
			\coordinate (D) at (0.8,0.75);    
			\coordinate (E) at (0.9,0.2);    
			\coordinate (F) at (0.2,0.3);
			\coordinate (G) at (0.8,0.2);
			%
			\fill[blue!70, opacity = 0.2] (A) -- (C) -- (D) -- (B) -- cycle;
			%
			%
			\draw[densely dashed, semithick] (A) -- (C);
			\draw[densely dashed, semithick] (B) -- (D);
			%
			\draw[-{Latex[length=1.5mm]}, thick] (C) -- (D) node[midway, above left, scale=0.7] {$\ell$};
			\draw[-{Latex[length=1.5mm]}, thick] (F) -- (G) node[midway, above left, scale=0.7] {$\bar{\ell}$};
			\draw[-{Latex[length=1.5mm]}, thick, red] (A) -- (B) node[midway, below, scale=0.7] {$\mathring{\ell}$};
			\draw[-{Latex[length=1.5mm]}, thick, black] (0.9,0.2) -- (0.9,0.85) node[midway, right, scale=0.7] {$\tilde{\ell}$};
			%
			\foreach \pt in {A,B,C,D,E,F,G}
			\fill[black] (\pt) circle (0.3pt);
		\end{tikzpicture}
		\subcaption{Two line segments~$\ell \simhor \bar{\ell} \in \cX_{\rmh}$ and their horizontol projection~$\mathring{\ell} := \rmP_{\cX_{\rmh\ax}}(\ell)$ (in red). A vertical line segment~$\tilde{\ell} \notin \cX_{\rmh}$. The shaded area is the polygon~$U_\ell$ defined in~\eqref{e:U_ell} below. \label{subfig:U_ell}}
	\end{minipage}
	\hspace{2em}
	\begin{minipage}[t]{.45\textwidth}
		\centering
			\tikzset{>={Latex[width=3pt,length=3pt]}}
		\begin{tikzpicture}[scale=3] 
			%
			%
			\draw[-{Latex[length=1.5mm]}, thick] (0,0) -- (1.1,0); 
			\draw[-{Latex[length=1.5mm]}, thick] (0,0) -- (0,1.1); 
			%
			\coordinate (A) at (0.2,0);       
			\coordinate (A') at (0.1,0);       
			\coordinate (B) at (0.8,0);       
			\coordinate (B') at (0.9,0);   
			\coordinate (C) at (0.2,0.3);     
			\coordinate (C') at (0.2,0.455);     
			\coordinate (D) at (0.9,0.75);    
			\coordinate (E) at (0.1,0.5);
			\coordinate (F) at (0.8,0.2);
			\coordinate (F') at (0.8,0.7);
			%
			\draw[densely dashed, semithick] (A) -- (C');
			\draw[densely dashed, semithick] (A') -- (E);
			\draw[densely dashed, semithick] (B') -- (D);
			\draw[densely dashed, semithick] (B) -- (F');
			\path[name path=up1] (C) -- (0.2,1);
			\path[name path=up2] (B) -- (0.8,1);
			\path[name path=ell] (C) -- (D);
			\path[name path=ellbar] (E) -- (F);
			\path[name intersections={of=up1 and ellbar, by=I1}];
			\path[name intersections={of=up2 and ell, by=I2}];
			\path[name intersections={of=ell and ellbar, by=Imid}];
			%
			%
			\draw[thick,
			blue,
			double=blue!10,          
			double distance=0.5pt,     
			line cap=round,           
			opacity=0.3
			] (I1) -- (F);
			\draw[thick,
			red,
			double=red!10,          
			double distance=0.5pt,     
			line cap=round,           
			opacity=0.3
			] (E) -- (I1);
			%
			%
			\draw[thick,
			blue,
			double=blue!10,          
			double distance=0.5pt,     
			line cap=round,           
			opacity=0.3
			] (C) -- (I2);
			\draw[thick,
			green,
			double=red!10,          
			double distance=0.5pt,     
			line cap=round,           
			opacity=0.3
			] (I2) -- (D);
			%
			%
			%
			%
			\draw[-{Latex[length=1.5mm]}, thick] (C) -- (D) node[midway, above left, scale=0.7] {$\ell$};
			\draw[-{Latex[length=1.5mm]}, thick] (E) -- (F) node[midway, below left, scale=0.7] {$\bar{\ell}$};
			%
			%
			%
			\fill[blue!70, opacity = 0.2] (I2) -- (Imid) -- (F) -- cycle;
			%
			%
			\fill[blue!70, opacity = 0.2] (I1) -- (C) -- (Imid) -- cycle;
			%
			%
			\draw (0,0) rectangle (1,1);
			\draw[-{Latex[length=1.5mm]}, thick, blue!70] (B) -- (B') node[midway, above left, scale=0.7] {};
			\draw[-{Latex[length=1.5mm]}, thick, blue!70] (A) -- (A') node[midway, above left, scale=0.7] {};
			\draw[-{Latex[length=1.5mm]}, thick, black] (A) -- (B) node[midway, below, scale=0.7] {\phantom{$\mathring{\ell}$}};
			\foreach \pt in {A,A',B,B',C,D,E,F,I1,I2}
			\fill[black] (\pt) circle (0.3pt);
		\end{tikzpicture}
		\subcaption{Two line segments~$\ell$ and~$\bar{\ell}$ with their hori\-zontal overlaps (shaded in blue) as well as~$\bar{\ell}_a$ (red) and~$\ell_b$ (green), cf.~Section~\ref{sec:Kolmogorov}.  Note that we have~$\abs[0]{\bar{\ell}_b} = \abs[0]{\ell_a} = 0$. The value of~$\rho(\ell, \bar{\ell})$ is given by the shaded area plus the length of the two light blue line segments on the horizontal axis.}
		\label{subfig:horizontal_overlap}
	\end{minipage}%
	\captionsetup{format=plain}
	\caption{Graphic representation of line segments}
	\label{fig:non_vert_lines}
\end{figure}
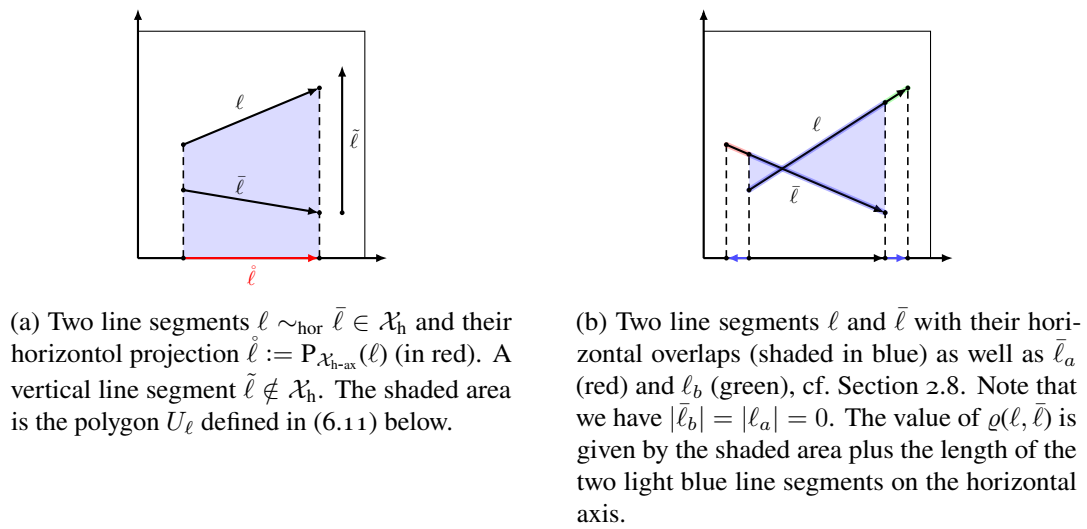

\begin{definition}[Enhanced additive functions in the axial gauge]\label{def:enhanced_AF_axial}
We denote by $\Axial\subset \fA$ the collection of enhanced additive functions in the \emph{axial gauge} with the defining property that $\bfA=(A,\bA)$ vanishes on $\cX\setminus \cX_\rmh$, i.e.
\[
A(\ell)=0, \qquad \bA(\ell)=0, \qquad \text{ for all } \ell\in\cX\setminus\cX_\rmh. 
\]
Moreover, we say $\bfA\in\Caxial$ is in the \emph{complete axial gauge} if it also vanishes on $\cX_{\rmh\ax}$.  
\end{definition}

\begin{remark}
Any $A\in\Omega^1_{\infty\ax}$ (in fact Lipschitz suffices) induces an element in $\Axial$ (similarly when in addition $A$ is in the complete axial gauge, it induces an element in $\Caxial$). 
\end{remark}
We now let $\alpha\in (\frac 1 3,\frac 12)$ and  introduce for $\bfA=(A,\bA)\in\Axial$ the quantities
\label{symb:RAF_hor_norms}\begin{align}\label{eq:RAF_hor_norms}
\begin{split}
|A|_{\alpha\hor}
&:=\sup_{\ell\simhor\bar\ell}\frac{|A(\ell)-A(\bar\ell)|}{\Area(\ell,\bar\ell)^\alpha}\;,
\\
\|\mathbb A\|_{2\alpha\hor}
&:=\sup_{\ell\simhor\bar\ell} \frac{|\bA(\ell)-\bA(\bar\ell)|}{(|\ell|\vee |\bar\ell|)^\alpha\Area(\ell,\bar\ell)^{\alpha}}
\;.
\end{split}
\end{align}
Denote $|\Cdot|_{\alpha\ax}:=|\Cdot|_{\alpha\gr}+|\Cdot|_{\alpha\hor}$ and similarly $\|\Cdot\|_{2\alpha\ax}:=\|\Cdot\|_{2\alpha\gr}+\|\Cdot\|_{2\alpha\hor}\label{symb:RAF_axial_norms}$ where $|\Cdot|_{\alpha\gr}$ and $\|\Cdot\|_{\alpha\gr}$ are defined in~\eqref{eq:gr_norm} and~\eqref{e:RAF_gr_norm_2} respectively.
Similarly as  \eqref{eq:RAF_gr_metric}, we define on $\Axial$ the metric  
\begin{equation}\label{eq:alpha_ax_norm}
\triple{\bfA;\bar\bfA}_{\alpha\ax}:=|A-\bar A|_{\alpha\ax}+\|\bA-\bar\bA\|_{\alpha\ax}^{1/2},\qquad \bfA,\bar\bfA\in\Axial.
\end{equation}
We also set $\triple{\bfA}_{\alpha\ax}:=\triple{\bfA;\bf0}_{\alpha\ax}$.

\begin{definition}[Rough additive functions in axial gauge]\label{def:RAF_ax}
Let $\alpha\in(\frac13,1]$. We define
\[
    \fA_{\alpha\ax}
    :=
    \{\bfA\in\Axial:\triple{\bfA}_{\alpha\ax}<\infty\},
\]
equipped with the metric $\triple{\Cdot;\Cdot}_{\alpha\ax}$.
Elements of
$\fA_{\alpha\ax}$ are called \emph{rough additive functions in the axial gauge}.
Moreover, we denote by $\bfOmega^1_{\alpha\ax}$ the closure in
$\fA_{\alpha\ax}$ of the canonically enhanced smooth $1$-forms in the axial gauge.
\end{definition}

The metric space $(\fA_{\alpha\ax},\triple{\Cdot;\Cdot}_{\alpha\ax})$ is complete due to completeness of the space of rough paths.
	We made a deliberate choice in writing  $\bfOmega_{\alpha\ax}^1$ with the subscript \enquote{ax}, short for axial, to emphasise that this space is purely for the axial gauge representatives. Indeed, it is not difficult to see that $\bfOmega_{\alpha\ax}^1\subset \bfOmega_{\alpha\gr}^1$ and together with the canonical projection $\bfOmega_{\alpha\ax}^1 \ni \bfA \mapsto A \in \Omega^1 C^{\alpha-1}$ from \rem{rem:inclusions_RAF}, one obtains that $A_2\equiv 0$. 

   Since for elements in $\fA_{\alpha\ax}$ we have additional control between line segments, we can improve~Lemma~\ref{l:raf_to_rp} to the following statement whose proof is straightforward:

    \begin{lemma} \label{l:raf_axial_to_rp}
    Let $\bfA$~$\in\fA_{\alpha\ax}$ and $\ell\simhor\bar\ell$, then $\ell_\bfA,\bar\ell_\bfA\in\scC^\alpha_\rmg$ and 
    $$|\ell_A-\ell_{\bar A}|_{C^\alpha}\leq \Area(\ell,\bar\ell)^\alpha |A|_{\alpha\hor}, \ \ \ |\ell_\bA-\bar\ell_\bA|_{C^{2\alpha}_2}\leq (|\ell|\vee |\bar\ell|)^{\alpha}\Area(\ell,\bar\ell)^\alpha\|\bA\|_{2\alpha\hor},$$
    and the usual corresponding Lipschitz estimates in $\bfA$ hold.     
    \end{lemma}

    \begin{remark}[Arbitrary axial-parallel rectangle]\label{rem:RAF_arbitrary_axial_domain}
The definition of axial gauge extends to any rectangle $U$ whose sides are parallel to the standard Euclidean axis, i.e.~any translation and dilation of $\Lambda$. In particular, we can extend the metrics defined above to $U$ (where all line segments are at most $\frac 1 4$) and we emphasise the domain with subscript, e.g.~$\triple{\Cdot;\Cdot}_{\alpha\ax;U}$. We denote the corresponding space by $\fA_{\alpha\ax}(U)$ or $\bfOmega^1_{\alpha\ax}(U)$ when it is the closure of smooth $1$-forms in the axial gauge. 
    \end{remark}

    \subsection{Domain extension of rough additive functions} 
  For two reasons that appear later we need to be able to extend a rough additive function in the axial gauge defined on $\Lambda$ to a rough additive function on $\R^2$ with support on $\Lambdaex\label{symb:Lambdaex}$. 

\begin{lemma}[Domain extension of rough additive functions]
\label{lem:domain_extension_RAF}
Let $\alpha\in(\frac13,\frac12)$ and let $A\in\Omega^1 C^1(\Lambda;\fg)$ be in the axial gauge with canonical lift $\bfA=(A,\bA)\in\bfOmega_{\alpha\ax}^1$. Then there exists $A^\ext\in\Omega^1 C^1(\R^2;\fg)$ such that $A^\ext=A$ on $\Lambda$ and $A^\ext$ vanishes outside $\Lambda^{+1}$,  denoting by
$\bfA^{\ext}=(A^{\ext},\bA^{\ext})$ its canonical lift,  one has
\begin{equation}\label{eq:domain_extension_RAF_bound}
    \triple{\bfA^{\ext}}_{\alpha\ax;\R^2}
    \lesssim_\alpha
    \triple{\bfA}_{\alpha\ax;\Lambda}.
\end{equation}
Given another $\bar A\in\Omega^1 C^1(\Lambda;\fg)$ with extension $\bar\bfA^\ext$ one has the estimate
\begin{equation}\label{eq:domain_extension_RAF_continuity}
    \triple{\bfA^{\ext};\bar\bfA^{\ext}}_{\alpha\ax;\R^2}
    \lesssim_\alpha 
    \triple{\bfA;\bar\bfA}_{\alpha\ax;\Lambda}.
\end{equation}

\end{lemma}

The proof of this lemma is basically a consequence of the following result from rough path theory:

\begin{lemma}[Weighted rough path transform]
\label{lem:four_point_weighted_rp_transform}
Let $(E,|\Cdot|)$ be a  finite-dimensional normed vector space, $\alpha\in(\frac13,\frac12)$ and let
\[
    \bfX^a=(X^a,\mathbb X^a),
    \qquad
    \bar\bfX^a=(\bar X^a,\bar{\mathbb X}^a),
    \qquad a\in\{0,1\},
\]
be weakly geometric $E$-valued $\alpha$-Hölder rough paths. Let
$h^0,h^1\in C^1([0,1];\R)$. For $a\in\{0,1\}$ define
\[
    Y^a_t:=\int_0^t h^a_r\dif X^a_r,
    \qquad
    \bar Y^a_t:=\int_0^t h^a_r\dif \bar X^a_r.
\]
Then $Y^a$ (resp.~$\bar Y^a$) are controlled rough path with respect to $X^a$ (resp.~$\bar X^a$) and have a canonical second level enhancement by denoted by $\bY^a$ (resp.~$\bar\bY^a$) (see~\cite[Rem.~4.12]{FH20}). 
Set
\[
    \Delta X^a:=X^a-\bar X^a,
    \qquad
    \Delta\mathbb X^a:=\mathbb X^a-\bar{\mathbb X}^a,
\]
and
\[
    \Delta X^{0,1}:=\Delta X^1-\Delta X^0,
    \qquad
    \Delta\mathbb X^{0,1}:=\Delta\mathbb X^1-\Delta\mathbb X^0.
\]
Then
\begin{equation}\label{eq:four_point_weighted_first_level}
\begin{split}
    &
   |(Y^1-\bar Y^1)-(Y^0-\bar Y^0)|_{\Hol\alpha}
\lesssim_\alpha
    |h^1|_{C^1}
    |\Delta X^{0,1}|_{\Hol\alpha}
    +
    |h^1-h^0|_{C^1}
    |\Delta X^0|_{\Hol\alpha},
\end{split}
\end{equation}
and, recalling the definition of the $C^{2\alpha}_2$-norm from~\eqref{eq:Hol_X_def}, we have
\begin{align}\label{eq:four_point_weighted_second_level}
\begin{split}
|(\bY^1-\bar\bY^1)-(\bY^0-\bar \bY^0)|_{C^{2\alpha}_2}
&\lesssim_\alpha
|h^1|_{C^1}^2
|\Delta\bX^{0,1}|_{C^{2\alpha}_2}
+
(|h^1|_{C^1}+|h^0|_{C^1})
|h^1-h^0|_{C^1}
|\Delta\bX^0|_{C^{2\alpha}_2}.
\end{split}
\end{align}
\end{lemma}

\begin{proof}
The bound regarding the first level~\eqref{eq:four_point_weighted_first_level} is immediate from Young integration by simply writing
\[
    (Y^1-\bar Y^1)-(Y^0-\bar Y^0)=
    \int_0^{\Cdot} h^1_r\dif \Delta X_r^{0,1}
    +
    \int_0^{\Cdot} (h^1_r-h^0_r)\dif \Delta X^0_r.
\]
For the second level, we start by introducing for $h\in C^1([0,1];\R)$
\[
    Y^{a,h}_t:=\int_0^t h_r\dif X^a_r,
    \qquad
    \bar Y^{a,h}_t:=\int_0^t h_r\dif \bar X^a_r,
\]
and defining the germs
\begin{align}\label{eq:Xi_ah}
    \Xi^{a,h}_{s,t}
    :=
    Y^{a,h}_s\otimes Y^{a,h}_{s,t}
    +
    h_s^2\bX^a_{s,t},
    \qquad
    \bar\Xi^{a,h}_{s,t}
    :=
    \bar Y^{a,h}_s\otimes \bar Y^{a,h}_{s,t}
    +
    h_s^2\bar\bX^a_{s,t}.
\end{align}
These are the germs whose sewing gives the integrals of $Y^{a,h}$ (resp~$\bar Y^{a,h})$ against itself, see~\cite[Rem.~4.12]{FH20}. 

We are interested in bounding $(\bY^1-\bar\bY^1)-(\bY^0-\bar \bY^0)$, which from the perspective of germs requires to study $(\Xi^{1,h^1}-\bar \Xi^{1,h^1})-(\Xi^{0,h^0}-\bar\Xi^{0,h^0})$. We write 
\begin{align*}
(\Xi^{1,h^1}-\bar \Xi^{1,h^1})-(\Xi^{0,h^0}-\bar\Xi^{0,h^0})&=(\Xi^{1,h^1}-\bar \Xi^{1,h^1})-(\Xi^{0,h^1}-\bar\Xi^{0,h^1}) \\
&\qquad +(\Xi^{0,h^1}-\bar \Xi^{0,h^1})-(\Xi^{0,h^0}-\bar\Xi^{0,h^0}).
\end{align*}
We set
\[
\Theta:=(\Xi^{1,h^1}-\bar \Xi^{1,h^1})-(\Xi^{0,h^1}-\bar\Xi^{0,h^1}),\qquad \Psi:=(\Xi^{0,h^1}-\bar \Xi^{0,h^1})-(\Xi^{0,h^0}-\bar\Xi^{0,h^0}). 
\]
We treat these terms separately. 

Before doing that, we need to introduce the following notation. 
For a smooth $F:[0,1]^2_{\leq}\to E\otimes E$ with $F_{s,s}=0$ for all $s\in [0,1]$, we define the operator
\[
\begin{split}
    \mathscr L^h_{s,u,t}F
    &:=
    \int_s^u\int_u^t
    (h_rh_q-h_s^2)
    \partial_r\partial_qF_{r,q}
    \dif q\dif r
    +
    (h_s^2-h_u^2)F_{u,t}.
\end{split}
\]
With a step of integrating by parts, one can readily verify the identity
\begin{align}\label{eq:Lh_IBP}
\begin{split}
    \mathscr L^h_{s,u,t}F
    &=
    h_u(h_t-h_u)F_{u,t}
    -
    h_s(h_t-h_s)F_{s,t}
    +
    h_s(h_u-h_s)F_{s,u}
    \\
    &\qquad
    -
    \int_s^u \dot h_rh_tF_{r,t}\dif r
    +
    \int_s^u \dot h_rh_uF_{r,u}\dif r
    \\
    &\qquad
    -
    \int_u^t h_u\dot h_qF_{u,q}\dif q
    +
    \int_u^t h_s\dot h_qF_{s,q}\dif q
    \\
    &\qquad
    +
    \int_s^u\int_u^t
    \dot h_r\dot h_qF_{r,q}
    \dif q\dif r .
\end{split}
\end{align}
To treat $\Theta$, we just write $h^1=h$, and define 
\[
F^{0,1}_{r,q}:=\Delta\bX^{0,1}_{r,q}
    =
    \bX^1_{r,q}-\bar\bX^1_{r,q}
    -
    \bX^0_{r,q}+\bar\bX^0_{r,q}.
\]
We use the convention
\[
    \delta Z_{s,u,t}:=Z_{s,t}-Z_{s,u}-Z_{u,t} ,\qquad Z:[0,1]^2_{\leq }\to E\otimes E,
\]
together with~\eqref{eq:Xi_ah}, to conclude
\begin{align}\label{eq:Theta}
\begin{split}
    \delta\Theta_{s,u,t}
    &=
    -\Big(
        Y^{1,h}_{s,u}\otimes Y^{1,h}_{u,t}
        -
        \bar Y^{1,h}_{s,u}\otimes \bar Y^{1,h}_{u,t}
        -
        Y^{0,h}_{s,u}\otimes Y^{0,h}_{u,t}
        +
        \bar Y^{0,h}_{s,u}\otimes \bar Y^{0,h}_{u,t}
    \Big)
    \\
    &\qquad
    +
    h_s^2
    \Big(
        F^{0,1}_{s,t}-F^{0,1}_{s,u}-F^{0,1}_{u,t}
    \Big)
    +
    (h_s^2-h_u^2)F^{0,1}_{u,t}.
\end{split}
\end{align}
We can approximate our rough paths by smooth paths in a worse regularity space, so for identities, we can assume our paths are smooth.
In particular, assuming all the paths $\bfX^a,\bar\bfX^a$ are smooth, we can write 
\[
    Y^{a,h}_{p,q}=\int_p^q h_r\dot X^a_r\dif r ,
    \qquad
    \bar Y^{a,h}_{p,q}=\int_p^q h_r\dot{\bar X}^a_r \dif r.
\]
Hence the first line of~\eqref{eq:Theta} is equal to
\[
\begin{split}
    -\int_s^u\int_u^t h_rh_q
    \Big(
        \dot X^1_r\otimes\dot X^1_q
        -
        \dot{\bar X}^1_r\otimes\dot{\bar X}^1_q
        -
        \dot X^0_r\otimes\dot X^0_q
        +
        \dot{\bar X}^0_r\otimes\dot{\bar X}^0_q
    \Big)
    \dif q\dif r.
\end{split}
\]
On the other hand, Chen's identity gives
\[
\begin{split}
    F^{0,1}_{s,t}-F^{0,1}_{s,u}-F^{0,1}_{u,t}
    &=
    X^1_{s,u}\otimes X^1_{u,t}
    -
    \bar X^1_{s,u}\otimes \bar X^1_{u,t}
    -
    X^0_{s,u}\otimes X^0_{u,t}
    +
    \bar X^0_{s,u}\otimes \bar X^0_{u,t}
    \\
    &=
    \int_s^u\int_u^t
    \Big(
        \dot X^1_r\otimes\dot X^1_q
        -
        \dot{\bar X}^1_r\otimes\dot{\bar X}^1_q
        -
        \dot X^0_r\otimes\dot X^0_q
        +
        \dot{\bar X}^0_r\otimes\dot{\bar X}^0_q
    \Big)
    \dif q\dif r .
\end{split}
\]
Combining these two, together with the fact that 
\[
\begin{split}
    \partial_r\partial_q F^{0,1}_{r,q}
    &=
    -\Big(
        \dot X^1_r\otimes\dot X^1_q
        -
        \dot{\bar X}^1_r\otimes\dot{\bar X}^1_q
        -
        \dot X^0_r\otimes\dot X^0_q
        +
        \dot{\bar X}^0_r\otimes\dot{\bar X}^0_q
    \Big),
\end{split}
\]
we conclude that
\begin{equation}\label{eq:delta_Theta_Lh}
    \delta\Theta_{s,u,t}
    =
    \int_s^u\int_u^t
    (h_rh_q-h_s^2)
    \partial_r\partial_q F^{0,1}_{r,q}
    \dif q\dif r
    +
    (h_s^2-h_u^2)F^{0,1}_{u,t}
    =
    \mathscr L^h_{s,u,t}F^{0,1}.
\end{equation}
But then by~\eqref{eq:Lh_IBP} (which holds then even if the rough paths are not smooth) we immediately see that 
\begin{equation}\label{eq:delta_Theta_bound_compact}
    |\delta\Theta_{s,u,t}|
    \lesssim_\alpha
    |h^1|_{C^1}^2
    |\Delta\bX^{0,1}|_{C^{2\alpha}_2}
    |t-s|^{1+2\alpha}.
\end{equation}

Now, doing similar computations with $\Psi$ yields
\[
\delta\Psi_{s,u,t}=(\mathscr L^{h^1}_{s,u,t}-\mathscr L^{h^0}_{s,u,t})F^{0},
\]
with $F^0:=\Delta \bX^0$. In particular, it is again obvious from~\eqref{eq:Lh_IBP} that 
\begin{equation}\label{eq:delta_Psi_bound_compact}
\begin{split}
    |\delta\Psi_{s,u,t}|
    &\lesssim_\alpha
    (|h^1|_{C^1}+|h^0|_{C^1})
    |h^1-h^0|_{C^1}
    |\Delta\bX^0|_{C^{2\alpha}_2}
    |t-s|^{1+2\alpha}.
\end{split}
\end{equation}

Combining \eqref{eq:delta_Theta_bound_compact} and
\eqref{eq:delta_Psi_bound_compact}, we get
\[
\begin{split}
|\delta ((\Xi^1-\bar\Xi^1)-(\Xi^0-\bar\Xi^0))_{s,u,t}|
&\lesssim_\alpha
|h^1|_{C^1}^2
|\Delta\bX^{0,1}|_{C^{2\alpha}_2}
|t-s|^{1+2\alpha}
\\
&\qquad
+
(|h^1|_{C^1}+|h^0|_{C^1})
|h^1-h^0|_{C^1}
|\Delta\bX^0|_{C^{2\alpha}_2}
|t-s|^{1+2\alpha}.
\end{split}
\]
Using that 
\[
\begin{split}
\left|
    (h^1_s)^2\Delta\bX^1_{s,t}
    -
    (h^0_s)^2\Delta\bX^0_{s,t}
\right|
&\lesssim
|h^1|_{C^1}^2
|\Delta\bX^{0,1}|_{C^{2\alpha}_2}
|t-s|^{2\alpha}
\\
&\qquad
+
(|h^1|_{C^1}+|h^0|_{C^1})
|h^1-h^0|_{C^1}
|\Delta\bX^0|_{C^{2\alpha}_2}
|t-s|^{2\alpha},
\end{split}
\]
the sewing lemma then implies~\eqref{eq:four_point_weighted_second_level}. 
\end{proof}

\begin{proof}[of Lemma~\ref{lem:domain_extension_RAF}]
Let $\tilde\Lambda:=[-\frac 1 2 ,\frac 3 2 ]^2$ and note that $\tilde\Lambda\subset\Lambdaex$. We construct the extension in two steps. First we extend vertically from
$\Lambda$ to the vertical strip
\[
    V:=[0,1]\times[-\tfrac 1 2,\tfrac 3 2],
\]
and then horizontally from $V$ to $\tilde \Lambda=[-\tfrac 1 2,\frac 3 2]^2$.
Let
\[
    \chi^\ua(r):=3-2r,\qquad r\in[1,\tfrac 3 2],
    \qquad
    \chi^\da(r):=1+2r,\qquad r\in[-\tfrac 12,0].
\]
Define $A^\ver=A^\ver_1\diff x^1$ on $V$ by
\[
    A^\ver_1(x_1,x_2)
    :=
    \begin{cases}
        \chi^\da(x_2) A_1(x_1,0), & x_2\in[-\tfrac 12,0],\\
        A_1(x_1,x_2), & x_2\in[0,1],\\
        \chi^\ua(x_2) A_1(x_1,1), & x_2\in[1,\tfrac 3 2].
    \end{cases}
\]
Thus $A^\ver=A$ on $\Lambda$, and $A^\ver_1$ vanishes on the horizontal sides
$x_2=-\frac 1 2$ and $x_2=\frac 3 2$. Since $A_1$ is $C^1$ on $\Lambda$, the function
$A^\ver_1$ is Lipschitz on $V$.

Let $\bfA^\ver=(A^\ver,\bA^\ver)$ be the canonical lift of $A^\ver$. We first
claim that
\begin{equation}\label{eq:vertical_extension_RAF_bound}
    \triple{\bfA^\ver}_{\alpha\ax;V}
    \lesssim
    \triple{\bfA}_{\alpha\ax;\Lambda},
\end{equation}
and the corresponding Lipschitz estimate. 
Note that it is enough to prove the required estimates for pairs of lines $\ell\simhor\bar\ell$ for which both are in one square, either $\Lambda$, the lower or upper square. The reason is that any other case can be reduced by a cutting/adding lines at the boundary of the squares together with additivity and Chen's identity. 
The estimate inside $\Lambda$ is immediate. So, it remains to discuss the upper
and lower squares. We only treat the upper square $V^\ua:=[0,1]\times[1,\frac 3 2]$, because the lower square requires just the same argument.

Let $\ell\simhor\bar\ell$ be line segments contained in
$V^\ua$. Write $\ell=(x,v)$ and $\bar\ell=(\bar x,\bar v)$. The case
$|v_1|=0$ is trivial in the axial gauge, so we assume $|v_1|>0$. Since the
two line segments are horizontally equivalent, their projections onto the upper side of $\Lambda$ coincide, which is the line segment  $\hat \ell:=((x_1,1),(v_1,0))$. 
We let $=\hat\ell_{\bfA}=(\hat\ell_A,\hat\ell_\bA)=:\bfX=(X,\bX)$ be the rough path sitting on the boundary.  By the growth bounds~Lemma~\ref{l:raf_to_rp}, we have
\begin{equation}\label{eq:boundary`_RP_bound}
    |X|_{\Hol\alpha}+|\mathbb X|_{C^{2\alpha}_2}^{1/2}
    \lesssim
    \triple{\bfA}_{\alpha\ax;\Lambda}|v_1|^\alpha .
\end{equation}
Define
\[
h^{\ell}:=\chi^\ua (x_2+t v_2)=3-2x_2-2tv_2,\qquad h^{\bar\ell}:=\chi^\ua (\bar x_2+t \bar v_2)=3-2\bar x_2-2t\bar v_2,
\]
and note that $|h^\ell|_{C^1}+|h^{\bar\ell}|_{C^1}\lesssim 1$ and 
\[
|h^\ell-h^{\bar\ell}|_{C^1}\lesssim \int_0^1 |h^\ell_t-h^{\bar\ell}_t|\dif t\lesssim  \frac{\Area(\ell,\bar\ell)}{|v_1|}
    \lesssim
    \Area(\ell,\bar\ell)^\alpha |v_1|^{-\alpha}
\]
where in the last inequality we have used that $\ell$ and $\bar\ell$
contained in the fixed square $V^\ua$ implies
$\Area(\ell,\bar\ell)\lesssim |v_1|$. 

Set 
\[
Y^\ell_t=\int_0^t h^\ell_r \dif X_r,\qquad Y^{\bar\ell}_t=\int_0^t h^{\bar\ell}_r\dif X_r,
\]
so that $A^\ver(\ell)=Y^\ell_1$ and $A^\ver(\bar\ell)=Y^{\bar\ell}_1$. Moreover, we are in the context of Lemma~\ref{lem:four_point_weighted_rp_transform}. In particular, using $|v_1|\leq |\ell|\vee|\bar\ell|$ and~\eqref{eq:four_point_weighted_first_level} implies 
\[
|A^\ver(\ell)|\lesssim \triple{\bfA}_{\alpha\gr;\Lambda}|v_1|^\alpha\lesssim \triple{\bfA}_{\alpha\gr;\Lambda}|\ell|^\alpha, 
\]
and 
\[
|A^\ver(\ell)-A^\ver(\bar\ell)|\lesssim \triple{\bfA}_{\alpha\ax;\Lambda}  \Area(\ell,\bar\ell)^\alpha |v_1|^{-\alpha}|v_1|^\alpha \lesssim \triple{\bfA}_{\alpha\ax;\Lambda} \Area(\ell,\bar\ell)^\alpha .
\]
Denoting by $\bY^\ell$ and $\bY^{\bar\ell}$ the canonical iterated integrals of $Y^\ell$ and $Y^{\bar\ell}$ respectively, by~\eqref{eq:four_point_weighted_first_level}, we get
\[
\begin{split}
|\bA^\ver(\ell)|
&=
|\bY^\ell_{0,1}|
\lesssim
|h^\ell|_{C^1}^2
|\bX|_{C^{2\alpha}_2}
\lesssim
\triple{\bfA}_{\alpha\ax;\Lambda}^2
|v_1|^{2\alpha}
\lesssim
\triple{\bfA}_{\alpha\ax;\Lambda}^2
|\ell|^{2\alpha},
\end{split}
\]
as well as
\[
\begin{split}
|\bA^\ver(\ell)-\bA^\ver(\bar\ell)|
&=
|\bY^\ell_{0,1}-\bY^{\bar\ell}_{0,1}|
\\
&\lesssim
\bigl(|h^\ell|_{C^1}+|h^{\bar\ell}|_{C^1}\bigr)
|h^\ell-h^{\bar\ell}|_{C^1}
|\bX|_{C^{2\alpha}_2}
\\
&\lesssim
\triple{\bfA}_{\alpha\ax;\Lambda}^2
\Area(\ell,\bar\ell)^\alpha
|v_1|^{-\alpha}|v_1|^{2\alpha}
\\
&\lesssim
\triple{\bfA}_{\alpha\ax;\Lambda}^2
\bigl(|\ell|\vee|\bar\ell|\bigr)^\alpha
\Area(\ell,\bar\ell)^\alpha,
\end{split}
\]
where we used $|v_1|\leq |\ell|\vee|\bar\ell|$ in the last step. This finishes the proof of~\eqref{eq:vertical_extension_RAF_bound}. The corresponding Lipschitz estimate follows by similar application of Lemma~\ref{lem:four_point_weighted_rp_transform}.

We now extend horizontally. Let
\[
    \rho_1(r)
    :=
    \begin{cases}
        -r, & r\in[-\tfrac 12,0],\\
        r, & r\in[0,1],\\
        2-r, & r\in[1,\tfrac 3  2],
    \end{cases}
    \qquad
    \chi_1(r)
    :=
    \begin{cases}
        1+2r, & r\in[-\tfrac 12,0],\\
        1, & r\in[0,1],\\
        3-2r, & r\in[1,\tfrac 3 2].
    \end{cases}
\]
For $x\in\tilde\Lambda$, define
\[
    A^{\ext}_1(x_1,x_2)
    :=
    \chi_1(x_1)A^\ver_1(\rho_1(x_1),x_2),
\]
and set $A^{\ext}_1=0$ on $\R^2\setminus\tilde\Lambda$. Then $A^{\ext}=A$ on
$\Lambda$, $A^{\ext}_1$ vanishes outside $\Lambdaex$, and
$A^{\ext}_1$ is Lipschitz. Indeed, the vertical extension vanishes on the
horizontal sides of $V$, while the horizontal cutoff vanishes on the vertical
sides of $\tilde\Lambda$.

It remains to check that the horizontal extension preserves the axial RAF bounds. We explain the argument on the right side
\[
    S^\rightarrow:=[1,\tfrac 32]\times[-\tfrac 1 2,\tfrac 3 2],
\]
as the same argument applies to the left side. 
Let  $\ell=(x,v)$ be a line segment in
$S^\rightarrow$, and set
\[
    x^\rho:=(2-x_1,x_2),
    \qquad
    v^\rho:=(-v_1,v_2).
\]
Set
\[
    h_t:=3-2x_1-2tv_1,
    \qquad
    X^\ell_{s,t}:=
    -A^\ver((x^\rho+sv^\rho,(t-s)v^\rho)),
\]
and define the second level by
\[
    \mathbb X^\ell_{s,t}:=
    \bA^\ver((x^\rho+sv^\rho,(t-s)v^\rho)).
\]
It is not difficult to then verify that $\bfX^\ell:=(X^\ell,\bX^\ell)$ is $\alpha$-H\"older weakly geometric rough path. Moreover, given another $\bar\ell\sim\ell$ in $S^\to$, we have the following bounds by Lemma~\ref{l:raf_axial_to_rp} 
\begin{equation}\label{eq:RP_bound_horizontal}
   \begin{split}
   |X^\ell|_{\Hol\alpha}+|\mathbb X^\ell|_{C^{2\alpha}_2}^{1/2}
    &\lesssim
    \triple{\bfA}_{\alpha\ax;\Lambda}|\ell|^\alpha \\
     |X^\ell-X^{\bar\ell}|_{\Hol\alpha}
    &\lesssim
    \triple{\bfA}_{\alpha\ax;\Lambda}\Area(\ell,\bar\ell)^\alpha \\
    |\mathbb X^\ell-\bX^{\bar\ell}|_{C^{2\alpha}_2}
    &\lesssim \triple{\bfA}_{\alpha\ax;\Lambda}(|\ell|\vee|\bar\ell|)^\alpha \Area(\ell,\bar\ell)^\alpha. 
\end{split}
\end{equation}

Note that $|h|_{C^1}\lesssim 1$. In particular, we have gathered all ingredients to apply Lemma~\ref{lem:four_point_weighted_rp_transform}, and since the details are similar to the vertical extension step, we leave the remaining of those details for the reader to finish the proof.  
\end{proof}

   \subsection{Kolmogorov continuity theorem}\label{sec:Kolmogorov}
In this section, we prove a Kolmogorov continuity theorem for rough additive functions in the complete axial gauge, as we have the complete axial gauge representative of the YM measure in mind. 
To that end, let us now fix a probability space $(\Sigma,\cF,\mathbb P)$ and a finite-dimensional normed vector space $(E,|\Cdot|)$. With these notations, we can now write the Kolmogorov continuity theorem as follows:

    \begin{theorem}[Kolmogorov continuity]\label{thm:Kolmogorov} Let  $\bfA=(A,\bA)$ be a random field with values in~$\Caxial$.  Assume that for some $\alpha\in (\frac 1 3,\frac 1 2]$, $p>{16}/\alpha$ and $L>0$
\begin{equation}\label{eq:Kolmogorov_moment_bounds}
\begin{split}
\E[|A(\ell)-A(\bar\ell)|^p]&\leq L\Area(\ell,\bar\ell)^{\alpha p}\\
\E[|\bA(\ell)-\bA(\bar\ell)|^{p/2}]&\leq L(|\ell|\vee|\bar\ell|)^{\alpha p/2}\Area(\ell,\bar\ell)^{\alpha p/2},
\end{split}
\end{equation}
for all $\ell\simhor\bar\ell$. 

    Then for any $0<\beta<\alpha-16/p$ there exists a modification $\bfA^\sharp=( A^\sharp,\bA^\sharp)\in\fA_{\beta\ax}$ of $(A,\bA)$ such that 
    \begin{align}\label{eq:Kolmogorov_result_bound}
    \E[\triple{\bfA^\sharp}^p_{\beta\ax}]\leq  C^pL
   (1-2^{-\frac 12(\alpha-16/p)})^{-p}
    ( 1-2^{-\frac p4(\alpha-\frac{16}{p}-\beta)}
    )^{-1},
    \end{align}
    for some universal constant $C>0$. Moreover, if a second random field $\bar\bfA=(\bar A,\bar\bA)\in\Caxial$ satisfies the above moment bounds with the same constant $L$, and the differences 
    \[
   \Delta A:= A-\bar A, \qquad \Delta\bA:=\bA-\bar \bA, 
    \]
    satisfy the same two bounds but with constant $\varepsilon L$, then the two modifications $\bfA^\sharp$ and $\bar\bfA^\sharp$ of $\bfA$ and $\bar\bfA$ respectively, satisfy the stability estimate
   \begin{align}\label{eq:Kolmogorov_result_bound_stability}
   \E[\triple{\bfA^\sharp;\bar\bfA^\sharp}^p_{\beta\ax}]\leq  C^p \eps L
   (1-2^{-\frac 12(\alpha-16/p)})^{-p}
    ( 1-2^{-\frac p4(\alpha-\frac{16}{p}-\beta)}
    )^{-1}.
    \end{align}
    \end{theorem}
A Kolmogorov continuity theorem was previously proved for the additive function space $\Omega_\beta^1$ in \cite[Lem.~4.11]{CCHS2d} and \cite[Lem.~3.12]{CCHS3d}. Our proof is slightly different as we are dealing with the rough case and, on top of that, the line segments that we consider in our norms are different. 
The proof combines ideas that are fairly standard in proving Kolmogorov continuity theorems. 
Since the remainder of this section consists only of the proof, the reader uninterested in the technicalities is welcome to skip ahead to Section~\ref{sec:solution_theory}.

    To prepare for the proof, we introduce several concepts. There is a natural orientation of lines in $\cX_h$ (the set of non-vertical lines), 
namely, each $\ell\in\cX_h$ points either \emph{left} or \emph{right}.
Fix two lines with the \emph{same} orientation,
$\ell=(x,v)$ and $\bar\ell=(\bar x,\bar v)$. 
For~$\hat\ell\in\{\ell,\bar\ell\}$, let $\hat\ell_{\hov(\ell,\bar\ell)}$ be the longest line segment on~$\hat{\ell}$ such that $\ell_{\hov(\ell,\bar\ell)}\simhor \bar\ell_{\hov(\ell,\bar\ell)}$. We chose the letters $\hov$ for \emph{\enquote{horizontal} overlap}. It is convenient to extend $\hat\ell_{\hov(\ell,\bar\ell)}$ for line segments with different orientations to the line segment $(0,0)$ (a line with zero length!). Define $[\ell]:=|\rmP_{\cX_{\rmh\ax}}(\ell)|$ to be the \emph{horizontal length}. Note that any two line segments $\ell,\bar\ell$ in the same orientation can be uniquely decomposed into 
\begin{equation} \label{e:decomp_line}
	\ell=\ell_a \sqcup \ell_{\hov(\ell,\bar\ell)}\sqcup \ell_b, \qquad \bar\ell=\bar\ell_a \sqcup \bar\ell_{\hov(\ell,\bar\ell)}\sqcup \bar\ell_b,
\end{equation}
where it is possible that some of these line segments have zero length. Set
\[
\out(\ell,\bar\ell):=[\ell_a]+[\ell_b]+[\bar\ell_a]+[\bar\ell_b].
\]
The letter $\out$ is chosen as it is the horizontal combined length of the \emph{outer} line segments that do not \enquote{horizontally} overlap.   We now define 
    \begin{align}
        \begin{split}\rho(\ell,\bar\ell):=\begin{cases}\Area(\ell_{\hov(\ell,\bar\ell)},\bar\ell_{\hov(\ell,\bar\ell)})+\out(\ell,\bar\ell) &  \text{ if } \ell,\bar\ell \text{ have same orientation},\\
        [\ell]+[\bar\ell] & \text{ otherwise}.
        \end{cases}
        \end{split}
    \end{align}
    It is not difficult to check that $\rho$ is a metric on $\cX_\rmh$; see Figure~\ref{subfig:horizontal_overlap} for an illustration of the previous quantities.
  
   \begin{proof}[of Theorem~\ref{thm:Kolmogorov}] First let $\beta<\alpha-16/p$ and take $\gamma:=(\beta+\alpha-16/p)/2$ so that $(\alpha-16/p)/2<\gamma<\alpha-16/p$. 
   Using the fact that $A$ and $\bA$ vanish on the horizontal axis,
   additivity and Chen's identity, one can obtain for all
   $\ell,\bar\ell\in\cX_{\rmh}$ the inequality
    \begin{align}\label{eq:Kolmogorov_rho_bounds}
    \begin{split}
        \E[|A(\ell)-A(\bar\ell)|^p]
        &\leq C^pL\rho(\ell,\bar\ell)^{p\alpha},\\
        \E[|\bA(\ell)-\bA(\bar\ell)|^{p/2}]
        &\leq
        C^p L
        (|\ell|\vee|\bar\ell|)^{\alpha p/2}
        \rho(\ell,\bar\ell)^{\alpha p/2},
        \end{split}
    \end{align}
    for some $C>0$.

    Let $D_N$ denote the set of line segments in $\cX_{\rmh}$ with starting
    and ending points having dyadic coordinates of scale $2^{-N}$. Let
    $D=\cup_{N\geq 0}D_N$ and note that $D$ is dense in $(\cX_{\rmh},\rho)$.
    \paragraph{The first order term.}
    Given distinct $\ell,\bar\ell\in D$, we want to bound
    $A(\ell)-A(\bar\ell)$. To that end, we use the following dyadic
    construction. Given $\ell\in D$ and a scale $N$, we choose
    \[
    \ell_N,\ell_{N+1},\ldots,\ell_{N+m}=\ell,
    \qquad
    \ell_M\in D_M,
    \]
    by replacing the endpoints of $\ell$ at scale $2^{-M}$ by neighbouring
    dyadic points, ordered according to the horizontal orientation of $\ell$.
    The approximating endpoints can be chosen so that the segment remains
    non-vertical. Moving endpoints by $O(2^{-M})$ changes the outer horizontal
    part and the area between overlapping parts by $O(2^{-M})$.
    Thus,
    \[
        \rho(\ell_M,\ell_{M+1})\leq C_0 2^{-M},
    \]
    for some constant $C_0>0$. Moreover, if $\ell,\bar\ell\in D$ and
    \[
        2^{-N}<\rho(\ell,\bar\ell)\leq 2^{-N+1},
    \]
    then we can choose the approximating sequence such that
    \[
        \rho(\ell_N,\bar\ell_N)\leq C_0 2^{-N},
    \]
    where we have possibly increased the value of $C_0>0$. 

    With all these notations, we write
    \[
    \begin{split}
        A(\ell)-A(\bar\ell)
        &=
        A(\ell_N)-A(\bar\ell_N)
        +
        \sum_{i\geq0}
        \bigl(A(\ell_{N+i+1})-A(\ell_{N+i})\bigr)  \\
        &\quad-
        \sum_{i\geq0}
        \bigl(A(\bar\ell_{N+i+1})-A(\bar\ell_{N+i})\bigr),
    \end{split}
    \]
    where the chains are kept constant after reaching their terminal values.
    Each term at scale $N+i$ is bounded by
    $K_A2^{-\gamma(N+i)}$, up to a universal constant, where
    \[
        K_A
        :=
        \sup_{M\geq0}
        \sup_{\substack{a,b\in D_M\\ \rho(a,b)\leq C_0 2^{-M}}}
        \frac{|A(a)-A(b)|}{2^{-\gamma M}}.
    \]
    Therefore, using
    \[
    \sum_{i\geq 0}2^{-\gamma(N+i)}
    =2^{-\gamma N}(1-2^{-\gamma})^{-1}\lesssim 2^{-\gamma N} (1-2^{-\frac 1 2 (\alpha-16/p)})^{-1},
    \]
    we get
    \[
        |A(\ell)-A(\bar\ell)|
        \lesssim
        (1-2^{-\frac 1 2 (\alpha-16/p)})^{-1}K_A2^{-\gamma N}
        \lesssim
        (1-2^{-\frac 1 2 (\alpha-16/p)})^{-1} K_A\rho(\ell,\bar\ell)^\gamma.
    \]
    By \eqref{eq:Kolmogorov_rho_bounds},
    \begin{align*}
    \begin{split}
        \E[K_A^p]
        &\leq
        \sum_{M\geq0}
        \sum_{\substack{a,b\in D_M\\ \rho(a,b)\leq C_0 2^{-M}}}
        2^{\gamma pM}\E[|A(a)-A(b)|^p] \\
        &\lesssim
        C^pL\sum_{M\geq0}
        2^{8M}2^{-(\alpha-\gamma)pM}
        \lesssim
        C^p L,
    \end{split}
    \end{align*}
    where we have used the bound $|D_M^2|\lesssim2^{8M}$ and the fact that
    series converges and is uniformly bounded because of
    $\gamma< \alpha-16/p$. 

    \paragraph{The second order term.}
    Let us now consider the second order process. First we define
    \[
        K_{\bA}
        :=
        \sup_{M\geq0}
        \sup_{\substack{a,b\in D_M\\ (|a|\vee |b|)\rho(a,b)\leq C_0 2^{-M}}}
        \frac{|\bA(a)-\bA(b)|}{2^{-\gamma M}}.
    \]
    Given distinct $\ell,\bar\ell\in D$, choose $N$ such that
    \[
        2^{-N}
        <
        (|\ell|\vee|\bar\ell|)\rho(\ell,\bar\ell)
        \leq
        2^{-N+1}.
    \]
    We can do the same construction as for the first order term, and, increasing $C_0$ if necessary, we get a sequence of line segments such that 
    \[
        (|\ell_N|\vee|\bar\ell_N|)\rho(\ell_N,\bar\ell_N)
        \leq C_0 2^{-N},
    \]
    and, for all relevant $M\geq N$,
    \[
    \begin{split}
        (|\ell_M|\vee|\ell_{M+1}|)
        \rho(\ell_M,\ell_{M+1})
        &\leq C_0 2^{-M},\\
        (|\bar\ell_M|\vee|\bar\ell_{M+1}|)
        \rho(\bar\ell_M,\bar\ell_{M+1})
        &\leq C_0 2^{-M}.
    \end{split}
    \]
    Now note that
    \[
    |\bA(\ell_N)-\bA(\bar\ell_N)|\leq K_\bA 2^{-\gamma N}.
    \]
    Hence
    \[
    \begin{split}
        \bA(\ell)-\bA(\bar\ell)
        &=
        \bA(\ell_N)-\bA(\bar\ell_N)
        +
        \sum_{i\geq0}
        \bigl(\bA(\ell_{N+i+1})-\bA(\ell_{N+i})\bigr)\\
        &\quad-
        \sum_{i\geq0}
        \bigl(\bA(\bar\ell_{N+i+1})-\bA(\bar\ell_{N+i})\bigr),
    \end{split}
    \]
    where, again, the chains are kept constant after reaching their terminal
    values. Each summand is bounded by $K_{\bA}2^{-\gamma(N+i)}$, up to a
    universal constant. Thus
    \[
        |\bA(\ell)-\bA(\bar\ell)|
        \lesssim
      (1-2^{-\frac 1 2 (\alpha-16/p)})^{-1}K_{\bA}2^{-\gamma N}
        \lesssim
        (1-2^{-\frac 1 2 (\alpha-16/p)})^{-1}
        K_{\bA}
        \bigl((|\ell|\vee|\bar\ell|)
        \rho(\ell,\bar\ell)\bigr)^\gamma.
    \]
    Using \eqref{eq:Kolmogorov_rho_bounds} and $|D_M^2|\lesssim2^{8M}$,
    \[
    \begin{split}
        \E[K_{\bA}^{p/2}]
        &\leq
        \sum_{M\geq0}
        \sum_{\substack{a,b\in D_M\\ (|a|\vee |b|)\rho(a,b)\leq C_0 2^{-M}}}
        2^{\gamma pM/2}
        \E[|\bA(a)-\bA(b)|^{p/2}]\\
        &\lesssim
        C^pL
        \sum_{M\geq0}
        2^{8M}2^{-(\alpha-\gamma)pM/2}\lesssim C^pL (1-2^{\frac p 4 (\beta-\alpha+16/p)})^{-1},
    \end{split}
    \]
    where we have used that $(\alpha+\beta-16/p)/2=\gamma< \alpha-16/p$ for the series to converge.
    \paragraph{Construction of the modification.}
    Let $\Sigma_0$ be the full-probability event on which the above dyadic
    seminorms are finite and the algebraic identities hold for all dyadic line
    segments. This event has probability one since $D$ is countable.

    For $\ell\in\cX_{\rmh}$, define $A^\sharp(\ell)$ as the limit of
    $A(\ell_n)$ along any sequence $\ell_n\in D$ with
    $\rho(\ell_n,\ell)\to0$, and set $A^\sharp(\ell)=0$ for vertical line
    segments. The first-level estimate on $D$ shows that the limit exists and
    is independent of the approximating sequence. Moreover,
    \[
        \E[|A(\ell_n)-A(\ell)|^p]
        \lesssim
        C^pL\rho(\ell_n,\ell)^{\alpha p}
        \to0,
    \]
    so $A^\sharp(\ell)=A(\ell)$ almost surely for every fixed $\ell$.

    Similarly, for $\ell\in\cX_{\rmh}$, define $\bA^\sharp(\ell)$ as the
    limit of $\bA(\ell_n)$ along any sequence $\ell_n\in D$ with
    $\rho(\ell_n,\ell)\to0$, and set $\bA^\sharp(\ell)=0$ for vertical line
    segments. The limit exists and is independent of the approximating
    sequence, because
    \[
        |\bA(\ell_n)-\bA(\ell_m)|
        \lesssim
        (1-2^{-\frac 1 2 (\alpha-16/p)})^{-1}
        K_{\bA}
        \bigl((|\ell_n|\vee|\ell_m|)
        \rho(\ell_n,\ell_m)\bigr)^\gamma,
    \]
    and all lengths are uniformly bounded. For every fixed $\ell\in\cX_{\rmh}$,
    the same estimate together with \eqref{eq:Kolmogorov_rho_bounds} gives
    \[
        \bA^\sharp(\ell)=\bA(\ell)
        \qquad\text{almost surely}.
    \]
    Thus $\bfA^\sharp=(A^\sharp,\bA^\sharp)$ is a modification of $\bfA$.

    The algebraic identities pass to the modification by approximation.
    Indeed, they hold on $D$ on the event $\Sigma_0$; approximating all
    involved endpoints simultaneously, preserving the relevant concatenations
    and horizontal-comparability relations, and then passing to the limit using
    the Hölder estimates gives additivity, Chen's identity and the axial
    vanishing conditions for $\bfA^\sharp$.

    It remains to identify the obtained bounds with the $\gamma$-axial norm. If
    $\ell\simhor\bar\ell$, then $
        \rho(\ell,\bar\ell)=\Area(\ell,\bar\ell)$. From here we see
    \[
        |A^\sharp(\ell)-A^\sharp(\bar\ell)|
        \lesssim
        (1-2^{-\frac 1 2 (\alpha-16/p)})^{-1}
        K_A\Area(\ell,\bar\ell)^\gamma,
    \]
    and
    \[
        |\bA^\sharp(\ell)-\bA^\sharp(\bar\ell)|
        \lesssim
        (1-2^{-\frac 1 2 (\alpha-16/p)})^{-1}
        K_{\bA}
        (|\ell|\vee|\bar\ell|)^\gamma
        \Area(\ell,\bar\ell)^\gamma.
    \]

    The growth bounds follow by comparison with the horizontal projection
    $\hat\ell:=\rmP_{\cX_{\rmh\ax}}(\ell)$. Since $\hat\ell$ lies on the
    horizontal axis and we are in the complete axial gauge, we have
    \[
        A^\sharp(\hat\ell)=0,
        \qquad
        \bA^\sharp(\hat\ell)=0.
    \]
    Therefore
    \[
        |A^\sharp(\ell)|
        \lesssim
        (1-2^{-\frac 1 2 (\alpha-16/p)})^{-1}
        K_A\rho(\ell,\hat\ell)^\gamma
        \lesssim
        (1-2^{-\frac 1 2 (\alpha-16/p)})^{-1}
        K_A|\ell|^\gamma,
    \]
    and
    \[
        |\bA^\sharp(\ell)|
        \lesssim
       (1-2^{-\frac 1 2 (\alpha-16/p)})^{-1}
        K_{\bA}
        (|\ell|\vee|\hat\ell|)^\gamma
        \rho(\ell,\hat\ell)^\gamma
        \lesssim
        (1-2^{-\frac 1 2 (\alpha-16/p)})^{-1}
        K_{\bA}|\ell|^{2\gamma}.
    \]
    In particular, since $\beta<\gamma $
    \[
      \triple{\bfA^\sharp}_{\beta\ax} \lesssim    \triple{\bfA^\sharp}_{\gamma\ax}
        \lesssim
       (1-2^{-\frac 1 2 (\alpha-16/p)})^{-1}K_A
        +
        (1-2^{-\frac 1 2 (\alpha-16/p)})^{-1/2}K_{\bA}^{1/2}.
    \]
    Taking $p$-th moments gives
 \[
    \E[\triple{\bfA^\sharp}_{\beta\ax}^p]
\leq 
    C^pL
   (1-2^{-\frac 12(\alpha-16/p)})^{-p}
    ( 1-2^{-\frac p4(\alpha-\frac{16}{p}-\beta)}
    )^{-1}.
\]
    This gives the bound stated in the
    theorem.

    \paragraph{Stability.}
    Apply the same argument to
    \[
        \Delta A:=A-\bar A,
        \qquad
        \Delta\bA:=\bA-\bar\bA.
    \]
    The assumed difference estimates replace $L$ by $\varepsilon L$
    throughout. Since the modifications are constructed by the same dyadic
    limiting procedure, the resulting modification of the difference is
    \[
        (A^\sharp-\bar A^\sharp,\,
        \bA^\sharp-\bar\bA^\sharp).
    \]
    In this case we get the bound~\eqref{eq:Kolmogorov_result_bound_stability} from where we conclude the proof.

\end{proof}

\section{Solution theory for local Coulomb gauge}
\label{sec:solution_theory}

Recall from Section~\ref{sec:main_results} that, to find a Coulomb gauge of $\sigma A=(\sigma A_1,0)$ with $\sigma>0$ small, it suffices solve the PDE \eqref{eq:intro_Psi} for $g\colon\Lambda\to G$, which we rewrite in `mild formulation' as
\begin{align}\label{eq:PDE_g_mild}
\Phi(\sigma,g) = G^{\Dir}*\partial_i((\partial_i g)g^{-1})-\sigma G^{\Dir}*\partial_1(gA_1g^{-1})=0, \qquad g|_{\partial\Lambda}=1_G,
\end{align}
where $G^{\Dir}$ is the Dirichlet Green's function on $\Lambda$ (the choice $g|_{\partial\Lambda}=1_G$ is somewhat arbitrary but convenient).
For $A_1\in C^{-\nfrac12^-}$ as in our case, this is a singular but subcritical PDE, and thus falls under the methodology of regularity structures.

However, neither the setup of regularity structures nor the solution theory for \eqref{eq:PDE_g_mild} that we employ is standard in the singular SPDE literature.
In particular, we do not employ fixed points theorems.
The basic issue in applying fixed point arguments for singular elliptic equations, like in e.g. \cite{Labbe19_Anderson}, is that, if we were to rewrite \eqref{eq:PDE_g_mild} in the more semi-linear form $\Delta g =\ldots$ to which, say, Banach's fixed point theorem could apply,
the right-hand side would hide the natural geometric structure of the equation, leading to a more involved regularity structure and making it even non-obvious that solutions are $G$-valued.

As described after Theorem \ref{thm:Rough_Uhlenbeck}, our method is inspired by the IFT and involves the tangent equation
\begin{equation}\label{eq:tangent_eq}
D_2\Phi(\sigma,g)[vg] = G^{\Dir}*\partial_i(\partial_i v + [v,(\partial_i g)g^{-1}])-\sigma G^{\Dir}*\partial_1[v,gA_1g]\;,
\end{equation}
where the first term is the identity on $v\colon\Lambda\to\mfg$ that vanish on the boundary.
We show that $D_2\Phi(\sigma,g)$ is invertible on a space of modelled distributions under suitable smallness assumptions
and use it to solve \eqref{eq:PDE_g_mild} through an ODE in $\sigma$ (see \eqref{eq:intro_ODE}).
In Sections \ref{sec:reg_struct}, \ref{sec:models}, and \ref{sec:modelled_distributions}, we introduce our regularity structures, space of models, and modelled distributions.
In Sections \ref{sec:existence_stability} and \ref{sec:properties_solution} we prove existence, stability, and important algebraic properties of solutions $g$.

\subsection{Set up of the regularity structure}
\label{sec:reg_struct}

Our first step is to construct the regularity structure underlying the equation \eqref{eq:PDE_g_mild}.
    Although we take inspiration from vector-valued regularity structures as in \cite[Sec.~5]{CCHS2d},
    we perform our construction essentially from scratch. We do this for two reasons: (a) to be self-contained, and (b) our derivations and homogeneities are non-standard (see Definitions \ref{def:homogeneity} and \ref{def:derivatives}).

\subsubsection{The set of all symbols and associated vector spaces}\label{sec:S_all}

We first introduce the set of all formal symbols $S_\all$ generated from $A_1$, the polynomial symbols, and the derivation and integration operators.
We will later restrict this set when defining our concrete regularity structure.
The set $S_\all$, defined in \eqref{eq:S_all_def} below, is generated by the symbols $A_1, \1,\rmX_1,\rmX_2$, and the unary operators
    \begin{equation*}
    \scL\label{symb:unary_operators} =\{\rmD_1,\rmD_2,\cI\}\;.
    \end{equation*}
	Here $\rmD_i$ and $\cI$ should be thought of as abstract versions of derivation in spatial direction $i$ and integration against (a truncation of) the Dirichlet Green function $G^\Dir$, respectively.
	We impose the following relations:
	
    \begin{notation}[Unary operators and products]
       \label{not:unary_ops_symbols}
       \leavevmode
\begin{itemize}
\item {\bf Commutativity of $\scL$.} For  any $L,L'\in\scL$ we impose $LL'=L'L$ 
\item {\bf Commutativity of products.} For any symbols  $\tau_1,\tau_2$ we impose $\tau_1\tau_2=\tau_2\tau_1$ and $\tau_1\1 = \tau_1$.

\item {\bf Polynomials.} We define $\rmX^k\label{symb:polynomial_symbols}=\rmX_1^{k_1}\rmX_{2}^{k_2}$ for multi-index $k=(k_1,k_2)\in\bN^2$ and $\rmX^0:=\1$.       With $\rme_1=(1,0)$, $\rme_2=(0,1)$, we impose
      \[
        \rmD_i\rmX^k = k_i\,\rmX^{k-\rme_i}\;, \qquad i\in\{1,2\}\;.
      \]
\item {\bf Integration.}\ For any $i,j\in\{1,2\}$ we introduce the notations $\cI_i=\rmD_i\cI=\cI\rmD_i$ , $\cI_{ij}=\rmD_i\cI_j$.
\end{itemize}
\end{notation}

\noindent We have all the required notations to build the set of all symbols recursively starting from 
\[
S_{0}:=\{A_1,\1,\rmX_1,\rmX_2\}\;, 
\]
and for $i\geq 0$, given $S_i$, we define
\[
S_{i+1}:=S_{i}\cup \{L\tau\colon \tau\in  S_i, L\in \scL\}\cup \{\tau_1\tau_2\colon \tau_1,\tau_2\in S_i\}\;.
\]
We then define the set of all symbols as
\label{symb:S_all}\begin{equation}\label{eq:S_all_def}
S_\all:=\bigcup_{i\geq 0}S_i\;. 
\end{equation}

\subsubsection{Construction of subspaces associated to symbols}\label{sec:construction_subspace}

 We now associate to a symbol $\tau$ a subspace $T[\tau]$. Let us fix a basis $(e_i)_{i=1}^{\dim\fg}$ of $\fg$ as well as the dual basis for $\fg^*$ denoted by $(e^i)_{i=1}^{\dim\fg}$, so that $e^i(e_j)=\delta_{ij}$. 
    To the symbol $A_1$ we can associate a set of basis elements, namely
    \[
    \cB[A_1]=\{\blue{A_1^{e^i}}\colon i\in\{1,\ldots,\dim\fg\}\}\;.
    \]
    We then define $T[A_1] = \Span{\cB[A_1]}$. 
    For any $u=\sum_{i=1}^{\dim\fg}\lambda_i e^i\in\fg^*$ with $\lambda_i\in\R$ we then define
    \[
    \blue{A_1^u}=\sum_{i=1}^{\dim\fg}\lambda_i \blue{A_1^{e^i}}\in T[A_1]\;.
    \]
    Of course $\mfg^*\ni u\mapsto \blue{A^u} \in T[A_1]$ is a canonical isomorphism.

  We now define a set of basis elements $\cB[\tau]\label{symb:basis_symbol_space}$ for any given $\tau\in S_\all$ through an inductive procedure. 
    Firstly, we give a definition for products; we define the set $\cB[\tau_1 \tau_2]$ to be
    \[
    \cB[\tau_1\tau_2] =\{\blue{\tau_1\tau_2}
	\,:\,
	\blue{\tau_1}\in \cB[\tau_1]\,,\,\blue{\tau_2}\in \cB[\tau_2]\}\;,
    \]
where we identify $\blue{\tau_1\tau_2}=\blue{\tau_2\tau_1}$. Note that the latter imposed commutativity is consistent with the commutativity on the level of the symbols, which, by definition, yields $\cB[\tau_1\tau_2]=\cB[\tau_2\tau_1]$.
For $i=0,1,2$ and $\cI_i\tau\in S_\all$
    \[
    \cB[\cI_i\tau]=\{\blue{\cI_i\tau}\,:\, \blue\tau\in \cB[\tau]\}\;,
    \]
    where, similarly to before, we denote $\blue{\cI_0}=\blue{\cI}$.
    We also set for $i=1,2$
    \[
    \cB[\rmD_i\tau]=\{\blue{\rmD_i\tau}\,:\, \blue\tau\in\cB[\tau]\}\;.
    \]
    As in \notat{not:unary_ops_symbols},
    we impose commutativity of derivations and abstract integration.
\begin{example}\label{ex:commutativity_der_int_consequence}
    By commutativity of derivation and integration for symbols, we have $\cI_2\cI_{12}A_1=\cI_1\cI_{22}A_1$ which implies $\cB[\cI_2\cI_{12}A_1]=\cB[\cI_1\cI_{22}A_1]$.
    This is consistent with $\blue{\cI_2\cI_{12}A_1^{e^i}}=\blue{\cI_1\cI_{22}A_1^{e^i}}$.
    \end{example}
    We define the free vector space
    \begin{align}
        T[\tau] =\Span{\cB[\tau]}\;,
    \end{align}
    and for polynomials
    \[
    T[\rmX^k]=\Span{\blue{\rmX^k}}\;. 
    \]
We then define the vector space generated by all these vector spaces, namely
\[
T_\all\label{symb:T_all} =\bigoplus_{\tau\in S_\all}T[\tau]\;. 
\]
It is also useful to introduce the following notation: For a given set of symbols $S\subset S_\all$, we define
\[
T[S] = \bigoplus_{\tau\in S}T[\tau]\;. 
\]

\subsubsection{Homogeneity}\label{sec:homogeneity}

Fix $\alpha\in (\frac 1 3,\frac 1 2)$ throughout this section. 
Our starting point is a rough additive function $A\in\bfOmega_{\alpha\ax}^1$.
In this case $A_1\in C^{\alpha-1}$, which gives the homogeneity of our first symbol, namely 
  \[
      |A_1| =\alpha-1=-\frac 1 2-\kappa\;,
  \] 
  where we have introduced $\kappa\in (0,\frac16)$.
We introduce the following notations: 
\[
|\cI|=2\;, \quad |\rmD_i|=-1\;, \quad i\in\{1,2\}\;. 
\]
With these notations, we define the notion of homogeneity as follows.
An important remark is that we treat and $\rmD_1 A_1$ and $\rmD_2 A_1$ in a different manner (see Remark~\ref{rem:D2}). 

\begin{definition}[Homogeneity of symbols]\label{def:homogeneity}
Define the map
$|\Cdot|\colon S_{\all}\to\R$
by 
\[
  |A_1| =  -\tfrac12-\kappa\;,\qquad  |\rmD_2A_1| =  -1-2\kappa\;, \qquad
  |\1| = 0\;,\qquad
  |\rmX_i| = 1 \quad(i=1,2)\;,
\]
and recursively by
\begin{enumerate}[label=(\roman*)]
\item \emph{products}\,:  $|\tau_1\tau_2| = |\tau_1| + |\tau_2|$;

\item \emph{derivative in first direction and integration}:  $|L\tau| = |L|+|\tau|$ for $L\in\{\rmD_1,\cI\}$;
\item \label{def:homogeneity:iii} \emph{derivative in second direction}\,:  
      \begin{align*}
        |\rmD_2(\tau_1\tau_2)|
        &= \min\{|\rmD_2\tau_1|+|\tau_2|,
                         |\tau_1|+|\rmD_2\tau_2|\}\\
              |\rmD_2L\tau|& = |L| + |\rmD_2\tau|
        \quad\text{for } L\in \scL = \{\rmD_1,\rmD_2,\cI\}\;.
      \end{align*}
\end{enumerate}
We call $|\tau|\label{symb:homogeneity}$ the homogeneity of $\tau$.
This definition extends to the homogenous subspaces $T[\tau]$ of $T_\all$, namely for  $\blue\tau\in T[\tau]$ we set $|\blue\tau| = |\tau|$.
\end{definition}

\begin{remark}\label{rem:D2}
The fact that $A\in\bfOmega_{\alpha\ax}^1$ implies that $\partial_2A_1\in C^{2\alpha-2}$,
which explains the choice $|\rmD_2A_1| = -1-2\kappa$ instead of the `usual' choice $|\rmD_2A_1| = |A_1|-1=-\nfrac{3}{2}-\kappa$.
This may appear surprising, but stems from the fact that we are working in the axial gauge.
\end{remark}
\begin{remark}
Since operators in $\scL$ commute, we have $ L_1\cdots L_N  \tau= L_{\pi(1)}\cdots L_{\pi(N)}\tau$ for any $\tau\in S_\all$, permutation $\pi$ on $\{1,...,N\}$ and $L_1,...,L_N\in\scL$.
The above definition of homogeneity is consistent with this commutativity in the sense that
\[
|L_1\cdots L_N  \tau|=|L_{\pi(1)}\cdots L_{\pi(N)}\tau|\;.
\]
\end{remark}
\begin{remark}
The homogeneity for $\rmD_2(\tau_1\tau_2)$ can be understood as follows: if Leibniz rule were to hold then we would have 
$\rmD_2(\tau_1\tau_2)
 =( \rmD_2\tau_1)\tau_2 + \tau_1(\rmD_2\tau_2)$, and as such the homogeneity for $\rmD_2(\tau_1\tau_2)$ is then determined by the minimum of 
$|\rmD_2\tau_1|+|\tau_2|$ and $|\tau_1|+|\rmD_2\tau_2|$. 
\end{remark}
\subsubsection{The solution regularity structure}

We next define the sets of symbols $\cU$, $\cF, \cF_1$, $\cF_2$, and $\cF_\rmD\label{symb:solution_symbol_collections}$, which will be used in our solution theory of \eqref{eq:PDE_g_mild}, in Table \ref{tab:U_F_sets_homogeneities}.

\begin{table}[ht]
\centering
\renewcommand{\arraystretch}{1.15}

\begin{minipage}[t]{0.48\textwidth}\vspace{0pt}
\centering
\begin{tabular}{l|ll}
\toprule
Set & Symbol & Homogeneity  \\
\midrule
$\cU$  & $\1$ & $0$\\
  & $\rmX_i$ \; ($i=1,2$) & $1$\\
  & $\cI_1A_1$ & $\frac12-\kappa$\\
  & $\cI_1(A_1\cI_1A_1)$ & $1-2\kappa$\\
  & $(\cI_1A_1)^2$ & $1-2\kappa$\\
  & $\cI_1\cI_{22}A_1$ & $1-2\kappa$\\
\bottomrule
\end{tabular}
\end{minipage}
\hfill
\begin{minipage}[t]{0.48\textwidth}\vspace{0pt}
\centering
\begin{tabular}{l|ll}
\toprule
Set & Symbol ($i=1,2$) & Homogeneity  \\
\midrule
$\cF$  & $\cI_{i2}(A_1\cI_1A_1)$  & $\frac12-3\kappa$\\
  &  $(\cI_1A_1)\,(\cI_{i2}A_1)$  & $\frac12-3\kappa$ \\
\midrule
$\cF_1$ & $A_1$ & $-\frac12-\kappa$\\
  & $A_1\cI_1A_1$ & $-2\kappa$\\
  & $\cI_{22}A_1$ & $-2\kappa$\\
\midrule
$\cF_2$ & $\cI_{12}A_1$ & $-2\kappa$\\
\midrule
$\cF_\rmD$ & $\rmD_2 A_1$ & $-1-2\kappa$\\
& $\rmD_2(A_1 \cI_1 A_1)$ & $-\frac12-3\kappa$\\
\bottomrule
\end{tabular}
\end{minipage}

\caption{Symbols in $\cU$ and in $\cF,\cF_1,\cF_2,\cF_\rmD$ with homogeneities (recall $\frac12-\kappa=\alpha$).}
\label{tab:U_F_sets_homogeneities}
\end{table}

To explain further, the set $\cF_i$ will define the domain of the abstract integration map $\cI_i$ for $i=1,2$, and $\cU$ is the set of symbols obtained by applying integration maps to $\cF_i$ and then taking products.
See also \eqref{eq:U_sol_def} for a further decomposition of $\cU$ into subsets that will play distinct roles in the solution theory.

The set $\mcF$ arises from applying $\rmD_1$ or $\rmD_2$ on the solution $g$, but then these terms are removed by truncation due to the fact that $\rmD_2$ has a non-standard deregularisation effect, see Definition \ref{def:homogeneity}. We keep these terms in the regularity structure since this helps in showing that the truncated modelled distributions have the correct regularity.
Similarly $\cF_\rmD$ is not used in the solution theory, but is useful defining models (Definition \ref{def:model}), in bounding the differentiation operators on modelled distributions (Proposition \ref{prop:derivatives}),
and in bounding models using rough additive functions in Section \ref{s:model_bounds}.

\begin{remark}
To generate the above sets of symbols, it is tempting to set up a rule akin to \cite{BHZ19} arising from the system of equations~\eqref{eq:PDE_g_mild}-\eqref{eq:tangent_eq}.
However, we found that the most straightforward rule generates too many symbols (see Remark \ref{rem:U_sol_reason} for a related issue) and requires introducing exceptions. 
Since the sets are small anyway, we find it simpler to give them directly.
\end{remark}

We now have set up all the notations needed to define the solution symbols $S_\sol$ and the solution model space $T_\sol$, namely 
\[
S_\sol\label{symb:solution_symbol_space}=\cF\cup\cF_1\cup \cF_2\cup \cF_\rmD \cup\cU\;,\qquad T_\sol=\bigoplus_{\tau\in S_\sol}T[\tau]\;. 
\]
For $\gamma\in\R$, we define the natural projection maps\label{symb:Q_gamma}
\[
\rmQ_{\gamma}
\colon T_\sol\to \bigoplus_{\substack{\tau\in S_\sol\\ |\tau|=\gamma}}T[\tau]\;,
\qquad
\rmQ_{<\gamma}
\colon T_\sol\to \bigoplus_{\substack{\tau\in S_\sol\\ |\tau|<\gamma}}T[\tau]
\;.
\]
Then $T_\sol$ is a finite dimensional vector space and we fix $|\Cdot|_{T_\sol}$ to be any norm on $T_\sol$.
We define the corresponding norm at level $\gamma$ by 
\begin{equation}\label{eq:T_sol_norms}
|\Cdot|_{T_\sol;\gamma} =|\rmQ_{\gamma}\Cdot|_{T_\sol}
\;.
\end{equation}

\paragraph{Structure group.}
We now define the structure group $G_\sol\label{symb:G_sol}$ associated to $T_\sol$, which is a collection of linear maps $\Gamma\in L(T_\sol)$ that are lower triangular in the sense that, for any $\blue\tau\in T[\sigma]$ with $\sigma\in S_\sol$,
\[
\Gamma\blue\tau -\blue\tau \in \bigoplus_{\substack{\bar\tau\in S_\sol\\ |\bar\tau|<|\sigma|}} T[\bar\tau] 
\;.
\]
The definition of $\Gamma\in G_\sol$ is given on basis elements according to Table \ref{tab:structure_group}, combined with the identity $\Gamma\blue{\tau_1\tau_2}=\Gamma\blue{\tau_1}\Gamma\blue{\tau_2}$
for products $\blue{\tau_1\tau_2}\in T_\sol$ (one can readily check that $\Gamma\blue{\tau_1}\Gamma\blue{\tau_2}$ is indeed in $T_\sol$ in this case).
In Table \ref{tab:structure_group}, $\bfh_{\CJ_1\tau}\label{symb:structure_group_coefficients}$ for $\blue\tau\in T[\CF_1]$, and $\bfh_{\rmX_i}$
and $\bfh_{\cJ_{i2}(A_1^u\cI_1A_1^v)}$ for $i=1,2$ are real coefficients parametrising the group element $\Gamma$.
For example, we have 
\[
\Gamma\blue{\cI_1 A_1^u}=\blue{\cI_1 A_1^u}+\bfh_{\cJ_{1} A_1^u}\blue{\1}, \quad u\in (e_i)_{i=1}^{\dim\fg},
\]
for a coefficient $\bfh_{\cJ_1 A_1^u} \in \R$.
One can readily verify that $G_\sol$ is indeed a subgroup of $L(T_\sol)$.
    \begin{table}[ht]
    \centering
       \begin{tabular}{l|l}
\hline
$\blue{\tau}$ \;\; $(i=1,2)$ & $\Gamma \blue{\tau}$ \\
\hline \\[-1.5ex]
$\blue{\1}$ & $\blue{\1}$ \\[1.5ex]
$\blue{\rmX_i}$ & $\blue{\rmX_i} + \bfh_{\rmX_i}\blue{\1}$ \\[1.5ex]
$\blue{A_1^u}$ & $\blue{A_1^u}$ \\[1.5ex]
$\blue{\cI_1 A_1^u}$ & $\blue{\cI_1 A_1^u}+\bfh_{\cJ_{1} A_1^u}\blue{\1}$ \\[1.5ex]
$\blue{\cI_{i2} A_1^u}$  & $\blue{\cI_{i2} A_1^u}$ \\[1.5ex]
$\blue{\cI_1\cI_{22} A_1^u}$ & $\blue{\cI_1\cI_{22} A_1^u}+\bfh_{\cJ_1\cI_{22}A_1^u}\blue\1$ \\[1.5ex]
$\blue{\cI_1(A_1^u\cI_1A^v)}$  & $\blue{\cI_1(A_1^u\cI_1A^v)}+\bfh_{\cJ_{1} A_1^v}\blue{\cI_1 A_1^u}+\bfh_{\cJ_1(A_1^u\cI_1 A_1^v)}\blue\1$ \\[1.5ex]
$\blue{\cI_{i2}(A_1^u \cI_1 A_1^v)}$  & $\blue{\cI_{i2}(A_1^u \cI_1 A_1^v)}+\bfh_{\cJ_{1} A_1^v}\blue{\cI_{i2} A_1^u}
+ \bfh_{\cJ_{i2}(A_1^u \cI_1 A_1^v)}\blue\1$ \\[1.5ex]
$\blue{\rmD_2 A_1^u}$ & $\blue{\rmD_2 A_1^u}$ \\[1.5ex]
$\blue{\rmD_2(A_1^u \cI_1 A_1^v)}$ & $\blue{\rmD_2(A_1^u \cI_1 A_1^v)}
+ \bfh_{\cJ_{1}A_1^v}\blue{\rmD_2A_1^u}$ \\[1.5ex]
\hline
\end{tabular}
\caption{Action of $\Gamma$ on basis elements where $u,v\in (e_i)_{i=1}^{\dim\fg}$}
\label{tab:structure_group}
\end{table}

We conclude that $\cT_\sol:=(T_\sol,G_\sol)\label{symb:solution_regular_structure}$ is a genuine regularity structure according to \cite[Sec.~2]{Hairer14}. 
 
We further define the following subsets of symbols of $\cU$:
\label{symb:solution_sectors}\begin{equation}\label{eq:U_sol_def}
\begin{split}
\cU_\lin
&=
\{\1 \, , \, \rmX_1\,,\, \rmX_2 \,,\,  \cI_1A_1 \,,\, \cI_1(A_1\cI_1A_1) \}
\;,
\\
\cU_\alg &= \cU_\lin \cup \{(\cI_1A_1)^2 \}
\;,
\\
\cU_\Int &= \cU_\lin \cup \{\cI_1\cI_{22}A_1 \} \;.
\end{split}
\end{equation}
The set $\cU_\alg$ will be the set of symbols where the modelled distribution $\blue g$ that we solve for take values in;
we will later show it generates an algebra under a suitable product, see Proposition \ref{prop:stability_U_sol}.
The set $\cU_\lin$ will play a role in defining the tangent solution space
and $\cU_\Int$ will define the target space of the abstract integration maps $\cI_i$.

We do not include $\cI_1\cI_{22}A_1$ in $\cU_\alg$ because this term will be cancelled in our formulation of the PDE~\eqref{eq:PDE_g_mild} at the level of modelled distribution; moreover, it is important that we do not include this term as doing so
would lead to a proliferation of symbols, see Remark \ref{rem:U_sol_reason}.

Recall that a graded subspace of $T_\sol$ is called a \emph{sector} of $\cT_\sol$ if it is invariant under the action of the structure group $G_\sol$.

\begin{lemma}
$T[\cU]$, $T[\cU_\lin]$, $T[\cU_\alg]$, $T[\cU_\Int]$, $T[\cF_i]$ for $i=1,2$ and $T[\cF_\rmD]$ are sectors of $\cT_\sol$.
\end{lemma}

\begin{proof}
This follows by inspecting the definition of the structure group in Table~\ref{tab:structure_group}. 
\end{proof}

\subsubsection{Multiplication, derivation and integration maps}

\begin{definition}[Products]\label{def:star_prod}\label{symb:star_product} Let $\tau_1,\tau_2\in S_\sol$.  Let $\blue{\tau_1}\in T[\tau_1],\blue{\tau_2}\in T[\tau_2]$. If $\tau_1\tau_2\in S_\sol$, then we define $\blue{\tau_1}\star\blue{\tau_2}=\blue{\tau_1\tau_2}\in T[\tau_1\tau_2]$. Otherwise, we set $\blue{\tau_1}\star\blue{\tau_2}=\blue 0$.
\end{definition}

It is immediate to check that the product $\star$ is commutative and associative.

\begin{definition}[Derivatives]
\label{def:derivatives}\label{symb:abstract_derivatives}
We define linear `derivation' maps $\rmD_i\colon T[\cU_\alg]\to T_\sol$ for $i=1,2$ as follows.
We define for $\blue{\cI_1\tau}\in T[\cI_1\tau]$, where $\tau\in \{A_1,A_1\cI_1A_1\}$,
\begin{equation}\label{eq:D1D2_def}
\rmD_1 \blue{\cI_1\tau }=\blue{\tau}-\blue{\cI_{22}\tau}
\in T[\cF_1]\;,
\qquad
\rmD_2\blue{\cI_1\tau}=\blue{\cI_{12}\tau} \in T[\cF_2]
\;. 
\end{equation}
On the polynomials we define $\rmD_i$ in the usual way,
i.e. $\rmD_i \blue{\rmX_j} = \delta_{ij}\blue\1$ and $\rmD_i \blue{\1} = 0$ for $i,j=1,2$.
Finally, on $T[(\cI_1A_1)^2]$, we impose Leibniz rule.
\end{definition}

We remark that the first identity in \eqref{eq:D1D2_def} is a reflection on the level of symbols that
$\cI$ represents integration against the Green's function of $\Delta = \partial_1^2+\partial_2^2$.
The asymmetry in directions $\partial_1$ and $\partial_2$ is related to the asymmetry of the axial gauge in these directions.

\begin{example}\label{ex:der_square}
Let $\blue{A_1^u},\blue{A_1^v}\in T[A_1] $.
Then by Leibniz rule and \eqref{eq:D1D2_def}
\begin{align*}
\rmD_1 ((\blue{\cI_1 A_1^u})(\blue{\cI_1 A_1^v}))
&=(\blue{A_1^u}-\blue{\cI_{22}A_1^u})(\blue{\cI_1 A_1^v})+(\blue{\cI_1 A_1^u})(\blue{A_1^v}-\blue{\cI_{22}A_1^v})
\\
&\in T[A_1\cI_1A_1]\oplus T[\cI_1 A_1\cI_{22}A_1]. 
\end{align*}
\end{example}

\begin{lemma}\label{lem:D_Gamma_commute}
Let $i=1,2$ and $\Gamma\in G_\sol$ given by Table \ref{tab:structure_group}.
Then
$\Gamma \rmD_i\blue{\tau}=\rmD_i\Gamma\blue{\tau}$ for all $\blue\tau\in T[\cU_\alg\setminus\{\cI_1(A_1\cI_1A_1)\}]$.
Moreover, for $\blue \tau = \blue{\cI_1 (A_1^u \cI_1 A_1^v)} \in T[\cI_1(A_1\cI_1 A_1)]$, we have
\begin{equation}\label{eq:D_Gamma_trunc}
	\Gamma \rmD_i\blue{\tau} - \rmD_i\Gamma\blue{\tau} 
	=
	\begin{cases}
		\bfh_{\cJ_{12}(A_1^u \cI_1 A_1^v)} \blue\1 & \quad \text{if} \quad i = 2, \\
		- \bfh_{\cJ_{22}(A_1^u \cI_1 A_1^v)} \blue\1 & \quad \text{if} \quad i = 1. 	
	\end{cases}
\;
\end{equation}
\end{lemma}

\begin{proof}
The claim is trivial to verify for~$\tau \in \cbr[0]{\1 \, , \, \rmX_1\,,\, \rmX_2}$. 
For the remaining symbols, namely $\tau \in \cbr[0]{\cI_1 A_1, (\cI_1 A_1)^2, \cI_1(A_1 \cI_1 A_1)}$, we check the respective claims only for~$i=1$; the argument for $i=2$ is similar and simpler.
For $\blue{\cI_1 A_1^u} \in T[\cI_1 A_1]$, we have
\[
\Gamma \rmD_1\blue{\cI_1 A_1^u}=\Gamma (\blue{A_1^u}-\blue{\cI_{22}A_1^u})= \blue{A_1^u} -\blue{\cI_{22}A_1^u} 
\;,
\]
while 
\[
\rmD_1 \Gamma \blue{\cI_1A_1^u}=\rmD_1 (\blue{\cI_1 A_1^u}+\bfh_{\cJ_{1} A_1^u}\blue{\1})
=
\rmD_1 \blue{\cI_1A_1^u}=
\blue{A_1^u}-\blue{\cI_{22}A_1^u}
\;,
\]
as required.
It then follows from the Leibniz rule for $\rmD_1$ and multiplicativity of $\Gamma$ that $\rmD_1$ and $\Gamma$ commute on $T[(\cI_1 A_1)^2]$.

Finally, for $\blue{\cI_1(A_1^u\cI_1 A_1^v)} \in T[\cI_1(A_1\cI_1 A_1)]$, on the one hand we have
\begin{equs}
	\thinspace 
	\Gamma  \rmD_1 
		\blue{\cI_1(A_1^u\cI_1 A_1^v)} 
	&= 
	\Gamma 
	\sbr[1]{
		\blue{A_1^u\cI_1 A_1^v} - \blue{\cI_{22}(A_1^u\cI_1 A_1^v)} } \\
	&=
		\Gamma \blue{A_1^u\cI_1 A_1^v} 
		- \blue{\cI_{22}(A_1^u\cI_1 A_1^v)}
		- \bfh_{\cJ_1 A_1^v} \blue{\cI_{22} A_1^u} 
		- \bfh_{\cJ_{22}(A_1^u\cI_1 A_1^v)}\blue\1 
\end{equs}
while, on the other hand, 
\begin{equs}
	\rmD_1\Gamma \blue{\cI_1(A_1^u\cI_1 A_1^v)}
	& =
	\rmD_1 \sbr[1]{
		\blue{\cI_1(A_1^u\cI_1 A_1^v)} + \bfh_{\cJ_{1}A_1^v} \blue{\cI_1 A_1^u} + \bfh_{\cJ_1(A_1^u\cI_1 A_1^v)}\blue \1} \\
	& =
		\blue{A_1^u\cI_1 A_1^v} - \blue{\cI_{22}(A_1^u\cI_1 A_1^v)} 
		+ \bfh_{\cJ_{1}A_1^v} \del[1]{
				\blue{A_1^u} - \blue{\cI_{22} A_1^u} 
		} \\
	& =
		\Gamma \blue{A_1^u\cI_1 A_1^v} - \blue{\cI_{22}(A_1^u\cI_1 A_1^v)} 
		- \bfh_{\cJ_{1}A_1^v} \blue{\cI_{22} A_1^u} \,. 
\end{equs}
The claim in~\eqref{eq:D_Gamma_trunc} follows.
\end{proof}

\begin{definition}[Integration]
\label{def:integration}
We define `integration' maps $\cI_i\colon T[\cF_i]\to T[\cU_\Int]$ for $i=1,2$ by
$\cI_i\colon \blue{\tau}\mapsto \blue{\cI_i\tau}$ for a basis element $\blue\tau\in T[\cF_i]$ and extended by linearity.
Recall from Example \ref{ex:commutativity_der_int_consequence} that $\blue{\cI_2\cI_{12}A_1^u} = \blue{\cI_1\cI_{22}A_1^u}$, so $\cI_2$ is well-defined according to Table \ref{tab:U_F_sets_homogeneities}.

We similarly define $\cI_i\colon T[\cF_\rmD] \to T[\cF\cup\cF_1\cup\cF_2]$ by $\cI_i\colon \blue{\tau}\mapsto \blue{\cI_i\tau}$ for a basis element $\blue\tau\in T[\cF_\rmD]$ and extended by linearity.
Since $\blue{\cI_i\rmD_2\tau} = \blue{\cI_{i2} \tau}$,
this operator is well defined.
\end{definition}

To solve our equation, we need to establish that suitable combinations of multiplication, derivation and integration are stable in suitable sectors of $\cT_\sol$.
To that end, we recall from \cite[Def.~4.6]{Hairer14} that a pair of sectors $(V,W)$ is called $\gamma$-regular if $\Gamma(\blue\tau\star\blue{\bar\tau})=\Gamma\blue\tau\star\Gamma\blue{\bar\tau}$ for every $\Gamma\in G_\sol$ and $\blue\tau\in V,\blue{\bar\tau}\in W$ such that $|\blue\tau|+|\blue{\bar\tau}|<\gamma$.
If $V=W$, we simply say that $V$ is $\gamma$-regular.

\begin{proposition}\label{prop:stability_U_sol}
\begin{enumerate}[label=(\roman*)]
\item\label{pt:T_U_sol} $T[\cU_{\alg}]$ is $(\frac32-3\kappa)$-regular. Moreover, $T[\cU_{\alg}]$
is closed under $\star$: for all $\blue\tau,\blue{\bar\tau}\in T[\cU_{\alg}]$ one has
$\blue\tau\star\blue{\bar\tau}\in T[\cU_{\alg}]$.
\item\label{pt:T_U_sol_A1} The pair $(T[A_1],T[\cU_{\alg}])$ is $(\frac12-3\kappa)$-regular. Moreover, if
$\blue\tau\in T[\cU_{\alg}]$ and $\blue{A_1^u}\in T[A_1]$, then 
\begin{equation}\label{eq:A_1_stable}
\rmQ_{<\frac12-3\kappa}(\blue{\tau}\star\blue{A_1^u})\in T[\cF_1]
\;, 
\qquad
\cI_1\rmQ_{<\frac12-3\kappa}(\blue{\tau }\star\blue{A_1^u})\in T[\cU_\lin]
\;.
\end{equation}
\item\label{pt:DT_U_sol} For each $i\in\{1,2\}$, the pair $(T[\cU_{\alg}],\rmD_iT[\cU_{\alg}])$ is $(\frac12-3\kappa)$-regular. Moreover, if  $\blue\tau,\blue{\bar\tau}\in T[\cU_{\alg}]$, then
\begin{equation}\label{eq:D_1_stable}
\rmQ_{<\frac12-3\kappa}(\blue\tau\star\rmD_i\blue{\bar\tau})\in T[\cF_i\cup\{\1\}]
\;,
\qquad
\sum_{i=1}^2\cI_i\rmQ_{<\frac12-3\kappa}( \blue{\tau}\star\rmD_i \blue{\bar\tau})\in T[\cU_\lin]
\;. 
\end{equation}
\end{enumerate}
\end{proposition}

\begin{remark}\label{rem:U_sol_reason}
Without the sum over $i$ in \eqref{eq:D_1_stable}, the term $\cI_i\rmQ_{<\frac12-3\kappa}( \blue{\tau}\star\rmD_i \blue{\bar\tau})$ is only in $T[\cU_\Int]$.
Moreover, by definition, $\rmD_1$ is not well-defined on $T[\cU_{\Int}]$ because we are missing the symbol $\cI_{22}\cI_{22}A_1$ in the regularity structure.
This reveals the reason we exclude $\cI_1\cI_{22}A_1$ from $\cU_\alg$:
its inclusion would force us to include $\cI_{22}\cI_{22}A_1$,
and therefore $\cI_1\cI_{22}\cI_{22}A_1$,
and so on, leading to a proliferation of symbols.
\end{remark}

\begin{proof}
To prove \ref{pt:T_U_sol}, let $\blue\tau, \blue{\bar\tau}\in T[\cU_\alg]$ such that $|\blue\tau|+|\blue{\bar\tau}|<\frac32-3\kappa$.
Then it is easy to verify that $\blue{\tau\bar\tau} \in T[\cU_\alg]$ and in particular that $\blue{\tau}\star\blue{\bar\tau}=\blue{\tau\bar\tau}$.
Then $\Gamma(\blue{\tau}\star\blue{\bar\tau})=\Gamma\blue\tau\star\Gamma\blue{\bar\tau}$ by definition of $\Gamma\in G_\sol$, hence $(T[\cU_\alg],T[\cU_\alg])$ is $(\frac32-3\kappa)$-regular.
The fact that $T[\cU_\alg]$ is stable under multiplication now also follows from the fact that, if $|\blue\tau|+|\blue{\bar\tau}|\geq \frac32-3\kappa$, then $\blue{\tau}\star\blue{\bar\tau}=0$ by definition of $\star$ since in this case $\blue{\tau\bar\tau} \notin T_\sol$.

To prove \ref{pt:T_U_sol_A1}, the claim on $(\frac12-3\kappa)$-regularity follows in a similar manner as for \ref{pt:T_U_sol}. To prove \eqref{eq:A_1_stable}, remark that
\[
\rmQ_{<\frac12-3\kappa}(\blue{\tau}\star\blue{A_1^u})\in T[\{A_1,A_1\cI_1A_1\}]
\;,
\]
which, after applying $\cI_1$, clearly lands us back in $T[\cU_\lin]$.

Finally, for \ref{pt:DT_U_sol}, the claim on $(\frac12-3\kappa)$-regularity is similar to before.
For \eqref{eq:D_1_stable}, consider first $i=1$.
Using the definition of derivative,
we obtain
\[
\rmQ_{<\frac12-3\kappa}(\blue{\tau}\star\rmD_1\blue{\bar\tau})\in T[\cF_1 \cup \{\1\}]
\;.
\]
Therefore, by Definition \ref{def:integration} of the integration map,
\[
\cI_1 \rmQ_{<\frac12-3\kappa}(\blue{\tau}\star\rmD_1\blue{\bar\tau})\in T[\cU_\Int]
\;.
\]
A similar---in fact, even easier---argument applies to $i=2$. For now, we argued that $\cI_i\rmQ_{<\frac12-3\kappa}(\blue{\tau}\star\rmD_i\blue{\bar\tau})\in T[\cU_\Int]$. The only difference between $\cU_\Int$ and $\cU_\lin$ is the symbol $\cI_1\cI_{22}A_1$. 
This term can only arise when we take $\blue{\tau}=\blue{\1}$ and $\blue{\bar\tau}\in T[\cI_1A_1]$, in which case $\blue{\bar\tau}=\blue{\cI_1\tau'}$ for $\blue{\tau'}\in T[A_1]$ and 
\[
\blue{\tau}\star\rmD_1\blue{\bar\tau}=\blue{\tau'}-\blue{\cI_{22}\tau'}
\;, 
\]
so that 
\[
\cI_1 \rmQ_{<\frac12-3\kappa}(\blue{\tau}\star\rmD_1\blue{\bar\tau})= \blue{\cI_1\tau'} -\blue{\cI_1\cI_{22}\tau'}
\;,
\]
and the second term is the one  inside $T[\cI_1\cI_{22}A_1]$. 
On the other hand, we have 
\[
\blue{\tau}\star\rmD_2\blue{\bar\tau}=\blue{\cI_{12}\tau'}
\;,
\]
which implies
\[
\cI_2 \rmQ_{<\frac12-3\kappa}(\blue{\tau}\star\rmD_2\blue{\bar\tau})=\blue{\cI_1\cI_{22}\tau'}
\;.
\]
Therefore, when we sum over $i=1,2$, the term $\blue{\cI_1\cI_{22}\tau'}$ vanishes and we conclude that
\[
\sum_{i=1}^2\cI_i\rmQ_{<\frac12-3\kappa}( \blue{\tau}\star \rmD_i \blue{\bar\tau})\in T[\cU_\lin]
\;.
\]
\end{proof}

\subsection{Models} \label{sec:models}
We recall the following definition from regularity structures \cite[Def~2.17]{Hairer14}.

\begin{definition}\label{def:model}
A \emph{model} $\sfZ = (\Pi,\Gamma)\label{symb:model}$ for the regularity structure $\cT_\sol$ consists of the following elements: 
\begin{itemize}
     \item A map $\Gamma:\R^2\times\R^2\to G_\sol$ such that $\Gamma_{xy}\Gamma_{yz}=\Gamma_{xz}$ for all $x,y,z\in \R^2$.
     \item A collection of continuous linear maps $\Pi_x:T_\sol\to \cD'(\R^2)$ such that $\Pi_x\Gamma_{xy}=\Pi_y$ for all $x,y\in \R^2$.
     \end{itemize}
       Furthermore, recalling the notation \eqref{eq:T_sol_norms},
	   we require that for every $\gamma>0$ and compact set $\fK\subset \R^2$, there exist $C_1,C_2>0$ such that
     \begin{align}\label{eq:model_bounds}
     |(\Pi_x\blue\tau)(\varphi_x^\lambda)|\leq C_1 |\blue\tau|_{T_\sol}\lambda^\ell\;,
	 \qquad
	 |\Gamma_{xy}\blue\tau|_{T_\sol;m}\leq C_2 |\blue\tau|_{T_\sol} |x-y|^{\ell-m}\;,
     \end{align}
       for all $x,y\in\fK$, $\lambda\in (0,1]$, $\ell<\gamma$, $m<\ell$, $\blue \tau\in T_\sol$ with $|\blue\tau|=\ell$,
	   and $\varphi\in \CB^2$ (recall the notation of Section \ref{sec:notation}).
We furthermore require that $\sfZ$ satisfies the following additional properties:

\begin{enumerate}[label=(\alph*)]
\item \label{pt:derivatives} \textbf{Derivatives.} For $\blue{\rmD_2\tau} \in T[\cF_\rmD]$ and $x\in\R^2$
\begin{equation}\label{eq:D_2_model_def}
\Pi_x \blue{\rmD_2\tau} = \partial_2 \Pi_x \blue\tau\;.	
\end{equation}

	\item \textbf{Integration.} \label{item:integration}
	We fix a truncation $K\label{symb:K}\colon \R^2\setminus\{0\}\to \R$ of the Green's function $G^\Free(x) = \frac{1}{2\pi}\log|x|\label{symb:Green}$ 
	of $\Delta$ on $\R^2$ as in \cite[Rem.~5.6]{Hairer14}
	such that $K=G^\Free$ on $[-4,4]^2$ and
	$K$ is compactly supported in $[-5,5]^2$
	and smooth outside the origin.\footnote{We take $K=G^\Free$ on a sufficiently large box, namely $[-4,4]^2$, as this ensures that canonical models from Definition \ref{def:canonical_model} suitably commute with derivatives, see Lemma \ref{lem:canonical}.}
Introduce the shorthand
\[
K_i = \partial_i K\;,\qquad K_{ij} = \partial_{i}\partial_j K
\;.
\]
Then $\sfZ$ realises $K_i$ for $\cI_i$ where $i=1,2$.
Explicitly, for every $x\in\R^2$, $i=1,2$ and every $\blue\tau\in T[\cF_i]$,
\begin{equation}\label{eq:model_realisation}
\Pi_x \blue{\cI_i\tau}=K_i\ast \Pi_x\blue{\tau} - (K_i\ast \Pi_x\blue{\tau})(x)\;,
\end{equation}
for $\blue\tau \in T[A_1]$
\begin{equation}
\label{eq:model_realisation_2}
\Pi_x \blue{\cI_{i}\rmD_2\tau}= K_i\ast \Pi_x\blue{\rmD_2\tau}
\;,
\end{equation}
and for $\blue\tau \in T[A_1\cI_1A_1]$
\begin{equation}
\label{eq:model_realisation_3}
\Pi_x \blue{\cI_{i}\rmD_2 \tau}=K_i\ast \Pi_x\blue{\rmD_2 \tau} - (K_i\ast \Pi_x\blue{\rmD_2 \tau})(x)\;.
\end{equation}
The last terms in \eqref{eq:model_realisation} and \eqref{eq:model_realisation_3} should be understood as $\sum_{|k|<|\cI_{i}\tau|} \Pi_x\frac{\blue{\rmX^k}}{k!}(\partial_k K_i\ast \Pi_x\blue{\tau})(x)$ appearing in \cite[Def.~5.9]{Hairer14}, but since $|\cI_{i}\tau| \in (0,1)$, the sum is over the single multi-index $k=0$.
The same remark applies to \eqref{eq:model_realisation_3} since $|\cI_{i2}\tau| \in (0,1)$.

\item\label{pt:polys} \textbf{Polynomials.} For all $x,y\in\R^2$, $\Pi_x\blue{\rmX_i}(\cdot)=(\cdot-x)_i$ for $i=1,2$ and $\Pi_x\blue{\1}(\cdot)=1$,
and correspondingly $\Gamma_{xy}\blue{\rmX_i}=\blue{\rmX_i}+(x_i-y_i)\blue{\1}$ and $\Gamma_{xy}\blue{\1}=\blue{\1}$ for $i=1,2$.
\end{enumerate}

We let\label{symb:model_norms} $\|\Pi\|_{\gamma;\mfK}$ and $\|\Gamma\|_{\gamma;\mfK}$
denote the the smallest possible choices for $C_1$ and $C_2$ respectively such that \eqref{eq:model_bounds} holds.
We then set $\triple{\sfZ}_{\gamma;\fK}=\|\Pi\|_{\gamma,\fK}+\|\Gamma\|_{\gamma,\fK}$.
If $\mfK= \Lambda^{\!+10}$, then we drop $\mfK$ from the notation and simply write $\triple{\Cdot}_{\gamma} =\triple{\Cdot}_{\gamma;\Lambda^{\!+10}}$. (We take $+10$ because $K$ is supported on $[-5,5]^2$.)

To compare two models $\sfZ=(\Pi,\Gamma)$ and $\bar \sfZ=(\bar\Pi,\bar\Gamma)$,
we denote by $\|\Pi-\bar\Pi\|_{\gamma;
\mfK}$ and $\|\Gamma-\bar\Gamma\|_{\gamma;
\mfK}$ the optimal choice of $C_1$ and $C_2$ respectively such that \eqref{eq:model_bounds} holds once we replace $\Pi$ by $\Pi-\bar\Pi$ and $\Gamma$ by $\Gamma-\bar\Gamma$.
We then set $\triple{\sfZ;\bar\sfZ}_{\gamma;\fK} = \|\Pi-\bar\Pi\|_{\gamma;\mfK} + \|\Gamma-\bar\Gamma\|_{\gamma; \mfK}$.
\end{definition}

\begin{definition}[Canonical models] \label{def:canonical_model}
Suppose $A_1 \in C^1(\R^2;\mfg)$ with support in $\Lambdaex$ (recall \eqref{eq:ex_def}).
We define the \emph{canonical model} $\sfZ = (\Pi,\Gamma)$ built from $A_1$ by defining $\Pi_x$ and $\Gamma_{xy}$ for all $x,y\in\R^2$ as follows:
\begin{itemize}
\item $\Pi_x\blue{A_1^{u}}= \langle A_1(\cdot),u\rangle$,
\item for $\blue\tau\in T[\cF_i]\cup T[\cF_\rmD]$, $i=1,2$,
setting $\Pi_x\blue{\cI_i\tau}$ according to \eqref{eq:model_realisation}-\eqref{eq:model_realisation_3} 
and $\Gamma_{xy}\blue{\cI_i\tau}$ by taking $\bfh_{\cJ_1\tau}(x,y)$ for $\blue\tau\in T[\CF_1]$ in Table \ref{tab:structure_group} as
\begin{equation}\label{eq:h_canonical}
\bfh_{\cJ_1\tau}(x,y)
=
(K_1*\Pi_y\blue\tau)(x) - (K_1*\Pi_y\blue\tau)(y)\;,
\end{equation}
and taking
\begin{equation}\label{eq:h_canonical_2}
\bfh_{\cJ_{i2}(A_1^u \cI_1 A_1^v)}(x,y)
=
(K_{i2}*\Pi_x\blue{A_1^u \cI_1 A_1^v})(x) - (K_{i2}*\Pi_y\blue{A_1^u \cI_1 A_1^v})(y)
\;,
\end{equation}
\item for $\rmD_2\sigma\in T[\cF_\rmD]$,
setting $\Pi_x\blue{\rmD_2\sigma} = \partial_2\Pi_x\blue{\sigma}$,

\item defining products as $\Pi_x(\blue{\tau_1}\blue{\tau_2})=(\Pi_x\blue{\tau_1})(\Pi_x\blue{\tau_2})$ whenever $\blue{\tau_1}\blue{\tau_2}$ is in $T_\sol$,
\item setting $\Pi_x,\Gamma_{xy}$ on polynomials as in point \ref{pt:polys} of the previous definition.
\end{itemize}
\end{definition}

Inspecting Table \ref{tab:U_F_sets_homogeneities}, the above definition completely specifies $\Pi_x$ and $\Gamma_{xy}$ by recursion.
The assumption that $A_1\in C^1$ ensures that $\Pi_x$ and $\bfh_{\sigma}(x,y)$ in the above definition are well-defined (see the proof of the next lemma).
\begin{lemma}
\label{lem:canonical}
The canonical model $\sfZ$ is a model in the sense of Definition \ref{def:model}.
Moreover, for all
$\blue\tau \in T[\cU_\alg\setminus\{\cI_1(A_1\cI_1A_1)\}]$, $i=1,2$, and $x\in\Lambda$,
\begin{equation}\label{eq:model_D_commute}
\Pi_x \rmD_i \blue\tau = \partial_i \Pi_x \blue \tau\;,
\end{equation}
and for $\blue\tau = \blue{\cI_{1}\sigma} \in T[\cI_1(A_1\cI_1A_1)]$,
\begin{equation}\label{eq:model_D_commute_2}
\Pi_x \rmD_1\blue{\tau} =\partial_1 \Pi_x\blue{\tau} + (K_{22}*\Pi_x\blue \sigma)(x)
\;,
\qquad
\Pi_x \rmD_2 \blue\tau = \partial_2 \Pi_x \blue \tau - K_{12}*\Pi_x\blue\sigma(x)\;.
\end{equation}
\end{lemma}

The fact that $\sfZ$ is a model does not use that $A_1$ in Definition \ref{def:canonical_model} has support in $\Lambdaex$, but the identities \eqref{eq:model_D_commute}-\eqref{eq:model_D_commute_2} do use this support condition.

Some aspects of Lemma \ref{lem:canonical}
follow from the general construction in \cite[Sec.~8.2]{Hairer14},
we give a hands-on proof because our regularity structure and the derivatives $\rmD_i$ are not standard
and because we need to use the specifics of $K$ (i.e. that $K=G^\Free$ on $[-4,4]^2$) to verify \eqref{eq:model_D_commute}.

\begin{proof}
We first consider the analytic bounds \eqref{eq:model_bounds}.
Since $A_1\in C^1$ and $K_i$ is $1$-regularising, one has $K_i*A_1\in C^{2-\theta}$ and $K_{ij}*A_1\in C^{1-\theta}$ for all $\theta>0$.
It follows that $K_i*\Pi_y\blue\tau \in C^{2-\theta}$ for all $\blue\tau\in T[\cF_i]$ and $\theta>0$,
from which we obtain that $\bfh_{\cJ_1\tau}$ in \eqref{eq:h_canonical} satisfies $|\bfh_{\cJ_1\tau}(x,y)|\lesssim|x-y|$.
Using that $\Pi_y=\Pi_x\Gamma_{xy}$, which we verify below,
one has
\[
\Pi_y\blue{A_1^u \cI_1 A_1^v}(z) = \Pi_x\blue{A_1^u \cI_1 A_1^v}(z) + \bfh_{\cJ_{1}A_1^v}(x,y)A_1(z)
\]
from which it follows that the map $x\mapsto \Pi_x\blue{A_1^u \cI_1 A_1^v}$ is in $C^1(\R^2,C^1)$ (recall $C^1$ means Lipschitz).
Since convolution with $K_{ij}$ is a bounded operator from $C^1$ to $C^{1-\theta}$ for any $\theta>0$, it follows that $|K_{i2}*(\Pi_x - \Pi_y)\blue{A_1^u \cI_1 A_1^v}|_{C^1} \lesssim |x-y|^{1-\theta}$,
from which we readily obtain
$
|\bfh_{\cJ_{i2}(A_1^u \cI_1 A_1^v)}(x,y)| \lesssim |x-y|^{1-\theta} 
$
as desired.
This concludes the bounds for $\Gamma$.

The bounds for $\Pi$ follow in a similar and simpler way after we note that $\Pi_x\blue\tau \in L^\infty$ for all $\blue\tau \in T[\cF_1\cup \cF_2 \cup \cF_\rmD]$ because $A_1\in C^1$.

We now turn to the algebraic properties of $\Gamma$ and $\Pi$.
The cocycle property $\Gamma_{xy}\Gamma_{yz} = \Gamma_{xz}$
is clear for all symbols with one instance of $A_1$ because $\Pi_x\blue\tau = \Pi_y\blue\tau$ for all $\blue\tau\in T[\cF_1]$ with one $A_1$ which implies the additivity property
\begin{equation}\label{eq:h_additive}
\bfh_{\cJ_1\tau}(x,z)
=
\bfh_{\cJ_1\tau}(x,y)
+
\bfh_{\cJ_1\tau}(y,z)\;.
\end{equation}
The cocycle property also then follows for the `pure products' $A_1\cI_1A_1$, $(\cI_1A_1)^2$ and $(\cI_1A_1)(\cI_{i2}A_1)$ by multiplicativity of $\Gamma$.
On all these terms, one also has $\Pi_x\Gamma_{xy} = \Pi_y$ from similar considerations.

It remains to consider
$\blue{\cI_1(A_1^u\cI_1A_1^v)}$ and $\blue{\cI_{i2}(A_1^u \cI_1 A_1^v)}$.
For $\blue{\cI_1(A_1^u\cI_1A_1^v)}$,
the cocycle property for which is equivalent to the additivity property \eqref{eq:h_additive} of $\bfh_{\cJ_1A_1^u}(x,y)$ and to the Chen-type identity
\begin{align*}
\bfh_{\cJ_1(A_1^u\cI_1A_1^v)}(x,z)
&=
\bfh_{\cJ_1(A_1^u\cI_1A_1^v)}(x,y)+\bfh_{\cJ_1(A_1^u\cI_1A_1^v)}(y,z)
+\bfh_{\cJ_1A_1^u}(x,y)\bfh_{\cJ_1A_1^v}(y,z)
\;,
\end{align*}
which can be verified directly by using $\Pi_x\Gamma_{xy}=\Pi_{y}$ for the symbols with one instance of $A_1$.
The identity $\Pi_x\Gamma_{xy} = \Pi_y$ on this term follows similarly.

For $\blue{\cI_{i2}(A_1^u \cI_1 A_1^v)}$,
we first verify $\Pi_x\Gamma_{xy} = \Pi_y$.
Using the definition of $\Gamma_{xy}$ on this term, we have
\begin{equation}\label{eq:h_computation}
\begin{split}
\Pi_x\Gamma_{xy}\blue{\cI_{i2}(A_1^u \cI_1 A_1^v)}
&=
\Pi_x\bigl(\cI_{i2}\Gamma_{xy}\blue{A_1^u \cI_1 A_1^v} + \bfh_{\cJ_{i2}(A_1^u \cI_1 A_1^v)}(x,y)\blue{\1}\bigr)\\
&=
K_{i2} * \Pi_x\Gamma_{xy}\blue{A_1^u \cI_1 A_1^v}
-
(K_{i2} *\Pi_x \blue{A_1^u \cI_1 A_1^v})(x)
+ \bfh_{\cJ_{i2}(A_1^u \cI_1 A_1^v)}(x,y)\\
&=
K_{i2} * \Pi_y\blue{A_1^u \cI_1 A_1^v}
-
(K_{i2} *\Pi_y \blue{A_1^u \cI_1 A_1^v})(y)
= \Pi_y\blue{\cI_{i2}(A_1^u \cI_1 A_1^v)}
\end{split}
\end{equation}
where in the second equality we used that $\Gamma_{xy}\blue{A_1^u \cI_1 A_1^v} = \blue{A_1^u \cI_1 A_1^v} + \bfh_{\cJ_1A_1^v}\blue{A_1^u}$
followed by \eqref{eq:model_realisation_2}-\eqref{eq:model_realisation_3},
in the third equality we used $\Pi_x\Gamma_{xy} = \Pi_y$ on the symbols of the form $A_1\cI_1A_1$
and the definition of $\bfh_{\cJ_{i2}(A_1^u \cI_1 A_1^v)}(x,y)$ in \eqref{eq:h_canonical_2},
and in the final equality we used again \eqref{eq:model_realisation_3}.
The cocycle property $\Gamma_{xy}\Gamma_{yz} = \Gamma_{xz}$ on this term follows directly from Table \ref{tab:structure_group}, which in particular implies
$\Gamma_{xy}\blue{\cI_{i2}A_1^u} =  \blue{\cI_{i2}A_1^u}$, and from the additivity properties
\eqref{eq:h_additive} and
\begin{align*}
\bfh_{\cJ_{i2}(A_1^u \cI_1A_1^v)}(x,z)
&=
\bfh_{\cJ_{i2}(A_1^u \cI_1A_1^v)}(x,y)+\bfh_{\cJ_{i2}(A_1^u \cI_1A_1^v)}(y,z)
\;,
\end{align*}
the latter of which follows immediately from the definition \eqref{eq:h_canonical_2}.

It remains to verify \eqref{eq:model_D_commute}-\eqref{eq:model_D_commute_2}.
The claim for $\blue\tau$ in the polynomial sector is clear. 
The claims \eqref{eq:model_D_commute}-\eqref{eq:model_D_commute_2} for $\rmD_2$ follow directly from Definition \ref{def:derivatives}, the Leibniz rule, the multiplicativity of $\Pi_x$, and from \eqref{eq:model_realisation_2}-\eqref{eq:model_realisation_3}.

For $\rmD_1$, it suffices to consider $\blue\tau = \blue{\cI_1\sigma}$ with $\sigma \in T[\{A_1,A_1\cI_1A_1\}]$
because for $\blue{\tau} \in T[(\CI_1A_1)^2]$, the claim
\eqref{eq:model_D_commute} would then follow from \eqref{eq:model_D_commute} applied to $\blue{\cI_1\sigma} \in T[\cI_1 A_1]$, the Leibniz rule for $\rmD_1$, and multiplicativity of $\Pi_x$.
To this end, consider first $\blue\sigma \in T[A_1\cI_1A_1]$
for which we prove \eqref{eq:model_D_commute_2}.
Using Definition \ref{def:derivatives} and \eqref{eq:model_realisation_3},
\[
\Pi_x \rmD_1\blue{\cI_1\sigma} = \Pi_x(\blue\sigma - \blue{\cI_{22}\sigma})
=
\Pi_x\blue\sigma - K_{22} * \Pi_x\blue\sigma
+ (K_{22}*\Pi_x\blue\sigma)(x)\;.
\]
By the facts that $A_1$ has support on $\Lambdaex$,
and $\Pi_x$ is multiplicative,
note that $\Pi_x \sigma$ in supported on $\Lambdaex\subset [-2,2]^2$.
Since $K$ agrees with $G^\Free$ on $[-4,4]^2$, we have on $[-2,2]^2\supset\Lambdaex$
\[
K_{22} * \Pi_x\blue\sigma = \partial_{2}^2 G^\Free * \Pi_x\blue\sigma
=
\Pi_x\sigma - \partial_1^2 G^\Free * \Pi_x\blue\sigma
=
\Pi_x\sigma - K_{11} * \Pi_x\blue\sigma
\;,
\]
where in the second equality we use that $G^\Free$ is the Green's function of $\Delta$.
In conclusion,
\[
\Pi_x \rmD_1\blue{\cI_1\sigma}
=  K_{11} * \Pi_x\blue\sigma
+ (K_{22}*\Pi_x\blue\sigma)(x)
= \partial_1\Pi_x\blue{\cI_1\sigma} + (K_{22}*\Pi_x\blue\sigma)(x)
\]
where in the final equality we used \eqref{eq:model_realisation}.
The proof of \eqref{eq:model_D_commute} for $\rmD_1$ and $\blue\tau=\cI_1\blue\sigma \in T[\cI_1A_1]$
proceeds in the same way except we use \eqref{eq:model_realisation_2} instead of \eqref{eq:model_realisation_3} so that the term $(K_{22}*\Pi_x\sigma)(x)$ is absent.
\end{proof}

\begin{remark}\label{rem:bfh_explicit}
The computation \eqref{eq:h_computation} reveals that \eqref{eq:h_canonical_2} holds also for any model $\sfZ$ (not necessarily canonical).
\end{remark}


\subsection{Modelled distributions}
\label{sec:modelled_distributions}

Fix two models $\sfZ = (\Pi,\Gamma), \bar\sfZ = (\bar\Pi,\bar\Gamma)$ throughout this subsection. 
We define\label{symb:boundary_distance}
     \[
     |x|_{\partial\Lambda}=1\wedge \dist(x,\partial\Lambda)\;, \qquad |x,y|_{\partial\Lambda}=|x|_{\partial\Lambda}\wedge |y|_{\partial\Lambda}\;,
     \]
     \[
     \Lambda_{\partial\Lambda}=\{(x,y)\in (\Lambda^\circ)^2\, :\, x\neq y \ \text{ and } \ 2|x-y|\leq |x,y|_{\partial\Lambda}\}
	 \;.
     \]
     %
	 Fix $\gamma, \eta \in\R$.  
     For functions $\blue f,\blue{\bar f}\colon\Lambda^\circ \to T_{\sol}$
	 taking values in a sector with homogeneity less than $\gamma$, we set
     \begin{align*}
|\blue f-\blue{\bar f}|_{\gamma,\eta}
&=
\sup_{x\in \Lambda^\circ}\sup_{\ell<\gamma}\frac{|\blue f(x) - \blue{\bar f}(x)|_{T_\sol;\ell}}{|x|_{\partial\Lambda}^{(\eta-\ell)\wedge 0}}\;,
		\\
		\triple{\blue f;\blue{\bar f}}_{\gamma,\eta}
		&=
		\sup_{(x,y)\in \Lambda_{\partial\Lambda}}\sup_{\ell<\gamma}\frac{|\blue f(x)-\Gamma_{xy}\blue f(y) - \blue {\bar f}(x) + \bar \Gamma_{xy}\blue {\bar f}(y)|_{T_\sol;\ell}}{|x-y|^{\gamma-\ell}|x,y|_{\partial\Lambda}^{\eta-\gamma}}
	 \end{align*}
     and
     \label{symb:modelled_distribution_norms}\begin{equation}\label{eq:f_norm_def}
     \|\blue f;\blue{\bar f}\|_{\gamma,\eta}=|\blue f-\blue{\bar f}|_{\gamma,\eta}+ \triple{\blue f;\blue{\bar f}}_{\gamma,\eta}
	 \;.
     \end{equation}
	 These quantities of course depend on $\Gamma,\bar\Gamma$, but we suppress this from the notation.
	 We also denote $|\blue f|_{\gamma,\eta} = |\blue f-0|_{\gamma,\eta}$,
	 $\triple{\blue f}_{\gamma,\eta} = \triple{\blue f;0}_{\gamma,\eta}$
	 and $\|\blue f\|_{\gamma,\eta} = \|\blue f;0\|_{\gamma,\eta}$,
	 which now only depend on $\Gamma$.
The space of $\blue f$ for which $\|\blue f\|_{\gamma,\eta} <\infty$ is denoted $\scD^{\gamma,\eta}\label{symb:scD}(\Gamma)$.
We view $\scD^{\gamma,\eta}(\Gamma)$ as a space of singular modelled distributions in the sense of \cite[Sec.~6]{Hairer14} by extending $\blue f\in\scD^{\gamma,\eta}(\Gamma)$ to $\R^2\setminus\Lambda$ as $0$.

When $\Gamma$ is clear from the context, we simply write $\scD^{\gamma,\eta}$ for $\scD^{\gamma,\eta}(\Gamma)$.
If $f$ takes values in a sector $V$, we write $\scD^{\gamma,\eta}(V)$ or $\scD^{\gamma,\eta}(\Gamma;V)$ to emphasise the range of $f$.
For $\nu\leq 0$, we also write $\scD^{\gamma,\eta}_\nu$ for $\scD^{\gamma,\eta}(\Gamma;V)$ where $V$ is a sector with lowest homogeneity $\nu$.
In this case, we call $\nu$ the \emph{regularity} of the sector $V$.

Most of our modelled distributions will be matrix valued, i.e.~elements of $\scD^{\gamma,\eta}\otimes \rmM_\bC(N)$.
The norm on $\scD^{\gamma,\eta}$ extends to $\scD^{\gamma,\eta}\otimes \rmM_\bC(N)$ by choosing a norm on $\rmM_{\C}(N)$
and then a cross tensor norm $\scD^{\gamma,\eta}\otimes \rmM_\bC(N)$.
We denote the norm with the same notation as for the scalar case.
Moreover, all linear map $\scD^{\gamma,\eta}\to E$ for a vector space $E$ extend linearly to $\scD^{\gamma,\eta}\otimes \rmM_\bC(N) \to E\otimes\rmM_\C(N)$.

\begin{notation}\label{not:f}
Any function $\blue f\colon \Lambda^\circ\to T_\sol\otimes \rmM_\bC(N)$, by definition, can be written as
   \begin{align*}
    \blue f(x)=\sum_{\tau} \blue{\tau}\otimes f^\tau(x)\label{symb:modelled_coefficients}
    \end{align*}
	where the sum is over a basis set of $T_\sol$ and where
	$f^{\tau}\colon \Lambda^\circ \to \rmM_\bC(N)$.	
	For a function as above, we often write $\blue f(x) = \sum_\tau f^\tau(x)\blue{\tau}$.
	We use the same notation for $\blue f\colon \Lambda^\circ \to T_\sol$ in which case $f^\tau(x)\in\R$.
\end{notation}

We now define several important operations on $\scD^{\gamma,\eta}$, namely reconstruction, multiplication, differentiation, integration and inversion.

\paragraph{Reconstruction.}
Let $\gamma>0$ and $-1< \nu\wedge\eta\leq \gamma$.
Then there is a unique linear map $\cR\colon \scD^{\gamma,\eta}_\nu(\Gamma)\to C^{\nu\wedge\eta}(\R^2)$ such that
\begin{equation}\label{eq:CR_defining}
|(\cR \blue f-\Pi_x\blue{f}(x))(\varphi_x^\lambda)|\lesssim \lambda^\gamma \|\blue f\|_{\gamma,\eta},
\end{equation}
uniformly over $\lambda\in (0,1]$, $x\in\R^2$ and $\phi\in \CB^2$ (recall the notation of Section \ref{sec:notation}).
Here, to make sense of $f(x)$ for all $x\in\R^2$, we extend $f$ to $\R^2\setminus\Lambda^\circ$ by $0$.
We call $\cR\label{symb:cR}$ the \emph{reconstruction map} associated to $\sfZ$.

We recall from \cite[Rem~3.15]{Hairer14} that, if $\sfZ$ is a canonical model, then $\cR$ is given by point evaluation, i.e.
\begin{equation}\label{eq:reconst_canonical}
(\cR\blue f)(x)=(\Pi_x\blue f(x))(x)
\;.
\end{equation}
Furthermore, letting $\bar\cR$ denote the reconstruction map associated to $\bar\sfZ$, then by \cite[Prop.~6.9]{Hairer14}
\begin{equation}\label{eq:recon_cont}
\|\cR \blue{f} - \bar\cR \blue{\bar f}\|_{C^{\nu\wedge\eta}(\R^2)} \lesssim \|\blue{f};\blue{\bar f}\|_{\gamma,\eta} + \triple{\sfZ;\bar\sfZ}_{\gamma}
\end{equation}
locally uniformly over $\blue f, \blue{\bar f},\sfZ,\bar\sfZ$.

\paragraph{Multiplication.}
Recall the product $\star$ defined on $T_\sol$ in Definition \ref{def:star_prod}.
For $\gamma\in\R$, we also write $\star_\gamma\label{symb:star_gamma}$ to denote the truncated product on $T_\sol$ given by
\[
\blue \tau \star_\gamma \blue{\bar\tau} = \rmQ_{<\gamma}(\blue \tau \star \blue{\bar\tau})
\;.
\]

We equip $ T_\sol\otimes \rmM_\bC(N)$ with the product $\circ_\gamma\label{symb:circ_gamma}$ induced by the matrix multiplication on $\rmM_\bC(N)$ and the product $\star_\gamma$ on $T_\sol$,
i.e. for $\blue{f} = \sum_{\tau}\blue{\tau} \otimes f^\tau$ and $\blue {\bar f}=\sum_{\bar\tau} \blue{\bar\tau}\otimes \bar{f}^{\bar\tau}$, we set
\begin{equation}\label{eq:circ_prod_def}
\blue{f} \circ_\gamma \blue{\bar f}=  \sum_{\tau,\bar\tau}\blue{\tau}\star_\gamma\blue{\bar\tau}\otimes f^\tau \bar{f}^{\bar\tau}
\;.
\end{equation}
Note that $T_\sol \otimes \rmM_{\C}(N)$ equipped with $\circ_\gamma$
is an associative algebra.
However, while $\star$ and $\star_\gamma$ are commutative, $\circ_\gamma$ is not commutative due to the matrix multiplication on $\rmM_\bC(N)$.

Recall the following estimates from \cite[Prop.~6.12]{Hairer14}.

\begin{proposition}\label{prop:mult}
Consider $\gamma_i \in\R, \alpha_i\leq 0$, sectors $V_i$ of regularity $\alpha_i$,
and $\blue{f_i}\in \scD^{\gamma_i,\eta_i}_{\alpha_i}(\Gamma;V_i)\otimes \rmM_\bC(N)$ for $i=1,2$. Let $\gamma=(\gamma_1+\alpha_2)\wedge(\gamma_2+\alpha_1)$ and assume $(V_1,V_2)$ is $\gamma$-regular. 
Then $\blue{f_1}\circ_\gamma \blue{f_2} \in \scD^{\gamma,\eta}_\alpha\otimes \rmM_\bC(N)$, where $\alpha=\alpha_1+\alpha_2$ and $\eta=(\eta_1+\alpha_2)\wedge (\eta_2+\alpha_1)\wedge (\eta_1+\eta_2)$. Furthermore,
\[
\|\blue{f_1}\circ_\gamma \blue{f_2}\|_{\gamma,\eta}\lesssim \|\blue{f_1}\|_{\gamma_1,\eta_1}\|\blue{f_2}\|_{\gamma_2,\eta_2}
\;,
\]
locally uniformly in $\sfZ$. Furthermore, for another pair of modelled distributions $\blue{\bar f_i}\in \scD^{\gamma_i,\eta_i}_{\alpha_i}(\bar\Gamma;V_i)$,
\[
\|\blue{f_1}\circ_\gamma \blue{f_2};\blue{\bar f_1}\circ_\gamma \blue{\bar f_2}\|_{\gamma,\eta}\lesssim \|\blue{f_1};\blue{\bar f_1}\|_{\gamma_1,\eta_1}+\|\blue{f_2};\blue{\bar f_2}\|_{\gamma_2,\eta_2}+\|\Gamma-\bar\Gamma\|_{\gamma_1+\gamma_2}
\;,
\]
locally uniformly over all quantities.
\end{proposition}

\paragraph{Differentiation.}
For $\blue f \in \scD^{\gamma,\eta}(T[\cU_\alg])$, we define the function
\label{symb:modelled_derivative}\begin{equation}\label{eq:rmD_gamma_def}
\cD_i^\gamma \blue f = \rmQ_{<\gamma-1} \rmD_i \blue f
+ (H_i\blue f)\blue\1
\colon
\R^2 \to T_\sol \;,
\end{equation}
where
\begin{equation}\label{eq:H_def}
(H_1\blue f)(x)
=
-\sum_{\blue\sigma}f^{\cI_1\sigma}(x) (K_{2} *\Pi_x \blue{\rmD_2\sigma})(x)
\;,
\qquad
(H_2\blue f)(x)
=
\sum_{\blue\sigma}f^{\cI_1\sigma}(x) (K_{1} *\Pi_x \blue{\rmD_2\sigma})(x)
\end{equation}
where the sum is over the basis set of $T[A_1\CI_1A_1]$, i.e. $\blue\sigma = \blue{A_1^{e^i}\cI_1A_1^{e^j}}$ for $e^i$ a basis of $\mfg^*$
and where we recall Notation \ref{not:f}.
Remark that, for a compact $\mfK\subset\R^2$ one has uniformly in $x\in\mfK$
\begin{equation}\label{eq:Ki2_exist}
(K_i*\Pi_x\blue{\rmD_2\sigma})(x)
\lesssim
\triple{\sfZ}_{\gamma;\mfK^{+10}}
\end{equation}
for any $\gamma>0$.
Indeed, $\blue{\rmD_2\sigma}\in T[\rmD_2(A_1\cI_1A_1)]$ has homogeneity
$-\frac12-3\kappa$, so by a multi-scale decomposition of $K_i$ as in \cite{Hairer14},
see Lemma~5.5 and Eq.~(5.11) therein, 
the singular part of $K_i$ is summable,
while the smooth large-scale part is
controlled by the model bounds on $\mfK^{+10}$, using a finite partition of
the support of $K_i(x-\cdot)$ and the identity $\Pi_x=\Pi_y\Gamma_{yx}$.

This rather complicated definition stems from the fact that $\rmD_2$ has a special deregularising effect from Definition \ref{def:homogeneity},
and moreover $\Pi_x$ and $\rmD_i$ in general (see e.g.~Lemma \ref{lem:canonical}).
The role of $H_i$ is similar to that of the correction in \cite[Sec.~4.5.1]{CS26} and \cite[Sec.~5.3]{EH17} and ensures that reconstruction and derivatives commute for canonical models.

\begin{proposition}\label{prop:derivatives}
Let $1<\gamma\leq \frac{3}{2}-3\kappa$, $\eta\leq \gamma$,
and $\blue f\in \scD^{\gamma,\eta}(T[\cU_\alg])$.
Then $\cD_i^\gamma \blue f\in \scD^{\gamma-1,\eta-1}$ for $i = 1,2$ and one has the bound
\begin{equation}\label{eq:Df_bound}
\|\cD_i^\gamma\blue f\|_{\gamma-1,\eta-1}\lesssim \|\blue f\|_{\gamma,\eta}
\end{equation}
locally uniformly in $\sfZ$.
Moreover, for $\blue{\bar f}\in \scD^{\gamma,\eta}(\bar\Gamma;T[\cU_\alg])$,
\begin{equation}\label{eq:Df_diff}
\|\cD_i^\gamma\blue f;\cD_i^\gamma\blue{\bar f}\|_{\gamma-1,\eta-1}\lesssim \|\blue f;\blue{\bar f}\|_{\gamma,\eta}+\triple{\sfZ; \bar\sfZ}_\gamma
\;,
\end{equation}
locally uniformly over all quantities.
Finally, if $\sfZ$ is canonical and $\eta>0$, then \begin{equation}\label{eq:CR_D_commute}
\cR\cD_i^\gamma\blue f=\partial_i\cR\blue f
\;.
\end{equation}
\end{proposition}

\begin{proof}
We start with the bound on $\|\cD_i^\gamma\blue f\|_{\gamma-1,\eta-1}$, more specifically the first summand in~\eqref{eq:rmD_gamma_def}.

We first look at the (non-)commutativity of $\Gamma_{xy}$ and $\rmQ_{<\gamma-1}$:
For $\ell<\gamma-1$ and $(x,y)\in\Lambda_{\partial\Lambda}$ we observe that 
\begin{equs}[][eq:commutativity_Gamma_Q]
|\Gamma_{xy}\rmQ_{<\gamma-1}\rmD_i \blue f (y)
-
\rmQ_{<\gamma-1} \Gamma_{xy}\rmD_i \blue f (y) |_{T_\sol;\ell}
&\lesssim \sum_{\beta\geq\gamma-1}
|x-y|^{\beta-\ell}\|\Gamma\|_{\ell}|\blue f|_{\gamma,\eta}|y|_{\d\Lambda}^{\eta-\beta-1}
\\
&\lesssim
|x-y|^{\gamma-1-\ell}|x,y|_{\d\Lambda}^{\eta-\gamma}
\|\Gamma\|_\gamma |\blue f|_{\gamma,\eta}
\;,
\end{equs}
where the sum is over the finite set of homogeneities $\beta\geq \gamma-1$ of $T_\sol$
and in the first bound we used that $|\rmD_i \blue f(y)|_{T_\sol;\ell} \lesssim |\blue f|_{\gamma,\eta}|y|_{\d\Lambda}^{(\eta-\ell-1)\wedge 0}$
since $\rmD_i$ lowers homogeneity by at most $1$
and that $\eta\leq\gamma$,
and in the second bound we used that $|x-y|\leq |x,y|_{\d\Lambda}\asymp|y|_{\d\Lambda}$.

Next, we commute~$\Gamma_{xy}$ and $\rmD_i$, up to a correction term, using Lemma~\ref{lem:D_Gamma_commute}:
\begin{equation}\label{eq:commutativity_Gamma_D}
\begin{split}
	& \abs[0]{\rmQ_{< \gamma-1}\del[1]{\rmD_i \blue{f}(x) - \Gamma_{xy} \rmD_i \blue{f}(y)}}_{T_\sol;\ell} \\
	\lesssim & 
	\abs[0]{\rmQ_{< \gamma-1}\rmD_i \blue f(x) - \rmQ_{<\gamma-1}\rmD_i \Gamma_{xy}\blue f (y)}_{T_\sol;\ell} 
	+
	|y|_{\d\Lambda}^{(\eta-1+2\kappa)\wedge 0} |\blue f|_{\gamma,\eta} \abs[0]{\bfh_{\cJ_{(3-i)2}(A_1^u \cI_1 A_1^v)}(x,y)} \1_{\ell=0} \\
	\lesssim & 
	\abs[0]{\rmQ_{< \gamma-1}\rmD_i \blue f(x) - \rmQ_{<\gamma-1}\rmD_i \Gamma_{xy}\blue f (y)}_{T_\sol;\ell} 
	+
	\|\Gamma\|_\gamma |\blue f|_{\gamma,\eta} |x-y|^{\gamma-1-\ell}
	|x,y|_{\d\Lambda}^{\eta-\gamma}
	\1_{\ell=0}
	%
	\end{split}
\end{equation}
where we used the estimate $|y|_{\d\Lambda}^{(\eta-1+2\kappa)\wedge 0} \leq |y|_{\d\Lambda}^{\eta-\gamma}$ since $\gamma>1$ and $|y|_{\d\Lambda}\asymp |x,y|_{\d\Lambda}$, and that
\begin{equ}[e:estimate_bfh_aux]
	|\bfh_{\cJ_{j2}(A_1\cI_1A_1)}(x,y)|
	\lesssim 
	\|\Gamma\|_\gamma|x-y|^{\frac{1}{2}-3\kappa}
	\lesssim 
	\|\Gamma\|_\gamma|x-y|^{\gamma-1}, \quad j = 1,2 \, ,
\end{equ}
since $|\cI_{j2}(A_1\cI_1A_1)| = \frac{1}{2}-3\kappa$ and $0 < \gamma - 1 \leq \frac{1}{2}-3\kappa$.

Finally, the combination of~\eqref{eq:commutativity_Gamma_Q} and~\eqref{eq:commutativity_Gamma_D} leads to the estimate
\begin{equs}{}
	& \abs[0]{\rmQ_{< \gamma-1}\rmD_i \blue{f}(x) - \Gamma_{xy} \rmQ_{< \gamma-1} \rmD_i \blue{f}(y)}_{T_\sol;\ell} \\
	\lesssim \ & 	
	\abs[0]{\rmQ_{< \gamma-1}\rmD_i \blue{f}(x) -  \rmQ_{< \gamma-1} \rmD_i \Gamma_{xy} \blue{f}(y)}_{T_\sol;\ell} 
	+
	|x-y|^{\gamma-1-\ell}|x,y|_{\d\Lambda}^{\eta-\gamma}
\|\Gamma\|_\gamma |\blue f|_{\gamma,\eta} \\
	= \ & 	
	\abs[0]{\rmD_i \blue{f}(x) - \rmD_i \Gamma_{xy} \blue{f}(y)}_{T_\sol;\ell} 
	+
	|x-y|^{\gamma-1-\ell}|x,y|_{\d\Lambda}^{\eta-\gamma}
\|\Gamma\|_\gamma |\blue f|_{\gamma,\eta} \\
	\lesssim \ & 	
	\triple{\blue f}_{\gamma,\eta}|x-y|^{\gamma-\ell-1}|x,y|_{\d\Lambda}^{\eta-\gamma}
	+
	|x-y|^{\gamma-1-\ell}|x,y|_{\d\Lambda}^{\eta-\gamma}
\|\Gamma\|_\gamma |\blue f|_{\gamma,\eta}
\end{equs}
where we additionally used that $\ell < \gamma - 1$ and that $\rmD_i$ lowers homogeneity by at most $1$.
We conclude that
\[
\triple{\rmQ_{<\gamma-1}\rmD_i\blue f}_{\gamma-1,\eta-1}
\lesssim
\triple{\blue f}_{\gamma,\eta} + \|\Gamma\| |\blue f|_{\gamma,\eta}\;.
\]
The necessary bounds on $|\rmQ_{<\gamma-1}\rmD_i\blue f|_{\gamma,\eta}$  are simpler.

In order to complete the proof of~\eqref{eq:Df_bound}, it remains to control the second summand in~\eqref{eq:rmD_gamma_def}, i.e. $(H_i\blue f)\blue\1$.
Note first that
\[
|H_i\blue f(x)| \lesssim \triple{\sfZ}_\gamma |\blue f|_{\gamma,\eta}|x|_{\d\Lambda}^{(\eta-1+2\kappa)\wedge 0}
\leq
\triple{\sfZ}_\gamma |\blue f|_{\gamma,\eta}|x|_{\d\Lambda}^{(\eta-1)\wedge 0}
\;,
\]
where we used that $|\cI_1(A_1\cI_1A_1)|=1-2\kappa$
and \eqref{eq:Ki2_exist}.
Moreover, recalling \eqref{eq:H_def},
for $j\neq i$,
\begin{align*}
|H_i\blue f(x)-H_i\blue f(y)|
&\lesssim
\sum_{\sigma}\{ |f^{\cI_1\sigma}(x)-f^{\cI_1\sigma}(y)|K_{j} *\Pi_x \blue{\rmD_2\sigma}(x)|
\\
\qquad &+
|f^{\cI_1\sigma}(x)|K_{j} *\Pi_x \blue{\rmD_2\sigma}(x)-K_{j} *\Pi_y \blue{\rmD_2\sigma}(y)|
\}
\;.
\end{align*}
The first term is bounded by a multiple of
\begin{align*}
\triple{\blue f}_{\gamma,\eta}|x-y|^{\gamma-1+2\kappa}|x,y|_{\d\Lambda}^{\eta-\gamma}\triple{\sfZ}_\gamma
\end{align*}
where we used  \eqref{eq:Ki2_exist},
while the second term is bounded by a multiple of
\begin{align*}
|\blue f|_{\gamma,\eta}|x|_{\d\Lambda}^{(\eta-1+2\kappa)\wedge 0}\|\Gamma\|_\gamma |x-y|^{\frac{1}{2}-3\kappa}
\lesssim
|\blue f|_{\gamma,\eta}|x|_{\d\Lambda}^{\eta-\gamma}\|\Gamma\|_\gamma |x-y|^{\gamma-1}
\;,
\end{align*}
where we used that $\bfh_{\cJ_{j2}(A_1^u\cI_1A_1^v)}(x,y) = (K_{j2} *\Pi_x \blue{A_1^u\cI_1A_1^v})(x) - (K_{j2} *\Pi_y \blue{A_1^u\cI_1A_1^v})(y)$ due to Remark \ref{rem:bfh_explicit} as well as the bound in~\eqref{e:estimate_bfh_aux},
and in the second bound we used that $1<\gamma\leq \frac{3}{2}-3\kappa$
to conclude that $|x|_{\d\Lambda}^{(\eta-1+2\kappa)\wedge 0}\leq |x|_{\d\Lambda}^{\eta-\gamma}$
and $|x-y|^{\frac{1}{2}-3\kappa}\leq |x-y|^{\gamma-1}$.
This completes the proof of \eqref{eq:Df_bound}.
The bound \eqref{eq:Df_diff}
follows in a similar way.

For the claim \eqref{eq:CR_D_commute}, suppose $\sfZ$ is canonical.
Since $\gamma>1$, note that for all $x\in\Lambda^\circ$, by \eqref{eq:reconst_canonical},
\begin{align*}
\cR\cD_i^\gamma\blue f(x) &= (\Pi_x \cD_i\blue f (x))(x)
= (\Pi_x \rmQ_{<\gamma-1}\rmD_i \blue f(x))(x) + (H_i\blue f)(x)\\
&= (\partial_i\Pi_x \blue f(x) - H_i \blue f)(x) + (H_i\blue f)(x)
=(\partial_i\Pi_x \blue f(x))(x) = (\partial_i\CR\blue f)(x)
\;,
\end{align*}
where in the third equality we used \eqref{eq:model_D_commute}-\eqref{eq:model_D_commute_2}
and the definition of $H_i$ in \eqref{eq:H_def}.
Then $\eta>0$ and $\CR \cD^\gamma_i f \in C^{\eta-1}(\R^2)$ imply that
$\cR\cD_i^\gamma\blue f = \partial_i\CR\blue f$ in the sense of distributions on $\R^2$.
\end{proof}

\paragraph{Integration.} Recall the Dirichlet Green's function $G^\Dir$ from Appendix \ref{app:integration_kernel}, in particular \eqref{eq:G^D_decomp}.

\begin{proposition}\label{prop:integration}
For $i=1,2$, $\gamma>0$ and $\eta\wedge\nu>-1$, there is an operator
\[
\cG^i_\gamma\label{symb:Dirichlet_integration_operator}\colon \scD^{\gamma,\eta}_\nu(\Gamma;T[\cF_i\cup\{\1\}])\to\scD^{\gamma+1,\eta\wedge\nu+1}_0(\Gamma;T[\cU_\Int])
\]
of the form
\[
(\cR\cG^i_\gamma\blue f )(x)= \cI_i \blue f(x) + \blue\cQ(x)
\;,
\]
where $\blue\cQ\colon \Lambda^\circ \to T[\{\1,\rmX_1,\rmX_2\}]$ is polynomial-valued and we set $\cI_i\blue\1 = 0$,
such that
\begin{equation}\label{eq:K_recon}
\cR\cG^i_\gamma\blue f(x) = G^\Dir*\partial_i \cR\blue f(x)
\end{equation}
for all $x\in\Lambda^\circ$.
For $\blue{\bar f}\in \scD^{\gamma,\eta}_\nu(\bar\Gamma;T[\cF_i\cup\{\1\}])$, one has
\[
\|\cG^i_\gamma \blue f;\cG^i_\gamma \blue{\bar f}\|_{\gamma+1,\eta\wedge\nu+1}
\lesssim
\|\blue f;\blue{\bar f}\|_{\gamma,\eta}
+
\triple{\sfZ;\bar\sfZ}_{\gamma},
\]
locally uniformly over $\sfZ,\bar\sfZ, \blue f,\blue{\bar f}$ in bounded sets.
\end{proposition}

\begin{proof}
Observe that $\partial_i\CR \blue f \in C^{\beta}(\R^2)$ with $\beta = \nu\wedge\eta-1>-2$ and with support on $\Lambda$, thus
$G^\Dir*\partial_i \cR\blue f\in C^{\beta+2}(\Lambda)$ by Lemma \ref{lem:Schauder_Dir} and is in particular well-defined pointwise.
We set
\[
\CG^i_\gamma \blue f= \CK^i_\gamma \blue f+ \CZ^i_\gamma(\CR\blue f)
\]
with $\CK^i_\gamma$ from Theorem \ref{thm:singular_kernel_composed_der} and $\CZ^i_\gamma$ from Theorem \ref{thm:integration_boundary_derivation}.
The claimed properties of $\CG^i_\gamma$ follow from these results, from the bounds on $\CR$ in \eqref{eq:recon_cont}, and from the decomposition \eqref{eq:GD_decomp} of $G^\Dir$.
\end{proof}

\paragraph{Inversion.}
Recall that, by Proposition \ref{prop:stability_U_sol}\ref{pt:T_U_sol}, $T[\cU_\alg]$ is closed under the $\star$ product defined on $T_\sol$
and that $T[\cU_\alg]$ is $\gamma$-regular for $\gamma \in [0,\frac32 - 3\kappa]$.
In particular, $T[\cU_\alg]\otimes \rmM_{\C}(N)$ is closed under $\circ_\gamma$ with identity element $1_G\blue\1$.
Remark that, if $g\in\rmM_{\C}(N)$ is invertible, then $g\blue\1$ is invertible under $\circ_\gamma$ with inverse $g^{-1}\blue\1$.
It follows that, writing $g= g^\1 \blue\1 + \blue{\tilde g}\in T[\cU_\alg]\otimes \rmM_{\C}(N)$
where $\blue{\tilde g}$ only has components of strictly positive homogeneity,
if $g^\1$ is invertible,
then $\blue g$ is invertible under $\circ_\gamma$ with inverse
\begin{equation}\label{eq:g_inv}
\blue{g}^{-1} := (g^\1)^{-1} - (g^\1)^{-1} \circ_\gamma\blue{\tilde g} \circ_\gamma (g^\1)^{-1} + (g^\1)^{-1}\circ_\gamma \blue{\tilde g}\circ_\gamma (g^\1)^{-1}\circ_\gamma \blue{\tilde g} \circ_\gamma (g^\1)^{-1} \;.
\end{equation}
The terms ($g^\1)^{-1}$ in \eqref{eq:g_inv} are understood as $(g^\1)^{-1} \blue\1$.
Remark that $\blue g \circ_\gamma \blue{g}^{-1} = \blue{g}^{-1} \circ_\gamma \blue g = 1_G \blue\1$
and that we see a finite sum in \eqref{eq:g_inv} instead of a series because $\blue{\tilde g}$ only has components of strictly positive homogeneity.

\begin{definition}\label{def:I}
Let $I\label{symb:I}$ denote the set of $g\in \rmM_\bC(N)$ such that, denoting by $\lambda$ the smallest singular value of $g$, one has $\lambda\geq \frac12$.
Let $\scI^{\gamma,\eta}\label{symb:scI}$ denote the set of $\blue g \in \scD^{\gamma,\eta}(T[\cU_\alg])\otimes \rmM_{\C}(N)$ such that $g^\1(x) \in I$ for all $x\in\Lambda$.
We write $\scI^{\gamma,\eta}(\Gamma)$ if we need to emphasise dependence on $\Gamma$.
\end{definition}

Our main interest in the set $I$ is that matrix inversion is a smooth map on $I$ with derivatives bounded of all orders and that $G\subset I$.
For our purposes, the choice of $\frac12$ in the above definition is somewhat arbitrary and can be replaced by any positive number smaller than $1$.
The following is now a direct consequence of \cite[Prop.~3.11]{WongZakai}.
\begin{proposition}\label{prop:inversion}
Let $\gamma \in [0, \frac32-3\kappa]$ and $\eta\in [0,\gamma]$.
The map $\blue g \mapsto \blue{g}^{-1}$ given by \eqref{eq:g_inv}
maps $\scI^{\gamma,\eta}(\Gamma)$ to $\scD^{\gamma,\eta}(\Gamma;T[\cU_\alg])\otimes \rmM_\C(N)$ and, for $\bar g \in \scI^{\gamma,\eta}(\bar\Gamma)$,
\[
\|\blue g^{-1} ; \blue{\bar g}^{-1}\|_{\gamma,\eta} \lesssim \|\blue g ; \blue{\bar g}\|_{\gamma,\eta} + \|\Gamma-\bar\Gamma\|_\gamma
\]
locally uniformly over all quantities.
\end{proposition}

We also record a low order expansion of $\blue{g}^{-1}$ that will be useful later.
Denoting
\[
\blue g = g^\1 \blue\1  +
\sum_{i} g^{i} \blue{\cI_1A_1^{e^i}}
+
\blue\mcQ
\;,
\]
where $e^i$ is a basis of $\mfg^*$ and $\blue\mcQ$ takes values in $\bigoplus_{|\tau|\geq 1-2\kappa} T[\tau]\otimes \rmM_{\C}(N)$, then, by \eqref{eq:g_inv},
\begin{equation}\label{eq:g_inv_explicit}
	\blue{g}^{-1}
= (g^\1)^{-1} \blue\1
-
\sum_{i} (g^\1)^{-1} g^{i} (g^\1)^{-1} \blue{\cI_1A_1^{e^i}} + \blue{\bar\mcQ}
\;,
\end{equation}
where $\blue{\bar\mcQ}$ also takes values in $\bigoplus_{|\tau|\geq 1-2\kappa} T[\tau]\otimes \rmM_{\C}(N)$.

\subsection{Existence and stability of solutions}
\label{sec:existence_stability}

Fix a model $\sfZ = (\Pi,\Gamma)$ throughout this subsection.
Define the modelled distribution spaces\label{symb:solution_spaces}
\[
\scE_\lin = \scD^{\frac32-3\kappa,\frac12-\kappa}(T[\cU_\lin])\otimes \rmM_{\C}(N)
\;
\subset
\;
\scE_\alg = \scD^{\frac32-3\kappa,\frac12-\kappa}(T[\cU_\alg])\otimes \rmM_{\C}(N)
\;,
\]
which are Banach spaces equipped with the norm $\|\Cdot\|_{\frac32-3\kappa,\frac12-\kappa}$.
We will write $\sfZ$ in the superscript, e.g. $\scE^\sfZ_\lin$, whenever we wish to emphasise the dependence on the model.

\begin{remark}\label{rem:f1_Calpha}
Uniformly in $\blue f\in\scE_\alg$ and $\sfZ=(\Pi,\Gamma)$, one has
\begin{equation}\label{eq:f1_Calpha}
|f^\1|_{C^{\frac12-\kappa}(\Lambda^\circ)} \lesssim \triple{\blue f}_{\frac32-3\kappa,\frac12-\kappa} + \|\Gamma\|_{\frac32-3\kappa}|\blue f|_{\frac32-3\kappa,\frac12-\kappa}
\;.
\end{equation}
Indeed, recall that for $\blue f\in\scD^{\gamma,\eta}$, one has
\[
|\blue f(x)-\Gamma_{xy}\blue f(y)|_{T_\sol;\ell} \lesssim
\triple{\blue f}_{\gamma,\eta} \|x-y\|^{\gamma-\ell} d(x,\partial\Lambda)^{\eta-\gamma}
\]
uniformly in $y\in B(x,\frac12 d(x,\partial \Lambda))$.
In particular, for $\ell\leq \eta\leq \gamma$,
\[
|\blue f(x)-\Gamma_{xy}\blue f(y)|_{T_\sol;\ell} \lesssim
\triple{\blue f}_{\gamma,\eta} |x-y|^{\eta-\ell}
\;.
\]
Specialising to $\ell=0$, since $\frac12-\kappa$ is the lowest strictly positive homogeneity in $T_\sol$,
we obtain for $\frac12-\kappa\leq\eta\leq\gamma$ that
\[
|f^\1(x)-f^\1(y)| \lesssim (\triple{\blue f}_{\gamma,\eta} + \|\Gamma\|_{\gamma}|\blue f|_{\gamma,\eta})|x-y|^{\frac12-\kappa}
\]
uniformly in $y\in B_{x}(\frac12 d(x,\partial \Lambda))$ and $x\in\Lambda^\circ$.
We can now use a basic dyadic sum argument to show that $|f^\1|_{C^{\frac12-\kappa}(\Lambda^\circ)} \lesssim \triple{\blue f}_{\gamma,\eta} + \|\Gamma\|_{\gamma}|\blue f|_{\gamma,\eta}$,
of which \eqref{eq:f1_Calpha} is a special case.

In particular, $f^\1$ extends to $\partial\Lambda$ and we treat in the sequel $f^\1$ as a function on $\Lambda$.
\end{remark}

Define\label{symb:modelled_gauge_spaces}
\[
\scG =\{ \blue g\in \scE_\alg \,:\, g^\1\restr_{\partial\Lambda}=1_G \text{ and } g^\1(x) \in  G \text{ for all } x\in\Lambda \}\;,
\quad
\]
equipped with the norm $\|\Cdot\|_{\frac32-3\kappa,\frac12-\kappa;\Lambda}$.
Denote further
\[
\scE_{\alg,0} = \{\blue v \in \scE_\alg \,:\, \blue v^\1 \restr_{\partial\Lambda} = 0\}
\;,
\quad
\scE_{\lin,0} = \scE_{\alg,0} \cap \scE_\lin
\;,
\]
\[
\scU_{\alg,0} =\{\blue u\in \scE_{\alg,0} \,:\, u^\1(x) \in \mfg \text{ for all } x\in\Lambda \}
\;,
\quad
\scU_{\lin,0} = \scU_\alg \cap \scE_{\lin,0}
\;.
\]
We equip these spaces with $\|\Cdot\|_{\frac32-3\kappa,\frac12-\kappa;\Lambda}$.
The subscript $0$ in $\scU_{\alg,0}$ and $\scU_{\lin,0}$ emphasises the zero boundary condition at $\partial\Lambda$.

We will soon see that $\scG$ is a Banach Lie group
with Lie algebra $\scU_{\alg,0}$.
Note that $\scE_{\lin,0}$, $\scE_{\alg,0}$, $\scU_{\lin,0}$, and $\scU_{\alg,0}$ are Banach spaces.

A Banach algebra is a Banach space $(E,\|\cdot\|)$ with a continuous bilinear  map (the product) $E^2\ni(x,y)\to xy\in E$.
Recall the product $\circ_\gamma$ on $T_\sol\otimes \rmM_\C(N)$ from \eqref{eq:circ_prod_def}.
Denote
\[
\bone = \blue\1 \otimes 1_G \in T[\1]\otimes \rmM_\C(N)\;.\label{symb:identity_RS}
\]

\begin{lemma}\label{lem:G_Lie_group}
$\scE_\alg$ is a Banach algebra under the product $\circ_{\frac32-3\kappa}$
and $\scG\subset\scE_\alg$ is a smooth Banach Lie group with Lie algebra $\scU_{\alg,0}$ and unit $\bone$.
In particular, the tangent space to $\scG$ at $\blue g\in \scG$ is $\scU_{\alg,0}{\blue g}$.
\end{lemma}

\begin{proof}
The claim that $\scE_\alg$ is a Banach algebra follows from Propositions \ref{prop:stability_U_sol}\ref{pt:T_U_sol} and \ref{prop:mult}.

The proof of the second claim is similar to the classical fact that $C^\gamma(\Lambda,G)$ is a smooth Banach Lie group with Lie algebra $C^\gamma(\Lambda,\mfg)$.
Indeed, since inversion $\scG\ni \blue g \mapsto \blue g^{-1} \in \scG$ is continuous by Proposition \ref{prop:inversion},
it follows that $\scG$ is a closed (topological) group in $\scE_\alg$.
Note moreover that $\scU_{\alg,0}$ is a closed Lie subalgebra of $\scE_\alg$.
Since $\scE_\alg$ is a Banach algebra,
$\exp \colon \scE_\alg \to \scE_\alg$, defined by its power series, is smooth,
and in particular $\exp \colon \scU_{\alg,0}\to \scG$ is smooth.
Furthermore, there exists a neighbourhood $U$ of $0\in \scU_{\alg,0}$ and a neighbourhood $V$ of $\bone\in\scG$ such that $\exp \colon U \to V$ is a 
diffeomorphism with inverse $\log\colon V\to U$ (also given by its power series).
It follows from standard arguments that $\scG$ is a Banach Lie group with Lie algebra $\scU_{\alg,0}$.
\end{proof}

Recalling $\scI^{\gamma,\eta}$ from Definition \ref{def:I}, let us denote
\[
\scI_\alg = \scI^{\frac32-3\kappa,\frac12-\kappa} \cap \scE_\alg\;.
\]
For $\sigma\in[0,1]$ and $\blue g \in \scI_\alg$, consider the modelled distribution
\label{symb:Phi}\begin{equation}\label{eq:Phi_def}
\Phi(\sigma,\blue g)
=
\cG^{i}((\cD^{\frac32-3\kappa}_i\blue{g}) \blue{g}^{-1})-\sigma \cG^{1}(\blue{g} \blue{A_1} \blue{g}^{-1})
\;,
\end{equation}
where
there is implicit summation over $i=1,2$, we recall $\cD^\gamma_i$ from \eqref{eq:rmD_gamma_def}, we denote by $\cG^i$ the operator $\cG^{i}_{\frac12-3\kappa}$ from Proposition \ref{prop:integration}, and the products are understood as $\circ_{\frac12-3\kappa}$.
Moreover, $\blue{A_1}$ here is understood as the constant function
$
\blue{A_1} = \sum_{i} \blue{A_1^{e^i}}\otimes e_i \colon\Lambda^\circ\to T[\cF_1]\otimes \mfg
$,
where we recall the dual bases
$(e_i)_{i=1}^{\dim\fg}$ and $(e^i)_{i=1}^{\dim\fg}$ of $\mfg$ and $\mfg^*$ respectively from Section \ref{sec:construction_subspace}.

\begin{remark}
By the canonical isomorphisms $T[A_1]\simeq \fg^*$ and $\fg\otimes\fg^*\simeq L(\fg)$,
one can also define $\blue{A_1}$ in a basis-independent way as the identity map in $L(\fg)$, i.e. 
   $\blue{A_1} = \psi(\id_{\fg})$ where $\psi\colon L(\fg)\to T[A_1]\otimes \fg$ is the canonical isomorphism.
\end{remark}
In the sequel, when it causes no confusion, all instances of $\cG^i$ will mean $\cG^i_{\frac12-3\kappa}$ and all products of modelled distributions will mean $\circ_{\frac12-3\kappa}$.

\begin{proposition}\label{prop:Phi_well_defined}
For any $\blue g\in\scI_\alg$, both
$\sum_{i=1}^2\cG^{i}((\cD^{\frac32-3\kappa}_i\blue{g}) \blue{g}^{-1})$ and $\cG^{1}(\blue{g} \blue{A_1} \blue{g}^{-1})$
are in $\scE_{\lin,0}$.
In particular, $\Phi\colon [0,1]\times \scI_\alg \to \scE_{\lin,0}$.
Moreover, if $\sfZ$ is a canonical model and $\blue g\in\scG$, then
$\sum_{i=1}^2\cG^{i}((\cD^{\frac32-3\kappa}_i\blue{g}) \blue{g}^{-1})$ and $\cG^{1}(\blue{g} \blue{A_1} \blue{g}^{-1})$ are in $\scU_{\lin,0}$.
\end{proposition}

As the proof will reveal, we do not need for the final part of Proposition \ref{prop:Phi_well_defined} that $\blue g^\1 \restr_{\partial\Lambda} =1_G$ for $\blue g\in \scG$
and any other fixed (sufficiently regular) boundary condition is possible.

\begin{proof}
In the proof, let $\scD^{\gamma,\eta}$ denote $\scD^{\gamma,\eta}\otimes\rmM_\C(N)$.
Fix $\sigma\in[0,1]$, $\blue g\in\scI$.
Then $\cD_i^{\frac32-3\kappa} \blue g \in \scD^{\frac12-3\kappa,-\frac12-\kappa}_{-\frac12-\kappa}(T[\cF_i\cup\{\1\}])$ by Proposition \ref{prop:derivatives}
and the homogeneities in Table \ref{tab:U_F_sets_homogeneities}.
Moreover $\blue g^{-1} \in \scD^{\frac32-3\kappa,\frac12-\kappa}_0(T[\cU_\alg])$ by Proposition \ref{prop:inversion}.
Therefore, by Propositions \ref{prop:stability_U_sol}\ref{pt:DT_U_sol} and \ref{prop:mult},
\begin{equation}\label{eq:D_g_inv}
	(\cD_i^{\frac32-3\kappa} \blue g) \blue g^{-1} \in \scD^{\frac12-3\kappa,-\frac12-\kappa}_{-\frac12-\kappa}(T[\cF_i\cup\{\1\}])
\;.
\end{equation}
Furthermore $\blue{A_1}\in\scD^{\infty,\infty}_{-\frac12-\kappa}$ and thus, by Propositions \ref{prop:stability_U_sol}\ref{pt:T_U_sol}-\ref{pt:T_U_sol_A1} and \ref{prop:mult},
\begin{equation}\label{eq:g_A1_g_inv}
\rmQ_{<\frac12-3\kappa}(\blue g \blue{A_1}\blue{g}^{-1} )\in \scD^{\frac12-3\kappa,-2\kappa}_{-\frac12-\kappa}(T[\cF_1])
\;.
\end{equation}
Since $\cG^{i}_{\frac12-3\kappa} \colon \scD^{\frac12-3\kappa,-\frac12-\kappa}_{-\frac12-\kappa}(T[\cF_i\cup\{\1\}]) \to \scD^{\frac32-3\kappa,\frac12-\kappa}_0(T[\cU_\Int])$ is a bounded operator by Proposition \ref{prop:integration},
it follows from Propositions \ref{prop:stability_U_sol}\ref{pt:DT_U_sol}, namely
the second identity in \eqref{eq:D_1_stable}, that both terms on the right-hand side of \eqref{eq:Phi_def} are in $\scE_\lin$.

To prove the first claim,
it only remains to show that $\cR\cG^{i}((\cD^{\frac32-3\kappa}_i\blue{g}) \blue{g}^{-1})\restr_{\partial\Lambda}=0$,
and similarly for the second term.
This in turn follows from $\cR\cG^{i} = G^\Dir*\partial_i\cR$ by Proposition \ref{prop:integration}, and from the fact that $G^\Dir *f$ vanishes on $\partial\Lambda$ for any $f\in C^{\beta}(\R^2)$ with support in $\Lambda$ and $\beta>-2$ by~Lemma~\ref{lem:Schauder_Dir}.

For the second claim, suppose $\sfZ$ is a canonical model and $\blue g\in\scG$.
Then $K_1*A_1\in C^{2-\theta}$ for all $\theta>0$, from which it readily follows that $\bfh_{\cJ_1 \tau}(x,y)$ defined in \eqref{eq:h_canonical} is $C^{2-\theta}$ in both variables for all $\blue\tau \in T[\{A_1,A_1\cI_1A_1\}]$.
Since $\frac32-3\kappa>1$, it follows that $g^\1\colon\Lambda^\circ\to G$ has a H\"older continuous first derivative away from the boundary.
The fact that $\sfZ$ is canonical implies, by \eqref{eq:recon_cont} and \eqref{eq:CR_D_commute}, that $\cR((\cD_i^{\frac32-3\kappa} \blue{g})\blue{g}^{-1}) = (\partial_i g^\1)(g^\1)^{-1}$, which is $\mfg$-valued.
Similarly, $\cR(\blue{g}\blue{A_1}\blue{g}^{-1})$ is $\fg$-valued.
Therefore $\Phi(\sigma,\blue g)$ is $\fg$-valued, i.e. $\Phi(\sigma,\blue g)^\1$ is $\fg$-valued. 
\end{proof}

\begin{lemma}\label{lem:D_2Phi}
The map $\Phi$ is Frechet differentiable in the second variable and,
for all $\blue g\in\scI_\alg$ and $\blue v\in\scE_\alg$,
\begin{equation}\label{eq:D_2Phi}
\begin{split}
D_2\Phi(\sigma,\blue g)[\blue v]
&=\CG^{i}((\cD_i^{\frac32-3\kappa} \blue{v})\blue{g}^{-1} - (\cD_i^{\frac32-3\kappa} \blue{g})\blue{g}^{-1}\blue v\blue{g}^{-1}) -\sigma \CG^{1}(\blue{v A_1}\blue{g}^{-1} - \blue{g A_1}\blue{g}^{-1}\blue v\blue{g}^{-1})
\\
&\in \scE_{\lin,0}
\end{split}
\end{equation}
and the map
\[
[0,1]\times\scI_\alg \times \scE_\alg
\ni
(\sigma,\blue g;\blue v)
\mapsto D_2\Phi(\sigma,\blue g)[\blue v]
\in
\scE_{\lin,0}
\]
is locally Lipschitz continuous.

Finally, if $Z$ is a canonical model and $\blue g\in\scG$, then $D_2\Phi(\sigma,\blue g) \colon \scU_{\alg,0} \blue g
\to \scU_{\lin,0}$.
\end{lemma}

\begin{proof}
Frechet differentiability and the local Lipschitz property of the corresponding derivative is an immediate consequence of continuity results from Propositions \ref{prop:mult}, \ref{prop:integration}, \ref{prop:derivatives}, and \ref{prop:inversion}.
To obtain the expression \eqref{eq:D_2Phi} for the derivative,
we note that $D (\blue{g}^{-1})[v] = -(\blue{g}^{-1}) v(\blue{g}^{-1})$,
which follows from $D(\blue g \blue g^{-1})[\blue v] = v\blue g^{-1} + \blue{g}D(\blue{g}^{-1})[\blue v] =0$,
and $D(\cD_i^{\frac32-3\kappa} g)[v] = \cD_i^{\frac32-3\kappa} v$.

The final claim follows from the expression \eqref{eq:D_2Phi} and from the same argument as in the final part of the proof of Proposition \ref{prop:Phi_well_defined}.
\end{proof}

\begin{lemma}\label{lem:D2Phi_at_id}
For $\blue v\in \scE_{\lin,0}$, one has $D_2\Phi(0,\bone)(\blue v)=\blue v$.
\end{lemma}

\begin{proof}
First note that, by \eqref{eq:D_2Phi},
\begin{equation*}
D_2\Phi(0,\bone)(\blue v)=\cG^{i}\cD_i^{\frac32-3\kappa} \blue v
\;.
\end{equation*}
Since $\sum_{i=1,2}\cI_i \rmD_i \blue\tau = \blue\tau$ for all $\blue\tau\in T[\{\cI_1A_1,\cI_1(A_1\cI_1A_1)\}]$ due to Definition \ref{def:derivatives},
the non-polynomial terms of $\cG^{i}\cD_i^{\frac32-3\kappa} \blue v$ coincide with those of $\blue v$.

To see that the the polynomial terms of $\cG^{i}\cD_i^{\frac32-3\kappa} \blue v$ also coincide with those of $\blue v$, observe that
\[
\CR\cG^{i}\cD_i^{\frac32-3\kappa} \blue v
=
G^{\Dir}*\partial_i^2 \cR\blue v
=
\cR \blue v
= v^\1
\;,
\]
where we used Propositions \ref{prop:derivatives} and \ref{prop:integration} in the first equality
and Corollary \ref{cor:GD_inverse} and that $v^\1\restr_{\partial \Lambda}=0$ in the second equality.
This shows that the $\1$ coefficient of both sides coincide.
The fact that the $\rmX_i$ coefficients coincide now follows from the fact that non-polynomial terms coincide and from the definition of modelled distributions of regularity $\frac32-3\kappa > 1$.
\end{proof}

\begin{lemma}\label{lem:D_2Phi_inv}
Consider $L>0$. There exists $\sigma_0>0$ with the following properties.
Suppose $\triple{\sfZ}_{\frac32-3\kappa} < L$ and let $\sigma \in [0,\sigma_0]$ and $\blue g\in\scI_\alg$ with $\|\blue g-\bone\|_{\frac32-3\kappa,\frac12-\kappa}\leq \sigma_0$.

\begin{enumerate}[label=(\roman*)]
	\item\label{pt:D_2E}$D_2\Phi(\sigma,\blue g)$ admits an inverse
\[
D_2\Phi(\sigma,\blue g)^{-1} \colon \scE_{\lin,0} \to \scE_{\lin,0} \blue g
\]
which has operator norm bounded by $2$.

Moreover, if $\bar\sfZ$ is another model bounded by $L$, then
uniformly in $\sigma\in[0,\sigma_0]$ and $\blue g\in\scI_\alg^\sfZ$, $\blue{\bar g} \in \scI_\alg^{\bar\sfZ}$ with
\begin{equation}\label{eq:g_bar_g_small}
\|\blue g-\bone\|_{\frac32-3\kappa,\frac12-\kappa}\vee \|\blue{\bar g}-\bone\|_{\frac32-3\kappa,\frac12-\kappa}\leq \sigma_0\;,
\end{equation}
and $\blue v \in \scE_{\lin,0}^\sfZ,\blue{\bar v} \in \scE_{\lin,0}^{\bar\sfZ}$,
\begin{multline}\label{eq:D2Phi_inv_stability}
\| D_2\Phi^{\sfZ}(\sigma, \blue g)^{-1} (\blue v) ;
D_2\Phi^{\bar \sfZ}(\sigma,\blue{\bar g})^{-1} (\blue{\bar v})\|_{\frac32-3\kappa;\frac12-\kappa}
\\
\lesssim
\|\blue v;\blue{\bar v}\|_{\frac32-3\kappa;\frac12-\kappa} + (\|\blue v\|_{\frac32-3\kappa;\frac12-\kappa}+\|\blue{\bar v}\|_{\frac32-3\kappa,\frac12-\kappa}) (\triple{\sfZ;\bar \sfZ}_{\frac32-3\kappa} + \|\blue g;\blue{\bar g}\|_{\frac32-3\kappa;\frac12-\kappa} )\;.
\end{multline}

\item\label{pt:D_2U} If $\sfZ$ is a canonical model and $\blue g\in\scG$, then $D_2\Phi(\sigma,\blue g)^{-1} \colon \scU_{\lin,0} \to \scU_{\lin,0} \blue g$.
\end{enumerate}
\end{lemma}

\begin{proof}
For \ref{pt:D_2E}, by Lemma \ref{lem:D2Phi_at_id}, $D_2\Phi(0,\bone)\restr_{\scE_{\lin,0}} = \id \restr_{\scE_{\lin,0}}$.
It follows from the local Lipschitz continuity of $D_2\Phi$ in Lemma \ref{lem:D_2Phi}
that for $\sigma_0$ sufficiently small,
$D_2\Phi(\sigma,\blue g)(\;\cdot \;\blue g) \colon \scE_{\lin,0}\to \scE_{\lin,0}$ is invertible with operator norm bounded by $2$.

For the Lipschitz continuity bound \eqref{eq:D2Phi_inv_stability}, we use Neumann series expansion.
Fix $\sigma\in[0,\sigma_0]$, $\blue g \in \scI_\alg^\sfZ, \blue{\bar g} \in \scI_\alg^{\bar\sfZ}$ as in the statement.
Write
\[
\Psi = \id - D_2\Phi^\sfZ(\sigma,\blue g)(\;\cdot\; \blue g)
\colon \scE_{\lin,0}^\sfZ \to \scE_{\lin,0}^\sfZ\;,
\]
and similarly for $\bar \Psi$.
Then, for $\sigma_0$ sufficiently small (depending on $L$),
$\|\Psi\|\leq \frac12$, where we denote by $\|\Psi\|$ the operator norm of $\Psi\colon \scE_{\lin,0}^\sfZ \to \scE_{\lin,0}^\sfZ$, and for $\blue v\in \scE_{\lin,0}^\sfZ$,
\begin{equation}\label{eq:Phi_Neumann}
D_2\Phi^\sfZ(\sigma,\blue g) (\;\cdot\; \blue g)^{-1}(v) = \sum_{k\geq0} \Psi^k(\blue v)\;.
\end{equation}
Likewise for $\bar\Psi$.
Note that, by \eqref{eq:D_2Phi} and Lemma \ref{lem:D2Phi_at_id},
\[
\Psi(\blue v)
=
-\cG^i[\blue v, (\cD_i^{\frac32-3\kappa} \blue g) \blue g^{-1}] + \sigma \cG^1 [\blue v, \blue{g A_1}\blue{g}^{-1}]\;,
\]
where the Lie brackets are defined by $[a,b] = a\circ_{\frac12-3\kappa} b - b \circ_{\frac12-3\kappa} a$.
It follows from the continuity results of Propositions \ref{prop:mult}, \ref{prop:integration}, \ref{prop:derivatives}, and \ref{prop:inversion}
that there exists $C>0$ such that,
denoting $\Sigma = \triple{\sfZ;\bar \sfZ}_{\frac32-3\kappa} + \|\blue g;\blue{\bar g}\|_{\frac32-3\kappa,\frac12-\kappa}$,
\begin{equation}\label{eq:Psi_diff}
\|\Psi(\blue v); \bar\Psi(\blue{\bar v})\|_{\frac32-3\kappa,\frac12-\kappa}
\leq
C (\|\blue v\|_{\frac32-3\kappa,\frac12-\kappa}
+\|\blue{\bar v}\|_{\frac32-3\kappa,\frac12-\kappa})\Sigma
+ \frac12 \|\blue v;\blue{\bar v}\|_{\frac32-3\kappa,\frac12-\kappa}
\;.
\end{equation}
We now claim that, for $k\geq 1$,
\begin{equation}\label{eq:Psik_diff}
\|\Psi^k(\blue v);\bar\Psi^k(\blue{\bar v})\|_{\frac32-3\kappa;\frac12-\kappa} \leq 2^{-k} \|\blue v;\blue{\bar v}\|_{\frac32-3\kappa;\frac12-\kappa} +
2^{-k+2} C (\|\blue v\|_{\frac32-3\kappa;\frac12-\kappa}+\|\blue{\bar v}\|_{\frac32-3\kappa;\frac12-\kappa})\Sigma\;.
\end{equation}
Indeed, this is implied by \eqref{eq:Psi_diff} for $k=1$. Then, proceeding by induction, for $k\geq 2$,
\begin{equs}
\|\Psi^k(\blue v);\bar\Psi^k(\blue{\bar v})\|_{\frac32-3\kappa;\frac12-\kappa}
&\leq
C(\|\Psi^{k-1}(\blue v)\|_{\frac32-3\kappa;\frac12-\kappa}+
\|\bar\Psi^{k-1}(\blue{\bar v})\|_{\frac32-3\kappa;\frac12-\kappa})\Sigma
\\
&+
\frac12 \|\Psi^{k-1}(\blue v);\bar\Psi^{k-1}(\blue{\bar v})\|_{\frac32-3\kappa;\frac12-\kappa}
\\
&\leq
2^{-k+1} C (\|\blue v\|_{\frac32-3\kappa,\frac12-\kappa}
+
\|\blue{\bar v}\|_{\frac32-3\kappa,\frac12-\kappa}) \Sigma
\\
&+
2^{-k} \|\blue v;\blue{\bar v}\|_{\frac32-3\kappa,\frac12-\kappa} 
+
2^{-k+1} C (\|\blue v\|_{\frac32-3\kappa,\frac12-\kappa} + \|\blue{\bar v}\|_{\frac32-3\kappa,\frac12-\kappa}) \Sigma
\\
&\leq
2^{-k} \|\blue v;\blue{\bar v}\|_{\frac32-3\kappa,\frac12-\kappa}
+
2^{-k+2} C (\|\blue v\|_{\frac32-3\kappa,\frac12-\kappa}  + \|\blue{\bar v}\|_{\frac32-3\kappa,\frac12-\kappa} ) \Sigma
\;,
\end{equs}
where we used \eqref{eq:Psi_diff} in the first bound, and $\|\Psi\|\leq \frac12$ and the induction hypothesis in the second bound.
This completes the proof of \eqref{eq:Psik_diff}.

By the definition of the distance $\|\Cdot;\Cdot\|_{\frac32-3\kappa,\frac12-\kappa}$ in \eqref{eq:f_norm_def},
\begin{equation}\label{eq:diff_sum}
\|\blue v + \blue w;\blue{\bar v}+\blue{\bar w}\|_{\frac32-3\kappa,\frac12-\kappa}  \leq \|\blue v;\blue{\bar v}\|_{\frac32-3\kappa,\frac12-\kappa}  + \|\blue w;\blue{\bar w}\|_{\frac32-3\kappa,\frac12-\kappa} \;.
\end{equation}
It follows that
\begin{equs}{}
&\| D_2\Phi^\sfZ(\sigma,\blue g) (\;\cdot\; \blue g)^{-1} (\blue v) ;
D_2\Phi^{\bar\sfZ}(\sigma,\blue{\bar g}) (\;\cdot\; \blue{\bar g})^{-1} (\blue{\bar v})\|_{\frac32-3\kappa;\frac12-\kappa}
\\
&\qquad\leq
\sum_{k\geq0} \|\Psi^k(\blue v); \bar\Psi^k(\blue{\bar v})\|_{\frac32-3\kappa;\frac12-\kappa}
\\
&\qquad
\lesssim
\|\blue v;\blue{\bar v}\|_{\frac32-3\kappa;\frac12-\kappa}
+
(\|\blue v\|_{\frac32-3\kappa;\frac12-\kappa}+\|\blue{\bar v}\|_{\frac32-3\kappa;\frac12-\kappa}) (\triple{\sfZ;\bar \sfZ}_{\frac32-3\kappa} + \|\blue g;\blue{\bar g}\|_{\frac32-3\kappa;\frac12-\kappa} )\;,
\end{equs}
where we used \eqref{eq:Phi_Neumann} and \eqref{eq:diff_sum} in the first bound, and \eqref{eq:Psik_diff} in the second bound.
Since multiplication by $\blue g,\blue{\bar g}$ satisfying \eqref{eq:g_bar_g_small} is Lipschtiz continuous in the sense of Proposition \ref{prop:mult},
one obtains the bound \eqref{eq:D2Phi_inv_stability}.
This completes the proof of \ref{pt:D_2E}.

For \ref{pt:D_2U}, by the final part of Lemma \ref{lem:D_2Phi}, $D_2\Phi^Z(\sigma,\blue g)(\;\cdot\; \blue g) \colon \scU_{\sol,0} \to \scU_{\sol,0}$
is the identity for $\sigma=0$ and $\blue g = \bone$,
therefore for $\sigma_0$ sufficiently small, $D_2\Phi^Z(\sigma,\blue g)(\;\cdot\; \blue g)^{-1}$ maps $\scU_{\sol,0} \to \scU_{\sol,0}$,
i.e. $D_2\Phi^Z(\sigma,\blue g)^{-1}$ maps $\scU_{\sol,0} \to \scU_{\sol,0}\blue g$.
\end{proof}

\begin{theorem}
\label{thm:existence_solution}
Let $L>0$. There exists $\sigma_\sol = \sigma_\sol(L)\in (0,1]$ with the following property.
Assume $\triple{\sfZ}_{\frac32-3\kappa} < L$.
For all $\sigma\in[0,\sigma_\sol]$,
there exists $\blue{g}^\sigma\in \scI_\alg$ such that $\Phi(\sigma,\blue{g}^\sigma)=0$.
Moreover, one can choose $\blue{g}^\sigma$ such that $\blue{g}^0=\bone$ and, for any other model $\bar\sfZ$ with $\triple{\bar\sfZ}_{\frac32-3\kappa} < L$, denoting by
$\blue{\bar g}^{\sigma}\in\scI_\alg^{\bar \sfZ}$ the corresponding solution for $\bar\sfZ$,
\begin{equation}\label{eq:sol_Lip}
\|\blue{g}^\sigma ; \blue{\bar g}^{\bar\sigma}\|_{\frac32-3\kappa,\frac12-\kappa}
\lesssim \triple{\sfZ;\bar\sfZ}_{\frac32-3\kappa} + |\sigma-\bar\sigma|\;,
\end{equation}
where the proportionality constant depends only on $L$.
Finally, if $\sfZ$ is canonical, then $\blue{g}^\sigma \in \scG$ for all $\sigma\in[0,\sigma_\sol]$.
\end{theorem}

\begin{remark}\label{rem:non_unique_g}
The reader should keep in mind that there is not a unique $\blue g^\sigma \in \scG$ such that $\Phi(\sigma,\blue g)=0$ for $\sigma>0$.
Even for $\sfZ$ as the canonical model for $A_1\equiv 0$,
by the the proof of Lemma \ref{lem:expression_g_2} below,
there is a family of solutions given by
$\blue g^\sigma = 1_G\1 + \sigma 1_G\blue{\cI_1A_1} + g_1\blue{\cI_1(A_1\cI_1A_1)} + g_2 \blue{(\cI_1 A_1)^2}$,
where $g_1 \in \rmM_\C(N)$ is arbitrary and $g_2 \in \rmM_\C(N)$ is determined by $g_1$ and $\sigma$.
This non-uniqueness essentially arises from our choice to apply Leibniz rule when defining $\rmD_1\blue{(\cI_1A_1)^2}$.
\end{remark}

\begin{proof}
We first claim that there exists a differentiable map
\[
[0,\sigma_0]\ni\sigma\mapsto \blue{g^\sigma} \in \scI_\alg
\]
for $\sigma_0$ sufficiently small such that
\begin{equation}\label{eq:g_ODE}
\begin{split}{}
&\partial_\sigma \blue{g}^\sigma
=
-D_2\Phi(\sigma,\blue{g}^\sigma)^{-1} [D_1\Phi(\sigma,\blue{g}^\sigma)]
=
D_2\Phi(\sigma,\blue{g}^{\sigma})^{-1} [\cG^{1}\blue{g}^{\sigma} \blue{A_1} (\blue{g}^{\sigma})^{-1}]
\;,
\\
&\blue{g}^0
=
\bone\;.
\end{split}
\end{equation}
Indeed, note that, by Proposition \ref{prop:Phi_well_defined},
$\cG^{1}\blue{g} \blue{A_1} \blue{g}^{-1} \in \scE_{\lin,0}$ for any $\blue g \in \scI_\alg$.
It follows from Lemma \ref{lem:D_2Phi_inv}\ref{pt:D_2E} and Propositions \ref{prop:mult} and \ref{prop:integration}
that
the map
\[
\scI_\alg \ni \blue g \mapsto D_2\Phi(\sigma,\blue{g})^{-1} [\cG^{1}\blue{g} \blue{A_1} \blue{g}^{-1}]
\in
\scE_{\lin,0}\blue g
\]
is Lipschitz in $\blue g\in\scI_\alg$ for which $\|\blue g-\bone\|_{\frac32-3\kappa,\frac12-\kappa}\leq \sigma_0$ with Lipschitz constant uniform in $\sigma\leq\sigma_0$.
Observe the inclusion $\{\blue g\in\scE_\alg\,:\, \|\blue g-\bone\|_{\frac32-3\kappa,\frac12-\kappa}\leq \sigma_0\} \subset\scI_\alg$ for $\sigma_0$ sufficiently small.
Local existence (and uniqueness) of solutions to the ODE \eqref{eq:g_ODE} with $\blue g^\sigma\in\scI_\alg$ now follows from the classical Picard--Lindel{\"o}f theorem.
Note that the interval of existence $[0,\sigma_\sol]$ and the Lipschitz constant of $\sigma \mapsto \blue g^\sigma\in\scI_\alg$ depend only on $L$.

Observe now that for any differentiable map $\sigma\mapsto \blue{g}^\sigma$,
by the chain rule,
\[
\partial_\sigma \Phi(\sigma,\blue{g}^\sigma) = D_1\Phi^Z(\sigma,\blue{g}^\sigma) + D_2\Phi^Z(\sigma,g^\sigma)[\partial_\sigma \blue{g}^\sigma]
\;.
\]
Therefore, since $\Phi(0,\bone)=0$,
if $\blue g^\sigma$ solves \eqref{eq:g_ODE}
then $\Phi(\sigma,\blue g^\sigma)=0$ for all $\sigma\in[0,\sigma_\sol]$.
This completes the proof of existence of $\blue g\in\scI_\alg$ for sufficiently small $\sigma$ such that $\Phi(\sigma,\blue g)=0$.

We now prove the bound \eqref{eq:sol_Lip}.
By Lipschtiz dependence of $\blue g^\sigma$ on $\sigma$, it suffices to consider $\sigma=\bar\sigma$.
Consider another model $\bar\sfZ$ with $\triple{\bar\sfZ}<L$ with corresponding solution $\blue{\bar g} \in\scI_\alg^{\bar\sfZ}$.
Then
\begin{equs}{}
&\sup_{\sigma\in [0,\sigma_\sol]}\|\blue{g}^\sigma; \blue{\bar g}^{\sigma}\|_{\frac32-3\kappa,\frac12-\kappa}
\leq  \int_0^{\sigma_\sol} \|\partial_s\blue{g}^s; \partial_s\blue{\bar g}^{s}\|_{\frac32-3\kappa,\frac12-\kappa} \dif s
\\
&=
\int_0^{\sigma_\sol} \|D_2\Phi^\sfZ(s,\blue{g}^s)^{-1} [D_1 \Phi^\sfZ(s,\blue{g}^s)];D_2\Phi^{\bar\sfZ}(s,\blue{\bar g}^s)^{-1} [D_1 \Phi^{\bar\sfZ}(s,\blue{\bar g}^s)]\|_{\frac32-3\kappa,\frac12-\kappa} \dif s
\\
&\lesssim
\int_0^{\sigma_\sol} \{
\|D_1 \Phi^\sfZ(s,\blue{g}^s); D_1 \Phi^{\bar \sfZ}(s,\blue{\bar g}^s)\|_{\frac32-3\kappa,\frac12-\kappa}
\\
&\qquad
+
(\|D_1 \Phi^\sfZ(s,\blue{g}^s)\|_{\frac32-3\kappa,\frac12-\kappa}+\|D_1 \Phi^{\bar \sfZ}(s,\blue{\bar g}^s)\|_{\frac32-3\kappa,\frac12-\kappa}) (\triple{\sfZ;\bar \sfZ}_{\frac32-3\kappa} + \|\blue{g}^s;\blue{\bar g}^s\|_{\frac32-3\kappa,\frac12-\kappa} ) \} \dif s
\end{equs}
where we used \eqref{eq:diff_sum} and $\blue g^0 = \blue{\bar g}^0$ in the first bound, that $\blue{g}^\sigma,\blue{\bar g}^\sigma$ solve \eqref{eq:g_ODE} in the equality,
and Lemma \ref{lem:D_2Phi_inv}\ref{pt:D_2E} in the second bound.
Using that $\|D_1 \Phi^\sfZ(s,\blue{g}); D_1 \Phi^{\bar \sfZ}(s,\blue{\bar g})\|_{\frac32-3\kappa,\frac12-\kappa} \lesssim \|\blue g;\blue{\bar g}\|_{\frac32-3\kappa,\frac12-\kappa} + \triple{\sfZ;\bar\sfZ}_{\frac32-3\kappa}$ locally uniformly in $\blue g,\blue{\bar g}$,
we obtain that, for $\sigma_\sol$ sufficiently small,
\[
\sup_{\sigma\in[0,\sigma_\sol]}\|\blue g^\sigma; \blue{\bar g}^{\sigma}\|_{\frac32-3\kappa,\frac12-\kappa} \lesssim \triple{\sfZ;\bar \sfZ}_{\frac32-3\kappa}
\;,
\]
which completes the proof of \eqref{eq:sol_Lip}.

Finally, suppose $\sfZ$ is canonical.
Then by Lemma \ref{lem:D_2Phi_inv}\ref{pt:D_2U} and Proposition \ref{prop:Phi_well_defined},
\[
D_2\Phi(\sigma,\blue{g})^{-1} [\cG^{1}\blue{g} \blue{A_1} \blue{g}^{-1}] \in \scU_{\lin,0}\blue g
\]
whenever $\blue g\in\scG$.
By Lemma \ref{lem:G_Lie_group}, $\scU_{\alg,0} \blue g\supset \scU_{\lin,0}\blue g$ is the tangent space to $\scG$ at $\blue g$.
Consequently, every solution to the ODE \eqref{eq:g_ODE} remains in $\scG$ for all $\sigma\in [0,\sigma_0]$ since the initial value $\bone$ is in $\scG$,
which proves the claim.

\end{proof}

\subsection{Properties of solution}
\label{sec:properties_solution}

For the rest of the section, fix $L>0$, a model $\sfZ$ with $\triple{\sfZ}_{\frac32-3\kappa}< L$, and $\sigma_\sol$ as in Theorem \ref{thm:existence_solution}.
For $\sigma \in [0,\sigma_\sol]$, let $\blue g = \blue g^\sigma\in\scI_\alg$ be as in Theorem \ref{thm:existence_solution} such that $\Phi(\sigma,\blue g)=\blue 0$.
We next derive important structural properties for the ceofficients of $\blue g$.
For $\tau\in S_\sol$ and $\blue f\colon \R^2\to T_\sol\otimes\rmM_\C(N)$, let us write $[f^\tau]$ for the term of $\blue f$ in $T[\tau]\otimes \rmM_{\C}(N)$.
As usual $f^\1$ will denote coefficient of $\1$ in the expansion of $\blue f$, i.e. $[f^\1] = \blue{\1}\otimes f^\1$.

\begin{lemma}\label{lem:expression_g_1}
$[g^{\cI_1A_1}]=\sigma g^\1\blue{\cI_1A_1}$.
In particular,
\begin{equation}\label{eq:g_g-1_expansion}
	\blue g = g^\1\blue{\1}+\sigma g^\1\blue{\cI_1A_1}
+ \blue \mcQ\;,
\qquad
\blue g^{-1}=(g^\1)^{-1}\blue{\1}-\sigma
\blue{\cI_1A_1} (g^{\1})^{-1}
+ \blue{\bar\mcQ}\;,
\end{equation}
where $\blue\mcQ,\blue{\bar\mcQ}$ take values in $\bigoplus_{|\tau|\geq 1-2\kappa} T[\tau]\otimes \rmM_{\C}(N)$.
\end{lemma}

\begin{proof}
The terms in $T[\CI_1 A_1]\otimes \rmM_{\C}(N)$ in the expansion of $\Phi(\sigma,\blue g)$
come either from the first term in \eqref{eq:Phi_def} by taking $\CI_1A_1$ term from $\blue g$ and $\1$ term from $\blue{g}^{-1}$ (we recall for this that $\rmD_1\blue{\CI_1A_1} = \blue{A_1}-\blue{\CI_{22}A_1}$),
or from the second term in \eqref{eq:Phi_def}
by taking $g^\1\blue\1$ from $\blue g$ and $(g^\1)^{-1}\blue\1$ from $\blue{g}^{-1}$.
Therefore
\[
[g^{\cI_1A_1}](g^\1)^{-1}-\sigma g^\1\blue{\cI_1A_1} (g^\1)^{-1}=0
\]
from which the first claim follows.
The expression for $\blue g^{-1}$ is a direct consequence of \eqref{eq:g_inv_explicit}.
\end{proof}

\begin{lemma}\label{lem:expression_g_square}
$[g^{(\cI_1A_1)^2}] = \frac{\sigma^2}{2}g^\1 \blue{(\cI_1A_1)^2}$.
\end{lemma}

\begin{proof}
Define
\[
\blue{u}_\sigma = D_2\Phi(\sigma,\blue{g}_\sigma)^{-1} [\cG^{1}\blue{g}_\sigma \blue{A_1} \blue{g}_\sigma^{-1}] \blue g^{-1} \in \scE_{\lin,0}\;,
\]
so that, from the proof of Theorem \ref{thm:existence_solution},
\begin{equation*}
\dot{\blue g}_\sigma = \blue{u}_\sigma \blue g_\sigma\;,
\end{equation*}
where $\dot{\blue g}_\sigma$ refers to the $\sigma$-derivative of $\blue g_\sigma$.
Comparing coefficients of $\1$, it follows that
\begin{equation}\label{eq:dot_g_1}
\dot{g}_\sigma^\1 = u_\sigma^\1 g_\sigma^\1
\end{equation}
while comparing coefficients of $\cI_1A_1$,
\[
(g^\1_\sigma + \sigma\dot{g}^{\1}_\sigma)\blue{\cI_1A_1}
=
\sigma u_\sigma^\1 g_\sigma^\1 \blue{\cI_1A_1} + [u_\sigma^{\cI_1A_1} ] g_\sigma^\1
\;.
\]
Substituting $\dot{g}_\sigma^\1=u_\sigma^\1 g_\sigma^\1$ due to \eqref{eq:dot_g_1} and cancelling $
\sigma u_\sigma^\1 g_\sigma^\1 \blue{\cI_1A_1}$ on both sides, we obtain
\[
[u_\sigma^{\cI_1A_1} ] = g^\1_\sigma \blue{\cI_1A_1} (g^\1_\sigma)^{-1}\;.
\]
We now see that the coefficient of $(\cI_1A_1)^2$ is given by, using that
$[u^{(\cI_1A_1)^2}_\sigma]=0$,
\[
[\dot g^{(\cI_1A_1)^2}_\sigma]
=
\sigma g^\1_\sigma \blue{\cI_1A_1} \blue{\cI_1A_1}
+
u^\1_\sigma [g^{(\cI_1A_1)^2}_\sigma]
\;.
\]
This is an ODE for $[g^{(\cI_1A_1)^2}_\sigma]$ with initial condition zero,
the unique solution of which, by \eqref{eq:dot_g_1}, is
$
[g^{(\cI_1A_1)^2}_\sigma] = \frac{\sigma^2}{2}g^\1_\sigma\blue{(\cI_1A_1)^2}$.
\end{proof}

\begin{lemma}\label{lem:expression_g_2}
$[g^{\cI_1A_1\cI_1A_1}]
=
\frac{\sigma^2}{2} g^{\1} \blue{\cI_1 [\cI_1A_1,A_1]}$,
where $\blue{[A_1,\cI_1A_1]}$ is the Lie bracket in $T[\cU_\alg]\otimes \rmM_\C(N)$ induced by the 
usual associative product $\circ_{\frac32-3\kappa}$.
\end{lemma}

\begin{proof}
There are five contributions to the term in $T[\CI_1(A_1\CI_1A_1)]\otimes \rmM_{\C}(N)$ of $\Phi(\sigma,\blue g)$:
\begin{enumerate}
	\item From the first term in \eqref{eq:Phi_def} by taking $(\cI_1 A_1)^2$ term from $\blue g$ and $\1$ from $\blue g^{-1}$;
	we recall for this the Leibniz rule $\rmD_1\blue{(\cI_1 A_1)^2} = (\blue{A_1} - \blue{\CI_{22} A_1})\blue{\CI_1 A_1} + \blue{\CI_1A_1}(\blue{A_1} - \blue{\CI_{22} A_1})$,
	\item from the first term in \eqref{eq:Phi_def} by taking $\cI_1(A_1\cI_1A_1)$ term from $\blue g$ and $\1$ from $\blue g^{-1}$,
	\item from the first term in \eqref{eq:Phi_def} by taking $\CI_1A_1$ term from $\blue g$ and $\CI_1A_1$ from $\blue g^{-1}$,
	\item from the second term in \eqref{eq:Phi_def} by taking $\CI_1A_1$ term from $\blue g$ and $\1$ from $\blue g^{-1}$,
	\item from the second term in \eqref{eq:Phi_def} by taking $\1$ term from $\blue g$ and $\CI_1A_1$ from $\blue g^{-1}$.
\end{enumerate}
Combining these contributions and using Lemmas \ref{lem:expression_g_1} and \ref{lem:expression_g_square},
we obtain
\begin{align*}
\frac{\sigma^2}{2} g^{\1} \blue{\cI_1(A_1\cI_1A_1 + (\cI_1A_1) A_1)}& (g^\1)^{-1}+
[g^{\cI_1A_1\cI_1A_1}] (g^\1)^{-1}
-\sigma^2 g^\1\blue{\cI_1 (A_1\cI_1 A_1)}(g^\1)^{-1}
\\
&-\sigma^2 g^\1\blue{\cI_1(\cI_1 (A_1)A_1)}(g^\1)^{-1} 
+\sigma^2 g^{\1} \blue{\cI_1(A_1\cI_1A_1)}(g^\1)^{-1}=0.
\end{align*}
Remark that the 3rd and 5th terms cancel and we obtain thd conclusion by rearranging the above expression.
\end{proof}

Lemmas \ref{lem:expression_g_1}, \ref{lem:expression_g_square} and \ref{lem:expression_g_2}
imply that
\begin{equation}\label{eq:expression_g}
\blue g=g^\1\blue{\1}+\sigma g^\1\blue{\cI_1A_1}+\frac{\sigma^2}{2} g^\1 (\blue{\cI_1 [\cI_1A_1,A_1]} + \blue{(\cI_1A_1)^2})+g^{\rmX_i}\rmX_i
\;.
\end{equation}
The following lemma is the main property of $\blue g$ that we will use in the rest of the paper.
Let us write $(\sigma \blue{A})^{\blue g}_1 = \sigma \blue g\blue{A_1}{\blue g}^{-1} - (\cD_1^{\frac32-3\kappa}\blue g){\blue g}^{-1}$
and
$(\sigma \blue{A})^{\blue g}_2 = - (\cD_2^{\frac32-3\kappa}\blue g){\blue g}^{-1}$,
which are modelled distributions in $\scD^{\frac12-3\kappa,-\frac12-\kappa}_{-\frac12-\kappa}\otimes \rmM_\C(N)$ by \eqref{eq:D_g_inv}-\eqref{eq:g_A1_g_inv}.

\begin{lemma}\label{lem:B_expresssion}
For $i=1,2$, let $\blue{B_i} =\rmQ_{<\frac12-3\kappa}(\sigma \blue{A})^{\blue g}_i$
and denote $f_i = (\cD_i^{\frac32-3\kappa}g)^\1$.
Then $\blue B_i \in \scD^{\frac12-3\kappa,-\frac12-\kappa}_{-\frac12-\kappa}$ and
\and
\begin{equation}\label{eq:B_expression}
\cR \blue B_i = \sigma g^\1 \zeta_i (g^\1)^{-1} - f_i(g^\1)^{-1}
\end{equation}
where $\zeta_1=\cR\blue{\cI_{22}A_1}$, $\zeta_2=-\cR\blue{\cI_{12}A_1}$,
and $\sigma g^\1 \zeta_i (g^\1)^{-1}$ are well-defined as Young products (see Lemma \ref{lem:Young_product}).
For any $\Lambda'\Subset\Lambda^\circ$, $|f_i(g^\1)^{-1}|_{C^{\frac12-3\kappa}(\Lambda')} \lesssim \dist(\d\Lambda,\d\Lambda')^{-\kappa-\frac12}$.
\end{lemma}

\begin{proof}
The fact that $\blue B \in \scD^{\frac12-3\kappa,-\frac12-\kappa}_{-\frac12-\kappa}$ follows from \eqref{eq:D_g_inv} and \eqref{eq:g_A1_g_inv}.
The claim that $\sigma g^\1 \zeta_i (g^\1)^{-1}$ are well-defined as Young products follows from $g^\1\in C^{\frac12-\kappa}(\Lambda)$ by Remark \ref{rem:f1_Calpha} and $\zeta_i\in C^{-2\kappa}$
(to apply Lemma \ref{lem:Young_product} we extend $g^\1$ by $1_G$ outside $\Lambda$).

The bound on $f_i(g^\1)^{-1}$ in $C^{\frac12-3\kappa}$ away from the boundary
is due to $g^{\1}\in C^{\frac12-\kappa}(\Lambda)$ and $f_i\in C^{\frac12-3\kappa}(\Lambda')$ with the corresponding blow up at the boundary because of $\blue g\in\scD^{\frac32-3\kappa,\frac12-\kappa}_0$ and Proposition \ref{prop:derivatives}.
It remains to prove \eqref{eq:B_expression},
for which we compute the two components $\blue{B_1}$ and $\blue{B_2}$ separately.
For $\blue{B_1}$, by Lemma \ref{eq:g_g-1_expansion},
\begin{equation}\label{eq:gA1g-1}
\begin{split}
\rmQ_{<\frac12-3\kappa}(\sigma\blue{gA_1}\blue{g}^{-1})
&=\sigma \rmQ_{<\frac12-3\kappa} (g^\1\blue{\1}+\sigma g^\1\blue{\cI_1A_1})\blue{A_1}((g^\1)^{-1}\blue\1-\sigma\blue{\cI_1A_1}(g^\1)^{-1})
\\
&=
\sigma g^\1\blue{A_1}(g^\1)^{-1}
-
\sigma^2 g^\1 \blue{A_1\cI_1A_1}(g^\1)^{-1}+\sigma^2g\blue{\cI_1(A_1)A_1}g^{-1}
\\
&=
\sigma g^\1\blue{A_1}(g^\1)^{-1}
+
\sigma^2 g^\1 [\blue{\cI_1A_1},\blue{A_1}](g^\1)^{-1}\;.
\end{split}
\end{equation}
We also compute, by \eqref{eq:expression_g},
\begin{align*}
\cD_1^{\frac32-3\kappa} \blue{g}
&=
\rmQ_{<\frac12-3\kappa}
\Big(
\sigma g^\1 \cD_1\blue{\cI_1A_1}
+ \frac{\sigma^2}{2} g^\1 \cD_1 (\blue{\cI_1 [\cI_1A_1,A_1]}
+ \blue{(\cI_1A_1)^2})
+g^{\rmX_1} \blue\1
\Big)
\;.
\end{align*}
By definition of $\rmD_1$ from Definition \ref{def:derivatives},
\begin{align*}
\rmD_1\blue{\cI_1A_1}&=\blue{A_1}-\blue{\cI_{22}A_1}\\
\rmD_1\blue{\cI_1[\cI_1A_1,A_1]}&=\blue{[\cI_1A_1,A_1]}-\blue{\cI_{22}[\cI_1A_1,A_1]} \\
\rmD_1\blue{(\cI_1A_1)^2}
&= (\blue{A_1}-\blue{\cI_{22}A_1})\blue{\cI_1 A_1}+\blue{\cI_1 A_1}(\blue{A_1}-\blue{\cI_{22}A_1})
\;,
\end{align*}
where the last expression is taken directly from Example \ref{ex:der_square}.
Substituting these expressions into $(\cD_1^{\frac32-3\kappa}\blue{g})\blue{g}^{-1}$,
and noting that the projection $\rmQ_{<\frac12-3\kappa}$ makes the contributions with $\blue{\cI_{22}\tau}$ from the last two lines vanish, we see that 
\begin{equation}\label{eq:D1gg-1}
\begin{split}
&\rmQ_{<\frac12-3\kappa}\bigl((\cD_1^{\frac32-3\kappa}\blue g)\blue g^{-1}\bigr)\\
&=
\rmQ_{<\frac12-3\kappa}
\Big[\Big(
\sigma g^\1\bigl(\blue{A_1}-\blue{\cI_{22}A_1}\bigr)
+\frac{\sigma^2}{2} g^\1 (\blue{[\cI_1A_1,A_1]}+
\blue{A_1\cI_1A_1}+\blue{(\cI_1A_1)A_1} )
\\
&\qquad\qquad\qquad +f_1\blue\1
\Big)
\Big((g^\1)^{-1}\blue\1-\sigma \blue{\cI_1A_1}(g^\1)^{-1}
\Big)
\Big]
\\
&=
\sigma g^\1\blue{A_1}(g^\1)^{-1}
-\sigma g^\1\blue{\cI_{22}A_1}(g^\1)^{-1}
\\
&\qquad
-\sigma^2 g^\1\blue{A_1\cI_1A_1}(g^\1)^{-1}
+\sigma^2 g^\1 \blue{(\cI_1A_1)A_1}(g^\1)^{-1}
+f_1(g^\1)^{-1}\blue\1
\\
&=
\sigma\, g^\1\blue{A_1}(g^\1)^{-1}
-\sigma\, g^\1\blue{\cI_{22}A_1}(g^\1)^{-1}
+\sigma^2 g^\1
\blue{[\cI_1A_1,A_1]}(g^\1)^{-1}
+f_1(g^{\1})^{-1}\,\blue\1
\;.
\end{split}
\end{equation}
Consequently, combining \eqref{eq:gA1g-1} and \eqref{eq:D1gg-1},
\begin{equation}\label{eq:B1_expression}
\begin{split}
\blue{B_1}
&=
\rmQ_{<\frac12-3\kappa}\bigl((\sigma\blue{gA_1g^{-1}})-(\cD_1^{\frac32-3\kappa}\blue g)\blue g^{-1}\bigr)
\\
&=
\sigma g^\1\blue{\cI_{22}A_1}(g^\1)^{-1}
-\;f_1(g^{\1})^{-1}\,\blue\1
\;.
\end{split}
\end{equation}
For $\blue{B_2}$, we have from \eqref{eq:expression_g} and the fact that symbols in $\cF$ in Table \ref{tab:U_F_sets_homogeneities} have homogeneity $\frac12-3\kappa$
\[
\cD_2^{\frac32-3\kappa} \blue g
=
\sigma g^\1 \blue{\cI_{12}A_1}
+f_2\blue\1
\;.
\]
Therefore, since each term in $\cD_2^{\frac32-3\kappa} \blue g$ has homogeneity at least $-2\kappa$,
\begin{equation}\label{eq:B2_expression}
\blue B_2
= - \rmQ_{<\frac12-3\kappa}\bigl((\cD_2^{\frac32-3\kappa}\blue g)\blue g^{-1}\bigr)
=
-\sigma g^\1 \blue{\cI_{12}A_1}(g^\1)^{-1}
- f_2(g^\1)^{-1}\blue\1
\;.
\end{equation}
The expression \eqref{eq:B_expression} now follows from \eqref{eq:B1_expression}-\eqref{eq:B2_expression} and the fact that $\sigma g^\1 \zeta_i (g^\1)^{-1}$ are well-defined as Young products,
which implies by the defining property of $\cR$ in \eqref{eq:CR_defining} and of Young products in \eqref{eq:Young_bound} that $\cR(g^\1 \blue{\cI_{22}A_1}(g^\1)^{-1}) = \sigma g^\1 \zeta_1 (g^\1)^{-1}$
and $\cR(g^\1 \blue{\cI_{12}A_1}(g^\1)^{-1}) = \sigma g^\1 \zeta_2 (g^\1)^{-1}$.
\end{proof}

Finally, we record the following simple property of $g^\1$ if $\sfZ$ is canonical.

\begin{lemma}\label{lem:g_smooth}
Suppose $\sfZ$ is a canonical model built from a smooth $A_1\colon \R^2\to \mfg$ (see Definition \ref{def:canonical_model}).
Then $g^\1 = \CR\blue g \colon \Lambda^\circ \to G$ is smooth and $\diff^* (\sigma A)^{g^\1} \equiv 0$ on $\Lambda^\circ$.
Finally, $g^\1 = 1_G$ if $A_1=0$.
\end{lemma}

\begin{proof}
Applying $\Delta$ to the reconstruction of \eqref{eq:Phi_def} and using that $\CR\Phi(\sigma,\blue g)=0$ and $\cR\cG^i = G^\Dir*\partial_i$,
we obtain that $g^\1 = \CR\blue g$ solves the PDE on $\Lambda^\circ$
\begin{equation}\label{eq:g_PDE_explicit}
\sigma \partial_1 (g^\1 (A_1) (g^\1)^{-1})
=
\partial_i((\partial_i g^\1 )(g^{\1})^{-1})
=
\Delta g^\1 (g^\1)^{-1} + \partial_i g^\1 \partial_i (g^\1)^{-1}\;.
\end{equation}
We write this as $\Delta g^\1 = f$, where $f$ is a multilinear function of $g^\1$, $(g^{\1})^{-1}$, $\partial_i g^\1$, and $A_1$.
Since $\sfZ$ is canonical, by the argument at the end of the proof of Proposition \ref{prop:Phi_well_defined},
$g^\1$ has a H\"older continuous first derivative away from the boundary.
It follows from classical Schauder estimates that $g$ is smooth away from the boundary.
The equality $\diff^*(\sigma A)^{g^\1} = 0$ is equivalent to \eqref{eq:g_PDE_explicit}.

For the final claim, if $A_1=0$, then $\blue g^\sigma$ as in \eqref{eq:expression_g} with $g^\1=1_G$ and $g^{\rmX_i}=0$
is in $\scG$ since $\bfh_{\cJ_1\tau}=0$ according to \eqref{eq:h_canonical}.
One can then check directly from the computations in the proofs of Lemmas \ref{lem:expression_g_1}, \ref{lem:expression_g_square}, and \ref{lem:expression_g_2} that $\blue g^\sigma$ solves the ODE \eqref{eq:g_ODE} with initial condition $\bone$.
\end{proof}

\section{Model bounds} \label{s:model_bounds}

The goal of this section is to construct a model $\sfZ=(\Pi,\Gamma)$ (see~Definition~\ref{def:model}) for the regularity structure constructed in  the previous Section~\ref{sec:solution_theory}. As we have explained in the introduction (see Subsection~\ref{sec:proof_idea_construction_model}) the model is constructed from a rough additive function~$\bfA\in\bfOmega_{\alpha\ax}^1$ in a way that it coincides with canonical model in the case $\bfA$ is the canonical lift of a smooth $1$-form $A\in\Omega^1_{\infty\ax}$ (see~Definition~\ref{def:canonical_model}).

    We start with a smooth $1$-form $A\in\Omega^1_{\infty\ax}(\Lambda)$  in the axial gauge and we extend it to an element in $\Omega^1 C^1(\R^2)$ in the axial gauge, denoted by the same letter $A$, using Lemma~\ref{lem:domain_extension_RAF} and with support on $\Lambdaex$. We  denote by 
	\begin{itemize}
		\item $\bfA \in \bfOmega_{\alpha\ax}(\R^2)$ its extended canonical lift to a rough additive function (in the sense of Definition~\ref{def:RAF_ax} and~Lemma~\ref{lem:domain_extension_RAF}) with support in $\Lambdaex$,
		\item $\sfZ^A\label{symb:canonical_model_ZA}$ the canonical model associated to $A$ (in the sense of Definition~\ref{def:canonical_model} and Lemma \ref{lem:canonical}).
	\end{itemize}

	To keep the notation simple, we will simply write~$\triple{\cdot \thinspace ; \thinspace \cdot} :=\triple{\cdot \thinspace ; \thinspace\cdot}_{\gamma}$ for $1<\gamma\leq \frac{3}{2}-3\kappa$.
    The main result of this section is the following:
	
    \begin{theorem} \label{thm:construction_model}
	Let $\alpha = \frac{1}{2} - \kappa$ for $\kappa \in (0,\frac{1}{6})$, i.e. $\alpha\in (\frac 13,\frac 12)$, and consider $A, \bar{A} \in \Omega^1_{\infty\ax}$ with corresponding canonical lifts~$\bfA$ and~$\bar\bfA$.
	Then, the  map $\bfA \mapsto \sfZ^A$ is locally Lipschitz continuous in the sense that 
    \begin{align}\label{eq:thm:model_bound}
       \triple{\sfZ^A;\sfZ^{\bar{A}}} \lesssim \triple{\bfA;\bar\bfA}_{\alpha\ax},
    \end{align}
    where the proportionality constant is locally uniform in the rough additive functions. 
    \end{theorem}

	\begin{remark}
	By approximation with smooth functions, Theorem 
 \ref{thm:construction_model} implies that, more generally, there is a canonical and continuous map taking any $\bfA\in\bfOmega_{\alpha\ax}^1$ to a model $\sfZ^{\bfA}\label{symb:model_from_RAF}$, which correspond to the 2nd map in \eqref{eq:A_to_Z}.
	\end{remark}

We require several prerequisites in order to prove the theorem, the first of which concerns the regularity of $F^A=-\partial_2 A_2$. 
In the following two lemmas, let $A\in \Omega^1C^1(\R^2)$ denote the extension of $A \in \Omega^1_{\infty\ax}(\Lambda)$ as described above.

    \begin{lemma}\label{lem:regularity_curvature}
        $
        |F^A|_{C^{2(\alpha-1)}(\R^2)}\lesssim |A|_{\alpha\ax;\Lambda}
        $. 
    \end{lemma}
    
    \begin{proof}
    Let us take a suitable test function $\varphi$ and do integration by parts to get
    \[
    \int_{\R^2}F^A(y)\varphi_x^\lambda(y)\dif y=\lambda^{-2}\int_{\R^2}\int_{[x_1,y_1]}\int_{[x_2,y_2]}F^A(z)\dif z\hat\varphi_x^\lambda(y)\dif y,
    \]
    where $\hat \varphi=\partial_{12}\varphi$. 
    Since by Stokes' theorem
    \[
    \left|\int_U F^A\right|=|A(\partial U)|\lesssim |A|_{\alpha\ax}|U|^\alpha,
    \]
    we obtain that 
    \[
    \int_{\R^2}F^A(y)\varphi_x^\lambda(y)\dif y\lesssim \lambda^{-2}\int_{\R^2}|x_1-y_1|^{\alpha}|x_2-y_2|^\alpha |\hat\varphi_x^\lambda(y)|\dif y\lesssim \lambda^{2(\alpha-1)}\|\hat\varphi\|_{L^\infty}|A|_{\alpha\ax}, 
    \]
    which is the desired result.
    \end{proof}
    
    For the analysis of the remaining terms in the regularity structure, we next relate the convolution with the kernel~$K_1$ to line integration, as outlined in Section~\ref{sec:main_results}.
	This is the content of the following proposition, which also contains a somewhat suprising result on the line-integrability regularity of $\diff^*K*F^A$.
    Recall the notation $|\Cdot|_{\beta\tri}$ from \eqref{eq:def_beta_tri}.  
    \begin{proposition}\label{prop:regularity_d*KFA}
	   For any $x,y\in\R^2$,
	   one has
		\begin{equation} \label{eq:regularity_d*KFA:line_integral}
			K_1*A_1(y)-K_1*A_1(x)=A(\ell^{x;y})+\zeta(\ell^{x;y}) 
		+ H(\ell^{x;y}),
		\end{equation}
		where $H\in \Omega^1 C^\infty_c(\R^2)$ and $\zeta := \diff^*K*F^A$ satisfy
        \begin{equation}\label{eq:regularity_d*KFA:zeta_bound}
        |H|_{C^1(\R^2)}+|\zeta|_{2\alpha\tri;\R^2}+ |\zeta|_{2\alpha\gr;\R^2}\lesssim |A|_{\alpha\ax;\Lambda}.
		\end{equation}
    \end{proposition}
    \begin{proof}
	We first prove~\eqref{eq:regularity_d*KFA:line_integral}. 
	By the fundamental theorem of calculus, we have
    \[
    K_1*A_1(y)-K_1*A_1(x)=\int_{\ell^{x;y}}\diff K_1*A_1
    \]
   where we may write
    \[
    \diff K_1 *A_1= K_{11} *A_1\,\diff x^1+K_{12}*A_1\,\diff x^2 \,. 
    \]
     For the first term, we complete the Laplacian by inserting the remaining derivatives as follows: 
    \[
    K_{11}*A_1=\Delta K *A_1-K_{22}*A_1=A_1+(\Delta K-\delta_0)*A_1-K_{22}*A_1.
    \]
	By our choice of $K$ to equal~$G^\Free$ on $[-4,4]^2$ in Definition \ref{def:model}\ref{item:integration} which implies that $\Delta K=\delta_0+Z$ for some $Z\in C^\infty_c(\R^2)$.
	As a consequence, we find
		\begin{align*}
		\diff K_1*A_1&=(A_1+Z*A_1-K_{22}*A_1)\,\diff x^1+K_{12}*A_1\,\diff x^2\\
		&=A_1 \diff x^1 +K_2 *F^A\,\diff x^1 - K_1* F^A\,\diff x^2 +Z*A_1\,\diff x^1
		\\
        &=A+\diff^*K*F^A +H
		\end{align*}
	where $H=Z*A_1\,\diff x^1$ and used that~$F^A=-\partial_2 A_1$.
	This completes the proof of~\eqref{eq:regularity_d*KFA:line_integral}. Furthermore, since $A_1$ has compact support and $Z\in C^\infty_c(\R^2)$, we also get that $H$ is smooth with compact support and $|H|_{C^1(\R^2)}\lesssim |A|_{\alpha\ax;\Lambda}$ proving the bound regarding $H$ in~\eqref{eq:regularity_d*KFA:zeta_bound}.  

	The remainder of this proof deals with~\eqref{eq:regularity_d*KFA:zeta_bound}. For notational simplicity, we set $\eta:=K*F^A$ so that $\zeta=\diff^*\eta$. 
	By~\eqref{eq:regularity_d*KFA:line_integral}, we know that $\zeta\in \Omega_{\alpha\gr}$ since $K_1*A_1\in C^\alpha$ and $A\in \Omega_{\alpha\gr}$; 
	we want to show an \emph{improvement in regularity}.  

        The part with the triangular norm follows from the following computation. For any polygon $U$
        \[
        \zeta(\partial U)=\int_{\partial U}\zeta=\int_U \diff \zeta,
        \]
        but, by the same argument as above, we know that
        \[
        \diff\zeta=\diff\diff^*K*F^A=-\Delta K*F^A =-F^A \,.
        \]
       Therefore, the area bound 
       \begin{equation}\label{eq:area_bound}
       |\zeta(\partial U)|\lesssim |A|_{\alpha\ax} |U|^\alpha.
	   \end{equation}
	follows by the same argument as in Lemma~\ref{lem:regularity_curvature}.
    It remains to prove the growth part of the estimate.  We shall use the triangular norm estimate just obtained
    together with the Schauder estimate
    \begin{equation}\label{eq:eta_Holder}
        |\eta|_{C^{2\alpha}}
        \lesssim |F^A|_{C^{2\alpha-2}}
        \lesssim |A|_{\alpha\ax}\;,
    \end{equation}
    which follows from Lemma~\ref{lem:regularity_curvature}.
    We begin with axis-parallel line segments.
	Set
    \[
        M =
        \sup_{\substack{\ell\text{ horizontal or vertical}\\
        0<|\ell|\le 1/4}}
        \frac{|\zeta(\ell)|}{|\ell|^{2\alpha}}\;.
    \]
    We shall prove that $M\lesssim |A|_{\alpha\ax}$.
    Let $\ell$ be a horizontal segment from $(a,h)$ to $(b,h)$ with
    $|\ell|=|b-a|\le 1/4$.  Fix $\eps \in (0,1)$ to be chosen later and
    set $\delta=\eps |\ell|$.
	Let $\chi\in C_c^\infty((0,1))$ satisfy
    $\int_0^1\chi=1$, and denote
    $\chi_\delta(s)=\delta^{-1}\chi(s/\delta)$.  Denote
    $\ell_s=\ell+s e_2$, the vertical translate of $\ell$ by $s\geq 0$. Now write
    \begin{equation}\label{eq:zeta_av}
		\zeta(\ell)
        =\int_0^\delta \chi_\delta(s)\zeta(\ell_s)\dif s
        +\int_0^\delta \chi_\delta(s)
        \{\zeta(\ell)-\zeta(\ell_s)\} \dif s \;.
	\end{equation}
For the first term, recall $\zeta=\diff^*\eta$, so
$\zeta_1=\partial_2\eta$, thus by integration by parts in the $s$ variable
    \begin{align*}
        \int_0^\delta \chi_\delta(s)\zeta(\ell_s)\dif s
        &=\int_a^b\int_0^\delta \chi_\delta(s)
            \partial_2\eta(u,h+s)\dif s\dif u  \\
        &=-\int_a^b\int_0^\delta \chi_\delta'(s)
            \{\eta(u,h+s)-\eta(u,h)\} \dif s \dif u \;.
    \end{align*}
    Therefore
    \begin{equation}\label{eq:zeta_av_1}
        \Big|\int_0^\delta \chi_\delta(s)\zeta(\ell_s)\dif s\Big|
        \lesssim |A|_{\alpha\ax} |\ell|
        \int_0^\delta |\chi_\delta'(s)|s^{2\alpha}\dif s  \lesssim \eps^{2\alpha-1}|A|_{\alpha\ax} |\ell|^{2\alpha}
    \end{equation}
	where we used \eqref{eq:eta_Holder} in the first estimate and $\delta=\eps|\ell|$ and $|\chi_\delta'|_\infty\lesssim \delta^{-2}$ in the second estimate.

    For the second term in \eqref{eq:zeta_av}, let $R_s$ be the rectangle with horizontal sides
    $\ell$ and $\ell_s$. Then
    \begin{equation*}
        |\zeta(\ell)-\zeta(\ell_s)|
        \le |\zeta(\partial R_s)| + 2M s^{2\alpha}
        \lesssim |A|_{\alpha\ax}(|\ell|s)^\alpha + Ms^{2\alpha}
    \end{equation*}
	where in the first bound we used the definition of $M$ and that the two vertical sides of $R_s$ have length $s$,
	and in the second bound we used the area bound \eqref{eq:area_bound} and $|R_s| = |\ell|s$. Recalling again $\delta = \eps|\ell|$, it follows that
	\begin{equation}\label{eq:zeta_av_2}
	\int_0^\delta \chi_\delta(s)
        \{\zeta(\ell)-\zeta(\ell_s)\} \dif s \lesssim \eps^{\alpha} |A|_{\alpha\ax}|\ell|^{2\alpha} + \eps^{2\alpha}M|\ell|^{2\alpha}\;.
	\end{equation}
Therefore, by \eqref{eq:zeta_av} and the two estimates \eqref{eq:zeta_av_1}-\eqref{eq:zeta_av_2},
	we obtain
    \[
        |\zeta(\ell)|
        \lesssim
        (\eps^{2\alpha-1}+\eps^\alpha)
        |A|_{\alpha\ax}|\ell|^{2\alpha}
        +\eps^{2\alpha}M|\ell|^{2\alpha}\;.
    \]
    The same argument applies to vertical segments, now averaging in the
    horizontal direction and using $\zeta_2=-\partial_1\eta$. 
	Taking the
    supremum over all horizontal and vertical segments of length at most $1/4$,
    we obtain for $C>0$ not depending on $\eps$ that
    \[
        M
        \leq
        C (\eps^{2\alpha-1}+\eps^\alpha)|A|_{\alpha\ax}
        +C \eps^{2\alpha}M\;.
    \]
    Now choose $\eps>0$ small so that
    $C\eps^{2\alpha}\leq 1/2$. This gives
	$
	M \lesssim |A|_{\alpha\ax}
	$
	for horizontal and vertical segments.
	The bound for general segments follows from this and the area bound \eqref{eq:area_bound}.
    \end{proof}

  	Finally, we have gathered all the required prerequisites to present the proof of \theo{thm:construction_model}.

	\begin{proof}[of~\theo{thm:construction_model}]
	We will only deal with the case~$\bar{A} = 0$ from which the general case follows by straightforward adaptations.
	Recall that we have assumed~$A \in \Omega^1_{\infty\ax}$, i.e. in particular that~$A_1$ is smooth; we will denote the associated canonical model (in the sense of Definition~\ref{def:canonical_model}) by~$\sfZ = (\Pi,\Gamma)$, the \enquote{norm}~$\triple{\sfZ} = \norm[0]{\Pi} + \norm[0]{\Gamma}$ of which has been introduced in Definition~\ref{def:model}.
	Our goal is to show the analytical bounds in~\eqref{eq:model_bounds} with the prefactor~$\triple{\bfA}_{\alpha\ax}$ as claimed in~\eqref{eq:thm:model_bound} where the list of symbols together with their homogeneities was introduced  in Table~\ref{tab:U_F_sets_homogeneities}.

The first observation is that we only need to show the bounds on $\Pi \blue{\tau}$ in~\eqref{eq:model_bounds} for $\blue{\tau} \in T[\tau]$ where  
\begin{equation*}
	\tau \in S_- := \{A_1, \rmD_2 A_1, A_1\cI_1A_1,\rmD_2(A_1\cI_1A_1)\} 
\end{equation*}
which occupies the rest of this proof.
This is a consequence of 
\begin{enumerate}[label=(\roman*)]
	\item the \emph{extension theorem} (see~\cite[Thm. 5.24]{Hairer14}) which enforces the correct bounds on the integrated symbols~$\blue{\cI_i \tau}$ for $\blue{\tau} \in T[\tau]$, $\tau \in S_-$, as well as the $\Gamma$-bounds, and  
	\item the fact that all the product terms are canonically defined in our setting.\footnote{Note, in particular, that we do not need to renormalise any product term.}
\end{enumerate}

Let us highlight that, in order to obtain the full model bounds on all of the symbols in~$S_\sol$, one will sometimes have to invoke the previously mentioned points iteratively, as is for example clear when looking at the symbol~$\cI_1 \cI_{22} A_1$; we refrain from giving any more details.  

The case for the linear terms $A_1$ and $\rmD_2A_1$ in~$S_-$ is straightforward:
While the bound for $A_1$ follows from \cite[Cor.~3.23]{Chevyrev19} and is continuous with respect to $\bfA$, the bound on $\rmD_2A_1$ follows by \lem{lem:regularity_curvature} upon observing that $\Pi\blue{\rmD_2A_1}=\partial_2A_1$.

The quadratic terms in $S_-$ in are more involved: We begin with~$A_1 \cI_1 A_1$.
To simplify the notation, we will write 
\begin{equation} \label{e:tensor_notation_A1I1A1}
	\blue{\cI_1 A_1 \otimes A_1} := \sum_{i,j=1}^{\dim \fg} \blue{\cI_1 A^{e^i}_1} \otimes \blue{A^{e^j}_1} \otimes e_i \otimes e_j \in T[\cI_1A_1] \otimes T[A_1] \otimes \fg^{\otimes 2}
\end{equation}
such that, for~$u, v \in \fg^*$, by a slight abuse of notation we may set 
\begin{equation} \label{e:tensor_notation_A1I1A1_2}
	\scal{\blue{\cI_1 A_1 \otimes A_1}, v \otimes u} 
	:=
	\blue{A_1^{u} \cI_1 A_1^{v}} \in T[A_1 \cI_1 A_1] 
\end{equation}
as well as 
\begin{equation} \label{e:tensor_notation_A1I1A1_Pi}
	\scal{\Pi_x \blue{\cI_1 A_1 \otimes A_1}, v \otimes u} :=\Pi_x \blue{A_1^{u} \cI_1 A_1^{v}} \,.
\end{equation}

    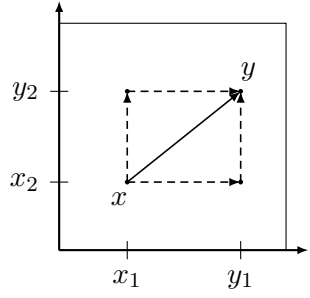
\begin{figure}[h] 
    	\centering
    	%
    		\tikzset{>={Latex[width=2pt,length=2pt]}}
    		\begin{tikzpicture}[scale=3] 
    			%
    			%
    			\draw (0,0) rectangle (1,1);
    			%
    			\draw[-{Latex[length=1.5mm]}, thick] (0,0) -- (1.1,0); 
    			\draw[-{Latex[length=1.5mm]}, thick] (0,0) -- (0,1.1); 
    			\draw (0.3,-0.04) -- (0.3,0.04) node[semithick, below, yshift=-7pt] {$x_1$};
    			\draw (-0.04,0.3) -- (0.04,0.3) node[semithick, left, xshift=-7pt] {$x_2$};
    			\draw (0.8,-0.04) -- (0.8,0.04) node[semithick, below, yshift=-7pt] {$y_1$};
    			\draw (-0.04,0.7) -- (0.04,0.7) node[semithick, left, xshift=-7pt] {$y_2$};
    			%
    			%
    			\coordinate (LD) at (0.3,0.3);       
    			\node[below, xshift=-3pt] at (0.3,0.3) {$x$}; 
    			\coordinate (LU) at (0.3,0.7);     
    			%
    			\coordinate (RD) at (0.8,0.3);    
    			%
    			\coordinate (RU) at (0.8,0.7);       
    			\node[above, xshift=3pt] at (0.8,0.7) {$y$}; 
    			%
    			%
    			%
    			\draw[-{Latex[length=1.5mm]}, semithick] (LD) -- (RU) node[midway, above left, scale=0.7] {}; 
				\draw[-{Latex[length=1.5mm]}, semithick, densely dashed] (LD) -- (LU) node[midway, above left, scale=0.7] {}; %
				\draw[-{Latex[length=1.5mm]}, semithick, densely dashed] (LU) -- (RU) node[midway, above left, scale=0.7] {}; %
				\draw[-{Latex[length=1.5mm]}, semithick, densely dashed] (LD) -- (RD) node[midway, above left, scale=0.7] {}; %
				\draw[-{Latex[length=1.5mm]}, semithick, densely dashed] (RD) -- (RU) node[midway, above left, scale=0.7] {}; %
    			%
    			\foreach \pt in {LD,LU,RD,RU}
    			\fill[black] (\pt) circle (0.3pt);
    		\end{tikzpicture}
    		%
    		\captionsetup{format=plain}
    		\caption{A graphical representation of the line segment~$\ell^{x;y}$ (solid) and the auxiliary lines~$\ell^{x;(x_1,y_2)}$ (left), $\ell^{x;(y_1,x_2)}$ (bottom), $\ell^{(y_1,x_2);y}$ (right), and $\ell^{(x_1,y_2);y}$ (top).}
    		\label{fig:lines_square}
    		%
   		%
   	\end{figure}
    
    Recall the identity~\eqref{eq:regularity_d*KFA:line_integral} wherein $\zeta$ and $H$ are additive functions which are more regular, namely $Q:=\zeta+H\in \Omega^1_{2\alpha}(\R^2)$.
	Then we have
    \begin{equs}[][e:I1A1A1]
    \Pi_x \blue{\cI_1A_1\otimes A_1}(y)
    &=(K_1 * A_1(y)-K_1 * A_1(x_1,y_2))\otimes A_1(y)\\
    &\qquad\quad+(K_1 * A_1(x_1,y_2)-K_1 * A_1(x))\otimes A_1(y)\\
    &= A(\ell^{(x_1,y_2);y})\otimes A_1(y)\!+\! Q(\ell^{(x_1,y_2);y})\otimes A_1(y)\!+\!Q(\ell^{x;(x_1,y_2)})\otimes A_1(y).
    \end{equs}
	where we have used that~$A(\ell^{x;(x_1,y_2)}) = 0$ because~$A$ is in the axial gauge. See Figure~\ref{fig:lines_square} for the auxiliary lines involved in the proof.
    We write these elements in a slightly different way: For instance, by the fundamental theorem of calculus, we find 
    \begin{equation} \label{e:derivative_A_product}
     A(\ell^{(x_1,y_2);y})\otimes A_1(y)
     =\partial_{y_1}\int_{x_1}^{y_1}  A(\ell^{(x_1,y_2);(t,y_2)})\otimes A_1(t,y_2)\dif t=\partial_{y_1}\bA(\ell^{(x_1,y_2);y})
    \end{equation}
	where~$\bA(\ell)$ was introduced in~\eqref{eq:def_bA_one_arg}.
    Given a test function $\varphi$, we apply integration by parts to obtain
    \begin{align*}
    \int_{\R^2}\varphi_x^\lambda(y)A(\ell^{(x_1,y_2);y})\otimes A_1(y)\dif y&=\int_{\R^2}\varphi_x^\lambda(y)\partial_{y_1}\bA(\ell^{(x_1,y_2);y})\dif y\\
    &=-\int_{\R^2}\partial_{y_1}\varphi_x^\lambda(y)\bA(\ell^{(x_1,y_2);y})\dif y.
    \end{align*}
    By the bound $\abs[0]{\bA(\ell^{(x_1,y_2);y})} \leq \|\bA\|_{\alpha\ax}|x_1-y_1|^{2\alpha}$, we obtain 
    \[
    \left|\int_{\R^2}\varphi_x^\lambda(y)A(\ell^{(x_1,y_2);y})\otimes A_1(y)\dif y \right|\lesssim \lambda^{2\alpha-1}\|\bA\|_{\alpha\ax} |\varphi|_{C^1}. 
    \]
    which is the claimed bound for the first term in~\eqref{e:I1A1A1}.
    To analyse the second term, we now define
    \begin{align}\label{eq:I(QA)}
    \rmI(Q,A)(\ell):=\int_0^1 \ell_{Q}(t)\otimes \dif\ell_A(t).
    \end{align}
	where the path~$\ell_{Q}$ as given below eq.~\eqref{eq:line_int_gauge_transformation_smooth} lies in~$C^{2\alpha}$ thanks to Proposition~\ref{prop:regularity_d*KFA}.
    Similarly to~\eqref{e:derivative_A_product}, we then have 
    \begin{align*}
    	Q(\ell^{(x_1,y_2);y})\otimes A_1(y)&=\partial_{y_1}\rmI(Q,A)(\ell^{(x_1,y_2);y}).
    \end{align*}
	Since~$3\alpha > 1$ because~$\alpha > \frac 1 3$, the bound \eqref{eq:regularity_d*KFA:zeta_bound} and the sewing lemma~\cite[Lem~4.2]{FH20} imply that
    \begin{align}\label{eq:bound_I(QA)}
    \left|\rmI(Q,A)(\ell^{(x_1,y_2);y})\right|\lesssim |A|_{\alpha\ax}^2|x_1-y_1|^{3\alpha},
    \end{align}
    which, similarly to before, leads to the following estimate for the second term in~\eqref{e:I1A1A1}:
    \[
    \left|\int_{\R^2}\varphi_x^\lambda(y)  Q(\ell^{(x_1,y_2);y})\otimes A_1(y)\dif y\right|\lesssim \lambda^{3\alpha-1} |A|_{\alpha\ax}^2|\varphi|_{C^1}\lesssim \lambda^{2\alpha-1} |A|_{\alpha\ax}^2|\varphi|_{C^1}.
    \]
    The third term in~\eqref{e:I1A1A1} is actually the easiest. We can write
    \[
    Q(\ell^{x;(x_1,y_2)})\otimes A_1(y)=\partial_{y_1}(Q(\ell^{x;(x_1,y_2)})\otimes A(\ell^{(x_1,y_2);y})).
    \]
    Furthermore
    \[
    |Q(\ell^{x;(x_1,y_2)})\otimes A(\ell^{(x_1,y_2);y})|\lesssim |A|_{\alpha\ax}^2 |y_1-y_2|^{2\alpha}|x_1-y_1|^\alpha.
    \]
    Therefore we get the claim similarly as before. 
    
    Concerning the continuity, one can easily redo the calculation above and see that
    \[
    |(\Pi_x-\bar \Pi_x) \blue{\cI_1A_1\otimes A_1}(\varphi_x^\lambda)|\lesssim \lambda^{2\alpha-1}|\varphi|_{C^1}(\|\bA-\bar\bA\|_{\alpha\ax}+(|A|_{\alpha\ax}+|\bar A|_{\alpha\ax})|A-\bar A|_{\alpha\ax}). 
    \]
    Using $\|\bA-\bar\bA\|_{\alpha\ax}\leq \|\bA-\bar\bA\|_{\alpha\ax}^{1/2}(\|\bA\|^{1/2}+\|\bar\bA\|_{\alpha\ax}^{1/2})$, we obtain
    \[
    |(\Pi_x-\bar \Pi_x) \blue{\cI_1A_1\otimes A_1}(\varphi_x^\lambda)|\lesssim \lambda^{2\alpha-1}|\varphi|_{C^1}(\triple{\bfA}_{\alpha\ax}+\triple{\bar\bfA}_{\alpha\ax})\triple{\bfA;\bar\bfA}_{\alpha\ax}. 
    \]
   
   Finally, let us look at the term $\blue{\rmD_2(\cI_1A_1\otimes A_1)}$ where the analogous notational convention as in~\eqref{e:tensor_notation_A1I1A1}, \eqref{e:tensor_notation_A1I1A1_2}, and \eqref{e:tensor_notation_A1I1A1_Pi} applies.\label{par:bound_D2I1AA} 
	This term is treated very similarly to the previous one.  
    By~\eqref{eq:regularity_d*KFA:line_integral}, we write
    \begin{align*}
    \begin{split}\Pi_x \blue{\rmD_2(\cI_1A_1\otimes A_1)}(y)&=\partial_{y_2}((K_1A_1(y)-K_1A_1(x_1,y_2))\otimes A_1(y))\\
    &\ \ \ +\partial_{y_2}((K_1A_1(x_1,y_2)-K_1 A_1(x))\otimes A_1(y))\\
    &=\partial_{y_2}(A(\ell^{(x_1,y_2);y})\otimes A_1(y))+ \partial_{y_2}(Q(\ell^{(x_1,y_2);y})\otimes A_1(y))\\
    &\ \ \ +\partial_{y_2}(Q(\ell^{x;(x_1,y_2)})\otimes A_1(y)).
    \end{split}
    \end{align*}
    With the notation from the calculations in the previous paragraph, we write
     \begin{align}\label{eq:bound_D2I1AA_2} 
     \begin{split}
    \Pi_x \blue{\rmD_2(\cI_1A_1\otimes A_1)}(y)&=\partial_{y_2}(A(\ell^{(x_1,y_2);y})\otimes A_1(y))+ \partial_{y_2}\partial_{y_1}\rmI(Q,A)(\ell^{(x_1,y_2);y})\\
    &\ \ \ +\partial_{y_2}\partial_{y_1}(Q(\ell^{x;(x_1,y_2)})\otimes A(\ell^{(x_1,y_2);y})).
    \end{split}
      \end{align}
    After testing with~$\varphi^\lambda_x$, in absolute value the last two terms are each bounded by
    \[\lambda^{3\alpha-2}|\varphi|_{C^2}|A|_{\alpha\ax}^2,\]
    as desired. 
    For the first term in~\eqref{eq:bound_D2I1AA_2}, we add a term that is constant in $y_2$ inside the derivative to obtain from~\eqref{e:derivative_A_product} that
    \begin{align*}
    \begin{split}
    \partial_{y_2}(A(\ell^{(x_1,y_2);y})\otimes A_1(y))&=\partial_{y_2}(A(\ell^{(x_1,y_2);y})\otimes A_1(y)-(A(\ell^{x;(y_1,x_2)})\otimes A_1(y_1,x_2)))\\
    &=\partial_{y_2}\partial_{y_1} (\bA(\ell^{(x_1,y_2);y})-\bA(\ell^{x;(y_1,x_2)})).
    \end{split}
    \end{align*}
    By definition of the horizontal norm~\eqref{eq:RAF_hor_norms}, we get
    \begin{align*}
    |\bA(\ell^{(x_1,y_2);y})-\bA(\ell^{x;(y_1,x_2)})|\lesssim \|\bA\|_{\alpha\ax} |x_1-y_1|^{2\alpha}|x_2-y_2|^\alpha. 
    \end{align*}
    After testing with $\varphi_x^\lambda$ we get a contribution of order $\lambda^{3\alpha-2}|\varphi|_{C^2}\|\bA\|_{\alpha\ax}$.
   Hence, we conclude that
    \begin{align*}
    |\Pi_x\blue{\rmD_2(\cI_1A_1\otimes A_1)}(\varphi_x^\lambda)|\lesssim \lambda^{3\alpha-2}|\varphi|_{C^2}(\|\bA\|_{\alpha\ax}+|A|_{\alpha\ax}^2). 
    \end{align*}
    Similar calculations show that
    \begin{align*}
    |(\Pi_x\!-\!\bar \Pi_x) \blue{\rmD_2(\cI_1A_1\otimes A_1)}(\varphi_x^\lambda)|\lesssim \lambda^{3\alpha-2}|\varphi|_{C^2}(\triple{\bfA}_{\alpha\ax}\!+\!\triple{\bar\bfA}_{\alpha\ax})\triple{\bfA;\bar\bfA}_{\alpha\ax}. 
    \end{align*}
	The whole proof is complete.
	\end{proof}

  \section{Patching}\label{sec:patching}
    In this section, we show that a  family of compatible and locally defined connection forms in $\Omega_\beta^1$ can be patched to a global connection form.
	This can be seen as a generalisation of the patching result from \cite{Uhlenbeck82} to distributional connections in the case of a rectangular cover.
    The results of this section are of independent interest and do not rely on those of the other sections.

	We start with the following definition, which is illustrated in Figure~\ref{fig:horizontal_patching}.

	\begin{definition}\label{def:rec_cover}\label{symb:rectangular_cover}
	Given an orthonormal basis $\sce=\{\rme_1,\rme_2\}$ of $\bR^2$, a scaling parameter $\mu=(\mu^1,\mu^2)\in [\frac 34,\infty)^2$, a point $y\in\bR^2$, and $\bfn=(n^1,n^2)\in\bN^2$, we define the rectangles:
	\begin{equation} \label{def:rec_cover:eq}
		U^{\sce,\mu,y}_\bfn
		:=
		\{y+\tfrac 23( n^1\mu^1\rme_1+ n^2\mu^2\rme_2)\} 
		+
		\{s^1\rme_1+s^2\rme_2\colon s^1\in [0,\mu^1], s^2\in [0,\mu^2]\}\,.
	\end{equation}
	For $L=(l^1,l^2)\in \bN^2$, we define the set $\bN_{L}^2:=\bN_{l^1}\times \bN_{l^2}$, where $\bN_{l}=\bN\cap [0,l] = \cbr[0]{0,\ldots,l}$ for $l\in\bN$. We then define the union of these rectangles 
	\[
	\cU^{\sce,\mu,y}_L:=\bigcup_{\bfn\in\bN_{L}^2} U_\bfn^{\sce,\mu,y},
	\] 
	which is a rectangle itself. 
	\end{definition}

\begin{notation}
	When the context is clear, we simply write $U_\bfn$ instead of $U^{\sce,\mu,y}_\bfn$ and similarly, we write $\cU$ instead of $\cU^{\sce,\mu,y}_L$. 
\end{notation}

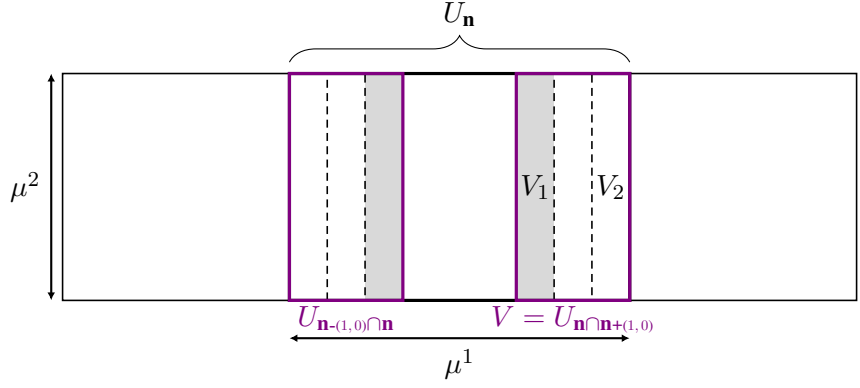
\begin{figure}[ht]
	\centering
	%
	\tikzset{>={Latex[width=3pt,length=3pt]}}
	\begin{tikzpicture}[scale=0.5] 
		%
		%
		%
		\draw[semithick] (0,0) rectangle (3,6);
		\fill[color=black!15] (2,0) rectangle (3,6);
		\draw[semithick, densely dashed] (1,0) rectangle (2,6);
		%
		%
		\draw[semithick] (6,0) rectangle (9,6);
		\fill[color=black!15] (6,0) rectangle (7,6);
		\draw[semithick, densely dashed] (7,0) rectangle (8,6);
		%
		%
		\draw[very thick] (0,0) rectangle (9,6);
		%
		%
		\draw[semithick] (-6,0) rectangle (3,6);
		%
		%
		\draw[semithick] (6,0) rectangle (15,6);
		%
		%
		\draw[very thick, violet] (0,0) rectangle (3,6);
		%
		%
		\draw[very thick, violet] (6,0) rectangle (9,6);
		%
		\draw[<->, thick] (0,-1) -- (9,-1) node[midway, below, scale=1] {$\mu^1$};
		\draw[<->, thick] (-6.3,0) -- (-6.3,6) node[midway, left, scale=1] {$\mu^2$};
		%
		\node at (1.5,-0.5) {$\textcolor{violet}{U_{\bfn\minush \cap \bfn}}$};
		\node at (7.5,-0.5) {$\textcolor{violet}{V = U_{\bfn \cap \bfn\plush}}$};
		%
		\node at (6.5,3) {$V_1$};
		\node at (8.5,3) {$V_2$};
		%
		\draw [decorate,decoration={brace,amplitude=10pt}]
		(0,6.3) -- (9,6.3)
		node[midway,yshift=18pt]{$U_{\bfn}$};
	\end{tikzpicture}
	\caption{\emph{Horizontal patching.} There are three rectangles,~$U_{\bfn\minush}$, $U_\bfn$, and~$U_{\bfn \cap \bfn\plush}$, together with their respective overlaps in purple. 
	These overlaps are further partitioned into three equal parts to interpolate gauge transformations by means of Lemma~\ref{lem:interpolating_gauge_transformation} applied with~$V$,~$V_1$, and~$V_2$ as given in the figure, see~\eqref{e:ex_V} and~\eqref{e:ex_V12} below.
	By construction, the function~$g_\bfn$ given in~\eqref{e:def_g_bfn} below equals~$1_G$ on the shaded regions and the one enclosed by them.
	}
	\label{fig:horizontal_patching}
\end{figure}

First recall the definition of $\Omega^1_\beta(U)$ for generic rectangle $U\subset \R^2$ from Section~\ref{sec:AF}.
In the following, we will consider additive functions $(B_\bfn)_{\bN_{L}^2}$ defined on these rectangles, where each $B_\bfn \in \Omega^1_\beta(U_\bfn^{\sce,\mu,y})$ for some $\beta \in (\frac{2}{3},1)$. The upcoming \theo{thm:patching} will show how, under a joint gauge equivalence condition, these local additive functions $B_\bfn$ can be patched together into a global additive function $B$ on $\cU^{\sce,\mu,y}_L$. This requires compatibility conditions with the gauge transformations which we introduce in the following definition:

\begin{definition}[Compatible gauge data]
\label{def:compatible_gauge_data}
Let $\beta\in (\frac 23,1]$ and $(U_\bfn^{\sce,\mu,y})_{\bfn\in\bN_{L}^2}$ be a collection of rectangles as in \defref{def:rec_cover}. We say $(B_\bfn,g_{\bfn,\bfm})_{\bfn,\bfm\in\bN_{L}^2}$  is a \emph{compatible gauge datum}\label{symb:compatible_gauge_datum}, where $B_\bfn\in\Omega_\beta^1(U_\bfn)$ and $g_{\bfn,\bfm}\in C^\beta(U_{\bfn}\cap U_\bfm,G)$, if the following conditions hold: 
\begin{itemize}
\item \textbf{(gauge equivalence)} $B_\bfn^{g_{\bfn,\bfm}}=B_\bfm$ on $U_\bfn\cap U_{\bfm}$ ;
\item \textbf{(cocycle condition)} $g_{\bfn,\bfp}^{-1}g_{\bfp,\bfm}^{-1}=g_{\bfn,\bfm}^{-1} \quad \text{ on } \ U_\bfm\cap U_\bfn \cap U_\bfp$.
\end{itemize}
Finally, we say a compatible gauge datum $(B_\bfn,g_{\bfn,\bfm})_{\bfn,\bfm\in\bN_L^2}$ is smooth if all $B_\bfn$ and $g_{\bfn,\bfm}$ are smooth $1$-forms respectively smooth gauge transformations. 
\end{definition}

The patching result is then provided by the following theorem.

\begin{theorem}[Patching]\label{thm:patching} 
   	Fix~$L \in \N$, let $\beta\in (\frac 2 3 ,1]$, and let
    $(U_\bfn^{\sce,\mu,y})_{\bfn\in\bN_{L}^2}$ be a collection of
    rectangles as in \defref{def:rec_cover}. There exists
    $\delta_{\patch} > 0$, depending only on $\beta$ and $G$, with the following property. 
   	For any compatible gauge datum
    $(B_\bfn,g_{\bfn,\bfm})_{\bfn,\bfm\in\bN^2_L}$ such that 
   \begin{align}\label{eq:patch_req}
   \max_{\bfn,\bfm\in\bN_L^2}\{|B_\bfn|_{\beta;U_\bfn}+
   |g_{\bfn,\bfm}-1_G|_{\infty;U_\bfn\cap U_\bfm}\}\leq \delta_\patch,
   \end{align}
   there exist $B\in \Omega^1_\beta(\cU)$ and gauge transformations
   $q_\bfn\in C^\beta(U_\bfn;G)$ such that
   \begin{align}\label{eq:patched_gauge_datum}
    B=B_\bfn^{q_\bfn}\quad\text{on }U_\bfn,
    \qquad
    q_\bfn=q_\bfm g_{\bfn,\bfm}\quad\text{on }U_\bfn\cap U_\bfm .
\end{align}
Moreover,
\begin{align}\label{e:patching_bound}
    |B|_{\beta;\cU}
    +
    \max_{\bfn\in\N_L^2}|q_\bfn-1_G|_{C^\beta(U_\bfn)}
    \lesssim
    \max_{\bfn,\bfm\in\N_L^2}
    \{|B_\bfn|_{\beta;U_\bfn}
    +
    |g_{\bfn,\bfm}-1_G|_{\infty;U_\bfn\cap U_\bfm}\}.
\end{align}
   Furthermore, if $(B_\bfn,g_{\bfn,\bfm})_{\bfn,\bfm\in\bN^2_L}$ is a
   smooth compatible gauge datum, then $B$ and the gauge transformations
$q_\bfn$ are smooth.
    
   Finally, given another compatible gauge datum
   $(\bar B_\bfn,\bar g_{\bfn,\bfm})_{\bfn,\bfm\in\N_L^2}$
   satisfying~\eqref{eq:patch_req}, with corresponding objects
   $\bar B$ and $\bar q_\bfn$, we have
   \begin{align}\label{e:patching_bound_cont} 
       |B-\bar B|_{\beta;\cU}
       +
       \max_{\bfn\in\N_L^2}
       |q_\bfn-\bar q_\bfn|_{C^\beta(U_\bfn)}
       \lesssim
       \max_{\bfn,\bfm\in\N_L^2}
       \{
       |B_\bfn-\bar B_\bfn|_{\beta;U_\bfn}
       +
       |g_{\bfn,\bfm}-\bar g_{\bfn,\bfm}|_{\infty;U_\bfn\cap U_\bfm}
       \},
   \end{align}
   where the proportionality depends only on $\beta$ and $G$. 
\end{theorem}

We will introduce several lemmas before getting to the proof. 
The first lemma discusses an easy case when we can obtain a global additive function.
\begin{lemma}
\label{lem:global_additive_function}
 Let $\beta\in (\frac 12 ,1]$ and $(U_\bfn^{\sce,\mu,y})_{\bfn\in\bN_{L}^2}$ be a collection of rectangles as in \defref{def:rec_cover}. Consider a sequence of additive functions $(B_\bfn)_{\bfn\in\bN_{L}^2}$ with $B_\bfn\in\Omega_\beta^1(U_\bfn)$  such that $B_\bfn=B_\bfm$ on $U_\bfn\cap U_\bfm$. Then there exists  $B\in \Omega_\beta^1(\cU)$ such that  $B=B_\bfn$ on $U_\bfn$ for any $\bfn\in\bN_{L}^2$. Moreover, 
 \[
 |B|_{\beta;\cU}\lesssim \max_{\bfn\in\bN_L^2}|B_\bfn|_{\beta;U_\bfn},
 \]
 where the proportionality constant is uniform in all parameters. 
 Finally, if $(B_\bfn)_{\bfn\in\bN_L^2}$ consists of smooth additive functions, then $B$ is smooth as well.
\end{lemma}

\begin{proof}
Firstly, since $B_\bfn=B_\bfm$ on $U_\bfn\cap U_\bfm$, we can define $B=B_\bfn$ on $\cX\cap U_\bfn$. This $B$ is well-defined as an additive function as the definition does not depend on $\bfn\in\bN^2_{L}$ that we have chosen. 

For the norm, recall $|\Cdot|_{\beta;\cU}=|\Cdot|_{\beta\gr;\cU}+|\Cdot|_{\beta\text{-tr};\cU}$ in Section~\ref{sec:AF}. By the choice $\mu^i\geq\frac 34$, we know that the overlaps of the rectangles have side lengths bigger or equal to $\frac 13\mu^i\geq \frac 14$. 
In particular, any line segment (resp.~triangle) with length (resp.~diameter) at most $\frac 14$ is contained entirely in precisely one rectangle, see Figure~\ref{fig:horizontal_vertical_overlaps} for an illustration.
Finally, if $(B_\bfn)_{\bfn\in\bN_L^2}$ consists of smooth $1$-forms, it is trivial that~$B$ inherits this property.
\end{proof}

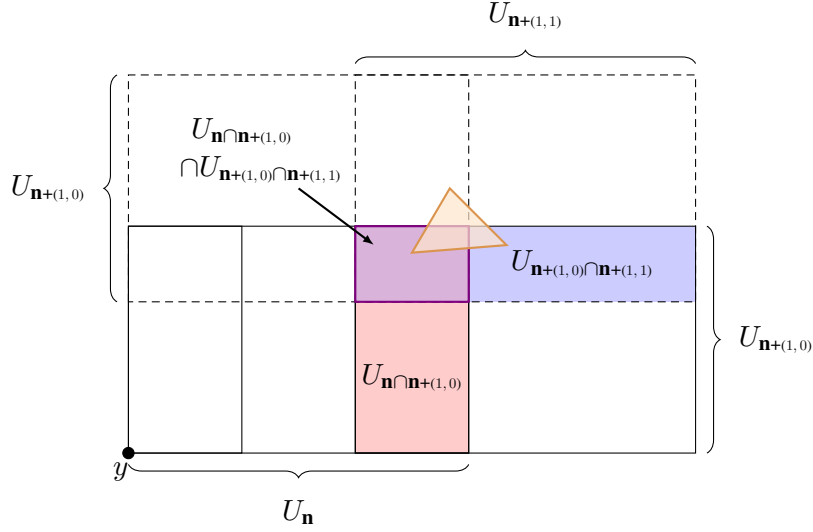
\begin{figure}[ht]
	\centering
	%
	\tikzset{>={Latex[width=3pt,length=3pt]}}
	\begin{tikzpicture}[scale=0.5] 
		\fill[color=blue!20] (6,4) rectangle (15,6);
		\fill[color=red!20] (6,0) rectangle (9,6);
		%
		\fill[color=violet!30] (6,4) rectangle (9,6);
		%
		%
		\draw (0,0) rectangle (3,6);
		%
		%
		\draw (6,0) rectangle (9,6);
		%
		%
		\draw(0,0) rectangle (9,6);
		%
		%
		\draw (6,0) rectangle (15,6);
		%
		%
		\draw[densely dashed] (0,4) rectangle (9,10);
		%
		%
		\draw[densely dashed] (6,4) rectangle (15,10);
		%
		%
		%
		\node at (12,5) {$U_{\bfn\plush\cap\bfn\plushv}$};
		%
		%
		%
		
		\draw[violet, thick] (6,4) rectangle (9,6);
		%
		%
		%
		%
		\draw [decorate,decoration={brace,amplitude=5pt}]
		(-0.3,4) -- (-0.3,10)
		node[midway,xshift=-26pt]{$U_{\bfn\plush}$};
		\node at (7.5,2) {$U_{\bfn \cap \bfn\plush}$};
		\node[draw,circle,fill=black, inner sep=1.5pt] at (0,0) {};
		\node at (-0.2,-0.5) {$y$};
		\draw [decorate,decoration={brace,mirror,amplitude=5pt}]
		(0,-0.3) -- (9,-0.3)
		node[midway,yshift=-18pt]{$U_{\bfn}$};
		\draw [decorate,decoration={brace,mirror,amplitude=5pt}]
		(15.3,0) -- (15.3,6)
		node[midway,xshift=26pt]{$U_{\bfn\plush}$};
		\draw [decorate,decoration={brace,amplitude=5pt}]
		(6,10.3) -- (15,10.3)
		node[midway,yshift=18pt]{$U_{\bfn\plushv}$};
		\draw[-{Latex[length=1.5mm]}, thick] (4.5,7) -- (6.5,5.5) node[below, scale=1] {};
		\node at (3,8.5) {$U_{\bfn \cap \bfn\plush}$};
		\node at (3.5,7.5) {$\cap U_{\bfn\plush\cap\bfn\plushv}$};
		\fill[color=orange!20, fill opacity=0.7] (7.5,5.3) -- (8.5,7) -- (10,5.5) -- cycle;
		\draw[black!20!orange!70, thick] (7.5,5.3) -- (8.5,7) -- (10,5.5) -- cycle;
	\end{tikzpicture}
	\caption{\emph{Horizontal-vertical overlaps between adjacent rectangles.} By choice of our parameters, the purple rectangle has side lengths greater equal than~$1/4$ and so the orange triangle of diameter~$1/4$ or less is contained in precisely one of the four rectangles, namely~$U_{\bfn\plushv}$.}
	\label{fig:horizontal_vertical_overlaps}
\end{figure}

In general, we will not have $B_\bfn=B_\bfm$ on $U_\bfn\cap U_\bfm$; instead, this identity will hold only up to a gauge transformation. 
The next ingredient is a way to interpolate gauge transformations which we will use to restore the previous equality, see~\eqref{e:identity_gauge_trafo} below.  
To that end, recall the notation $O_G$ from Section{sec:notation}.  
With these notations in mind, we introduce the following lemma.
    \begin{lemma}\label{lem:interpolating_gauge_transformation}
        Let $V,V_1,V_2\subset \bR^2$ be compact sets such that $V_1\cap V_2=\varnothing$ and $V_1\cup V_2\subset V$. Let $\beta\in (0,1]$ and $g\in C^\beta(V;O_G)$. Then there exists $\hat g\in C^\beta(V;G)$ such that $\hat{g}|_{V_1}= 1_G$ and $\hat g|_{V_2}=g|_{V_2}$. We also have the bound
        \[
    |\hat g-1_G|_{C^\beta(V)}\lesssim d(V_1,V_2)^{-\beta}|g-1_G|_{C^\beta(V)},
    \]
    where the proportionality constant only depends on $G$. Moreover, if $g\in C^\infty(V;O_G)$, then $\hat g\in C^\infty(V;G)$. 
       Finally, one can choose $\hat g$ continuously in the sense that given another $\bar g$ with corresponding $\hat{\bar g}$, one has 
    \[
    |\hat g-\hat{\bar g}|_{C^\beta(V)}\lesssim d(V_1,V_2)^{-\beta} |g-\bar g|_{C^\beta(V)},
    \]
    where the proportionality constant only depends on $G$. 
    \end{lemma}
    
    Examples of regions~$V$, $V_1$, and~$V_2$ are given in Figure~\ref{fig:horizontal_patching} above; they will become relevant in the proof of \theo{thm:patching} below.

    \begin{proof}
    Note that since $V_1,V_2$ are compact and disjoint, we have $d(V_1,V_2)>0$. Then we know that we can find a smooth bump function $\psi\in C^\infty_c(\bR^2;[0,1])$ supported on $V^{+1}$ (recall \eqref{eq:ex_def}), such that $\psi|_{V_1}\equiv 0$  and $\psi|_{V_2}\equiv 1$, with the bound 
    \[
    |\psi|_{C^\beta}\leq 2d(V_1,V_2)^{-\beta}. 
    \]
    Now we define~$\hat g := \exp(\psi\log g)$ and find 
   	\begin{equs}
   		\abs[0]{\hat g-1_G}_{\Hol\beta}
   		& = \abs[0]{\hat g}_{\Hol\beta}
   		\lesssim \abs[0]{\psi \log g}_{\Hol\beta}
   		\lesssim \abs[0]{\psi}_{\Hol\beta} \abs[0]{\log{g}}_{\infty} + \abs[0]{\psi}_\infty \abs[0]{\log g}_{\Hol\beta} \\[0.5em]
   		& = \abs[0]{\psi}_{\Hol\beta} \abs[0]{\log{g} - \log 1_G}_{\infty} + \abs[0]{\psi}_\infty \abs[0]{\log g - \log 1_G}_{\Hol\beta}
   		\lesssim \abs[0]{\psi}_{C^\beta} \abs[0]{g - 1_G}_{C^\beta} 
   	\end{equs}
    which is the claimed bound. 
    The statement about smoothness is clear.
    \end{proof}

    The final ingredients consist of two results which we state without proof as they basically follow the proof of~\cite[Thm.~3.27]{CCHS2d}. 
    The first one is a more precise bound for gauge transformed additive functions in $\Omega^1_\beta$.
    
    \begin{proposition}\label{prop:gauge_transformation_bound}
     Let $\beta\in (\frac 2 3 ,1)$, $B\in\Omega_{\beta}^1(V)$ and $g\in C^\beta(V;G)$. One has
     \[
     |B^g|_{\beta}\lesssim  (1+|g|_{\Hol\beta})(|B|_{\beta}+|g|_{\Hol\beta}). 
     \]
     Furthermore, if we have  another $ \bar B\in\Omega_{\beta}^1$ and $\bar g\in C^\beta(V;G)$, then we get 
     \[
     |B^g-\bar B^{\bar g}|_{\beta}\lesssim (1+|g|_{\Hol\beta}+|\bar g|_{\Hol\beta}+|B|_\beta + |\bar{B}|_\beta)(|B-\bar B|_\beta + |g-\bar g|_{C^\beta} ). 
     \]
    \end{proposition}
    The second result is a stability version of the bound $|g|_{\Hol\beta}\lesssim |B|_\beta+|B^g|_\beta$ from \cite[Eq.~(3.24)]{CCHS2d} given in the lemma that follows. 
    \begin{lemma}\label{lem:bound_gauge_Holder_seminorm}
        Let $B,\bar B\in\Omega^1_\beta(V)$, and $g,\bar g\in C^\beta(V;G)$. For any $x\in V$, we have  
        \[
        |g-\bar g|_{\Hol\beta}\leq C(|B-\bar B|_\beta+|B^g-\bar B^{\bar g}|_\beta)+|g(x)-\bar g(x)|,
        \]
        where $C>0$ is locally uniformly in $B,\bar B,B^g$ and $\bar B^{\bar g}$. 
    \end{lemma}

\begin{proof}[of \theo{thm:patching}]
Let us suppress the notation $\sce,\mu,y$ and simply write $U_\bfn$ as well as $U_{\bfn\cap\bfm}:=U_\bfn\cap U_\bfm$. 
We first patch in the first direction, i.e.\ in the direction of $\rme_1$, and then in the second direction of $\rme_2$. 
	In the following paragraph, we fix the second (i.e. the \enquote{vertical}) level~$m \in \N$ which we will then vary in the second paragraph.
Fix furthermore $\rlog\in (0,1]$ such that the ball $B(1_G,\rlog)$ in $G$ is contained in $O_G$ where we recall the notation of Section \ref{sec:notation}.
    \medskip

    \noindent \textbf{Patching in the first direction.}
   Assume that~\eqref{eq:patch_req} holds with~$\delta_\patch'=\rlog$; we will choose the actual $\delta_\patch$ small enough later.    Let us take any $\bfn=(n^1,n^2) = (n^1,m) \in\bN^2_{L}$.
    We know that by assumption  $B_\bfn^{g_{\bfn,\bfn\plush}}=B_{\bfn\plush}$ on $U_{\bfn\cap \bfn\plush}$.  Furthermore, 
    \begin{align}\label{eq:patching_g_in_log_chart}
    |g_{\bfn,\bfn\plush}-1_G|_{\infty;U_{\bfn\cap\bfn\plush}}\leq \rlog.
    \end{align}
    Note that 
    \[
    U_\bfn=[y_1+\tfrac 23 n^1\mu^1, y_1+\tfrac 23 n^1\mu^1+\mu^1]\times [y_2+\tfrac 23 n^2\mu^2
    ,y_2+\tfrac 23 n^2\mu^2+\mu^2].
    \]
    We apply \lem{lem:interpolating_gauge_transformation}, with 
    \begin{equation} \label{e:ex_V}
	    V=U_{\bfn\cap \bfn\plush}=\{s^1 \rme_1\!+\! s^2 \rme_2 : s^1 \in [\tfrac{2}{3}\mu^1, \mu^1], s^2 \in [0, \mu^2]\} + \{y + \tfrac{2}{3}(n^1\mu^1 \rme_1 + n^2\mu^2 \rme_2)\},
    \end{equation}
    and 
    \begin{equs}[][e:ex_V12]
    V_1&=\{s^1 \rme_1 + s^2 \rme_2 : s^1 \in [\tfrac{6}{9}\mu^1, \tfrac 79\mu^1], s^2 \in [0, \mu^2]\} + \{y + \tfrac{2}{3}(n^1\mu^1 \rme_1 + n^2\mu^2 \rme_2)\}\\
     V_2&=\{s^1 \rme_1 + s^2 \rme_2 : s^1 \in [\tfrac{8}{9}\mu^1, \tfrac 99\mu^1], s^2 \in [0, \mu^2]\} + \{y + \tfrac{2}{3}(n^1\mu^1 \rme_1 + n^2\mu^2 \rme_2)\}.
    \end{equs}
	We may think of~$\rme_1$ as the horizontal direction and $\rme_2$ as the vertical direction and in that case, the regions~$V$, $V_1$, and~$V_2$ are illustrated in Figure~\ref{fig:horizontal_patching}. 
	Note that~$d(V_1,V_2) = (1/9)\mu^1$ in that case.

By~\eqref{eq:patching_g_in_log_chart} and the choice of $\rlog$ (defined in the paragraph before Theorem~\ref{thm:patching}), we have $g_{\bfn,\bfn\plush}\in C^\beta(U_{\bfn\cap\bfn\plush};O_G)$.
Hence from \lem{lem:interpolating_gauge_transformation}, we obtain $\hat g_{\bfn,\bfn\plush}\in C^\beta(U_{\bfn\cap \bfn\plush},G)$ such that $\hat g_{\bfn, \bfn\plush}=1_G$ on $V_1$ and $\hat g_{\bfn, \bfn\plush}= g_{\bfn, \bfn\plush}$ on $V_2$. 
	Moreover, by the same lemma and~Lemma~\ref{lem:bound_gauge_Holder_seminorm}, there exists some $C_0>0$ depending only on $G$ such that
    \begin{align}\label{eq:patching_bound_interpolated}
    \begin{split}
    |\hat g_{\bfn, \bfn\plush}-1_G|_{C^\beta(U_{\bfn\cap\bfn\plush})}&\leq  C_0(\tfrac 19\mu^1)^{-\beta}\max_{\bfn,\bfm\in\N_L^2} \{|B_\bfn |_{\beta;U_\bfn},|g_{\bfn,\bfm}-1_G|_{\infty;U_{\bfn\cap\bfm}}\}\\
    &\leq  C_0 2^{4\beta} \max_{\bfn,\bfm\in\N_L^2} \{|B_\bfn |_{\beta;U_\bfn},|g_{\bfn,\bfm}-1_G|_{\infty;U_{\bfn\cap\bfm}}\},%
    \end{split}
    \end{align}
    where we have used that~$\mu^1 \geq 3/4$ and that~$(12)^\beta \leq 2^{4\beta}$.
    For $\bfn=(n^1,n^2)$ such that $1\leq n^1\leq l^1-1$, we then define $g_\bfn:U_\bfn\to G$, by 
    \begin{equation} \label{e:def_g_bfn}
    	g_\bfn:=
		\begin{cases}
		    \hat g_{\bfn\minush,\bfn}g_{\bfn\minush,\bfn}^{-1} \quad  &\text{    on }\   U_{\bfn\minush\cap \bfn} \\ 
		   1_G \quad  &\text{ on }\  U_\bfn\setminus ( U_{\bfn\minush \cap \bfn}\cup U_{\bfn\cap 	\bfn\plush}) \\ 
		   \hat g_{\bfn, \bfn\plush}  \quad  &\text{ on }\   U_{\bfn\cap \bfn\plush}.
		\end{cases} 
    \end{equation}
We extend this definition at the boundary indices $n^1=0$ and $n^1=l^1$ by the obvious one-sided version of the same definition.
	This choice of~$g_\bfn$ is clear on the first two sets; let us briefly explain the reasoning for~$U_{\bfn\minush\cap \bfn}$. 
	Viewed as the right overlap on~$U_{\bfn\minush}$, we require~$B^{g_\bfn}_{\bfn} = B_{\bfn\minush}^{\hat{g}_{\bfn\minush,\bfn}}$. 
	On the other hand, viewed as the left overlap on~$U_{\bfn}$, we need~$B_{\bfn}^{g_{\bfn}} = (B_{\bfn\minush}^{g_{\bfn\minush,\bfn}})^{g_{\bfn}}$.
	Equating the gauge transformations (which form a~\emph{left} action) in both cases explains our choice on~$U_{\bfn\minush\cap \bfn}$.

On the three different segments~$g_{\bfn}$ is $\beta$-H\"older, and by continuity on the connecting strips, we get that $g_\bfn\in C^\beta$. 
In fact, in the case $(g_{\bfn,\bfm})_{\bfn,\bfm\in \bN^2_L}$ consists of smooth functions, then $g_\bfn$ is also smooth since we then have $\hat g_{\bfn\minush,\bfn}g_{\bfn\minush,\bfn}^{-1} $ (resp.~$\hat g_{\bfn, \bfn\plush}$) are equal to $1_G$ on a $\frac 13$-portion of the rectangle in $U_{\bfn\minush\cap\bfn}$ (resp.~$U_{\bfn\cap\bfn\plush}$) that is bordering $ U_\bfn\setminus ( U_{\bfn\minush \cap \bfn}\cup U_{\bfn\cap \bfn\plush})$.
Those two regions are shaded in~Figure~\ref{fig:horizontal_patching} and is true by construction of~$\hat{g}_{\bfn\minush,\bfn}$ (resp.~$\hat{g}_{\bfn,\bfn\plush}$).

On $U_{\bfn\cap\bfn\plush}$, in conjunction with the gauge equivalence condition in Definition~\ref{def:compatible_gauge_data}, our construction implies that
	\begin{align*}
	B_\bfn^{g_\bfn}
	=
	B_\bfn^{\hat g_{\bfn,\bfn\plush}}
	=
	B_{\bfn\plush}^{\hat g_{\bfn,\bfn\plush}g_{\bfn,\bfn\plush}^{-1}}
	=
	B_{\bfn\plush}^{g_{\bfn\plush}} \,;
	\end{align*}
therefore, by Lemma~\ref{lem:global_additive_function}, $B_\bfn^{g_\bfn}$ defines an additive function on $\check U_{m}:=\bigcup_{n\in\bN_{l^1}}U_{(n,m)}$, which we call $ A_m\in\Omega_\beta^1(\check U_m)$. 
For any $\bfn=(n,m)$, we have
\begin{equation} \label{e:def_A_m}
	A_m=B_\bfn^{g_\bfn} \quad \text{on} \quad U_\bfn,
\end{equation}
proving the first condition in~\eqref{eq:patched_gauge_datum} (at least for this patching step). In the case $(B_{\bfn},g_{\bfn,\bfm})_{\bfn,\bfm\in\bN_L^2}$ consists of smooth $1$-forms and gauge transformations, then $A_m$ is also smooth. 

To verify, the second condition in~\eqref{eq:patched_gauge_datum}, we note that on $U_{\bfn\cap\bfn\plush}$, the definition of $g_\bfn$ gives
$g_\bfn=\hat g_{\bfn,\bfn\plush}$. On the other hand, this same overlap is the left overlap of $U_{\bfn\plush}$, and therefore
\[
    g_{\bfn\plush}
    =
    \hat g_{\bfn,\bfn\plush}g_{\bfn,\bfn\plush}^{-1}.
\]
Hence
\begin{equation}\label{eq:first_direction_gauge_compatibility}
    g_\bfn
    =
    g_{\bfn\plush}g_{\bfn,\bfn\plush}
    \qquad\text{on }U_{\bfn\cap\bfn\plush}.
\end{equation}

Finally, using \eqref{eq:patch_req},~\eqref{eq:patching_bound_interpolated} and~\eqref{e:def_g_bfn}, by Lemma~\ref{lem:global_additive_function},  Proposition~\ref{prop:gauge_transformation_bound}, and Lemma~\ref{lem:bound_gauge_Holder_seminorm} then allows to obtain the bound
\begin{equation} \label{e:bound_A_m}
	\max_{m\in \N_{l^2}} |A_m|_{\beta;\check U_m}+\max_{\bfn\in\N_L^2}
    |g_{\bfn}-1_G|_{C^\beta(U_\bfn)}\lesssim \max_{\bfn,\bfm\in\N_L^2} \{|B_\bfn |_{\beta;U_\bfn},|g_{\bfn,\bfm}-1_G|_{\infty;U_{\bfn\cap \bfm}}\} \,.
\end{equation}
Since the same lemmas have continuity estimates,  given another compatible gauge datum $(\bar B_\bfn,\bar g_{\bfn,\bfm})_{\bfn,\bfm\in\bN_{L}^2}$, we can obtain $\bar A_m$ in analogy to before and the bound
\begin{equation} \label{e:difference_A_bar_A_m}
\max_{m\in\N_{l^2}}	|A_m-\bar A_m|_{\beta;\check U_m}+\max_{\bfn\in\N^2_L}
    |g_{\bfn}-\bar g_{\bfn}|_{C^\beta(U_\bfn)}\lesssim \max_{\bfn,\bfm\in\N_L^2} \{|B_\bfn-\bar B_\bfn |_{\beta;U_\bfn},|g_{\bfn,\bfm}-\bar g_{\bfn,\bfm}|_{\infty;U_{\bfn\cap \bfm}}\} \,.
\end{equation}
The proportionality constant in both~\eqref{e:bound_A_m} and~\eqref{e:difference_A_bar_A_m} depend only on $G$ and $\beta$. 

\medskip 
\noindent \textbf{Patching in the second direction.} Let us now assume that~\eqref{eq:patch_req} holds with $\delta_\patch\in (0,\rlog]$ to be specified later. Our plan is to re-use the previous step and, to this end, we have to construct gauge\footnote{Note that~$\check{U}_m \cap \check{U}_{m+k} \neq \emptyset$ is equivalent to~$k \in \cbr[0]{0,\pm1}$, so it is enough to focus on~$h_{m,m+1}$ for any~$m \in \N_{l^2}$.} transformations~$h_{m,m+1}$ such that~$(A_m,h_{m,m+1})_{m \in \N_{l^2}}$ is a compatible gauge datum which also satisfies~\eqref{eq:patch_req} with $\delta_\patch'=\rlog$. 

At first, note that we have $A_{m+1}=B_{\bfn\plusv}^{g_{\bfn\plusv}}$ on $U_{\bfn\cap\bfn\plusv}$ by construction and
\begin{equ}[e:vertical_patching_1]
A_{m+1}=B_{\bfn\plusv}^{g_{\bfn\plusv}}=(B_\bfn^{g_{\bfn,\bfn\plusv}})^{g_\bfn\plusv}
=(B_\bfn^{g_\bfn})^{g_{\bfn\plusv}g_{\bfn,\bfn\plusv}g_\bfn^{-1}}=A_m^{g_{\bfn\plusv}g_{\bfn,\bfn\plusv}g_\bfn^{-1}}.
\end{equ}
We will argue that  $\tilde g_{\bfn,\bfn\plusv}:=g_{\bfn\plusv}g_{\bfn,\bfn\plusv}g_\bfn^{-1}$ defines a function  on $\check U_m\cap \check U_{m+1}$, which is the sought after $h_{m,m+1}$. 
In order to do so, we need to check the consistency of the corresponding expression on $U_{\bfn\cap\bfn\plusv}\cap U_{\bfn\plush\cap\bfn\plushv}$, see Figure~\ref{fig:horizontal_vertical_overlaps} for an illustration. 
On one hand, by definition of~$g_{\bfn\plusv}$ on $U_{\bfn\plusv\cap\bfn\plushv}$, i.e. on the right overlap of~$U_{\bfn\plusv}$, we have\footnote{Note that the cocycle condition in Definition~\ref{def:compatible_gauge_data} implies~$g_{\bfn,\bfm}^{-1} = g_{\bfm,\bfn}$ for any~$\bfn, \bfm \in \N^2$.} 
\begin{equation} \label{e:identity_gauge_trafo}
	\tilde g_{\bfn,\bfn\plusv}=\hat g_{\bfn\plusv,\bfn\plushv}g^{-1}_{\bfn\plusv,\bfn}\hat g^{-1}_{\bfn,\bfn\plush} \,.
\end{equation}
On the other hand, by definition of~$g_{\bfn\plushv}$ on~$U_{\bfn\plusv\cap\bfn\plushv}$ (i.e. the left overlap of~$U_{\bfn\plushv}$) and by definition of~$g_{\bfn\plush}$ on~$U_{\bfn\cap\bfn\plush}$ (i.e. the left overlap of~$U_{\bfn\plush}$), we have 
\begin{align*}
\tilde g_{\bfn\plush,\bfn\plushv}
&=
g_{\bfn\plushv} g_{\bfn\plush,\bfn\plushv} g_{\bfn\plush}^{-1} \\ 
&=
\hat g_{\bfn\plusv,\bfn\plushv} g_{\bfn\plusv,\bfn\plushv}^{-1}g_{\bfn\plushv,\bfn\plush}^{-1}g^{-1}_{\bfn\plush,\bfn} \hat g^{-1}_{\bfn,\bfn\plush}\\
&=\hat g_{\bfn\plusv,\bfn\plushv} g_{\bfn\plusv,\bfn}^{-1}\hat g^{-1}_{\bfn,\bfn\plush},
\end{align*}
where we have used the cocycle condition from \defref{def:compatible_gauge_data} twice in the last step.  
Therefore, $h_{m,m+1}$ is a well-defined function in $C^\beta(\check U_m\cap \check U_{m+1})$. We also set $h_{m,m}\equiv 1_G$. 

As we have seen in~\eqref{e:vertical_patching_1}, we have
\[
A_m^{h_{m,m+1}}=A_{m+1} \quad \text{on} \quad \check{U}_m \cap \check{U}_{m+1}
\]
and so the gauge equivalence condition of Definition~\ref{def:compatible_gauge_data} is satisfied.
The cocycle condition is trivially satisfied because~$\check{U}_m \cap \check{U}_{m+k} \cap \check{U}_{m+l} = \emptyset$ for~$k \neq l$ and~$k, l \neq 0$.

We have verified the criteria for a compatible gauge datum, so it remains to check the bounds in~\eqref{eq:patch_req} (with $\delta_\patch'=\rlog$). 
Note that~\eqref{e:bound_A_m} implies 
\[
\max_{m\in \N_{l^2}} |A_m|_{\beta;\check U_m}+\max_{m\in \N_{l^2}}|h_{m,m+1}-1_G|_{\infty}\leq C_1\delta_\patch,
\]
for some $C_1>0$ depending only on $G$ and $\beta$. 
We then choose $\delta_\patch:=(C_1+1)^{-1}\rlog$, so that 
\[
\max_{m\in \N_{l^2}} |A_m|_{\beta;\check U_m}+\max_{m\in \N_{l^2}}|h_{m,m+1}-1_G|_{\infty}\leq \rlog,
\]
implying that $(A_m,h_{m,m+1})_{m\in\N_{l^2}}$ satisfies~\eqref{eq:patch_req} with $\delta_\patch'=\rlog$. So we can apply the previous patching step. 

Indeed, it is not difficult to see that the argument for patching in the first direction did not depend on the rectangles, so now we simply take rectangles with respect to $\tilde\sce=\{\tilde\rme_1,\tilde \rme_2\}$, with $\tilde\rme_1=\rme_2$ and $\tilde\rme_2=\rme_1$ and $\tilde l^1=l^2$ and $\tilde l^2=0$, the same $y$, as well as $\tilde\mu^1=\mu^2$ and $\tilde\mu^2=\mu^1+\frac 2 3 l^1\mu^1$. We can define $A_{(m,0)}=A_m$ and $h_{(m,0),(m+1,0)}=h_{m,m+1}$ and the result applies to this new set of additive functions, as we have verified above.
The picture that we are dealing with this time is patching connection $1$-forms on a horizontally long rectangles as seen in Figure~\ref{fig:vertical_patching}. 
%
%
%
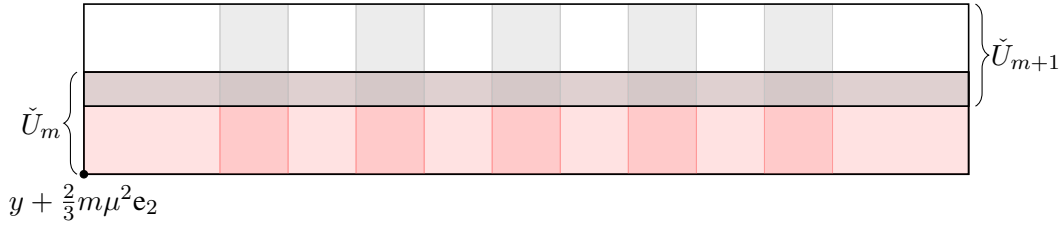
\begin{figure}[H]
	\centering
	%
	\tikzset{>={Latex[width=3pt,length=3pt]}}
	\begin{tikzpicture}[scale=0.45] 
		\draw[semithick, fill=red!15, fill opacity=0.75] (0,0) rectangle (26,3); 		
		%
		\foreach \k in {0,...,5}{
			\draw[line width=0.25pt,color=red!40] (4*\k,0) rectangle (4*\k+6,3);
		}
		\foreach \k in {1,...,5}{
			\draw[line width=0.25pt,color=red!40, fill opacity=0.5, fill=red!30] (4*\k,0) rectangle (4*\k+2,3);
		}		
		\foreach \k in {0,...,5}{
			\draw[line width=0.25pt,color=gray!40] (4*\k,2) rectangle (4*\k+6,5);
		}
		\foreach \k in {1,...,5}{
			\draw[line width=0.25pt,color=gray!40, fill opacity=0.5, fill=gray!30] (4*\k,2) rectangle (4*\k+2,5);
		}
		\draw[semithick] (0,0) rectangle (26,3); 		
		\draw[semithick] (0,2) rectangle (26,5); 
		\draw[semithick,fill opacity=0.5, fill=gray!50] (0,2) rectangle (26,3); 
		\draw [decorate,decoration={brace,amplitude=5pt}]
		(-0.2,0) -- (-0.2,3);
		\node[left=-2pt] at (-0.5,1.5) {$\check{U}_{m}$};
		\draw [decorate,decoration={brace,mirror,amplitude=5pt}]
		(26.2,2) -- (26.2,5);
		\node[right=2pt] at (26.2,3.5) {$\check{U}_{m+1}$};
		%
		\node[draw,circle,fill=black, inner sep=1pt] at (0,0) {};
		\node[below] at (0,0) {$y+\frac{2}{3}m\mu^2\rme_2$};
	\end{tikzpicture}
	\caption{\emph{Patching in the second (\enquote{vertical}) direction.} 
		For any~$m \in \N_{l^2}$, the new \enquote{long} rectangle~$\check{U}_m$  arises as the union of the smaller rectangles~$U_{(n,m)}$ for~$n \in \N_{l^1}$. 
		For presentational clarity, their horizontal overlaps on each level~$m$ are shaded in darker red. resp. gray.
		The new one-form~$A_m$ given in~\eqref{e:def_A_m} is now defined on~$\check{U}_m$ and the patching concerns the \enquote{vertical} overlap~$\check{U}_m \cap \check{U}_{m+1}$ shaded in dark gray.
	}
	\label{fig:vertical_patching}
\end{figure}
Now we have obtained the final global additive function
$B\in\Omega_\beta^1(\cU)$ together with gauge transformations
$p_m\in C^\beta(\check U_m;G)$ such that
\[
    B=A_m^{p_m}
    \qquad\text{on }\check U_m,
\]
and
\begin{equation}\label{eq:second_direction_p_compatibility}
    p_m=p_{m+1}h_{m,m+1}
    \qquad\text{on }\check U_m\cap\check U_{m+1}.
\end{equation}

It remains to relate the final output to the original compatible gauge datum. For $\bfn=(n,m)$, define
\[
    q_\bfn:=p_m g_\bfn
    \qquad\text{on }U_\bfn.
\]
Then, by~\eqref{e:def_A_m},
\[
    B=A_m^{p_m}=(B_\bfn^{g_\bfn})^{p_m}=B_\bfn^{p_mg_\bfn}=B_\bfn^{q_\bfn}
    \qquad\text{on }U_\bfn,
\]
proving the first condition in~\eqref{eq:patched_gauge_datum}. We now show the second condition therein. If $\bfm=\bfn\plush$, then~\eqref{eq:first_direction_gauge_compatibility} gives
\[
    q_\bfn
    =
    p_m g_\bfn
    =
    p_m g_{\bfn\plush}g_{\bfn,\bfn\plush}
    =
    q_{\bfn\plush}g_{\bfn,\bfn\plush}.
\]
If $\bfm=\bfn\plusv$, then~\eqref{eq:second_direction_p_compatibility} and the definition of $h_{m,m+1}=g_{\bfn\plusv}g_{\bfn,\bfn\plusv}g_\bfn^{-1}$ give
\[
    q_\bfn
    =
    p_m g_\bfn
    =
    p_{m+1}h_{m,m+1}g_\bfn
    =
    p_{m+1}g_{\bfn\plusv}g_{\bfn,\bfn\plusv}
    =
    q_{\bfn\plusv}g_{\bfn,\bfn\plusv}.
\]
The remaining overlap relations follow from these two cases and the cocycle condition of the original compatible gauge datum. Thus
\[
    q_\bfn=q_\bfm g_{\bfn,\bfm}
    \qquad\text{on }U_{\bfn\cap\bfm}.
\]

Finally, the bounds in~\eqref{e:patching_bound}-\eqref{e:patching_bound_cont}, and smoothness  in the case where the compatible gauge datum is smooth, follow by the previous step as well.  
\end{proof}

\section{Rough Uhlenbeck compactness}\label{ch:RUC}

In this section, we combine all the results that we have proved so far to conclude the proof of rough Uhlenbeck compactness Theorem~\ref{thm:Rough_Uhlenbeck} and its application to the YM measure on $\Lambda$ in Theorem~\ref{thm:RUC_YM}. 

  \subsection{Coulomb gauge}\label{sec:Coulomb_gauge}
The first main step in proving the rough Uhlenbeck compactness is finding a Coulomb gauge in the space $\Omega^1_\beta$ for some $\beta$ sufficiently close to $1$. That is a matter of combining what we know so far from Section~\ref{sec:solution_theory} and Section~\ref{s:model_bounds}.
Recall from Remark \ref{rem:left_action} that $L_g$ denotes left multiplication by $g$.

\begin{theorem}[Coulomb gauge]\label{thm:Coulomb_gauge}
		Let $\alpha\in (\frac13,\frac 12)$ and $\beta\in (0,6\alpha-2)$. 
		Then there exists $\sigma_\coul\in (0,1]$ such that,
		for any $\bfA\in\bfOmega_{\alpha\ax}^1$  satisfying $\triple{\bfA}_{\alpha\ax}\leq \sigma_\coul$, there exists
		a (controlled) gauge transformation~$g: \Lambda \to G$ such that for any rectangle~$\Lambda' \Subset \Lambda^\circ$, the following properties hold:
		\begin{enumerate}[label=(\roman*)]
			\item \label{pt:regularity}\emph{Regularity:} $(g,L_g) \in \mfG_\bfA^{2\alpha}(\Lambda')$
            ,~$B := A^g$ is in the \emph{Coulomb gauge} $\diff^* B=0$,
			and $B\sVert[0]_{\Lambda'} \in \Omega_\beta^1(\Lambda')$.
			\item \label{pt:smoothness}\emph{Smoothness:}  If~$\bfA$ is the canonical lift of a smooth $A\in\Omega^1_{\infty\ax}$, both $B$ and $g$ are smooth on $\Lambda^\circ$.
			Moreover, for~$\bfA = \bf0$, we have~$g \equiv 1_G$ and thus~$B = 0$.
			\item \label{pt:local_lipschitz} \emph{Local Lipschitz continuity:} Given another $\bar\bfA \in\bfOmega_{\alpha\ax}^1$ satisfying $\triple{\bar\bfA}_{\alpha\ax}\leq \sigma_\coul$ with corresponding $\bar g$ and $\bar B$, one has
			\begin{equ}[e:Coulomb_gauge:bound]
			\triple{(g,L_g);(\bar g,L_{\bar g})}_{2\alpha;\Lambda'}+|B-\bar B|_{\beta;\Lambda'}\lesssim  \triple{\bfA;\bar\bfA}_{\alpha\ax}
			\end{equ}
			where the proportionality constant depends only on $\dist(\partial\Lambda,\partial\Lambda')$, $\alpha$, $\beta$ and the Lie group~$G$. 
			In particular, for~$\bar\bfA=\bf0$, we get $\triple{(g,L_g);(1_G,L_{1_G})}_{2\alpha;\Lambda'}
		+
		|B|_{\beta;\Lambda'}
		\lesssim
		\triple{\bfA}_{\alpha\ax}$.
		\end{enumerate}
	\end{theorem}

      \begin{remark}
  It is interesting to determine in what situations there is uniqueness of $g$ under a given set of boundary data
  (the modelled distribution $\blue g$ from Theorem \ref{thm:existence_solution} is \emph{not} unique, see Remark \ref{rem:non_unique_g}).
   This question is connected to the Gribov ambiguity and is closely related to the positivity of the Fadeev--Popov operator, see \cite[Sec. 2]{VZ12}.
\end{remark}

The Coulomb gauge will be in a H\"older-Besov space by the reconstruction theorem, but the previous statement gives the  stronger $\Omega^1_\beta$ regularity which requires the following new result:  
    \begin{lemma}\label{lem:integration_af_Young}
        Let $(E,|\Cdot|)$ be a Banach space,   $\alpha\in (\frac 13,\frac 12)$ and $\beta'\in(\frac 12,1)$ such that $\alpha+\beta'>1$. Let $Y\in \lc^\alpha (\Lambda;L(E))$ and $\zeta\in\Omega_{\beta'}^1(\Lambda;E)$. Define 
        \begin{align}\label{eq:def_Y_zeta}
        B(\ell):=\int_0^1\ell_Y(t)\dif\ell_\zeta(t)\;,
        \end{align}
		where we recall the notation $\ell_\zeta$ from \eqref{eq:ell_A} and $\ell_Y$ following \eqref{eq:line_int_gauge_transformation_smooth}. 
        Then for any $\beta\in (0,1)$ satisfying $\alpha+\beta'-\frac{\beta} 2>1$, there exists $C=C(\alpha,\beta,\beta')>0$ such that
        \[
        |B|_{\beta}\leq C|Y|_{C^\alpha}|\zeta|_{\beta'}.
        \]
        Furthermore, given another $\bar Y\in \lc^\alpha(\Lambda;L(E))$ and $\bar \zeta\in\Omega_{\beta'}^1(\Lambda;E)$, we can define $\bar B$ similarly as in \eqref{eq:def_Y_zeta} with $\bar Y$ and $\bar \zeta$, and the following inequality holds: 
        \[
        |B-\bar B|_{\beta}\leq C(|Y-\bar Y|_{C^\alpha}|\zeta|_{\beta'}+|\bar Y|_{C^\alpha}|\zeta-\bar \zeta|_{\beta'}).
        \]
    \end{lemma}
\begin{remark}
This result is beyond the regime treated in \cite[Sec.~3]{CCHS2d} which would be  $\alpha+\frac\beta 2>1$ and $\frac \alpha2+\beta'>1$. The first inequality cannot be satisfied in our setting. 
\end{remark}

\begin{proof}
We use similar ideas that exist in \cite[Sec.~3.4]{CCHS2d}, but we interpolate inequalities  to go beyond the regime covered therein. Recall the notion of the vee seminorm $|\Cdot|_{\beta\text{-vee}}$ from \cite[Def.~3.8]{CCHS2d}. That seminorm is defined by 
\begin{align*}
|B|_{\beta\text{-vee}}=\sup_{\ell,\bar\ell}\frac{|B(\ell)-B(\bar\ell)|}{\Area(\ell,\bar\ell)^{\beta/2}},
\end{align*}
where the supremum is taken over line segments that form a \emph{vee}, visually, and are reasonably close to each other, such that $\Area(\ell,\bar\ell)$ is the unique area \enquote{enclosed} by these line segments. The reason we introduce this notion is that it is easier to bound than $|\Cdot|_{\beta\tri}$ (defined in \eqref{eq:def_beta_tri}) and one has $|\Cdot|_\beta\asymp|\Cdot|_{\beta\gr}+ |\Cdot|_{\beta\text{-vee}}$ by \cite[Thm.~3.11]{CCHS2d}.

 For the growth norm we use the sewing lemma \cite[Lem~4.2]{FH20}, to obtain
\[
\left|\int_0^1\ell_Y(t)\dif\ell_\zeta(t)\right|\leq |Y|_{C^\alpha}|\zeta|_{\beta'} |\ell|^{\beta'}\;,
\]
which implies the desired $\beta\gr$ bound since $\beta<\beta'$ four ur choice of parameters.

Turning to the vee seminorm, we start by writing for $\ell$ and $\bar\ell$ in $\cX$ close to each other
\[
\int^1_0 \ell_Y(t)\dif \ell_{\zeta}(t)-\int^1_0\bar\ell_Y(t)\dif \bar\ell_{\zeta}(t)=\int_0^1 (\ell_Y(t)-\bar\ell_Y(t))\dif\ell_{\zeta}(t)+\int^1_0 \bar\ell_Y(t)\dif \,(\ell_{\zeta}-\bar\ell_{\zeta})(t).
\]
We start with the first term. Since $Y\in C^\alpha$, we get for any $\mu \in [0,1]$ that
\[
|\ell_Y-\bar\ell_Y|_{C^{\alpha (1-\mu)}}\lesssim  |Y|_{C^\alpha} |\ell|^{\alpha(1-\mu)}|\ell_1-\bar\ell_1|^{\mu\alpha}.
\]
Let us take $\mu\in [0,1]$ such that  $\beta'+(1-\mu)\alpha>1$ holds to obtain
\[
\left| \int_0^1 (\ell_Y(t)-\bar\ell_Y(t))\dif\ell_{\zeta}(t)\right|\lesssim |\zeta|_{\beta'}|Y|_{C^\alpha} |\ell|^{\beta'+\alpha(1-\mu)}|\ell_1-\bar\ell_1|^{\mu\alpha}.
\]
For such $\mu$ one has $2\mu\alpha\leq 1<\alpha+\beta'$ and as such  $\beta'+\alpha(1-\mu)\geq \mu\alpha$, which implies
\[
\left| \int_0^1 (\ell_Y(t)-\bar\ell_Y(t))\dif\ell_{\zeta}(t)\right|\lesssim|\zeta|_{\beta'}|Y|_{C^\alpha}\Area(\ell,\bar\ell)^{\mu\alpha}.
\]

Let us now do the second term. For $\mu'\in [0,1]$ we note that by interpolation
\[
|\ell_\zeta-\bar\ell_\zeta|_{C^{\mu' \nfrac{\beta'}{2}+(1-\mu')\beta'}}\lesssim |\ell_\zeta-\bar\ell_\zeta|_{C^{\beta'/2}}^{\mu'}  |\ell_\zeta-\bar\ell_\zeta|_{C^{\beta'}}^{1-\mu' }\lesssim |\ell|^{(1-\mu')\beta'}\Area(\ell,\bar\ell)^{\beta'\mu'/2}|\zeta|_{\beta'}.
\]
We take $\mu'\in [0,1]$ such that $\mu'\frac{\beta'}{2}+(1-\mu')\beta'+\alpha>1$. Then we note that
\begin{align*}
\left|\int^1_0 \bar\ell_Y(t)\dif\,(\ell_\zeta-\bar\ell_\zeta)(t)\right|
&\lesssim|Y|_{C^\alpha}|\zeta|_{\beta'} |\ell|^{2(1-\mu')\beta'}\Area(\ell,\bar\ell)^{\beta'\mu'/2}\\
&\lesssim |Y|_{C^\alpha}|\zeta|_{\beta'}\Area(\ell,\bar\ell)^{\beta'\mu'/2}.
\end{align*}
To make both bounds coincide we take $\mu\alpha=\mu'\frac{\beta'}{2}$ and the condition on $\mu$ and $\mu'$ is $\alpha+\beta'-\mu'\frac{\beta'}{2} >1$. We set $\beta:=\mu'\beta'$ and the claim follows.

The continuity statement can be obtained similarly. From there one can also obtain that $B\in\Omega^1_{\beta}$, defined through the closure of smooth functions under $|\Cdot|_\beta$-norm, as we have assumed that $Y\in\lc^\alpha$ (recall the notation for little H\"older space from Section~\ref{sec:notation}).
\end{proof}
We finally get to the promised result on the existence of Coulomb gauge: 
\begin{proof}[of Theorem~\ref{thm:Coulomb_gauge}]
	Consider first $\hat A\in\Omega^1_{\infty\ax}$ with canonical lift $\hat \bfA\in \bfOmega_{\alpha\ax}^1$ satisfying $\triple{\hat\bfA}_{\alpha\ax}\leq 1$.
By Theorem~\ref{thm:construction_model}, the canonical model $\sfZ$ constructed from $\hat A$
satisfies $\triple{\sfZ}_{3\alpha} \leq L$ for some $L>0$.
Let $\sigma_\sol=\sigma_\sol(L)\in (0,1]$ be as in Theorem~\ref{thm:existence_solution}.
Fix for now $\sigma\in (0,\sigma_\sol]$.
Let $\blue g = \blue g^\sigma\in \scG^\sfZ$ be as in Theorem~\ref{thm:existence_solution}.
Following Lemma \ref{lem:B_expresssion}, we write $B=\cR\rmQ_{<3\alpha}\blue{(\sigma A)^g}$.
By~Lemma~\ref{lem:g_smooth}, $B$ and $g^\1$ are smooth in $\Lambda^\circ$,
thus $B$ agrees the classical smooth group action $A^{g^\1}$, and $\diff^*B\equiv 0$ on $\Lambda^\circ$.
If $\hat A=0$, then furthermore $g=1_B$ and $B=0$ by Lemma~\ref{lem:g_smooth}.

Furthermore, in the notation of~Lemma~\ref{lem:B_expresssion},
\[
B_i=\sigma g^\1\zeta_i (g^\1)^{-1}-f_i (g^\1)^{-1},
\]
where $\zeta=-\diff^*K*F^{\hat A}$.
Denoting $Q=(f_1(g^\1)^{-1},f_2(g^\1)^{-1})$, by Lemma~\ref{lem:B_expresssion} and the first embedding in~\eqref{eq:omega_beta_embedding},
\[
|Q|_{6\alpha-2;\Lambda'}
\lesssim
|Q|_{C^{3\alpha-1}(\Lambda')}
\lesssim
\dist(\d\Lambda,\d\Lambda')^{\alpha-1}\;.
\]
Furthermore, by \prop{prop:regularity_d*KFA}, $\zeta \in \Omega_{2\alpha}^1$ with $|\zeta|_{2\alpha;\Lambda}\lesssim 1$,
while $g\in C^\alpha(\Lambda)$.
Applying \lem{lem:integration_af_Young} with $\beta' = 2\alpha$, and noting that $\alpha + \beta' - \frac\beta2 >1 \Leftrightarrow 6\alpha-2>\beta$,
it follows that
$
g^\1\zeta (g^\1)^{-1} \in \Omega^1_{\beta;\Lambda}
$.
We conclude that $B\restr_{\Lambda'}\in \Omega^1_{\beta}(\Lambda')$.

We next verify that the $g^\1$ is controlled by $\sigma \hat A$.
By the expansion of $\blue g$ from Lemma~\ref{lem:expression_g_1} , we know that 
           	\[
           	\rmQ_{<2\alpha}\blue g=g^\1\blue{\1}+\sigma g^\1 \blue{\cI_1 A_1},
           	\]
           	and by the definition of modelled distributions, it satisfies for all $x,y\in \Lambda'$
           	\[
           |g(y)-g(x)-\sigma g(x) (K_1* \hat A_1(y)-K_1*\hat A_1(x)|\lesssim \|\blue g-\blue{\1_G}\|_{3\alpha,\alpha} |x-y|^{2\alpha}\lesssim |x-y|^{2\alpha}\;,
           	\]
   			with proportionality constant depending on $\dist(\partial\Lambda',\partial\Lambda)$. 
           	Then, by Proposition~\ref{prop:regularity_d*KFA},
           	\[
           	|g(y)-g(x)-\sigma g(x) \hat A(\ell^{x;y})|\lesssim  |x-y|^{2\alpha}+|A|_{\alpha\ax}|x-y|^{2\alpha}\lesssim |x-y|^{2\alpha}\;,
           	\]
           	implying that $(g,L_g)\in\fG^{2\alpha}_{\sigma\hat A}(\Lambda')$.
	
	Consider now $A\in\Omega^1_{\infty\ax}$ with canonical lift $\bfA\in \bfOmega_{\alpha\ax}^1$ satisfying $\triple{\bfA}_{\alpha\ax}\leq \sigma_\sol$.
	Items \ref{pt:regularity}-\ref{pt:smoothness} follow from the above by taking $\hat A = A/\sigma_\sol$, for which we remark that $\hat\bfA = (A/\sigma_\sol,\bA/\sigma_\sol^2)$ satisfies $\triple{\hat\bfA}_{\alpha\ax} = \triple{\bfA}_{\alpha\ax}/\sigma_\sol \leq 1$ by our choice of \emph{homogenous} norm $\triple{\Cdot}_{\alpha\ax}$ in \eqref{eq:alpha_ax_norm}.
           
           The local Lipschitz estimate in \ref{pt:local_lipschitz} follows by repeating the same
           bounds for two smooth $\bfA,\bar\bfA$ with $\triple{\bfA}_{\alpha\ax}\vee\triple{\bar\bfA}_{\alpha\ax} \leq \sigma_\sol$. Indeed, every map used above is locally Lipschitz on
           sets where the model norm is bounded by $L$, and the construction of the
           model is locally Lipschitz in $\hat\bfA$. Since throughout the argument
           $\triple{\hat\bfA}_{\alpha\ax},\triple{\hat{\bar\bfA}}_{\alpha\ax}\leq 1$,
           the constants are uniform in the rough additive functions. Hence, for another smooth  $\bar\bfA$ with
           corresponding objects $\bar g,\bar B$, we obtain
           \[
           \triple{(g,L_g);(\bar g,L_{\bar g})}_{2\alpha;\Lambda'}
           +
           |B-\bar B|_{\beta;\Lambda'}
           \lesssim
           \triple{\bfA;\bar\bfA}_{\alpha\ax}\;.
           \]
           
           	We have now proved the theorem statement for smooth $A,\bar A$ with $\triple{\bfA}_{\alpha\ax}\vee\triple{\bar\bfA}_{\alpha\ax} \leq \sigma_\sol$.
			Setting $\sigma_\coul=\frac 1 2\sigma_\sol$,the general statement follows by continuity and density since any $\bfA\in\Omega^1_{\alpha\ax}$ with $\triple{\bfA}_{\alpha\ax}\leq \sigma_\coul$ can be approximated by smooth $\bfA_n$ with $\triple{\bfA_n}_{\alpha\ax}\leq \sigma_\sol$.
\end{proof}

    \subsection{Proof of rough Uhlenbeck compactness} \label{sec:proof_RUC}
    We have now gathered the tools to prove a refined bounded-set version of Theorem~\ref{thm:Rough_Uhlenbeck}: The existence of a Coulomb gauge for a \enquote{small} connection form~(Theorem~\ref{thm:Coulomb_gauge}) and the patching statement~(Theorem~\ref{thm:patching}).
To that end, we first define for any $R\geq 0$,
\[
\bfOmega^1_{\alpha\ax;\leq R}\label{symb:RAF_bounded_ball}:=\{\bfA\in\bfOmega^1_{\alpha\ax}\colon
\triple{\bfA}_{\alpha\ax}\leq R\}.
\]
With this notation, we state the bounded-set version as follows:

\begin{theorem}[Rough Uhlenbeck compactness on bounded sets]\label{thm:RUC_bounded_sets}
Consider $\alpha\in (\frac {4}{9},\frac 12)$, $\beta\in (0,6\alpha-2)$, $\gamma \in (\frac{2}{3},6\alpha-2)$.
There exists a constant $C>0$ depending only on $\alpha,\beta,\gamma$ and $G$
with the following property: for any radius $R\geq 1$, there exists a
continuous map
\[
\bsg_R\label{symb:gauge_selection_map}:\bfOmega_{\alpha\ax;\leq R}^1\to C(\Lambda;G)
\]
such that, for every $\bfA=(A,\bA)\in\bfOmega_{\alpha\ax;\leq R}^1$,
one has $\bsg_R(\bfA)\in\fG^{\gamma}_A$,
$A^{\bsg_R(\bfA)}\in \Omega^1_\beta$, and
\begin{align*}
|A^{\bsg_R(\bfA)}|_{\beta}
\leq
CR^{\frac {2+\beta}{2\alpha}-1}\triple{\bfA}_{\alpha\ax}\;,
\qquad
\triple{(\bsg_R(\bfA),L_{\bsg_R(\bfA)});(1_G,L_{1_G})}_{\gamma}
\leq
CR^{\frac {1}{\alpha}-1}\triple{\bfA}_{\alpha\ax}\;.
\end{align*}
Moreover, for all
$\bfA,\bar\bfA\in\bfOmega_{\alpha\ax;\leq R}^1$,
we have
\begin{equation}\label{eq:cont_in_A}
\begin{aligned}
|A^{\bsg_R(\bfA)}-\bar A^{\bsg_R(\bar \bfA)}|_\beta
&\leq
CR^{\frac {2+\beta}{2\alpha}-1}\triple{\bfA;\bar\bfA}_{\alpha\ax}
\;,
\\
\triple{(\bsg_R(\bfA),L_{\bsg_R(\bfA)});(\bsg_R(\bar\bfA),L_{\bsg_R(\bar\bfA)})}_{\gamma}
&\leq
CR^{\frac {1}{\alpha}-1}\triple{\bfA;\bar\bfA}_{\alpha\ax}\;.
\end{aligned}
\end{equation}
\end{theorem}

To prove this result, we  fix $R\geq 1$ and
$\bfA\in\bfOmega_{\alpha\ax;\leq R}^1$. In order to obtain the smallness parameter~$\sigma$ that renders~$\bfA$ \enquote{small}, we need to scale and translate it;
    this is equivalent to scaling and translating the domain~$\Lambda$, as we will see. Let  us introduce for $z\in \R^2$ and $\sigma>0$, the translated and dilated square
    \begin{align}\label{eq:scaled_Lambda}
    \Lambda_z^\sigma:=\{z\}+[0,\sigma]^2. 
    \end{align}
    Recall that we can define for $\alpha\in (\frac 1 3,\frac 1 2)$ the space $\bfOmega^1_{\alpha}(\Lambda^\sigma_z)$ as indicated in Remark~\ref{rem:RAF_arbitrary_axial_domain}. Given $\bfA=(A,\bA)\in\bfOmega_{\alpha\ax}^1(\Lambda^\sigma_z)$, we define $\sfs^\sigma_z\label{symb:sfs}\bfA:=(\sfs^\sigma_z A,\sfs^\sigma_z\bA)\in\bfOmega_{\alpha\ax}^1$ as follows: For any line $\ell=(x,v)\in \cX$, we set
\[
    \sfs^\sigma_z A(\ell):=A(\sigma x+z,\sigma v), \qquad \sfs^\sigma_z \bA(\ell):=\bA(\sigma x+z,\sigma v).
    \]
    Note that the line segments $(\sigma x+z,\sigma v)$ are in $\Lambda^\sigma_z$. 
    This definition also applies to $B\in \Omega^1_\beta(\Lambda_z^\sigma)$ with $\beta\in (0,1]$: 
    For $\ell=(x,v)\in\cX$, we define 
    \[
    \sfs^\sigma_zB(\ell):=B(\sigma x+z,\sigma v). 
    \]    
	%
    We need to compare norms of scaled to those of non-scaled (rough) additive functions. 
    This is done in the lemma to follow. 
 \begin{lemma}\label{lem:scaling_raf_af_combined}
Let  $z\in\R^2$, $\alpha\in (\frac 13,\frac 12)$ and $\beta\in (0,1)$. Then we have the following two inequalities:
\begin{enumerate}[label=(\roman*)]
  \item \label{pt:A_ax} If $\sigma\in (0, 1]$ and $\bfA,\bar\bfA\in\bfOmega^1_{\alpha\ax}(\Lambda_z^\sigma)$, then $\sfs_z^\sigma\bfA,\sfs_z^\sigma\bar\bfA\in\bfOmega^1_{\alpha\ax}(\Lambda)$ and
  \[
    \triple{\sfs_z^\sigma\bfA,\sfs_z^\sigma\bar\bfA}_{\alpha\ax;\Lambda}
    \leq
    \sigma^\alpha
    \triple{\bfA,\bar\bfA}_{\alpha\ax;\Lambda_z^\sigma}.
  \]
  \item \label{pt:B_beta} If $\sigma\in [1,\infty)$ and $B\in\Omega^1_{\beta}(\Lambda_z^\sigma)$, then $\sfs_z^\sigma B\in\Omega^1_{\beta}(\Lambda)$ and
  \[
    \lvert\sfs_z^\sigma B\rvert_{\beta;\Lambda}
    \lesssim
    \sigma^{1+\beta/2}
    \lvert B\rvert_{\beta;\Lambda_z^\sigma}.
  \]
\end{enumerate}
\end{lemma}

\begin{proof}
The statement in \ref{pt:A_ax} follows easily from the definition of the norms \eqref{eq:RAF_gr_metric} and \eqref{eq:RAF_hor_norms} (remark that \eqref{eq:gr_norm} and \eqref{eq:RAF_gr_metric} are homogeneous and scale exactly as $\sigma^\alpha$, while the two components of\eqref{eq:RAF_hor_norms} scale as $\sigma^{2\alpha}$ and $\sigma^{3\alpha/2}$ respectively, providing even better bounds for these terms)
and from the fact that line segments shrink when applying $\sfs^\sigma_z$ so the length stays at most $\frac 14$.

Let us now focus on \ref{pt:B_beta}, where we have $\sigma\geq 1$ and $B\in\Omega^1_\beta$. We first prove the growth bound. Let $\ell$ be a line segment in $\Lambda$ with $|\ell|\leq \frac 14$. We split $\ell$ into $m$ line segments $\ell^1,\ldots,\ell^m$ of equal length, with $m\leq \lceil\sigma\rceil$, so that each scaled line segment has length at most $\frac 14$. By additivity,
\[
    \sfs_z^\sigma B(\ell)
    =
    \sum_{i=1}^m \sfs_z^\sigma B(\ell^i).
\]
Therefore,
\[
    |\sfs_z^\sigma B(\ell)|
    \leq
    \sum_{i=1}^m |\sfs_z^\sigma B(\ell^i)|  \leq
    \sum_{i=1}^m \sigma^\beta |\ell^i|^\beta |B|_{\beta;\Lambda_z^\sigma} \leq
    \sigma^\beta m^{1-\beta} |\ell|^\beta |B|_{\beta;\Lambda_z^\sigma}
\]
where we used Jensen's inequality in the final bound.
Since $m\leq \lceil\sigma\rceil\leq 2\sigma$, this gives
\[
    |\sfs_z^\sigma B(\ell)|
    \lesssim
    \sigma |\ell|^\beta |B|_{\beta;\Lambda_z^\sigma}.
\]

It remains to prove the triangle seminorm bound. For that, we take a triangle $P$ with diameter at most $\frac 1 4$. 
 Write $|P|=\frac 1 2 ah$ where $a$ is the length of the longest side and $h$ be the corresponding height perpendicularly relative to the longest side.  We know that $0<h\leq a\leq \frac 1 4$. We treat the cases $\sigma h\geq 1$  and $\sigma h\leq 1$ separately. For the case $\sigma h\geq 1$,  we use the growth estimate and get
\[
|\sfs^\sigma_z B(\partial P)| \lesssim \sigma a^\beta |B|_{\beta;\Lambda^\sigma_z} \leq \sigma^{1+\beta/2}  (\sigma h)^{-\beta/2}  (ah)^{\beta/2}  |B|_{\beta;\Lambda^\sigma_z}\lesssim \sigma^{1+\beta/2} |P|^{\beta/2}|B|_{\beta;\Lambda^\sigma_z}. 
\]
For the other  case $\sigma h\leq 1$, we can use a triangulation of $P$ with triangles $(P^i)_{i=1}^n$ for some $n\lesssim 1+\sigma ah$ such that for each $P^i$ the corresponding scaled triangle has diameter at most $\frac 1 4$. Using additivity and cancellations of internal boundaries, we get 
\[
\sfs^\sigma_z B(\partial P)=\sum_{i=1}^n \sfs^\sigma_z B(\partial P^i).
\]
This then yields
\[
    |\sfs_z^\sigma B(\partial P)|
    \leq
    \sum_{i=1}^n |\sfs_z^\sigma B(\partial P^i)|  \leq \sigma^\beta |B|_{\beta;\Lambda_z^\sigma} 
    \sum_{i=1}^n  |P^i|^{\beta/2} \leq \sigma^\beta n^{1-\beta/2} |B|_{\beta;\Lambda_z^\sigma} 
    \Big(\sum_{i=1}^n  |P^i|\Big)^{\beta/2},
\]
where we have used Jensen's inequality to obtain the last inequality. Now we continue with the fact that $n\lesssim 1+\sigma ah\lesssim \sigma $ to obtain 
\[
|\sfs_z^\sigma B(\partial P)|
    \lesssim \sigma^\beta \sigma^{1-\beta /2} |B|_{\beta;\Lambda_z^\sigma}|P|^{\beta/2} \lesssim \sigma^{1+\beta /2} |B|_{\beta;\Lambda_z^\sigma}|P|^{\beta/2},
\]
finishing the proof. 
\end{proof}

%
%
%
%
We are now ready to prove Theorem~\ref{thm:RUC_bounded_sets}. 
 \begin{proof}[of Theorem~\ref{thm:RUC_bounded_sets}]\label{proof:RUC_bounded_sets}
 Without loss of generality, we assume that $\gamma\geq \beta$.
 Fix $R>0$ and let $\bfA\in\bfOmega_{\alpha\ax;\leq R}^1(\Lambda)$. If $\bfA={\bf0}$, then we simply set $g_R(\bf0)=1_G$ and we are done. So let us assume $\bfA$ to be non-zero. 
 We can extend $\bfA$ on $[-1,2]^2$ continuously by~Lemma~\ref{lem:domain_extension_RAF} which we denote by the same letter $\bfA$. 
 
 We now introduce the setting that will allow us to apply Theorem~\ref{thm:Coulomb_gauge}.
 We set 
 \begin{align}\label{eq:choice_sigma_RUC}
\sigma=\sigma_R
:=
\min\left\{
\left(\frac{\sigma_{\coul}}{R}\right)^{1/\alpha},
\left(\frac{\delta_\patch}{2C_\coul R}\right)^{1/\alpha},
1
\right\},
\end{align}
where $C_\coul$ is the implicit constant in \eqref{e:Coulomb_gauge:bound} of Theorem \ref{thm:Coulomb_gauge} for the domain $\Lambda':=[ \frac 18,\frac7 8]^2  \Subset \Lambda^\circ$
and $\delta_\patch$ is the constant in Theorem~\ref{thm:patching}
(with $\beta$ in both theorems taken as $\gamma$ here).
Let 
\[
z_\bfn = -\left(\tfrac{\sigma}{4}, \tfrac{\sigma}{4}\right) + \tfrac{\sigma}{2}(n^1,n^2)
\]
for $\bfn=(n^1,n^2) \in\{ 0, 1, \ldots, l\}^2=\bN^2_L$ where $l=\lceil \frac 3\sigma -\frac 32\rceil$ and $L=(l,l)$ are as in Definition~\ref{def:rec_cover}. 
We recall from~\eqref{eq:scaled_Lambda}
\[
\Lambda_{z_{\bfn}}^\sigma=z_\bfn+[0,\sigma]^2=z_\bfn+(\tfrac \sigma 2,\tfrac\sigma 2)+[-\tfrac \sigma 2,\tfrac \sigma 2]^2.
\]
One can verify that 
\[
\bigcup_{\bfn\in\bN^2_L} \Lambda_{z_\bfn}^\sigma=[-\tfrac\sigma 4,\sigma -\tfrac \sigma 4+\tfrac\sigma 2 l]^2,
\]
which, in turn, implies that
\[
\Lambda
\subset [0,\tfrac 3 2]^2
\subset \bigcup_{\bfn\in\bN^2_L} \Lambda_{z_\bfn}^\sigma\subset [-\tfrac 14,2]^2\subset [-1,2]^2.
\]
See Figure~\ref{fig:RUC_square_cover} for a visual representation of this covering and the setup we have just described.
    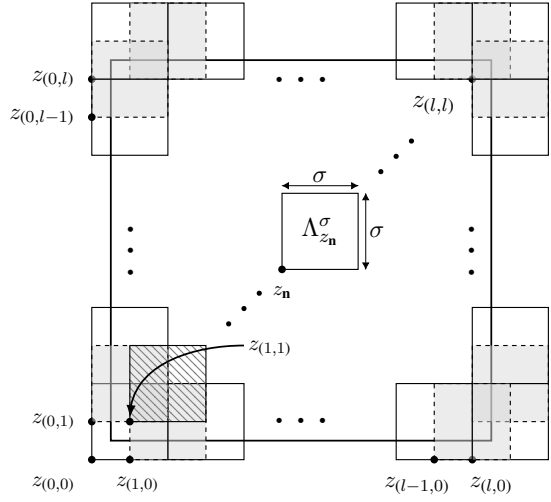
\begin{figure}[h]
 	\centering
 	\tikzset{>={Latex[width=1.5pt,length=1.5pt]}}
 	\resizebox{0.5\linewidth}{!}{%
 		\begin{tikzpicture}[x=3cm,y=3cm,line cap=butt,line join=miter]
 			\def\s{0.2}                   
 			\pgfmathsetmacro{\h}{0.5*\s}  
 			
 			\pgfmathtruncatemacro{\Kmax}{ceil((1+0.75*\s)/\h)} 
 			\pgfmathtruncatemacro{\Lmax}{ceil((1+0.75*\s)/\h)} 
 			
 			\begin{scope}
 				\draw[black] (0,0) rectangle (1,1);               
 				
 				\draw[line width=0.25pt] (-\s/4,-\s/4) rectangle ++(\s,\s);
 				\node[tinydot] at (-\s/4,-\s/4){};
 				\node[below left=0.01pt] at (-\s/4,-\s/4){\resizebox{0.023\textwidth}{!}{$z_{(0,0)}$}};
 				\draw[line width=0.25pt, dash pattern=on 1pt off 1pt, fill=gray!30, fill opacity=0.5] (-\s/4+\s/2,-\s/4) rectangle ++(\s,\s);
 				\node[tinydot] at (-\s/4+\s/2,-\s/4){};
 				\node[below=0.01pt] at (-\s/4+1.2*\s/2,-\s/4){\resizebox{0.023\textwidth}{!}{$z_{(1,0)}$}};
 				\draw[line width=0.25pt, fill opacity=0.5] (-\s/4+\s,-\s/4) rectangle ++(\s,\s);
 				\draw[line width=0.25pt, fill opacity=0.5] (-\s/4,-\s/4+\s) rectangle ++(\s,\s);
 				\node[tinydot] at (-\s/4,-\s/4+\s/2){};
 				\node[left=0.01pt] at (-\s/4,-\s/4+\s/2){\resizebox{0.023\textwidth}{!}{$z_{(0,1)}$}};
 				\draw[line width=0.25pt, dash pattern=on 1pt off 1pt, fill=gray!30, fill opacity=0.5] (-\s/4,-\s/4+\s/2) rectangle ++(\s,\s);
 				
 				\draw[line width=0.25pt,pattern=north west lines, pattern color=gray] (-\s/4+\s/2,-\s/4+\s/2) rectangle ++(\s,\s);
 				\node[tinydot] at (-\s/4+\s/2,-\s/4+\s/2){};
 				\draw [-{Latex[length=1mm]}] (-\s/4+4*\s/2,-\s/4+3*\s/2) to [out=180,in=90] (-\s/4+\s/2,-\s/4+1.1*\s/2);
 				\node[above right=-4pt] at (-\s/4+4*\s/2,-\s/4+2.6*\s/2){\resizebox{0.023\textwidth}{!}{$z_{(1,1)}$}};
 				%
 				
 				\node[tinydot] at (-\s/4+5*\s/2,-\s/4+5*\s/2){};
 				\node[below=0.01pt] at (-\s/4+5*\s/2,-\s/4+5*\s/2){\resizebox{0.01\textwidth}{!}{$z_\bfn$}};
 				\node at (-\s/4+6*\s/2,-\s/4+6*\s/2){\resizebox{0.02\textwidth}{!}{$\Lambda_{z_{\bfn}}^\sigma$}};
 				\draw[line width=0.25pt] (-\s/4+5*\s/2,-\s/4+5*\s/2) rectangle ++(\s,\s);
 				\draw[line width=0.1pt, <->] (-\s/4+5*\s/2,-\s/4+7.2*\s/2) -- (-\s/4+7*\s/2,-\s/4+7.2*\s/2);
 				\node[above=-3pt] at (-\s/4+6*\s/2,-\s/4+7.2*\s/2){\resizebox{0.008\textwidth}{!}{$\sigma$}};
 				\draw[line width=0.1pt, <->] (-\s/4+7.2*\s/2,-\s/4+5*\s/2) -- (-\s/4+7.2*\s/2,-\s/4+7*\s/2);
 				\node[right=-3pt] at (-\s/4+7.2*\s/2,-\s/4+6*\s/2){\resizebox{0.008\textwidth}{!}{$\sigma$}};
 				%
 				
 				\draw[line width=0.25pt] (-\s/4,-\s/4+10*\s/2) rectangle ++(\s,\s);
 				\node[tinydot] at (-\s/4,-\s/4+10*\s/2){};
 				\node[left=0.01pt] at ((-\s/4,-\s/4+10*\s/2){\resizebox{0.023\textwidth}{!}{$z_{(0,l)}$}};
 				
 				\draw[line width=0.25pt, dash pattern=on 1pt off 1pt, fill=gray!30, fill opacity=0.5] (-\s/4+\s/2,-\s/4+10*\s/2) rectangle ++(\s,\s);
 				
 				\draw[line width=0.25pt] (-\s/4+\s,-\s/4+10*\s/2) rectangle ++(\s,\s);

 				\draw[line width=0.25pt, dash pattern=on 1pt off 1pt, fill=gray!30, fill opacity=0.5] (-\s/4,-\s/4+9*\s/2) rectangle ++(\s,\s);
 				\node[tinydot] at (-\s/4,-\s/4+9*\s/2){};
 				\node[left=0.01pt] at ((-\s/4,-\s/4+9*\s/2){\resizebox{0.035\textwidth}{!}{$z_{(0,l-1)}$}};
 				
 				\draw[line width=0.25pt] (-\s/4,-\s/4+8*\s/2) rectangle ++(\s,\s);
 				
 				\draw[line width=0.25pt] (-\s/4+10*\s/2,-\s/4) rectangle ++(\s,\s);
 				\node[tinydot] at (-\s/4+10*\s/2,-\s/4){};
 				\node[below=0.01pt] at ((-\s/4+1.05*10*\s/2,-\s/4){\resizebox{0.023\textwidth}{!}{$z_{(l,0)}$}};
 				
 				\draw[line width=0.25pt, dash pattern=on 1pt off 1pt, fill=gray!30, fill opacity=0.5] (-\s/4+10*\s/2,-\s/4+\s/2) rectangle ++(\s,\s);
 				%
 				
 				\draw[line width=0.25pt] (-\s/4+10*\s/2,-\s/4+\s) rectangle ++(\s,\s);

 				\draw[line width=0.25pt, dash pattern=on 1pt off 1pt, fill=gray!30, fill opacity=0.5] (-\s/4+9*\s/2,-\s/4) rectangle ++(\s,\s);
 				\node[tinydot] at (-\s/4+9*\s/2,-\s/4){};
 				\node[below=0.01pt] at ((-\s/4+0.95*9*\s/2,-\s/4){\resizebox{0.035\textwidth}{!}{$z_{(l-1,0)}$}};
 				
 				\draw[line width=0.25pt] (-\s/4+8*\s/2,-\s/4) rectangle ++(\s,\s);
 				
 				\draw[line width=0.25pt] (-\s/4+10*\s/2,-\s/4+10*\s/2) rectangle ++(\s,\s);
 				\node[tinydot] at (-\s/4+10*\s/2,-\s/4+10*\s/2){};
 				\node[below left=0.01pt] at ((-\s/4+10*\s/2,-\s/4+10*\s/2){\resizebox{0.023\textwidth}{!}{$z_{(l,l)}$}};
 				
 				\draw[line width=0.25pt, dash pattern=on 1pt off 1pt, fill=gray!30, fill opacity=0.5] (-\s/4+9*\s/2,-\s/4+10*\s/2) rectangle ++(\s,\s);
 				
 				\draw[line width=0.25pt] (-\s/4+8*\s/2,-\s/4+10*\s/2) rectangle ++(\s,\s);
 				
 				\draw[line width=0.25pt, dash pattern=on 1pt off 1pt, fill=gray!30, fill opacity=0.5] (-\s/4+10*\s/2,-\s/4+9*\s/2) rectangle ++(\s,\s);
 				
 				\draw[line width=0.25pt] (-\s/4+10*\s/2,-\s/4+8*\s/2) rectangle ++(\s,\s);

 				\node[above] at (0.5,0){$\ldots$};
 				
 				\node[below] at (0.5,1){$\ldots$};
 				
 				\node[right] at (0,0.5){\rotatebox[origin=c]{90}{$\ldots$}};
 				
 				\node[left] at (1,0.5){\rotatebox[origin=c]{90}{$\ldots$}};
 				
 				\node[below] at (0.35,0.45){\rotatebox[origin=c]{45}{$\ldots$}};
 				
 				\node[below] at (0.75,0.85){\rotatebox[origin=c]{45}{$\ldots$}};
 				
 						%
 			\end{scope}
 		\end{tikzpicture}%
 	}
  	\caption{The large square~$\Lambda$ gets covered in squares~$\Lambda_{z_\bfn}^\sigma$ of type~$[0,\sigma]^2$, anchored in~$z_\bfn$ for~$\bfn \in \N_L^2$.}
 	\label{fig:RUC_square_cover}
 \end{figure}
 Let $z=z_\bfn$ for some $\bfn\in\bN^2_L$.
 On each little square~$\Lambda_{z_\bfn}^\sigma$ we want to re-scale the rough additive function~$\bfA$ in such a way that we work on a square of size $\Lambda$ again. 
 To this end, we define 
 $\bfA_{z,\sigma}:=\sfs_z^\sigma(\bfA|_{\Lambda_z^\sigma})$ on $\Lambda$
and apply \lem{lem:scaling_raf_af_combined} to get 
 \[
 \triple{\bfA_{z,\sigma}}_{\alpha\ax}
 \leq \sigma^{\alpha} \triple{\bfA}_{\alpha\ax}
 \leq \sigma^{\alpha} R
 \leq \sigma_\coul,
 \]
 where the last  inequality holds by our choice of $\sigma$ in
 \eqref{eq:choice_sigma_RUC}. 
As a result, the assumptions of Theorem~\ref{thm:Coulomb_gauge} are satisfied and we obtain a connection $B_{z,\sigma}$ in the Coulomb gauge together with a gauge transformation~$g_{z,\sigma}$ that  satisfies  the following bounds on $\Lambda'=[ \frac 18,\frac7 8]^2  \Subset \Lambda^\circ$: 
\begin{equation} \label{e:bound_coulomb}
	|B_{z,\sigma}|_{\gamma;\Lambda'}+|g_{z,\sigma}-1_G|_{\infty;\Lambda'}\leq C_\coul\sigma^\alpha\triple{\bfA}_{\alpha\ax}\leq C_\coul \sigma^\alpha R. 
\end{equation}

We now want to apply the patching result, Theorem~\ref{thm:patching}.
However, recall that $z \in \cbr[0]{z_\bfn}_{\bfn \in \N_L^2}$ and, by the previous procedure, we obtain a collection of $1$-forms $B_{z_\bfn,\sigma}$ in the Coulomb gauge and gauge transformations $g_{z_\bfn,\sigma}$ which are defined \emph{on the same square}~$\Lambda'$. 
Therefore, we need to suitably translate~$B_{\bfn}$ as in~\eqref{e:translation_B_bfn} to be defined on rectangles~$U_\bfn$ of the same size as $\Lambda'$ which then together cover a large square. 
In the setting of Definition~\ref{def:rec_cover}, we set $\sce:=\{(1,0),(0,1)\}$, $\mu=(\mu^1,\mu^2)=(\frac 34,\frac 34)$, and $y = (-\tfrac 18,-\tfrac 18)$ and consider the set~\footnote{This is consistent with Definition~\ref{def:rec_cover} since~$2/3 \cdot \mu^i = 2/3 \cdot 3/4 = 1/2.$} 
    \[
    U_\bfn = U_\bfn^{\sce,\mu,y}=
    -(\tfrac 18,\tfrac 18)+\tfrac 12 \bfn+[0,\tfrac 34]^2 \,.
    \]
   As above, we take $L=(l,l)$ for  $l=\lceil \frac 3\sigma-\frac 32\rceil$ to obtain the \enquote{large} rectangle
   \[
   \cU=\bigcup_{\bfn\in\bN^2_L}U_{\bfn}^{\sce,\mu,y} \,. 
   \]
As mentioned above, $B_{z_\bfn,\sigma}$ is defined on $\Lambda'=[\frac 18,\frac78]^2$ but its translated version
    \begin{equation} \label{e:translation_B_bfn}
    	B_\bfn := \sfs^1_{-\sigma^{-1}z_\bfn}B_{z_\bfn,\sigma} \,.
    \end{equation}
is defined on~$U_\bfn$.

	We now need to check the that the conditions of Theorem~\ref{thm:patching} are satisfied.
	We begin by building gauge transformations~$(g_{\bfm,\bfn})_{\bfn,\bfm \in \N^2_L}$ from~$(g_{z_\bfn,\sigma})_{\bfn \in \N^2_L}$ such that~$(B_\bfn,g_{\bfn,\bfm})_{\bfn,\bfm \in \N^2_L}$ is a compatible gauge datum in the sense of Definition~\ref{def:compatible_gauge_data}.
    To this end, define
\[
h_\bfn(x):=g_{z_\bfn,\sigma}(x-\sigma^{-1}z_\bfn), \qquad x\in U_\bfn,
\]
and $g_{\bfn,\bfm}:U_\bfm\cap U_\bfm\to G$ by
\[
g_{\bfn,\bfm}:=h_\bfm h_\bfn^{-1}.
\]
From this definition the cocycle condition from~Definition~\ref{def:compatible_gauge_data} follows immediately. The only thing remaining is the gauge equivalence condition $B_\bfn^{g_{\bfn,\bfm}}=B_{\bfm}$ on $U_\bfn\cap U_\bfm$. %
We first claim that
\begin{align}\label{eq:B_n_identity_RUC_proof}
B_\bfn=(\sfs^\sigma_0\bfA)^{h_\bfn}
\qquad\text{on }U_\bfn.
\end{align}
Indeed, recalling that
\[
B_\bfn=\sfs^1_{-\sigma^{-1}z_\bfn}B_{z_\bfn,\sigma},
\qquad
B_{z_\bfn,\sigma}=\bfA_{z_\bfn,\sigma}^{\,g_{z_\bfn,\sigma}},
\]
we obtain by compatibility of translations with gauge transformations that
\[
B_\bfn
=\sfs^1_{-\sigma^{-1}z_\bfn}\bigl(\bfA_{z_\bfn,\sigma}^{\,g_{z_\bfn,\sigma}}\bigr)
=\bigl(\sfs^1_{-\sigma^{-1}z_\bfn}\bfA_{z_\bfn,\sigma}\bigr)^{h_\bfn}.
\]
It therefore remains to identify the translated background. For any horizontal line segment
$\ell=(x,v)\in\cX_{\rm h}$, we have
\begin{align*}
\bigl(\sfs^1_{-\sigma^{-1}z_\bfn}\bfA_{z_\bfn,\sigma}\bigr)(\ell)
&=\bfA_{z_\bfn,\sigma}(x-\sigma^{-1}z_\bfn,v) \\
&=\bfA\bigl(\sigma(x-\sigma^{-1}z_\bfn)+z_\bfn,\sigma v\bigr) \\
&=\bfA(\sigma x,\sigma v)
=(\sfs^\sigma_0\bfA)(\ell).
\end{align*}
Hence
\[
\sfs^1_{-\sigma^{-1}z_\bfn}\bfA_{z_\bfn,\sigma}=\sfs^\sigma_0\bfA,
\]
proving~\eqref{eq:B_n_identity_RUC_proof}.
Therefore,
\[
B_\bfn^{g_{\bfn,\bfm}}
=\bigl((\sfs^\sigma_0\bfA)^{h_\bfn}\bigr)^{h_\bfm h_\bfn^{-1}}
=(\sfs^\sigma_0\bfA)^{h_\bfm}
=B_\bfm.
\]
This implies that $g_{\bfn,\bfm} \in C^{\gamma}$ and also finishes the verification that $(B_\bfn,g_{\bfn,\bfm})_{\bfn,\bfm\in\N_L^2}$ is a compatible gauge datum.

    The second requirement for Theorem~\ref{thm:patching} to apply is~\eqref{eq:patch_req}. 
    To verify it, we use the bound in~\eqref{e:bound_coulomb} to see that
    \begin{equation*}
    |g_{\bfm,\bfn}(x)-1_G|_{\infty;U_{\bfn}\cap U_\bfm}
    \leq 
    |g_{z_\bfn,\sigma}-1_G|_{\infty;\Lambda'}+|g_{z_\bfm,\sigma}-1_G|_{\infty;\Lambda'}
    \leq 
    2C_\coul \sigma^\alpha R
    \end{equation*}
   and
    \[
    |B_\bfn|_{\gamma;U_\bfn}
    \leq |B_{z_\bfn,\sigma}|_{\gamma;\Lambda'}
    \leq C_\coul \sigma^\alpha R \,.
    \]
	The RHS on the previous inequalities are less than $\delta_\patch$ because
    \[
    2C_\coul \sigma^{\alpha} R\leq \delta_\patch \iff \sigma\leq  \left(\frac{\delta_\patch}{2C_\coul R}\right)^{1/\alpha}
    \]
    and the latter holds because we chose $\sigma$ as in \eqref{eq:choice_sigma_RUC}. 

   	By Theorem~\ref{thm:patching}, which applies because $\gamma>\frac23$ by assumption, we now obtain $\hat B$ defined on the \enquote{large} square~$\cU_L$ and it satisfies 
    \[
    |B|_{\beta;\cU_L}\leq |\hat B|_{\gamma;\cU_L}\lesssim \sigma^\alpha \triple{\bfA}_{\alpha\ax},
    \]
    where we have used that $\beta\leq \gamma$. 
    So now we need to scale it down again and then restrict to the domain~$\Lambda$;
   	accordingly, we set~$B := (\sfs^{\sigma^{-1}}_0 \hat B)\sVert[0]_\Lambda$. 
   	By \lem{lem:scaling_raf_af_combined}, we have
    \[
    |B|_{\beta;\Lambda}\lesssim \sigma^{-1-\beta/2} \sigma^\alpha \triple{\bfA}_{\alpha\ax} = \sigma^{-(\frac{2+\beta} 2 -\alpha)} \triple{\bfA}_{\alpha\ax} \,. 
    \]
    Furthermore, it is not difficult to see that by the choice of $\sigma$ in~\eqref{eq:choice_sigma_RUC}, we have 
    \[
    \sigma^{-1}\lesssim R^{1/\alpha},
    \]
    which then implies
    \[
     |B|_{\beta;\Lambda}
     \lesssim
     R^{\frac {2+\beta}{2\alpha}-1}\triple{\bfA}_{\alpha\ax} \,.
    \]
    
Finally, we verify that $B$ is indeed obtained by gauge transforming $\bfA$. 
By Theorem~\ref{thm:patching},
there are gauge transformations $q_\bfn\in C^{\gamma}(U_\bfn;G)$ such that
\[
    \hat B=B_\bfn^{q_\bfn}
    \qquad\text{on }U_\bfn,
\]
and
\[
    q_\bfn=q_\bfm g_{\bfn,\bfm}
    \qquad\text{on }U_\bfn\cap U_\bfm .
\]
We combine this with~\eqref{eq:B_n_identity_RUC_proof} to obtain
\begin{align}\label{eq:local_gauge_transformed_RUC_proof}
    \hat B
    =B_\bfn^{q_\bfn}= ((\sfs^\sigma_0\bfA)^{h_\bfn})^ {q_\bfn}=
    (\sfs^\sigma_0\bfA)^{q_\bfn h_\bfn}
    \qquad\text{on }U_\bfn ,
\end{align}
where the last equality follows by~Remark~\ref{rem:left_action}.

The functions $q_\bfn h_\bfn$ agree on overlaps. Indeed, using
$g_{\bfn,\bfm}=h_\bfm h_\bfn^{-1}$, we have
\[
    q_\bfn h_\bfn
    =
    q_\bfm g_{\bfn,\bfm}h_\bfn
    =
    q_\bfm h_\bfm  \qquad \text{ on } U_\bfn\cap U_\bfm.
\]
In particular $q_\bfn h_\bfn$ constitutes a function $k\in C^\alpha(\cU;G)$ with the defining property $k=q_\bfn h_\bfn$ on $U_\bfn$. We now define the gauge transformation on the original square by
\[
    \bsg_R(\bfA)(x):=k(\sigma^{-1}x),
    \qquad x\in\Lambda .
\]
We verify that $(\bsg_R(\bfA),L_{\bsg_R(\bfA)})$ is a controlled gauge transformation.

First, note that, since $h_\bfn$ is controlled by $\sfs^\sigma_0\bfA$, we have for any $\ell^{x;y}\in U_\bfn$ 
\[
h_\bfn(y)-h_\bfn(x)=h_\bfn(x)(\sfs^\sigma_0 A)(\ell^{x;y})+R^{h_\bfn}(x,y).
\]
By the estimates in Theorem~\ref{thm:Coulomb_gauge} together with Lemma~\ref{lem:scaling_raf_af_combined}, we know that 
\[
\triple{(h_\bfn,L_{h_\bfn});(1_G,L_{1_G})}_{2\alpha;U_\bfn}\lesssim \sigma^\alpha\triple{\bfA}_{\alpha\ax}.
\]
Furthermore, from the estimates in Theorem~\ref{thm:patching}, we have
\[
\max_{n\in\bN^2_L} |q_\bfn-1_G|_{C^{\gamma}(U_\bfn)}\lesssim \sigma^\alpha \triple{\bfA}_{\alpha\ax}.
\]
Since the $L^\infty$-norm does not matter for the scaling, we find 
\[
|\bsg_R(\bfA)-1_G|_{\infty;\Lambda}=|k-1_G|_{\infty;\sigma^{-1}\Lambda}\lesssim \sigma^\alpha\triple{\bfA}_{\alpha\ax}. 
\]
To compute the H\"older seminorm, we distinguish cases. If $x,y\in\Lambda$ satisfy $|x-y|>\frac 1 4\sigma$, we use the supremum norm to bound 
\[
|\bsg_R(\bfA)(y)-\bsg_R(\bfA)(x)|\lesssim \sigma^\alpha \triple{\bfA}_{\alpha\ax}\lesssim |x-y|^\alpha \triple{\bfA}_{\alpha\ax}. 
\]
If $|x-y|\leq \frac 1 4\sigma$, 
then this implies $\sigma^{-1}x$ and $\sigma^{-1}y$ are in one $U_\bfn$, so that  we can use the H\"older seminorms of $q_\bfn$ and $h_\bfn$ to get
\[
|\bsg_R(\bfA)(y)-\bsg_R(\bfA)(x)|\lesssim \sigma^\alpha \triple{\bfA}_{\alpha\ax}|\sigma^{-1}(x-y)|^\alpha\lesssim \sigma^{\alpha-1} \triple{\bfA}_{\alpha\ax}|x-y|^\alpha. 
\]
We have now obtained 
\[
|\bsg_R(\bfA)-1_G|_{C^{\gamma/2}(\Lambda)}
\lesssim
|\bsg_R(\bfA)-1_G|_{C^{\alpha}(\Lambda)}
\lesssim \sigma^{\alpha-1}\triple{\bfA}_{\alpha\ax}. 
\]

For the remainder, we let $x,y\in\Lambda$, and define 
\[
 R^{\bsg_R(\bfA)}(x,y)
    :=
    \bsg_R(\bfA)(y)-\bsg_R(\bfA)(x)
    -
    \bsg_R(\bfA)(x)A(\ell^{x;y}).
\]
Let us first consider the case $|x-y|>\frac1 4\sigma$. Then we know from the computation above that 
\[
|R^{\bsg_R(\bfA)}(x,y)|\lesssim  \triple{\bfA}_{\alpha\ax} |x-y|^\alpha\lesssim \sigma^{\alpha-\gamma}  \triple{\bfA}_{\alpha\ax} |x-y|^{\gamma}. 
\]
In the other case, when  $|x-y|\leq \frac 1 4\sigma$, we have that
the segment $\ell^{\sigma^{-1} x;\sigma^{-1} y}$ is contained in one of the rectangles
$U_\bfn$. Then, since $k=q_\bfn h_\bfn$ on $U_\bfn$, we have using the controlledness of $h_\bfn$ from Theorem~\ref{thm:Coulomb_gauge}
\begin{align*}
\bsg_R(\bfA)(y)-\bsg_R(\bfA)(x)
&=
q_\bfn(\sigma^{-1} y)h_\bfn(\sigma^{-1} y)
-
q_\bfn(\sigma^{-1} x)h_\bfn(\sigma^{-1} x)
\\
&=
q_\bfn(\sigma^{-1} x)
\bigl(h_\bfn(\sigma^{-1} y)-h_\bfn(\sigma^{-1} x)\bigr)
+
\bigl(q_\bfn(\sigma^{-1} y)-q_\bfn(\sigma^{-1} x)\bigr)
h_\bfn(\sigma^{-1} y)\\
&=q_\bfn(\sigma^{-1}x)h_\bfn(\sigma^{-1}x) (\sfs^\sigma_0 A)(\ell^{\sigma^{-1}x;\sigma^{-1}y})+q_\bfn(\sigma^{-1} x)R^{h_\bfn}(\sigma^{-1} x,\sigma^{-1} y)\\
&\qquad 
    + 
    \bigl(q_\bfn(\sigma^{-1} y)-q_\bfn(\sigma^{-1} x)\bigr)
    h_\bfn(\sigma^{-1} y).
\end{align*}
Since $(\sfs^\sigma_0 A)(\ell^{\sigma^{-1}x;\sigma^{-1}y})=A(\ell^{x;y})$, we find that 
\[
    R^{\bsg_R(\bfA)}(x,y)
    =
    q_\bfn(\sigma^{-1} x)R^{h_\bfn}(\sigma^{-1} x,\sigma^{-1} y)
    +
    \bigl(q_\bfn(\sigma^{-1} y)-q_\bfn(\sigma^{-1} x)\bigr)
    h_\bfn(\sigma^{-1} y).
\]
Therefore 
\begin{align*}
|R^{\bsg_R(\bfA)}(x,y)|&\lesssim \sigma^\alpha \triple{\bfA}_{\alpha\ax}|\sigma^{-1}(x-y)|^{2\alpha}+\sigma^\alpha\triple{\bfA}_{\alpha\ax}|\sigma^{-1}(x-y)|^{\gamma}\\
&\lesssim  \sigma^\alpha(\sigma^{-2\alpha}+\sigma^{-\gamma})\triple{\bfA}_{\alpha\ax}|x-y|^{\gamma}\\
&\lesssim \sigma^{-\alpha}\triple{\bfA}_{\alpha\ax}|x-y|^{\gamma}.  
\end{align*}
We conclude that $(\bsg_R(\bfA),L_{\bsg_R(\bfA)})\in\fG^{\gamma}_A$ and
\[
\triple{(\bsg_R(\bfA),L_{\bsg_R(\bfA)});(1_G,L_{1_G})}_{\gamma}\lesssim (\sigma^{-\alpha}\vee \sigma^{\alpha-1})\triple{\bfA}_{\alpha\ax}\lesssim R^{\frac 1\alpha-1}\triple{\bfA}_{\alpha\ax}\;.
\]
Moreover, by~\eqref{eq:local_gauge_transformed_RUC_proof}, we get
\[
A^{\bsg_R(\bfA)}
=
A^{ k(\sigma^{-1}\Cdot)}
=
\left(
\sfs^{\sigma^{-1}}_0
\bigl((\sfs^\sigma_0\bfA)^k\bigr)
\right)\sVert_\Lambda
=
\left(
\sfs^{\sigma^{-1}}_0 \hat B
\right)\sVert_\Lambda
=
B .
\]

Finally, to prove \eqref{eq:cont_in_A}, we follow the exact same steps and use the continuity estimates from Theorems \ref{thm:patching} and \ref{thm:Coulomb_gauge}.
    \end{proof}

\subsection{Gauge-fixed Yang--Mills measure}\label{sec:proof_RUC_YM}

In this section, we prove that one can gauge-fix the YM measure as claimed in Theorem~\ref{thm:RUC_YM}. The proof is given at the end of this section.

Consider a mollifier $\rho$, a white noise $\xi$ defined on a probability space $(\Sigma,\cF,\bP)$, and $\xi^\eps = \xi*\rho^\eps\label{symb:mollifier_white_noise}$ as in Theorem \ref{thm:RUC_YM}.
Since~$\rho$ and thus $\rho^\varepsilon$ are radial, we have that
\begin{equation} \label{e:radial_mollifier}
	\xi^\varepsilon(f) =  (\xi*\rho^\varepsilon)(f)=\xi(\rho^\varepsilon * f), \quad f\in C^\infty_c(\bR^2) \,.
\end{equation} 
Recall Definition~\ref{def:non_vert_lines} in which the space~$\cX_{\rmh}$ of non-vertical (or~\enquote{horizontal}) line segments has been introduced alongside the projection~$\rmP_{\cX_{\rmh\ax}}$.
For the construction of the random fields in the axial gauge~(see~Definition~\ref{def:enhanced_AF_axial}), it is useful to  define $\cX_\rmh^\to$ (resp $\cX_\rmh^{\ot}$)\label{symb:oriented_lines} as all lines in $\cX_\rmh$ whose orientation is pointing right (resp.~left). From now onwards, we mostly focus on line segments pointing right. 
For any line $\ell\in\cX_\rmh^\to$ that has parametrisation $\gamma:[a,b]\to \Lambda$ with $\gamma(t)=(t,m+ct)$, we define 
\label{symb:U_ell}\begin{equation} \label{e:U_ell}
	U_\ell:=\{(x_1,x_2)\in \bR^2\colon a\leq x_1\leq b, 0\leq x_2\leq m+cx_1\}\,,
\end{equation}
see Figure~\ref{subfig:U_ell} on page~\pageref{subfig:U_ell} for a visual representation.
Note that $U_\ell$ is the trapezium on the horizontal axis of $\Lambda$ with parallel vertical sides and the top side being $\ell$.
The area of $U_\ell$ is
\begin{equation*}
|U_\ell|=\Area(\ell,\rmP_{\cX_{\rmh\ax}}(\ell))=\int_a^b m+ct\dif t=(b-a)(m+\tfrac{1}{2}c(b+a)).
\end{equation*}
We want to define a random field in $\Caxial$. To that end, we start by defining 
\begin{equation} \label{e:def_A_cax}
	A(\ell) := \xi(\1_{U_\ell}), \quad \ell\in\cX_\rmh^\to \,.
\end{equation}
Note that reversing any $\ell\in \cX^{\ot}_\rmh$ corresponds to a unique line $\ell^\to\in \cX^{\to}_\rmh$ and we define $A(\ell):=-A(\ell^\to)$ for $\ell\in\cX_\rmh^{\ot}$. Finally, we set $A(\ell)=0$ for all $\ell\in\cX\setminus\cX_\rmh$.

%

We think of $A$ as an integrated $1$-form in the axial gauge for which the curvature $-\partial_2 A_1$ is a white noise.
This corresponds to Driver's construction of YM measure on $\Lambda$ \cite{Driver89}.\footnote{Even though Driver in \cite{Driver89} worked on $\bR^2$, the YM measure is analogous on $\Lambda$ as both manifolds are contractible.}
Normally, one would need to consider the equivalence classes (via gauge transformations), but if one restricts to the  axial gauge representative, then the law of $A$ can be genuinely understood as the law of YM measure in the sense of holonomies. To clarify, for any  for any $\ell=(x,v)\in \cX_\rmh$, the process $(\ell_A(t))_{t\in [0,1]}$ is a martingale with respect to its own filtration. Therefore, one can solve the SDE
 \[
 \diff y(t)=y(t)\circ\dif\ell_A(t),\quad y(0)=1_G,
 \]
and find that $\hol(A,\ell):=y(1)\in G$ is the parallel transport.
For any vertical line $\ell$, we set $\hol(A,\ell)=1_G$.
Then for any piecewise affine curve $\gamma$, one can take the ordered products of the holonomies on each affine line segment, and call it $\hol(A,\gamma)$. The law of these holonomies would also correspond to the construction given in \cite{Levy03}.

For $\beta\in (\frac 12,1)$, $\hol(B,\gamma)$ is well-defined for every piecewise affine curve $\gamma$ and $B\in\Omega^1_\beta$ (or even $\Omega^1_{\beta\gr}$) as a Young ODE \cite[Sec.~3.5]{CCHS2d}.
By \cite[Prop.~3.35]{CCHS2d}, these holonomies, modulo the central action of $G$, characterise the gauge orbit of $B$.
We now make precise what we mean by a $\Omega^1_\beta$-valued random variable being a representative of the YM measure.

\begin{definition}\label{def:YM_measure_precise}
Let $\beta\in (\frac 1 2,1]$ and $B\in \Omega^1_\beta$ be a random field on a probability space  $(\tilde\Sigma,\tilde\cF,\tilde\bP)$. We say that $B$ is a representative of $\mu_{\YM}$ on $\Lambda = [0,1]^2$, if 
\[
\tilde\E[f(\hol(B,\gamma_1),...,\hol(B,\gamma_n))]=\E[f(\hol(A,\gamma_1),...,\hol(A,\gamma_n))],
\]
for all $n\in\bN$, piecewise affine loops $\gamma_1,...,\gamma_n$ in $\Lambda$, and smooth  functions $f:G^n\to \bC$ such that $f$ is invariant under the action of $\Ad_G$, i.e.
\[
f(\Ad_hg_1,...,\Ad_hg_n)=f(g_1,...,g_n). 
\]
\end{definition}
We continue to construct the rough additive function $\bA$.
Since, for any $\ell\in\cX_\rmh^\to$, the process $(\ell_A(t))_{t\in [0,1]}$ is a martingale with respect to its own filtration, we can define the iterated integral 
\begin{equation} \label{e:def_A_stratonovich}
	\bA(\ell):=\int_0^1 \ell_A(t)\otimes \circ \dif \ell_A(t), \quad  \ell\in\cX_\rmh^\to
\end{equation}
in the Stratonovich sense. 
For $\ell\in\cX^\ot$, by means of Chen's identity we define  
\[
\bA(\ell):=A(\ell)\otimes A(\ell)-\bA(\ell^\to),
\]
where $\ell^\to$ is the reversed line segment of $\ell$. 
From here it is not difficult to see that $\bfA=(A,\bA)\in\Caxial$. 

Now let $A^\varepsilon$ be defined in terms of $\xi^\varepsilon$ as in Theorem \ref{thm:RUC_YM}.
Define $\bA^\varepsilon$ precisely as $\bA$, but with $A^\varepsilon$ in place of~$A$ and where the iterated integral in~\eqref{e:def_A_stratonovich} is now understood in Riemann--Stieltjes sense, i.e. without recourse to stochastic integration. We use the convention $A^0=A$ and $\bA^0=\bA$ and consequently $\bfA^0=\bfA\label{symb:mollified_RAF}$. 

Our goal is to derive the moment and stability bounds on~$\bfA$ resp. on~$\bfA - \bfA^\eps$ that are required to apply the Kolmogorov theorem on~$\Caxial$, see Theorem~\ref{thm:Kolmogorov}.
For this purpose, we will need the following auxiliary bounds on mollified indicator functions of polytopes.

\begin{lemma}\label{lem:indicator_bound}
Let $U\subset \mathbb R^d$ be a bounded polytope. Let $\eta^\varepsilon(x)=\varepsilon^{-d}\eta(x/\varepsilon)$, 
where $\eta\in C_c^\infty(\mathbb R^d)$ is non-negative, supported in $B(0,1)$, and satisfies
$\int \eta=1$. Then, there exists $C=C(\eta)>0$, such that for every $\kappa\in [0,\frac 1 2]$ and $\varepsilon,\delta\geq 0$,
\[
\|(\eta^\varepsilon-\eta^\delta)*\1_U\|_{L^2}
\leq C
|\varepsilon-\delta|^\kappa
|\partial U|^\kappa
|U|^{\frac12-\kappa}.
\]
with the convention $\eta^0*\1_U=\1_U$.  

\end{lemma}

\begin{proof}
Using the scaling of the mollifier, we write
\[
\eta^\varepsilon * \1_U(x)
=
\int_{\mathbb R^d} \eta(z)\1_U(x-\varepsilon z)\dif z,
\]
so that 
\[
(\eta^\varepsilon-\eta^\delta)*\1_U(x)
=
\int_{\mathbb R^d}
\eta(z)
\big(
\1_U(x-\varepsilon z)-\1_U(x-\delta z)
\big)\dif z.
\]
Since $\eta(z)\dif z$ is a probability measure, Jensen's inequality gives
\begin{align*}
\|(\eta^\varepsilon-\eta^\delta)*\1_U\|_{L^2}^2
&\leq
\int_{\mathbb R^d}\eta(z)
\int_{\mathbb R^d}
\big|
\1_U(x-\varepsilon z)-\1_U(x-\delta z)
\big|^2
\dif x\dif z \\
&\leq \int_{\mathbb R^d}\eta(z)
\|\1_U(\Cdot-\varepsilon z)-\1_U(\Cdot-\delta z)\|_{L^2}^2
\dif z.
\end{align*}
Now we use the standard translation estimate for sets of finite perimeter:
\[
\|\1_U(\Cdot-h)-\1_U\|_{L^2}^2
\leq |h||\partial U|.
\]
Applying this with $h=(\varepsilon-\delta)z$, we obtain
\[
\|(\eta^\varepsilon-\eta^\delta)*\1_U\|_{L^2}^2
\leq
|\varepsilon-\delta||\partial U|
\int_{\mathbb R^d}|z|\eta(z)\,dz.
\]
Since $\eta$ is compactly supported, the last integral is finite, from where we get
\begin{align}\label{eq:interpolation_bound_indicator}
\|(\eta^\varepsilon-\eta^\delta)*\1_U\|_{L^2}
\lesssim |\partial U|^{1/2}|\varepsilon-\delta|^{1/2},
\end{align}
where the proportionality constant depends on $\eta$. 

On the other hand, by Young's inequality,
\[
\|(\eta^\varepsilon-\eta^\delta)*\1_U\|_{L^2}
\leq
\|\eta^\varepsilon*\1_U\|_{L^2}
+
\|\eta^\delta*\1_U\|_{L^2}
\leq
2|U|^{1/2}.
\]
Interpolating this bound with~\eqref{eq:interpolation_bound_indicator} one gets, for
$\kappa\in[0,\frac12]$,
\[
\|(\eta^\varepsilon-\eta^\delta)*\1_U\|_{L^2}
\leq
C(\eta)
|\varepsilon-\delta|^\kappa
|\partial U|^\kappa
|U|^{\frac12-\kappa},
\]
as we were after. 
\end{proof}

We will also need the following definition, see Remark~\ref{r:intersection_decomposition} below for an explanation.
\begin{definition}[Intersection decomposition] \label{d:intersection_decomposition}
	Let~$\ell, \bar{\ell} \in \cX^{\to}_\rmh$ such that~$\ell \simhor \bar{\ell}$ with respective parametrisations 
	\begin{equation} \label{e:parametrisation_gammas}
		\gamma, \bar{\gamma}: [a,b]\to \Lambda, 
		\quad 
		\gamma(t)=(t,m+ct), \quad \bar{\gamma}(t)=(t,\bar{m}+\bar{c}t)
	\end{equation}
	and assume that~$\ell$ and~$\bar{\ell}$ \emph{intersect} within~$[a,b] \times [0,1]$.
	In this case, denote by
	\begin{equation*}
		z^{\ell,\bar{\ell}} = (t^{\ell,\bar{\ell}}_\times, m+c t_\times^{\ell,\bar{\ell}}), 
		\quad 
		t_\times^{\ell,\bar{\ell}} = \frac{m-\bar{m}}{\bar{c}-c}
	\end{equation*}
	their unique \emph{intersection point}, see Figure~\ref{fig:intersection_decomposition}.
	Then, we can uniquely decompose~$\ell = \ell_1 \sqcup \ell_2$ and~$\bar{\ell} = \bar{\ell}_1 \sqcup \bar{\ell}_2$ with respective parametrisations given by 
	\begin{equs}
		\gamma_1 = \gamma\sVert[0]_{[a,t^{\ell,\bar{\ell}}_\times]}, \quad \bar{\gamma}_1 = \bar{\gamma}\sVert[0]_{[a,t^{\ell,\bar{\ell}}_\times]}, \quad \gamma_2 = \gamma\sVert[0]_{[t^{\ell,\bar{\ell}}_\times,b]}, \quad \bar{\gamma}_2 = \bar{\gamma}\sVert[0]_{[t^{\ell,\bar{\ell}}_\times,b]}, 
	\end{equs}
	In the case where~$\ell$ and~$\bar{\ell}$ do not intersect in~$[a,b] \times [0,1]$, we set~$\ell_1 := \ell$ resp.~$\bar\ell_1 := \bar\ell$ and~$\ell_2 := (\ell_2^{f},0)$ resp. ~$\bar\ell_2 := (\bar\ell_2^{f},0)$ where~$\ell_2^f$ resp.~$\bar\ell_2^f$ denotes the final point of~$\ell$ resp.~$\bar\ell$.

	The pair~$\del[1]{(\ell_1,\ell_2), (\bar{\ell}_1,\bar{\ell}_2)}$ is called the \emph{intersection decomposition} of~$\ell$ and~$\bar{\ell}$.

	\begin{figure}[H]
		\centering
			\tikzset{>={Latex[width=3pt,length=3pt]}}
			\begin{tikzpicture}[scale=3.3] 
				%
				%
				\draw[-{Latex[length=1.5mm]}, thick] (0,0) -- (1.1,0); 
				\draw[-{Latex[length=1.5mm]}, thick] (0,0) -- (0,1.1); 
				%
				\coordinate (A) at (0.2,0);       
				\coordinate (bottomleft) at (0.1,0);       
				\coordinate (B) at (0.8,0);       
				\coordinate (bottomright) at (0.9,0);  
				\coordinate (C') at (0.2,0.455);      
				\coordinate (lstart) at (0.1,0.25);     
				\coordinate (lend) at (0.9,0.5);    
				\coordinate (lbarstart) at (0.1,0.7);
				\coordinate (lbarend) at (0.9,0.3);
				\coordinate (F') at (0.8,0.7);
				\coordinate (topleft) at (0.1,1);
				\coordinate (topright) at (0.9,1);
				\path[name path=up1] (lstart) -- (0.2,1);
				\path[name path=up2] (B) -- (0.8,1);
				\path[name path=ell] (lstart) -- (lend);
				\path[name path=ellbar] (lbarstart) -- (lbarend);
				\path[name intersections={of=up1 and ellbar, by=I1}];
				\path[name intersections={of=up2 and ell, by=I2}];
				\path[name intersections={of=ell and ellbar, by=Imid}];
				\path let \p1 = (Imid) in coordinate (Imidprojdown) at (\x1,0);
				\path let \p1 = (Imid) in coordinate (Imidprojup) at (\x1,1);
				%
				\draw[densely dashed, semithick] (bottomleft) -- (topleft);
				\draw[densely dashed, semithick] (bottomright) -- (topright);
				\draw[densely dashed, semithick] (Imidprojdown) -- (Imidprojup);
				\node[below] at (bottomleft) {\scalebox{1}{$a$}};
				\node[below] at (bottomright) {\scalebox{1}{$b$}};
				%
				%
				%
				\draw[thick,
				red,
				double=red!10,          
				double distance=0.5pt,     
				line cap=round,           
				opacity=0.3
				] (lbarstart) -- (Imid) node[midway, above, scale=1, color=red, opacity=1.0] {$\bar{\ell}_1$}; 
				%
				%
				\draw[thick,
				red,
				double=red!10,          
				double distance=0.5pt,     
				line cap=round,           
				opacity=0.3
				] (lstart) -- (Imid) node[midway, below, scale=1, color=red, opacity=1.0] {$\ell_1$};
				%
				%
				%
				%
				%
				\draw[thick,
				blue,
				double=blue!10,          
				double distance=0.5pt,     
				line cap=round,           
				opacity=0.3
				] (Imid) -- (lend) node[midway, above, scale=1, color=blue, opacity=1.0] {$\ell_2$};
				%
				%
				\draw[thick,
				blue,
				double=blue!10,          
				double distance=0.5pt,     
				line cap=round,           
				opacity=0.3
				] (Imid) -- (lbarend) node[midway, below, scale=1, color=blue, opacity=1.0] {$\bar{\ell}_2$};
				%
				\draw[-{Latex[length=1.5mm]}, thick] (lstart) -- (lend) node[midway, above left, scale=1] {};
				\draw[-{Latex[length=1.5mm]}, thick] (lbarstart) -- (lbarend) node[midway, below left, scale=1] {};
				\node[right=-1pt] at (lend){\scalebox{1}{$\ell$}};
				\node[right=-1pt] at (lbarend){\scalebox{1}{$\bar{\ell}$}};
				%
				%
				%
				\fill[blue!70, opacity = 0.2] (lend) -- (Imid) -- (lbarend) -- cycle;
				%
				%
				\fill[red!70, opacity = 0.2] (lbarstart) -- (lstart) -- (Imid) -- cycle;
				%
				%
				\draw (0,0) rectangle (1,1);
				\draw[thick, black] (A) -- (B) node[midway, below, scale=1] {\phantom{$\mathring{\ell}$}};
				\foreach \pt in {bottomleft,bottomright,topleft, topright, lstart,lend,lbarstart,lbarend,Imidprojdown,Imidprojup}
				\fill[black] (\pt) circle (0.3pt);
				\node[tinydot] at (Imid){};
				\draw[-{Latex[length=1.5mm]}, thick] ($(Imid) + (-0.2,0.4)$) to [out=0,in=90] (Imid); 
				\node[left] at ($(Imid) + (-0.2,0.4)$) {\scalebox{1}{$z^{\ell,\bar{\ell}}$}};
				%
				%
				%
			\end{tikzpicture}
				%
			%
		%
		\vspace{-3mm}
		\captionsetup{format=plain}
		\caption{Two line segments~$\ell$ and~$\bar{\ell}$ with their intersection point~$z^{\ell,\bar{\ell}}$ and their intersection decomposition~$((\ell_1,\ell_2),(\bar\ell_1,\bar\ell_2))$.
			Shaded in red (resp. blue) is the area enclosed by~$\ell_1$ and~$\bar\ell_1$ (resp.~$\ell_2$ and~$\bar\ell_2$). }
		\label{fig:intersection_decomposition}
	\end{figure}
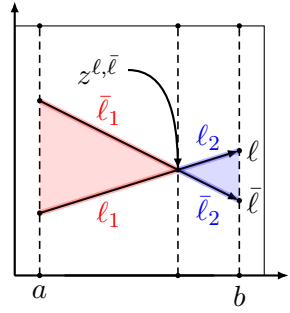
\end{definition}

\begin{remark} \label{r:intersection_decomposition}
	Let us briefly comment on the reason for introducing the intersection decomposition.
	In general, the identity~$\abs[0]{A(\ell) - A(\bar{\ell})} = \abs[0]{\xi(\1_{U_{\ell,\bar\ell}})}$ for $U_{\ell,\bar\ell}=U_{\ell}\triangle  U_{\bar\ell}$ and $\triangle$ the symmetric difference is only true if~$\ell$ and~$\bar\ell$ are parallel lines; otherwise, there is a sign change at the intersection point.
	In contrast, it is true that $\abs[0]{A(\ell) - A(\bar\ell)} =
	\abs[0]{\xi(\1_{U_{\ell_1,\bar\ell_1}}) - \xi(\1_{U_{\ell_2,\bar\ell_2}})}$ and this is precisely what we require in~\eqref{e:wn_intersection} below. 
\end{remark}

Our first moment bound result concerns $A$ and $A^\varepsilon$.
\begin{lemma} \label{lem:moment_bound_first_level}
Then there exists $C>0$ such that for any $p\in [2,\infty)$ and $\ell\simhor\bar\ell$, it holds that
\begin{equation*}
\E[|A(\ell)-A(\bar\ell)|^p]^{1/p}
\leq C\sqrt p\,\Area(\ell,\bar\ell)^{1/2}.
\end{equation*}
Moreover, for any $\varepsilon,\bar\varepsilon\in [0,1]$ and $\kappa\in [0,\frac12]$, we have
\begin{align*}
\E\big[
\big|
A^\varepsilon(\ell)-A^\varepsilon(\bar\ell)
-
A^{\bar\varepsilon}(\ell)+A^{\bar\varepsilon}(\bar\ell)
\big|^p
\big]^{1/p}\leq
C\sqrt p\,
|\varepsilon-\bar\varepsilon|^\kappa
\Area(\ell,\bar\ell)^{\frac12-\kappa}.
\end{align*}
\end{lemma}

\begin{proof}
Let~$\ell, \bar{\ell} \in \cX_\rmh^{\to}$ with~$\ell\simhor\bar\ell$ and corresponding intersection decomposition
$((\ell_1,\ell_2),(\bar\ell_1,\bar\ell_2))$, see Definition~\ref{d:intersection_decomposition}.
By additivity of~$A$, we have
\begin{equation} \label{e:wn_intersection}
\abs[0]{A(\ell)-A(\bar\ell)}
=
\abs[0]{
\xi(\1_{U_{\ell_1,\bar\ell_1}})
-
\xi(\1_{U_{\ell_2,\bar\ell_2}})
},
\end{equation}
where $U_{\ell_i,\bar\ell_i}=U_{\ell_i}\triangle U_{\bar\ell_i}$ and $\triangle$ denotes the symmetric difference.
Then, by definition of~$\xi$ and the elementary observation that
$\abs[0]{U_{\ell_i,\bar\ell_i}}\leq \abs[0]{U_{\ell,\bar\ell}}$, we have
\begin{equs}
\E[|A(\ell)-A(\bar\ell)|^2]
=
\norm[0]{\1_{U_{\ell_1,\bar\ell_1}}-\1_{U_{\ell_2,\bar\ell_2}}}_{L^2}^2
\leq
\del[1]{
\abs[0]{U_{\ell_1,\bar\ell_1}}^{1/2}
+
\abs[0]{U_{\ell_2,\bar\ell_2}}^{1/2}
}^2
\leq
4\abs[0]{U_{\ell,\bar\ell}}.
\end{equs}
Since $\abs[0]{U_{\ell,\bar\ell}}=\Area(\ell,\bar\ell)$, the first bound follows by Gaussian hypercontractivity.

For the second estimate, we let $\1_U^\varepsilon:=\rho^\varepsilon*\1_U$ for any set $U\subseteq\Lambda$.
Using~\eqref{e:radial_mollifier}, the identity~\eqref{e:wn_intersection} also holds with $A^\varepsilon$ in place of $A$ and $\1^\varepsilon$ in place of $\1$.
As a consequence, we find
 \begin{equs}
	    \E[|A^\varepsilon(\ell)-A^\varepsilon(\bar\ell)-(A^{\bar\epsilon}(\ell)-A^{\bar\eps}(\bar\ell))|^2]
	   	& \leq
	    \sum_{i=1}^2 \E[|\xi(\1^\varepsilon_{U_{\ell_i,\bar\ell_i}} - \1^{\bar\eps}_{U_{\ell_i,\bar\ell_i}})|^2] \\
	    & = 
	    \sum_{i=1}^2 \|\1^\varepsilon_{U_{\ell_i,\bar\ell_i}} - \1^{\bar\eps}_{U_{\ell_i,\bar\ell_i}}\|^2_{L^2}  \,.
    \end{equs}
By Lemma~\ref{lem:indicator_bound}, for every $\kappa\in[0,\frac12]$,
\[
\norm[0]{
\1^\varepsilon_{U_{\ell_i,\bar\ell_i}}
-
\1^{\bar\varepsilon}_{U_{\ell_i,\bar\ell_i}}
}_{L^2}
\lesssim 
|\varepsilon-\bar\varepsilon|^\kappa
|\partial U_{\ell_i,\bar\ell_i}|^\kappa
|U_{\ell_i,\bar\ell_i}|^{\frac12-\kappa}.
\]
Since $|\partial U_{\ell_i,\bar\ell_i}|\lesssim 1$ uniformly and
$|U_{\ell_i,\bar\ell_i}|\leq \Area(\ell,\bar\ell)$, we obtain
\[
\E\big[
\big|
A^\varepsilon(\ell)-A^\varepsilon(\bar\ell)
-
A^{\bar\varepsilon}(\ell)+A^{\bar\varepsilon}(\bar\ell)
\big|^2
\big]^{1/2}
\lesssim 
|\varepsilon-\bar\varepsilon|^\kappa
\Area(\ell,\bar\ell)^{\frac12-\kappa}.
\]
The desired $L^p$ bound follows again by Gaussian hypercontractivity.
\end{proof}

We now prove the moment bounds for the iterated integral $\bA$ and, to this end, write it in terms of its Wiener chaos decomposition. 
We will first derive the expressions for the mollified $\bA^\varepsilon$ which then easily generalise to the unmollified case~$\eps = 0$.
Let $\ell$ with parametrisations as in~\eqref{e:parametrisation_gammas}.
%
By~\eqref{e:RAF_gr_norm_2} as well as the definition of~$A_1$ in~\eqref{e:def_A_cax}, we have 
	\begin{equs}
	\bA^\varepsilon(\ell)
	& =\int^b_a \int^t_a A_1^\varepsilon(s,m+cs)\otimes A_1^\varepsilon(t, m+ ct)\dif s\dif t\\
	& =\int^b_a \int^{y_1}_a  \int_0^{m+cx_1}  \int_0^{ m+ cy_1}\xi^\varepsilon(x)\otimes \xi^\varepsilon(y)\dif y_2\dif x_2\dif x_1\dif y_1  \\
	& = 
	\int_{\R^2}  \int_{\R^2} \xi^\varepsilon(x)\otimes \xi^\varepsilon(y) k_{\ell}(x,y) \dif x \dif y
\end{equs}
where the function~$k_{\ell}$ and its symmetrisation~$\tilde k_{\ell}$ are given by
\[
k_{\ell}(x,y) := \1_{0\leq x_1\leq y_1\leq b}\1_{U_{\ell}}(x)\1_{U_{\ell}}(y), 
\quad 
\tilde k_{\ell}(x,y)=\frac 1 2 (k_{\ell}(x,y)+k_{\ell}(y,x)) \,
\]
Further, we let $k^\varepsilon_{\ell}:=(\rho^\varepsilon\otimes\rho^\varepsilon)*k_{\ell}$ where the convolution is to be understood coordinate-wise with symmetrisated counterpart~$\tilde k^\varepsilon_{\ell}$.
Finally, we denote by~$\bfI_2$ the second Wiener--It\^{o} isometry with values in~$\fg^{\otimes 2}$ and let $\Cas\label{symb:Cas}\in \fg^{\otimes 2}$ be the \emph{Casimir element}, i.e.~$\Cas =\sum_{i=1}^{\dim\fg} e_i\otimes e_i$ for an orthonormal basis~$(e_i)_{i=1}^{\dim\fg}$ of the Lie algebra. 

Then, it is immediate to see that~$\bA^\eps$ has Wiener chaos decomposition given by
\begin{equation} \label{e:WIC_decomp_bA_eps}
	\bA^\varepsilon(\ell)=\bfI_2(\tilde k^\varepsilon_{\ell})+\Cas\int_{\bR^2} k^\varepsilon_{\ell}(x,x) \dif x
\end{equation}
and, by definition of~$\bA$ in~\eqref{e:def_A_stratonovich}, one can obtain the following decomposition:
\begin{equation} \label{e:WIC_decomp_bA}
	\bA(\ell)
	=
	\bfI_2(\tilde k_{\ell})+\Cas\frac 1 2\int_{\bR^2}k_{\ell}(x,x)\dif x\;.
\end{equation}
The additional factor~\enquote{$\frac{1}{2}$} appears because we have defined~$\bA$ as a Stratonovich integral.
In the proof of the following lemma, e.g.~around~\eqref{eq:D_ell_delta},  we see that it arises as limiting expression of~\eqref{e:WIC_decomp_bA_eps} in the spirit of Wong--Zakai, that is 
	\begin{equation*}
		\lim_{\eps \downarrow 0} \int_{\bR^2} k^\varepsilon_{\ell}(x,x) \dif x 
		= 
		\frac{1}{2}\int_{\bR^2} k_{\ell}(x,x) \dif x \,.
	\end{equation*}
However, the purpose of the lemma is to show moment estimates for $\bA$, $\bA^\eps$ and differences thereof. 
%
%
\begin{lemma}\label{lem:moment_bound_second_level}
There exists $C>$ such that, for any $p\in [2,\infty)$ and $\ell\simhor\bar\ell$, it holds that
\begin{align}\label{lem:moment_bound_second_level:e1}
\E[|\bA(\ell)-\bA(\bar\ell)|^p]^{1/p}
&\leq C  p (|\ell|\vee|\bar\ell|)^{\frac 12}\Area(\ell,\bar\ell)^{\frac 12}.
\end{align}
Moreover, for any $\varepsilon,\bar\varepsilon\in [0,1]$ and $\kappa\in [0,\frac12]$, we have
\begin{align}\label{lem:moment_bound_second_level:e2}
\E[
|(\bA^\varepsilon(\ell)-\bA^\varepsilon(\bar\ell))
-(\bA^{\bar\varepsilon}(\ell)-\bA^{\bar\varepsilon}(\bar\ell))|^p
]^{1/p}
&\leq Cp |\varepsilon-\bar\varepsilon|^\kappa
(|\ell|\vee|\bar\ell|)^{\frac 12-\kappa}
\Area(\ell,\bar\ell)^{\frac 12-\kappa}.
\end{align}
\end{lemma}

\begin{proof}
We only focus on the bound in~\eqref{lem:moment_bound_second_level:e2} as the one in~\eqref{lem:moment_bound_second_level:e1} is easier to derive.
For the sake of clarity, we will only deal with the case where~$\ell$ and~$\bar\ell$ do \emph{not intersect} in~$[a,b] \times [0,1]$; the general case follows by using the intersection decomposition of~$\ell$ and~$\bar\ell$ (see Definition~\ref{d:intersection_decomposition}) in the same way as in the proof of Lemma~\ref{lem:moment_bound_first_level}. 

For $\delta\in[0,1]$, we use the notation
\[
D^\delta_{\ell}
:=
\begin{cases}
\displaystyle
\int_{\bR^2} k^\delta_{\ell}(x,x)\dif x,
& \delta>0,
\\[1.2em]
\displaystyle
\frac12\int_{\bR^2} k_{\ell}(x,x)\dif x,
& \delta=0.
\end{cases}
\]
With $k^0_{\ell}:=k_{\ell}$ and $\tilde k^0_{\ell}:=\tilde k_{\ell}$, the decompositions~\eqref{e:WIC_decomp_bA_eps} and~\eqref{e:WIC_decomp_bA} can be written as
\[
\bA^\delta(\ell)
=
\bfI_2(\tilde k^\delta_{\ell})
+
\Cas D^\delta_{\ell},
\qquad
\delta\in[0,1].
\]
Therefore,
\begin{align*}
\bA^\varepsilon(\ell)-\bA^\varepsilon(\bar\ell)
-
\bA^{\bar\varepsilon}(\ell)+\bA^{\bar\varepsilon}(\bar\ell)
 =
\bfI_2\del[1]{
\tilde k_{\ell}^\varepsilon
-
\tilde k_{\bar\ell}^\varepsilon
-
\tilde k_{\ell}^{\bar\varepsilon}
+
\tilde k_{\bar\ell}^{\bar\varepsilon}
}
+
\Cas\del[1]{
D^\varepsilon_{\ell}
-
D^\varepsilon_{\bar\ell}
-
D^{\bar\varepsilon}_{\ell}
+
D^{\bar\varepsilon}_{\bar\ell}
}.
\end{align*}
We bound the part in the second chaos separately from the one in the zero-th chaos. 
By It\^o isometry, we find
\begin{align*}
\E[
|\bfI_2(\tilde k_{\ell}^\varepsilon-\tilde k_{\bar\ell}^\varepsilon)
-\bfI_2(\tilde k_{\ell}^{\bar\varepsilon}-\tilde k_{\bar\ell}^{\bar\varepsilon})|^2
]^{1/2}
\lesssim
\|k_{\ell}^\varepsilon
-
k_{\bar\ell}^\varepsilon
-
k_{\ell}^{\bar\varepsilon}
+
k_{\bar\ell}^{\bar\varepsilon}\|_{L^2}.
\end{align*}
We write
\[
k_{\ell}(x,y)-k_{\bar\ell}(x,y)
=
\1_{a\leq x_1\leq y_1\leq b}
\del[1]{
\1_{U_\ell}(x)\1_{U_\ell}(y)
-
\1_{U_{\bar\ell}}(x)\1_{U_{\bar\ell}}(y)
}.
\]
The right hand side is a finite signed sum of indicator functions of polytopes
$V_j\subset \bR^2\times\bR^2$ satisfying
\[
    |V_j|
    \leq
    C (|\ell|\vee|\bar\ell|)\Area(\ell,\bar\ell),
    \qquad
    |\partial V_j|\leq C.
\]
Indeed, this follows from
\[
\1_{U_\ell}(x)\1_{U_\ell}(y)
-
\1_{U_{\bar\ell}}(x)\1_{U_{\bar\ell}}(y)
=
\del[1]{\1_{U_\ell}(x)-\1_{U_{\bar\ell}}(x)}\1_{U_\ell}(y)
+
\1_{U_{\bar\ell}}(x)\del[1]{\1_{U_\ell}(y)-\1_{U_{\bar\ell}}(y)},
\]
together with
\[
    |U_\ell\triangle U_{\bar\ell}|=\Area(\ell,\bar\ell),
    \qquad
    |U_\ell|+|U_{\bar\ell}|\leq C |\ell|\vee|\bar\ell|.
\]
Applying Lemma~\ref{lem:indicator_bound} to each of the finitely many polytopes $V_j$ gives
\begin{align*}
\|k_{\ell}^\varepsilon
-
k_{\bar\ell}^\varepsilon
-
k_{\ell}^{\bar\varepsilon}
+
k_{\bar\ell}^{\bar\varepsilon}\|_{L^2}
&\lesssim 
|\varepsilon-\bar\varepsilon|^\kappa
\del[1]{
(|\ell|\vee|\bar\ell|)\Area(\ell,\bar\ell)
}^{\frac 12-\kappa}
\\
&\lesssim |\varepsilon-\bar\varepsilon|^\kappa
(|\ell|\vee|\bar\ell|)^{\frac 12-\kappa}
\Area(\ell,\bar\ell)^{\frac 12-\kappa}.
\end{align*}
After applying Gaussian hypercontractivity, this yields the desired bound for the term in the second chaos.

The term in the zero-th chaos is deterministic and reads
\[
 D_{\ell,\bar\ell}^{\varepsilon,\bar\varepsilon}
:=
D^\varepsilon_{\ell}
-
D^\varepsilon_{\bar\ell}
-
D^{\bar\varepsilon}_{\ell}
+
D^{\bar\varepsilon}_{\bar\ell}.
\]
We first prove an $\eps$-dependent bound. 
For $\delta>0$, by Fubini and the change of variables $z=y-\delta r$, we have
\begin{align}\label{eq:D_ell_delta}
D^\delta_{\ell}
=
\int_{\{r_1\geq0\}}
\eta(r)
\int_{\bR^2}
\1_{U_{\ell}}(y-\delta r)
\1_{U_{\ell}}(y)
\dif y\dif r.
\end{align}
Indeed, after this change of variables, the condition $z_1\leq y_1$ becomes $r_1\geq0$. The remaining endpoint conditions are already enforced by the two indicators. For $\delta=0$, the same identity holds by the definition of $D^0_\ell$ and the radiality of $\eta$, since
\[
    k_\ell(y,y)=\1_{U_\ell}(y),
    \qquad
    \int_{\{r_1\geq0\}}\eta(r)\dif r=\frac12.
\]
Consequently, for all $\varepsilon,\bar\varepsilon\in[0,1]$,
\begin{align*}
|D^\varepsilon_{\ell}-D^{\bar\varepsilon}_{\ell}|
&\leq
\int_{\{r_1\geq0\}}
\eta(r)
\int_{\bR^2}
\abs[0]{
\1_{U_{\ell}}(y-\varepsilon r)
-
\1_{U_{\ell}}(y-\bar\varepsilon r)
}
\dif y
\dif r
\\
&=
\int_{\{r_1\geq0\}}
\eta(r)
\norm[0]{
\1_{U_{\ell}}(\Cdot-\varepsilon r)
-
\1_{U_{\ell}}(\Cdot-\bar\varepsilon r)
}_{L^1}
\dif r .
\end{align*}
By the standard translation estimate for sets of finite perimeter,
\[
\norm[0]{
\1_{U_{\ell}}(\Cdot-\varepsilon r)
-
\1_{U_{\ell}}(\Cdot-\bar\varepsilon r)
}_{L^1}
\leq
|\varepsilon-\bar\varepsilon|\,|r|\,|\partial U_{\ell}|,
\]
together with $|\partial U_{\ell}|\lesssim 1$, we obtain
\[
|D^\varepsilon_{\ell}-D^{\bar\varepsilon}_{\ell}|
\lesssim |\varepsilon-\bar\varepsilon|.
\]
The same estimate holds with $\bar\ell$ in place of $\ell$, and hence
\begin{equation} \label{eq:zeroth_chaos_two_scale_bound}
| D_{\ell,\bar\ell}^{\varepsilon,\bar\varepsilon}|
\lesssim |\varepsilon-\bar\varepsilon|.
\end{equation}

On the other hand, for every $\delta\in[0,1]$,
\begin{align*}
|D^\delta_{\ell}-D^\delta_{\bar\ell}|
&\leq
\int_{\{r_1\geq0\}}
\eta(r)
\int_{\bR^2}
\abs[0]{
\1_{U_{\ell}}(y-\delta r)
\1_{U_{\ell}}(y)
-
\1_{U_{\bar\ell}}(y-\delta r)
\1_{U_{\bar\ell}}(y)
}
\dif y\dif r
\\
&\leq
\int_{\{r_1\geq0\}}
\eta(r)
\int_{\bR^2}
\abs[0]{
\1_{U_{\ell}}(y-\delta r)
-
\1_{U_{\bar\ell}}(y-\delta r)
}
\dif y\dif r
\\
&\qquad
+
\int_{\{r_1\geq0\}}
\eta(r)
\int_{\bR^2}
\abs[0]{
\1_{U_{\ell}}(y)
-
\1_{U_{\bar\ell}}(y)
}
\dif y\dif r
\\
&\lesssim |U_\ell\triangle U_{\bar\ell}|\\
&\lesssim \Area(\ell,\bar\ell).
\end{align*}
Thus
\[
| D_{\ell,\bar\ell}^{\varepsilon,\bar\varepsilon}|
\lesssim \Area(\ell,\bar\ell).
\]
Moreover, by Fubini,
\[
    |D^\delta_{\ell}|
    \lesssim  |U_\ell|
    \lesssim |\ell|,
    \qquad
    |D^\delta_{\bar\ell}|
    \lesssim |U_{\bar\ell}|
    \lesssim |\bar\ell|,
\]
and therefore
\[
| D_{\ell,\bar\ell}^{\varepsilon,\bar\varepsilon}|
\lesssim |\ell|\vee|\bar\ell|.
\]
Combining the previous two trivial bounds yields
\begin{equation} \label{eq:zeroth_chaos_trivial_bound}
| D_{\ell,\bar\ell}^{\varepsilon,\bar\varepsilon}|
\lesssim \del[1]{(|\ell|\vee|\bar\ell|)\wedge \Area(\ell,\bar\ell)} .
\end{equation}
Interpolating~\eqref{eq:zeroth_chaos_two_scale_bound} with~\eqref{eq:zeroth_chaos_trivial_bound}, we get, for every $\kappa\in[0,\frac12]$,
\[
| D_{\ell,\bar\ell}^{\varepsilon,\bar\varepsilon}|
\leq
C
|\varepsilon-\bar\varepsilon|^\kappa
\del[1]{L_{\ell,\bar\ell}\wedge \Area(\ell,\bar\ell)}^{1-\kappa}.
\]
Since $|\ell|,|\bar\ell|,\Area(\ell,\bar\ell)\leq 1$ and $\kappa\in[0,\frac12]$, we conclude 
\[
| D_{\ell,\bar\ell}^{\varepsilon,\bar\varepsilon}|
\leq
C
|\varepsilon-\bar\varepsilon|^\kappa
(|\ell|\vee|\bar\ell|)^{\frac12-\kappa}
\Area(\ell,\bar\ell)^{\frac12-\kappa}.
\]
This proves the desired estimate for the zero-th chaos. Combining this with the second chaos estimate completes the proof of~\eqref{lem:moment_bound_second_level:e2}.
\end{proof}

\begin{proof}[of Theorem~\ref{thm:RUC_YM}]
Lemmas~\ref{lem:moment_bound_first_level} and
\ref{lem:moment_bound_second_level}
imply that the hypothesis \eqref{eq:Kolmogorov_moment_bounds} of the Kolmogorov Theorem~\ref{thm:Kolmogorov} is satisfied for $(A^\eps,\bA^\eps)$ and for $(A^{\bar\eps},\bA^{\bar\eps})$ with $\alpha=1/2$ and every $p\geq 2$ with constant $L= C^p p^{p/2}$.
Moreover the differences $\delta A = A^\eps - A^{\bar\eps},\bA^\eps - \bA^{\bar\eps}$ satisfy the same hypothesis \eqref{eq:Kolmogorov_moment_bounds} with $\alpha=1/2-\kappa,\kappa \in [0,1/2]$ and every $p\geq 2$ with constant $C^p p^{p/2} |\eps-\bar\eps|^{\kappa p/2}$,
uniformly in $\eps,\bar\eps\in[0,1]$.
In both statements, the constant $C$ is independent of $p,\kappa,\eps,\bar\eps$.

Therefore Theorem~\ref{thm:Kolmogorov} implies that (after possibly replacing the fields by modifications), 
uniformly in $\eps,\bar\eps\in[0,1]$, for all $\kappa\in[0,1/2)$, $0<\alpha < 1/2-\kappa$, and $p\geq 1$,
\begin{align}
    \E\big[
        \triple{\bfA^\varepsilon}_{\alpha\ax}^p
    \big]^{1/p}
    &\lesssim_\alpha\sqrt p
\label{eq:RUCYM_Kolmogorov}
	\\
\E\big[
        \triple{\bfA^\varepsilon;\bfA^{\bar\varepsilon}}_{\alpha\ax}^p
    \big]^{1/p}
    &\lesssim_{\alpha,\kappa}\sqrt p\,
    |\varepsilon-\bar\varepsilon|^{\kappa/2}\;.
	\label{eq:RUCYM_Kolmogorov_difference}
\end{align}
Now viewing $\eps\mapsto \bfA^\eps$ as a stochastic process indexed by $\eps\in[0,1]$ with values in the Polish space $(\bfOmega^1_{\alpha\ax},\triple{\Cdot}_{\alpha\ax})$,
the classical Kolmogorov continuity theorem (see \cite[Thm.~A.10]{FV10} and remark that the constant $C$ therein can be made independent of $p$ once $p$ is sufficiently large,
or just use \cite[Thm.~A.19]{FV10})
implies that, for $\kappa \in [0,1/4)$, $0<\alpha < 1/2-2\kappa$, and $p\geq 1$,
\begin{equation}\label{eq:A_eps_Holder}
\E\Big[ \Big|
		\sup_{0\leq\bar\eps<\eps\leq 1}
		\frac{\triple{\bfA^\varepsilon;\bfA^{\bar\varepsilon}}_{\alpha\ax}}{|\eps-\bar\eps|^{\kappa}}
		\Big|^p
	\Big]^{1/p}
	\lesssim_{\alpha,\kappa}\sqrt p
	\;.
\end{equation}

Now fix $\beta\in(0,1)$ and choose $\kappa\in (0,1/4)$ sufficiently small such that Theorem \ref{thm:RUC_bounded_sets} applies with $\alpha = 1/2-3\kappa$ and output regularity $\beta$.
Set
\[
    Y:=
    \sup_{\eps\in[0,1]}
    \triple{\bfA^\eps}_{\alpha\ax}\;.
\]
Then~\eqref{eq:RUCYM_Kolmogorov} and \eqref{eq:A_eps_Holder} imply that, for all $p\geq 1$,
\begin{equation}\label{eq:RUCYM_Y_moment}
    \E[Y^p]^{1/p}
    \lesssim_{\alpha,\kappa}\sqrt p\;.
\end{equation}

We now perform the gauge fixing simultaneously on the whole family
$(\bfA^\eps)_{\eps\in[0,1]}$. Define the measurable partition
\[
    \Sigma_1:=\{Y\leq 1\},
    \qquad
    \Sigma_n:=\{n-1<Y\leq n\},
    \qquad n\geq2.
\]
For each $n\geq1$, let
$
    \bsg_n:\bfOmega^1_{\alpha\ax;\leq n}\to C(\Lambda;G)
$
be the continuous map from
Theorem~\ref{thm:RUC_bounded_sets}, with output regularity $\beta \in (0,1)$. On
$\Sigma_n$, we define
\[
    g^\eps:=\bsg_n(\bfA^\eps)\;,
    \qquad
    B^\eps:=(A^\eps)^{g^\eps}\;,
    \qquad
    \eps\in[0,1]\;.
\]
This defines measurable random variables
$B^\eps$ and $B$ with values in $\Omega^1_\beta$. Since $B=A^g$ is
gauge equivalent to the complete axial representative of the YM measure,
$B$ is a gauge-fixed representative of the YM measure in the sense of~Definition~\ref{def:YM_measure_precise} by~Lemma~\ref{lem:holonomy_gauge_tr}.

By Theorem~\ref{thm:RUC_bounded_sets}, on $\Sigma_n$ we have
\[
   \sup_{\eps\in[0,1]} |B^\eps|_{\beta}
    \lesssim_{\alpha,\beta} n^{\frac {2+\beta}{2\alpha}-1}
	\sup_{\eps\in[0,1]}
    \triple{\bfA^\eps}_{\alpha\ax}
\]
and the right-hand side is bounded by $C(\alpha)(Y+1)^{\frac {2+\beta}{2\alpha}}$.
By~\eqref{eq:RUCYM_Y_moment}, it follows that for all $p\geq 1$,
\begin{equation*}
\E\Big[
    \sup_{\eps\in[0,1]}
    |B^\eps|_{\beta}^p
\Big]^{1/p}
\lesssim_{\alpha,\beta}
\E\big[(Y+1)^{\frac{(2+\beta)p}{2\alpha}}\big]^{1/p}
\lesssim_\alpha p^{\frac {2+\beta}{4\alpha}} 
\;
\end{equation*}
Taking $\alpha$ sufficiently close to $1/2$ implies \eqref{eq:B_moments}.
It remains to prove convergence of the mollified gauge-fixed fields. By the stability estimate \eqref{eq:cont_in_A} of Theorem~\ref{thm:RUC_bounded_sets}, on $\Sigma_n$ we have
\[
    |B^\eps-B^{\bar\eps}|_{\beta}
    \lesssim_{\alpha,\beta}
    n^{\frac {2+\beta}{2\alpha}-1}
    \triple{\bfA^\eps;\bfA^{\bar\eps}}_{\alpha\ax}.
\]
where the right-hand side is bounded by $C(\alpha)(Y+1)^{\frac {2+\beta}{2\alpha}-1}\triple{\bfA^\eps;\bfA^{\bar\eps}}_{\alpha\ax}$.
By Cauchy--Schwarz, \eqref{eq:A_eps_Holder}, and
\eqref{eq:RUCYM_Y_moment}, we obtain for all $p\geq 1$
\begin{align*}
\E\Big[\Big|
 \sup_{0\leq\bar\eps<\eps\leq 1} \frac{|B^\eps-B^{\bar\eps}|_{\beta}}{|\eps-\bar\eps|^\kappa}\Big|^p
\Big]^{1/p}
&\lesssim_{\alpha,\beta}
\|(1+Y)^{\frac {2+\beta}{2\alpha}-1}\|_{L^{2p}(\P)}
\E\Big[\Big|
 \sup_{0\leq\bar\eps<\eps\leq 1} \frac{\triple{\bfA^\eps;\bfA^{\bar\eps}}_{\alpha\ax}}{|\eps-\bar\eps|^\kappa}\Big|^{2p}
\Big]^{\frac{1}{2p}}                         \notag\\
&\lesssim_{\alpha}
p^{\frac{2+\beta}{4\alpha}}\;.
\end{align*}
Taking again $\alpha$ sufficiently close to $1/2$ implies \eqref{eq:B_eps_moments}.
\end{proof}

\appendix

\section{Abstract integration kernel}\label{app:integration_kernel}

The main results of this appendix are Theorems \ref{thm:singular_kernel_composed_der} and \ref{thm:integration_boundary_derivation}, which are used in the proof of Proposition \ref{prop:integration} on the integration of modelled distributions.
In preparation for this, we give a self-contained construction of the Dirichlet Green function $G^\Dir$ on $\Lambda$ which allows us to use results from \cite{Hairer14}.
Our setting is inspired by \cite{GH19}, however it slightly differs since the elliptic Green function on $\R^2$ does not decay at infinity, unlike the heat kernel considered therein.

    \subsection{Dirichlet boundary condition via odd reflections}\label{sec:odd_reflections}
This section sets up the method of images for the Dirichlet problem on the square by means of \emph{odd reflections} across the lattice lines.
To that end, denote by $\diag_2\{-1,1\}$ the set of $2\times 2$ diagonal matrices with diagonal entries $-1$ or $1$.
For $b\in\diag_2\{-1,1\}$ and $k\in\bZ^2$ we set
\begin{align}\label{eq:def_q}
q_{b,k}:\bR^2\to\bR^2,\qquad q_{b,k}(x):=bx+2k .
\end{align}
These are isometries, generated by reflections in the coordinate directions and translations by vectors in $2\bZ^2$.
We define
\[
\cQ=\{q_{b,k}: b\in\diag_2\{-1,1\},\ k\in\bZ^2\}
\]
and the sign
\[
a_{q_{b,k}}=\det b\in\{-1,1\}.
\]

\begin{definition}[Odd reflection symmetry]
Let $E$ be a vector space. We say that a function $f:\bR^2\to E$ is \emph{odd-reflected}  if
\begin{equation}\label{eq:odd_reflection_symmetry}
f(q(x)) = a_q f(x),
\qquad \text{for all } q\in\cQ,\ x\in\bR^2 .
\end{equation}
\end{definition}

Taking $q_{b,k}$ in \eqref{eq:odd_reflection_symmetry}  with $b=\id$ 
shows immediately that an odd-reflected function is $2$-periodic:
\[
f(x+2k)=f(x),\qquad k\in\bZ^2.
\]
On the other hand, the reflections $q_{b,k}$ with $\det b=-1$ are precisely the reflections across the lines of the lattice induced by $\bZ^2$.
In particular, any odd-reflected $f$ is odd about these lines and is thus determined by its restriction to the fundamental domain $[0,1)^2$.
We also record the following immediate consequence.

\begin{lemma}[Zero Dirichlet on $\partial\Lambda$]\label{lem:odd_reflection_vanishes}
Any odd-reflected function $f:\bR^2\to E$ vanishes on the boundary $\partial\Lambda$, i.e.~$f$ satisfies the zero Dirichlet boundary condition. 
\end{lemma}
For a compactly supported function on $\bR^2$, one can build an odd-reflected extension as follows.

\begin{definition}[Odd-reflected extension]
Let $f\colon \R^2\to E$ have compact support. Define the function $\CO (f)\colon\R^2\to E$ by
\begin{equation}\label{eq:odd_extension_def}
\CO (f)(x)
=
\sum_{q\in\cQ} a_q f(q(x))
=
\sum_{k\in\bZ^2}\ \sum_{b\in\diag_2\{-1,1\}} (\det b)\, f(q_{b,k}(x)).
\end{equation}
\end{definition}

Observe that, 
since $f$ has compact support,
there exists $M>0$, depending only of $f$,
such that for any $x\in\R^2$, there are at most $M$
pairs $(b,k)$ for which $f(q_{b,k}(x))\neq 0$.
In particular, $\CO(f)$ is well-defined.
Moreover, $\CO (f)$ is odd-reflected in the sense of \eqref{eq:odd_reflection_symmetry},
which follows by a change of variables in the summation index, using that
$q_{b_1,k_1}\circ q_{b_2,k_2}=q_{b_1b_2,\ k_1+b_1k_2}$, $\det(b_1b_2)=\det b_1\,\det b_2$,
and $\det b_i \in\{-1,1\}$.

\subsection{Green function decomposition and further properties}
Let $G^\Free(x) = \frac{1}{2\pi}\log|x|$ be the free space Green function of the Laplacian $\Delta$ on $\R^2$.
It is natural to enforce zero Dirichlet boundary conditions by considering
\[
x\mapsto \sum_{q\in\cQ} a_q \int_{\R^2} G^\Free(q(x)-y)f(y)\dif y\;,
\]
which formally corresponds to the representation $G^{\Dir}(x,y)\,``="\sum_{q\in\cQ} a_q\, G^\Free(q(x)-y)$,
but this is not well-defined since the sum is not convergent.

We use the decomposition $G^\Free=K+R$ from \cite[Lem~5.5, Rem~5.6]{Hairer14} as introduced already in~Definition~\ref{def:model} where $K=G^\Free$ on $[-4,4]^2$ and $K$ is supported in $[-5,5]^2$.
We define
\[
K^{\Dir}(x,y)=\sum_{q\in\cQ} a_q K(q(x)-y).
\]
For $x,y\in\Lambdaex$ the sum restricts to a finite set $\cQ'\subset\cQ$ due to the compact support of $K$. Concretely, we set
\begin{align}\label{eq:_def_Q'_dir}
\cQ'=\{q\in\cQ\colon |q(x)-y|\leq 5 \text{ for all } x,y\in\Lambdaex\},
\end{align}
so that for $x,y\in\Lambdaex$ we have
\begin{equation}\label{eq:KD_Q_finite}
K^{\Dir}(x,y)=\sum_{q\in\cQ'} a_q K(q(x)-y).
\end{equation}
We then define $R^\Dir(x,y)$ for $x\in\Lambda$ and $y\in\R^2$ by
the Dirichlet problem
\begin{align}\label{eq:def_R^D}
\Delta_x R^{\Dir}(x,y)=\sum_{q\in\cQ'} a_q\Delta_x R(q(x)-y)
\;,
\qquad
R^\Dir(\cdot,y)\restr_{\partial\Lambda}=0
\;.
\end{align}
Note that $\Delta_x R(q(x)-y) = (\Delta R)(q(x)-y)$
and since $R=G^\Free$ outside of $[-5,5]^2$ and $\Delta G^\Free = 0$ outside the origin, we have that $\Delta R$ is a smooth function with compact support.
Hence, $R^\Dir$ is well-defined
and $R^\Dir(x,y)=0$ whenever $x\in\Lambda$ and $|y|$ is sufficiently large.

We then define for $x\in\Lambda$ and $y\in\R^2$
\begin{align}\label{eq:G^D_decomp}
G^{\Dir}(x,y)=K^{\Dir}(x,y)+R^{\Dir}(x,y)
\;.
\end{align}

\begin{lemma}\label{lem:regularity_KDir}
Let $\beta>-2$ and $f\in C^\beta(\R^2)$
with support in $\Lambdaex$.
Then, for $x\in\Lambda$,
\[
K^\Dir*f(x):=\av{K^\Dir(x,\Cdot),f}
\]
is well-defined and one has 
\begin{equation}\label{eq:KD_Schauder}
|K^\Dir *f|_{C^{\beta+2}(\Lambda)}\lesssim |f|_{C^{\beta}(\R^2)}
\;. 
\end{equation}
Moreover, $K^\Dir*f|_{\partial\Lambda}=0$, i.e.~it satisfies the Dirichlet boundary condition. 
\end{lemma}

\begin{proof}
Consider first $f\in C^\infty_c(\R^2)$.
Then $K*f$ has compact support and,
as a simple special case of \cite[Thm.~5.12]{Hairer14},
\[
|K*f|_{C^{\beta+2}(\R^2)}\lesssim |f|_{C^{\beta}(\R^2)}
\;.
\]
The bound \eqref{eq:KD_Schauder} readily follows
by the fact that the sum in \eqref{eq:KD_Q_finite} is finite.
Moreover, $K^\Dir*f$ is the odd-reflected extension of $K*f$, i.e. 
$K^\Dir *f=\CO(K*f)$,
and thus $\CO(K*f)$ vanishes on $\partial\Lambda$ by Lemma~\ref{lem:odd_reflection_vanishes}.
The result for non-smooth $f$ follows by density and continuity.
\end{proof}

\begin{lemma}\label{lem:remainder_Green_regularity_Dir}
Let $\beta\in (0,2)$. Then the map $x\mapsto R^\Dir(x,\Cdot)$ is in $C^\beta(\Lambda;C^\infty_c(\R^2))$ and vanishes on $\partial\Lambda$.
In particular, for any distribution $\zeta\in \cD'(\R^2)$,
the convolution
\[
R^{\Dir}*\zeta(x)=\av{R^{\Dir}(x,\Cdot),\zeta}
\]
is well-defined and defines a function in $C^\beta(\Lambda)$ that vanishes on $\partial\Lambda$.
\end{lemma}

\begin{proof}
Consider first $f\in C^\infty(\Lambda\times \R^2)$ with bounded support and, for fixed $y\in \R^2$, the equation
\begin{equation}\label{eq:u_f}
\Delta u(\Cdot,y)=f(\Cdot,y)\ \  \ \text{ on }\Lambda^\circ\;,
\qquad
u(\Cdot,y) \restr_{\partial\Lambda}=0
\;.
\end{equation}
We claim that the map $x\mapsto u(x,\Cdot)$ is in
$C^\beta(\Lambda;C^\infty_c(\R^2))$ for any $\beta\in (0,2)$ and vanishes on $\partial\Lambda$.
Indeed, the function $y\mapsto f(\Cdot,y)$ is in $C^\infty_c(\R^2,L^p(\Lambda))$ for any $p\geq1$.
By Calderon--Zygmund estimates,
there exits a unique solution $u(\Cdot,y) \in W^{2,p}(\Lambda)$ to \eqref{eq:u_f}.
Taking $p$ sufficiently large, we have $W^{2,p}(\Lambda)\hookrightarrow C^\beta(\Lambda)$ by Sobolev embedding.
The map $L^p(\Lambda)\ni f(\Cdot,y) \mapsto u(\Cdot,y)\in C^\beta(\Lambda)$ is then a continuous linear function, and thus $y\mapsto u(\Cdot,y)$ is in $C^\infty_c(\R^2;C^\beta(\Lambda))$, as claimed.
The proof of the lemma now follows from taking $f(x,y)=\sum_{q\in\cQ'} a_q\Delta_x R(q(x)-y)$.
\end{proof}

Recalling $G^\Dir=K^\Dir+R^\Dir$ from \eqref{eq:G^D_decomp}, a consequence of Lemmas \ref{lem:regularity_KDir} and \ref{lem:remainder_Green_regularity_Dir} is that
\[
G^\Dir * f(x) = \av{G^\Dir(x,\Cdot),f}
\]
is well-defined for $x\in\Lambda$ and $f\in C^\beta(\R^2)$ with support in $\Lambdaex$
(in fact, for any compactly supported $f$),
where $\beta\in (-2,0)$.
The following result shows that $G^\Dir$ is indeed the Green function for the Laplacian with zero Dirichlet boundary conditons.

\begin{lemma}\label{lem:Schauder_Dir}
Let $\beta\in (-2,0)$ and assume $f\in C^\beta(\R^2)$ with support in $\Lambda$. Then $u:=G^\Dir*f$ is the unique weak solution to 
    \begin{align*}
        \begin{cases}
        \begin{aligned}
            \Delta u &= f  \quad &&\text{on } \Lambda^\circ\;, \\
            u &= 0 \quad &&\text{on } \partial\Lambda\;.
        \end{aligned}
    \end{cases}
    \end{align*}
Moreover, $G^\Dir*f\in C^{\beta+2}(\Lambda)$ and $
|G^\Dir*f|_{C^{\beta+2}(\Lambda)}\lesssim |f|_{C^{\beta}(\R^2)}$.
\end{lemma}

\begin{proof}
The regularity and boundary conditions follow from Lemma~\ref{lem:regularity_KDir} and Lemma~\ref{lem:remainder_Green_regularity_Dir}.
To show that $u$ solves the equation, it suffices to consider $f\in C^\infty_c(\R^2)$ with support in $\Lambda$ as the general case follows by approximation.
While the statement follows from classical ideas, we provide a proof for completeness.
Let $\varphi\in C^1(\Lambda^\circ)\cap C(\Lambda)$ with $\varphi|_{\partial\Lambda}=0$.
To show that $u$ is a weak solution, we are required to prove that $\av{\nabla (G^{\Dir}*f),\nabla \varphi}=-\av{f,\varphi}$.
For any $q\in\cQ$,
\[
\av{\nabla ((G^\Free*f)\circ q),\nabla \varphi}=-\av{\Delta ((G^\Free*f)\circ q),\varphi}\;.
\]
Using $\Delta((G^\Free*f)\circ q)=(\Delta(G^\Free*f))\circ q=f\circ q$ and summing over $q\in\cQ'$ with weights $a_q$ yields
\begin{align}\label{eq:proof_G^D_decomp_1}
\sum_{q\in\cQ'}a_q\,\av{\nabla ((G^\Free*f)\circ q),\nabla \varphi}
=-\sum_{q\in\cQ'}a_q\,\av{f\circ q,\varphi}
\;.
\end{align}
Since $f$ is supported in $\Lambda$, the Lebesgue measure of $\supp(f\circ q)\cap \supp(\varphi)$ is zero unless $q(x)=x$, 
in which case $a_q=1$; hence the right-hand side equals $-\av{f,\varphi}$.

Splitting $G^\Free=K+R$ in \eqref{eq:proof_G^D_decomp_1} and using the definition of $K^{\Dir}$ gives
\begin{align}\label{eq:proof_G^D_decomp_1b}
\av{\nabla K^{\Dir}*f,\nabla \varphi}
+\sum_{q\in\cQ'}a_q\,\av{\nabla ((R*f)\circ q),\nabla \varphi}
=-\av{f,\varphi}.
\end{align}

On the other hand, the defining equation \eqref{eq:def_R^D} for $R^{\Dir}$ implies
\[
\Delta (R^{\Dir}*f)=\sum_{q\in\cQ'}a_q\,\Delta ((R*f)\circ q)
\qquad\text{on }\Lambda^\circ.
\]
Integration by parts yields
\begin{align}\label{eq:proof_G^D_decomp_2}
\av{\nabla R^{\Dir}*f,\nabla \varphi}
&=-\av{\Delta(R^{\Dir}*f),\varphi}
=-\sum_{q\in\cQ'}a_q\,\av{\Delta((R*f)\circ q),\varphi}\\
&=\sum_{q\in\cQ'}a_q\,\av{\nabla ((R*f)\circ q),\nabla \varphi}.\notag
\end{align}
Combining \eqref{eq:proof_G^D_decomp_1b} and \eqref{eq:proof_G^D_decomp_2} yields
\[
\av{\nabla (K^{\Dir}*f),\nabla \varphi}+\av{\nabla (R^{\Dir}*f),\nabla \varphi}=-\av{f,\varphi},
\]
i.e.\ $\av{\nabla (G^{\Dir}*f),\nabla \varphi}=-\av{f,\varphi}$ as required.
\end{proof}

The following is an immediate corollary of Lemma \ref{lem:Schauder_Dir}.

\begin{corollary}\label{cor:GD_inverse}
Let $\beta > 0$ and $u\in C^\beta(\R^2)$ with support on $\Lambda$ and $u\restr_{\partial \Lambda}=0$.
Then $u = G^\Dir * \Delta u$.
\end{corollary}

\subsection{Abstract integration operators}\label{sec:integration_operator}

Throughout this section we work with the Dirichlet Green function $G^{\Dir}$ constructed above. We now write
\begin{equation}\label{eq:GD_decomp}
G^{\Dir}= K+Z^\Dir+R^\Dir,
\end{equation}
where $K$ is as earlier and $Z^\Dir=K^\Dir-K$.

\subsubsection{Integration against the singular kernel}\label{sec:singular_kernel}
 
Throughout this section we fix $i=1,2$ and a model as in Definition \ref{def:model}.
Recall from Section \ref{sec:modelled_distributions} that modelled distributions are functions $\blue f\colon \Lambda^\circ\to T_\sol$.
We assume in this section that all modelled distributions  take values in the sector $T[\cF_i \cup\{\1\}]$;
recall that the abstract integration operator $\CI_i$ is defined on $T[\cF_i \cup\{\1\}]$ by Definition \ref{def:integration} and by setting $\cI_i\blue\1 = 0$ as in Proposition \ref{prop:integration}.
We now recall the following fundamental result on the lift of the $1$-regularising integration operator $f\mapsto \partial_i K*f$ to modelled distributions from~\cite[Sec.~5-6]{Hairer14}.
\begin{theorem}\label{thm:singular_kernel_composed_der}
Let $-1<\alpha\wedge\eta\leq \gamma$ and $\gamma>0$. There exists a bounded linear map $\cK_{\gamma}^{i}\colon \scD^{\gamma,\eta}_\alpha\to \scD^{\gamma+1,\eta\wedge\alpha+1}$
of the form $\cK^i_\gamma\blue f (x)= \cI_i \blue f(x) + \blue\cQ(x)$, where $\blue\cQ$ is polynomial-valued, and such that
\[
\cR \cK_{\gamma}^{i}\blue f (x) = (K*\partial_i\cR \blue f)(x)
\;,
\]
for all $x\in\Lambda^\circ$,
where $\CR\colon  \scD^{\gamma,\eta}_\alpha \to C^{\eta\wedge\alpha}(\R^2)$ is the reconstruction operator from \eqref{eq:CR_defining}.
Moreover, the map $\cK_\gamma^i$ is locally Lipschitz continuous with respect to the models and modelled distributions.  
\end{theorem}

\begin{proof}
Observe that $\partial_iK*\cR \blue f = K*\partial_i\cR \blue f$ by translation-invariance of $K$.
Moreover $\partial_i\cR \blue f\in C^{\beta}(\R^2)$ with $\beta=\alpha\wedge\eta-1>-2$,
thus $K*\partial_i\cR \blue f$ is $(\beta+2)$-H\"older continuous and $(K*\partial_i\cR \blue f)(x)$ makes sense for all $x\in\R^2$.
The existence of $\CK^i_\gamma$ with the claimed properties now follows from~\cite[Prop.~6.16]{Hairer14}
applied to the modelled distribution $\blue f$ extended to $\R^2\setminus\Lambda$ by zero.
\end{proof}
\subsubsection{Integration against kernels with singularity at the boundary}\label{sec:remainder_kernel}
Recall that $Z^\Dir=K^\Dir-K$ which we then can write
\[
Z^\Dir (x,y)=\sum_{q\in\cQ'\setminus\{\id\}} K(q(x)-y),
\]
where $\cQ'\subset\cQ$ is the finite set from \eqref{eq:_def_Q'_dir}.
For fixed $x\in\Lambda^\circ$, $Z^\Dir(x,\Cdot)$ and $R^\Dir(x,\Cdot)$ are smooth and compactly supported functions.
For $\gamma\in(0,1)$ and $x\in\Lambda^\circ$, we then define $\cZ^{i}_{\gamma}$ on $\cD'(\R^2)$ by
\begin{equation}\label{eq:def_Zi_gamma}
\cZ^{i}_{\gamma}\zeta(x)
=\rmR_\Lambda(x)\sum_{|k|<\gamma+1}\frac{\blue{\rmX^k}}{k!}\,\partial_x^k\av{\partial_i\zeta,Z^\Dir(x,\Cdot)+R^\Dir(x,\Cdot)},
\end{equation}
where $\partial_x^k=(\partial_{x_1})^{k_1}(\partial_{x_2})^{k_2}$.
We have the following result regarding the regularity of this operator.

\begin{theorem}\label{thm:integration_boundary_derivation}
Let $\alpha < 0$ be non-integer. Then $\cZ^i_{\gamma}$ is a continuous linear map from $C^\alpha(\R^2)$  to $\scD^{\gamma+1,\alpha+1}$, where
$\alpha\leq \gamma < 1$.
Moreover, for $x\in\Lambda^\circ$,
\[
\cR\cZ^{i}_{\gamma}\zeta(x) = (Z^\Dir+R^\Dir)*\partial_i\zeta(x).
\]
\end{theorem}

\begin{proof}
We decompose $\cZ^i_\gamma=\tilde\cZ^i_\gamma+\hat\cZ^i_\gamma$ according to $Z^\Dir$ and $R^\Dir$:
\[
\tilde\cZ^i_\gamma\zeta(x)
:=\sum_{|k|<\gamma+1}\frac{\blue{\rmX^k}}{k!}\,\partial_x^k\av{\partial_i\zeta,Z^\Dir(x,\Cdot)},
\qquad
\hat\cZ^i_\gamma\zeta(x)
:=\sum_{|k|<\gamma+1}\frac{\blue{\rmX^k}}{k!}\,\partial_x^k\av{\partial_i\zeta,R^\Dir(x,\Cdot)}.
\]
For $\tilde\cZ^i_\gamma$, we use the decomposition of $Z^\Dir$ to conclude that it satisfies the assumptions of \cite[Lem.~4.16]{GH19}, yielding the stated mapping properties and the reconstruction identity for the $Z^\Dir$-part.

For $\hat\cZ^i_\gamma$, we use that $R^\Dir(\Cdot,\Cdot)$ is genuinely smooth in the second variable and
$C^{\gamma+1}$ in the first one by Lemma~\ref{lem:remainder_Green_regularity_Dir}. Since we truncate at order $\gamma+1<2$, the Taylor-type expression in
\eqref{eq:def_Zi_gamma} is well-defined and yields a modelled distribution in $\scD^{\gamma+1,\gamma+1}$
which is continous with respect to $\zeta$.
This also gives
$\cR\hat\cZ^i_\gamma\zeta(x)=\av{\partial_i\zeta,R^\Dir(x,\Cdot)}=R^\Dir*\partial_i\zeta(x)$.

Combining the two parts yields $\cR\cZ^i_\gamma\zeta(x)=(Z^\Dir+R^\Dir)*\partial_i\zeta(x)$ for $x\in\Lambda^\circ$.
\end{proof}

    \section{Symbolic index}

    In this appendix, we compile frequently utilised symbols from the article along with their corresponding meanings and the page where they initially appear. 
    
\begin{center}
\renewcommand{\arraystretch}{1.1}
\begin{longtable}{lll}
\toprule
Symbol & Meaning & Page\\
\midrule
\endfirsthead
\toprule
Symbol & Meaning & Page\\
\midrule
\endhead
\bottomrule
\endfoot
\bottomrule
\endlastfoot


$\triple{\Cdot}_{\alpha\ax}$ 
& Rough additive function metric in axial gauge 
& \pageref{eq:alpha_ax_norm} \\

$\triple{\Cdot}_{\alpha\gr}$ 
& Rough additive function growth metric 
& \pageref{eq:RAF_gr_metric} \\

$|\Cdot|_{\beta\gr},\,|\Cdot|_\beta$
& Growth (resp.~full) additive function norms
& \pageref{symb:additive_norms} \\

$\|\Cdot\|_{\gamma,\eta}$
& Singular modelled distribution norm
& \pageref{symb:modelled_distribution_norms} \\


$\triple{\Cdot}_{\gamma;\fK}$
& Model norm on the compact set $\fK$
& \pageref{symb:model_norms} \\

$1_G$
& Identity element of $G$
& \pageref{symb:identity_G} \\

$\bone$
& Identity element of $\cG$, i.e.\
  $\blue\1\otimes 1_G\in T[\1]\otimes\rmM_\C(N)$
& \pageref{symb:identity_RS} \\

$\circ_\gamma$
& Truncated product on $T_\sol\otimes\rmM_\bC(N)$
& \pageref{symb:circ_gamma} \\

$\star$
& Product on the solution model space $T_\sol$
& \pageref{symb:star_product} \\

$\star_\gamma$
& Truncated product on $T_\sol$
& \pageref{symb:star_gamma} \\

$\simhor$
& Horizontal equivalence relation on $\cX_{\rmh}$
& \pageref{def:non_vert_lines} \\



$\Axial,\Caxial$
& Enhanced additive function space in axial
  (resp.~complete axial) gauge
& \pageref{def:axial_gauge} \\

$\fA_{\alpha\ax}$
& Space of rough additive functions in axial gauge with metric
  $|\Cdot|_{\alpha\ax}$
& \pageref{d:RAF} \\

$\fA_{\alpha\gr}$
& Space of rough additive functions with growth metric
  $|\Cdot|_{\alpha\gr}$
& \pageref{d:RAF} \\

$\Area(\ell,\bar\ell)$
& Euclidean area enclosed between line segments
& \pageref{symb:Area} \\






$\cD_i^\gamma$
& Truncated derivative operator on modelled distributions
& \pageref{symb:modelled_derivative} \\

$\rmD_i$
& Abstract derivation maps on $T[\cU_\alg]$
& \pageref{symb:abstract_derivatives} \\

$\scD^{\gamma,\eta}$
& Space of modelled distributions of regularity $(\gamma,\eta)$
& \pageref{symb:scD} \\






$f^\tau$
& Coefficient of the basis symbol $\blue\tau$ in a modelled distribution
  $\blue f$
& \pageref{not:f} \\


$G$
& Compact connected Lie group
& \pageref{symb:G_Lie_group} \\

$G^\Dir$
& Dirichlet Green function on $\Lambda$
& \pageref{symb:Green_Dir} \\

$G^\Free$
& $G^\Free(x)=\frac{1}{2\pi}\log|x|$, the Green function of $\Delta$
  on $\R^2$
& \pageref{symb:Green} \\


$\cG^i_\gamma$
& Abstract Dirichlet integration operator on modelled distributions
& \pageref{symb:Dirichlet_integration_operator} \\

$\scG$
& Lie group-valued modelled distributions with boundary
 $1_G$
& \pageref{symb:modelled_gauge_spaces} \\

$\fG^{2\alpha}_A$
& Controlled gauge transformations over $\bfA = (A,\bA) \in \fA_{\alpha\gr}$
& \pageref{def:gauge_transformation} \\


$\fg$
& Lie algebra $T_{1_G}G$ of $G$
& \pageref{symb:Lie_algebra} \\





$\cI$, $\cI_i$, $\cI_{ij}$
& Abstract integration operators in the regularity structure
& \pageref{not:unary_ops_symbols} \\



$K$, $K_i$, $K_{ij}$
& Truncation of the Green function $G^\Free$  and its derivatives
& \pageref{symb:K} \\


$L(V;W)$, $L(V)$
& Space of linear operators $V\to W$ (resp.\ $V\to V$)
& \pageref{symb:LVW} \\

$L_g$
& Left multiplication by $g$
& \pageref{symb:left_multiplication} \\




$\ell^{x;y}$
& Straight line from $x$ to $y$
& \pageref{symb:ell_xy} \\

$\ell_B, \ell_\bfA$
& Path (resp. rough path)
induced by $B\in \Omega^1_{\beta\gr}$
(resp.  $\bfA\in\bfOmega_{\alpha\gr}^1$)
& \pageref{l:raf_to_rp} \\



$\bfOmega_{\alpha\ax}^1$
& Closure of smooth $1$-forms  in the axial gauge in $\Axial$
& \pageref{def:RAF_ax} \\

$\bfOmega_{\alpha\gr}^1$
& Closure of smooth $1$-forms  in
  $\fA_{\alpha\gr}$
& \pageref{d:RAF} \\

$\Omega^1_\beta$
& Completion of smooth $1$-forms   under
  $|\Cdot|_\beta$
& \pageref{symb:Omega_beta} \\

$\Omega^1_{\beta\gr}$
& Completion of smooth $1$-forms  under $|\Cdot|_{\beta\gr}$
& \pageref{symb:Omega_gr} \\

$\Omega^1_{\infty\ax}$
& Smooth $1$-forms  in the axial gauge
  $(A_2\equiv0)$
& \pageref{def:axial_gauge} \\


$\rmP_{\cX_{\rmh\ax}}(\ell)$
& Horizontal projection of $\ell$ onto the $x_1$-axis
& \pageref{symb:horizontal_projection} \\


$\rmQ_\gamma, \rmQ_{<\gamma}$
& Projection onto homogeneity exactly (resp strictly lower than) $\gamma$ in $T_\sol$
& \pageref{symb:Q_gamma} \\



$\cR$
& Reconstruction map associated to a model
& \pageref{symb:cR} \\





$|\tau|$
& Homogeneity of the formal symbol $\tau$
& \pageref{symb:homogeneity} \\

$T[\tau]$, $T[S]$
& Model space associated with a symbol $\tau$ or a collection of symbols $S$
& \pageref{sec:construction_subspace} \\

$\cT_\sol$
& Solution regularity structure $(T_\sol,G_\sol)$
& \pageref{symb:solution_regular_structure} \\


$U^\circ$
& Interior of $U$
& \pageref{symb:interior} \\

$U^{+\varepsilon}$
& $\varepsilon$-fattening of $U$
& \pageref{symb:fattening} \\

$U\Subset V$
& Compact inclusion $\overline U\subset V^\circ$
& \pageref{symb:Subset} \\





$\cX$
& Set of line segments $(x,v)$ in $\Lambda$ with
  $|v|\leq\frac14$
& \pageref{symb:cX} \\

$\cX_{\rmh}$
& Set of non-vertical line segments
& \pageref{def:non_vert_lines} \\




$\sfZ=(\Pi,\Gamma)$
& Model, with model maps $\Pi_x$ and recentering maps $\Gamma_{xy}$
& \pageref{symb:model} \\

\end{longtable}
\end{center}

\endappendix

\subsection*{Acknowledgements}
I.C. acknowledges support from the European Research Council (ERC) via the Starting Grant SQGT 101116964.
T.K. gratefully acknowledges funding through a UKRI Horizon Europe Guarantee MSCA Postdoctoral Fellowship (UKRI, SPDEQFT, grant reference EP/Y028090/1). Views and opinions expressed are however those of the authors only and do not necessarily reflect those of UKRI. 
In particular, UKRI cannot be held responsible for them. 
He also thanks Giuseppe Cannizzaro for the support during his postdoctoral position at the University of Warwick, when parts of this work were completed.
A.M. gratefully acknowledges the support of the Edinburgh Doctoral College Scholarship (EDCS), under which this work was carried out as part of his PhD studies. Moreover, he thanks the University of Vienna and TU Berlin for hosting him during the 2024/2025 academic year, where part of this work was carried out. Finally, he thanks Sorbonne University for hosting his visit with Thierry L\'evy and Elias Nohra, during which parts of this work were discussed.

\medskip

\noindent
For the purpose of open access, the authors have applied a CC BY public copyright licence to any author accepted manuscript arising from this submission.

     \bibliographystyle{Martin}
        
    \bibliography{refs}

\end{document}